\documentclass{memo-l}

\usepackage{titletoc}

\usepackage{imakeidx}
\makeindex

\usepackage{tikz}
\usepackage[colorlinks=true, linkcolor=blue, anchorcolor=blue, citecolor=blue, filecolor=blue, menucolor= blue, urlcolor=blue,pdfencoding=auto, psdextra]{hyperref}
\usepackage{bookmark}
\usepackage{amsmath,amssymb,amsthm,amscd}
\usepackage{verbatim}
\usepackage{comment}
\usepackage{multirow}
\usepackage{mathtools}
\usepackage{enumitem}
\usepackage{pgf,tikz,mathrsfs}
\usetikzlibrary{arrows}
\usepackage[normalem]{ulem}
\usepackage{comment}
\usepackage{multicol}
\usepackage{xfrac,xcolor}
\usepackage[margin=1in]{geometry}
\usepackage{graphicx}
\numberwithin{equation}{section}
\usepackage{pstricks-add,pst-plot,pst-func}

\newtheorem{theorem}{Theorem}[chapter]
\newtheorem{proposition}[theorem]{Proposition}
\newtheorem{lemma}[theorem]{Lemma}
\newtheorem{corollary}[theorem]{Corollary}

\newtheorem{prob*}{Problem}
\newtheorem{question}[theorem]{Question}
\newtheorem{conjecture}[theorem]{Conjecture}
\newtheorem*{theorem*}{Theorem}

\theoremstyle{definition}
\newtheorem{definition}[theorem]{Definition}
\newtheorem{notation}[theorem]{Notation}
\newtheorem{example}[theorem]{Example}
\newtheorem{remark}[theorem]{Remark}

\numberwithin{section}{chapter}
\numberwithin{equation}{chapter}

\definecolor{green}{rgb}{.1,.75,.1}

\newcounter{step}

\usepackage{xparse}
\usepackage{tikz}
\usetikzlibrary{matrix,backgrounds}
\pgfdeclarelayer{myback}
\pgfsetlayers{myback,background,main}

\tikzset{mycolor/.style = {line width=1bp,color=#1}}%
\tikzset{myfillcolor/.style = {draw,fill=#1}}%

\NewDocumentCommand{\highlight}{O{blue!40} m m}{%
	\draw[mycolor=#1] (#2.north west)rectangle (#3.south east);
}

\NewDocumentCommand{\fhighlight}{O{blue!40} m m}{%
	\draw[myfillcolor=#1] (#2.north west)rectangle (#3.south east);
}

\newcommand{\field}{ \ensuremath{\mathbb{C}}}

\newcommand{\CC}{ \ensuremath{\mathbb{C}}}

\newcommand{\PP}{ \ensuremath{\mathbb{P}}}
\newcommand{\calP}{ \ensuremath{\mathcal{P}}}

\newcommand{\mI}{\mathcal{I}}

\newcommand{\mO}{\mathcal{O}}
\newcommand{\mP}{\mathcal{P}}

\newcommand{\calq}{\mathcal{Q}}
\newcommand{\calc}{\mathcal{C}}

\DeclareMathOperator{\coker}{coker}
\DeclareMathOperator{\Aut}{Aut}

\DeclareMathOperator{\sgn}{sgn}

\DeclareMathOperator{\adim}{adim}
\DeclareMathOperator{\vdim}{vdim}

\newcommand{\rank}{\ensuremath{\mathrm{rank}}\hspace{1pt}}
\newcommand{\reg}{\ensuremath{\mathrm{reg}}\hspace{1pt}}

\definecolor{MyDarkGreen}{cmyk}{0.7,0,1,0}

\def\cocoa{{\hbox{\rm C\kern-.13em o\kern-.07em C\kern-.13em o\kern-.15em A}}}

\usepackage{etoolbox}
\patchcmd{\subsection}{-.5em}{.5em}{}{}
\patchcmd{\subsection}{2}{3}{}{}
\makeatletter
\newsavebox\myboxA
\newsavebox\myboxB
\newlength\mylenA

\newcommand*\xoverline[2][0.75]{%
    \sbox{\myboxA}{$\m@th#2$}%
    \setbox\myboxB\null
    \ht\myboxB=\ht\myboxA%
    \dp\myboxB=\dp\myboxA%
    \wd\myboxB=#1\wd\myboxA
    \sbox\myboxB{$\m@th\overline{\copy\myboxB}$}
    \setlength\mylenA{\the\wd\myboxA}
    \addtolength\mylenA{-\the\wd\myboxB}%
    \ifdim\wd\myboxB<\wd\myboxA%
       \rlap{\hskip 0.5\mylenA\usebox\myboxB}{\usebox\myboxA}%
    \else
        \hskip -0.5\mylenA\rlap{\usebox\myboxA}{\hskip 0.5\mylenA\usebox\myboxB}%
    \fi}
\makeatother

\begin{document}
\frontmatter
\title{Configurations of points in projective space and their projections}

\author[L.~Chiantini]{Luca Chiantini}
\address[L.~Chiantini]{Dipartimento di Ingegneria dell'Informazione e Scienze Matematiche, Universit\`a di Siena, Italy}
\email{luca.chiantini@unisi.it}

\author[{\L}.~Farnik]{{\L}ucja Farnik}
\address[{\L}.~Farnik]{Department of Mathematics, Pedagogical University of Cracow,
   Podcho\-r\c a\.zych~2,
   PL-30-084 Krak\'ow, Poland}
\email{lucja.farnik@gmail.com}

\author[G.~Favacchio]{Giuseppe Favacchio}
\address[G.~Favacchio]{Dipartimento di Ingegneria, Universit\`a degli studi di Palermo,
Viale delle Scienze,  90128 Palermo, Italy}
\email{giuseppe.favacchio@unipa.it}

\author[B.~Harbourne]{Brian Harbourne}
\address[B.~Harbourne]{Department of Mathematics,
University of Nebraska,
Lincoln, NE 68588-0130 USA}
\email{brianharbourne@unl.edu}

\author[J.~Migliore]{Juan Migliore} 
\address[J.~Migliore]{Department of Mathematics,
University of Notre Dame,
Notre Dame, IN 46556 USA}
\email{migliore.1@nd.edu}

\author[T.~Szemberg]{Tomasz Szemberg}
\address[T.~Szemberg]{Department of Mathematics, Pedagogical University of Cracow,
   Podcho\-r\c a\.zych~2,
   PL-30-084 Krak\'ow, Poland}
\email{tomasz.szemberg@gmail.com}

\author[J.~Szpond]{Justyna Szpond}
\address[J.~Szpond]{Institute of Mathematics, Polish Academy of Sciences, \'Sniadeckich~8,
   PL-00-656 Warsaw, Poland}
\email{szpond@gmail.com}


\thanks{Date of Draft: September 11, 2022}
\thanks{Chiantini and Favacchio are members of the Italian GNSAGA-INDAM}
\thanks{Farnik was partially supported by National Science Centre, Poland, grant 2018/28/C/ST1/00339.
}
\thanks{Favacchio was partially supported by Fondo di
Finanziamento
per la Ricerca di
Ateneo, Università degli studi di Palermo, FFR$\_$2021$\_$D26.}
\thanks{Harbourne was partially supported by Simons Foundation grant \#524858.}
\thanks{Migliore was partially supported by Simons Foundation grant \#309556.
}
\thanks{Szemberg and Szpond were partially supported by National Science Centre, Poland, grant 2019/35/B/ST1/00723.}

\thanks{We thank:
CIRM for its funding the Levico Terme June 2018 conference
``Lefschetz Properties and Jordan Type in Algebra, Geometry and Combinatorics", where we initiated our work on geproci sets;
MFO for hosting the September 2020 ``Lefschetz Properties in Algebra, Geometry and Combinatorics" mini-workshop where work on this project continued; and the Banach Center of the Mathematics Institute of the Polish Academy of Sciences and the Pedagogical University of Cracow IDUB grant ESWG/2022/05/00021 for supporting a visit in June 2022 during which the authors worked in person on this book. And we acknowledge Karolina's presence for the last nine months of our meetings working on this book.}



\maketitle

\tableofcontents

\phantomsection
\listoffigures

\begin{abstract}
\phantomsection 
\addcontentsline{toc}{chapter}{Abstract}

We call a set of points $Z\subset\PP^{3}_{\field}$ an $(a,b)$-geproci set 
(for GEneral PROjection is a Complete Intersection)
if its projection from a general point $P$ to a plane is a complete intersection of curves of degrees $a$ and $b$. 
Examples which we call grids have been known since 2011.
The only nongrid nondegenerate examples previously known
 had $ab=12, 16, 20, 24, 30, 36, 42, 48, 54$ or $60$. Here, for any $4 \leq a \leq b$, we construct nongrid nondegenerate $(a,b)$-geproci sets in a systematic way. We also show that the only such
example with $a=3$ is a $(3,4)$-geproci set coming from the $D_4$ root system, and we describe the $D_4$ configuration in detail. 
We also consider the question of the equivalence (in various senses) of geproci sets, as well as which sets occur over the reals, and which cannot. We identify several additional examples of geproci sets with interesting properties. We also explore 
the relation between unexpected cones and geproci sets and introduce the notion of $d$-Weddle schemes arising from special projections of finite sets of points. 
This work initiates the exploration of new perspectives on classical areas of geometry. We formulate and discuss a range of open problems in the final chapter.


\vskip.5in
The POLITUS group, consisting of authors 
from Poland, Italy and the United States,
formed to produce this work.
\begin{center}
\includegraphics[scale =.1]{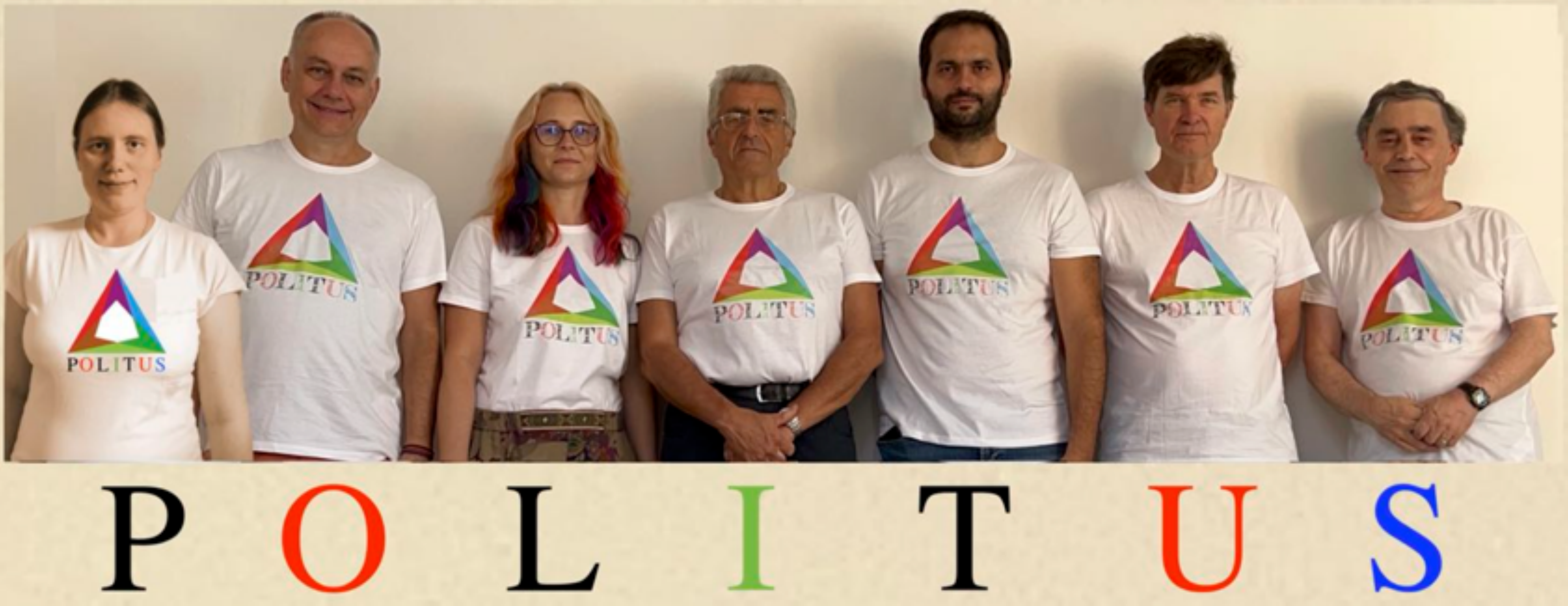}
\end{center}
\end{abstract}




\chapter*{Introduction}
\section{General Context}
We work over a ground field\index{ground field} $K$. Our standing assumption, unless explicitly stated otherwise, is that $K$ is the complex numbers.

The study of projective varieties admitting unexpected hypersurfaces 
has expanded rapidly, with significant connections to combinatorics, 
representation theory and commutative algebra. Introduced in 2018 in \cite{CHMN}, 
unexpectedness led to the newly
established topic of geproci subsets of projective space \cite{PSS}, which is itself
a special case of a range of problems in commutative algebra and algebraic 
geometry motivated by the inverse scattering problem. In inverse scattering
(whose study has resulted in advances such as computed tomography and
the determination of the double helix structure of DNA) one aims to
determine intrinsic structures from projected and reflected 
data. Algebraically, this translates to classifying varieties
whose projections have a given property $\calP$. For geproci sets, the question is:
if the projected image of Z from a general point to a hyperplane has the property $\calP$ of
being a complete intersection, what can we say about Z?

It has been a long standing area of research in algebraic geometry to study finite subsets $Z\subset\PP^n$ in projective space
satisfying given special properties $\mathcal P$. 
An important component of this work has involved studying the conditions
points impose on forms of given degree. 
For example, say $Z$ consists of distinct points $p_1,\ldots,p_s\in\PP^n$.
We have the ideal $I(p_i)\subset K[\PP^n]=K[x_0,\ldots,x_n]$ generated by all forms
vanishing at $p_i$ and the ideal $I(Z)=I(p_1)\cap\cdots\cap I(p_s)$ 
generated by the forms which vanish at all of the points $p_i$.
Its $m$-th symbolic power\index{symbolic power} is $(I(Z))^{(m)}=I(p_1)^m\cap\cdots\cap I(p_s)^m$.
This is a homogeneous ideal; understanding the dimension
$\dim [(I(Z))^{(m)}]_t$ of the homogeneous component in degree $t$
is a significant issue in algebraic geometry.

It is easy to see that
$$\dim [(I(Z))^{(m)}]_t\geq \binom{n+t}{n}-s\binom{n+m-1}{n};$$
i.e., vanishing to order at least $m$ at the points of $Z$ imposes at most 
$s\binom{n+m-1}{n}$ conditions on forms of degree $t$. When the points $p_i$ are general,
one would like to understand for which $m,s$ and $n$ we can have 
\begin{equation}\label{equation1}
\dim [(I(Z))^{(m)}]_t > \max\{0, \binom{n+t}{n}-s\binom{n+m-1}{n}\}.
\end{equation}
So in this case $\mathcal P$ would be the property that $Z$ consists of general
points for which display \eqref{equation1} holds.
(For $n=2$, this is a special case of a problem addressed by the well-known SHGH Conjecture\index{SHGH Conjecture}
\cite{Seg,HaVanc,Gim,Hirsch}.)
We note display \eqref{equation1} holds for $m=s=t=n=2$, for example, but no classification is known
for any $n>1$ of all $\{m,s,t\}$ for which it holds.
For no $n>1$ is it even known if the set of all $s$ for which there is such an $\{m,s,t\}$ is finite.

The above problem is quite old. A more recent question which has attracted a lot of attention is to classify
all $m,r$ and $Z\subset\PP^n$ with the property $\mathcal P$ that $(I(Z))^{(nm-n+1)}\not\subseteq (I(Z))^r$.
This is open even for $n=2$; in fact the first such example dates to 2013 \cite{DST}. 
Similarly, one can study $Z\subset\PP^n$ with the property $\mathcal P$ that all powers $I(Z)$ are symbolic;
i.e., that $(I(Z))^r=(I(Z))^{(r)}$ for all $r\geq1$. Any such $Z$ must be a complete intersection \cite[Theorem 2.3]{HKZ}, but
if $n>1$ and $Z$ is merely a 0-dimensional scheme
(or even a {\it fat point scheme}\index{points!fat}, meaning a scheme defined
by an ideal of the form $I(p_1)^{m_1}\cap\cdots\cap I(p_s)^{m_s}$), the phenomenon is not fully understood.

There are many more interesting possibilities for $\mathcal P$.
For example take $\mathcal P$ to be the property that $Z$ is arithmetically Gorenstein.
Or we could take $\mathcal P$ to be that $Z$ has an unexpected cone of some degree $t$,
in which case we will write $\mathcal P=C(t)$ and say $Z$ satisfies $C(t)$.
(One says that $Z$ has an {\it unexpected cone}\index{unexpected!cone} of degree $t$ if for a general point $P$ we have
$$\dim [I(Z)\cap (I(P))^t]_t > \max\{0,\dim [I(Z)]_t - \binom{n+t-1}{n}\};$$
i.e., there are more codimension 1 cones of degree $t$ containing $Z$ with vertex $P$ than one would
expect based on the number of conditions vanishing at $P$ to order $t$ imposes 
on forms of degree $t$ vanishing on $Z$. See \cite{HMNT}.)

More generally, one can study those $Z\subset\PP^{n+1}$ whose image $\overline{Z}\subset\PP^n$
under projection from a general point $P$ to a hyperplane $\PP^n\subset\PP^{n+1}$
satisfies $\mathcal P$. For such a $Z$, we say $Z$ is gepro-$\mathcal P$\index{gepro-$\mathcal P$}, meaning
the GEneral PROjection $\overline{Z}$ of $Z$ has property $\mathcal P$.

The main focus of this work is to do this in the case that $\mathcal P$ is the property of being a complete intersection,
in which case we say geproci (for GEneral PROjection is a Complete Intersection) instead of gepro-$\mathcal P$.
This turns out to be very interesting, with many problems still open.
In fact, nontrivial examples are known only for $Z\subset\PP^3$, and the first of them date back only to 2018.

\section{History}
Let $\PP^n=H'\subset\PP^{n+1}$ be a hyperplane. If $Z\subset H'$ is finite then projection from a general point $P\in\PP^{n+1}$
to another hyperplane $H$ will give an isomorphism $Z\to\overline{Z}\subset H$.
In particular, $Z$ is a finite complete intersection in $H'$ if and only if $\overline{Z}$ is a finite complete intersection in $H$. 
Hence $Z$ is geproci if and only if it is already a complete intersection. 
Thus what is of interest is when $Z$ is nondegenerate.

Nondegenerate examples are known only for $n=3$. These first arose in answer to a question posed 
by F.\ Polizzi\index{Polizzi} in a 2011 post on 
Math-Overflow \cite{P}; see also \cite[Question 15]{CDFPGR}.
Briefly, Polizzi asked if nondegenerate geproci sets exist in $\PP^3$.
In response, D.\ Panov\index{Panov} pointed out that what here we will call {\em grids}\index{grid}
give nondegenerate geproci examples.\index{geproci! nondegenerate}
More precisely, an {\em $(a,b)$-grid} \index{grid} is a set of $ab$ points of the form $Z=A\cap B$, where
$A$ is the union of $a\geq2$ pairwise disjoint lines, and $B$ is the union of $b\geq a$ pairwise disjoint
lines, such that each line of $A$ meets each line of $B$ in exactly one point 
(in particular, $A$  and $B$ are curves with no components in common).
We will refer to the lines in $A$ and $B$ as {\em grid lines}.\index{grid! lines}
Clearly such a $Z$ is $(a,b)$-geproci, meaning
its general projection $\overline{Z}$
is a complete intersection of plane curves of degrees $a\leq b$. 

It is not hard to construct an $(a,b)$-grid.
Any set $Z$ of $2b$ points with $b\geq 2$ points on each of two skew lines is a $(2,b)$-grid,
where the grid lines are the 2 skew lines, together with any choice of lines
connecting up each point of $Z$ on one of these two lines with distinct points of $Z$ on the other skew line.
When $a>2$, it is easy to see using B\'ezout's Theorem\index{Theorem! B\'ezout} 
that there is a unique smooth quadric $\calq$ containing
both $A$ and $B$, with the components of $A$ being lines in one of the rulings on $\calq$, and
the components of $B$ all being taken from the other ruling. (If $a=2$ and $b>3$,
there are always smooth quadrics containing the $a=2$ lines and hence the $ab$ points, but there need not be
any smooth quadric for which the 
$b$ grid lines are also all contained in the quadric.) It is also worth remarking that extending work of Diaz \cite{diaz} and of Giuffrida \cite{giuffrida}, it was shown in \cite[Proposition 3.1]{CM} that in $\PP^3$, among reduced curves with non-degenerate union, a curve $C_1$ of degree $a$ can be found meeting a curve $C_2$ of degree $b \geq a$ in $ab$ points if and only if either $a=1$ and $C_2$ is a disjoint union of plane curves, or $a \geq 2$ and $C_1$ and $C_2$ form a grid as described above. Thus grids are the only ``obvious” examples coming from the intersections of curves.

In response to Panov's grid examples, Polizzi edited his post to now ask
whether nondegenerate nongrid geproci\index{geproci! nongrid} sets $Z$ exist, and if so, can they be classified, at least 
when $|Z|$ is small. We will call such examples {\it nontrivial}\index{geproci! nontrivial}, so Polizzi's question was:
are there nontrivial geproci sets?

There matters remained until a conference on Lefschetz Properties in 2018 at Levico Terme\index{Levico Terme} where 
a working group noticed that results of \cite{HMNT} gave nontrivial geproci sets
(see the appendix to \cite{CM} for the participants
of the working group and for further discussion).
The topic of \cite{HMNT} was the notion of unexpectedness\index{unexpectedness},
first introduced for point sets in the plane in \cite{CHMN} and extended to higher dimensions in \cite{HMNT}.
In particular, \cite{HMNT} shows that when $Z$ is the $D_4$ configuration\index{configuration! $D_4$} (a nondegenerate nongrid set in $\PP^3$
coming from the $D_4$ root system having cardinality $c=12$, also known as the Reye configuration \cite{dolgachev2004})\index{configuration!Reye} or the 
$F_4$ configuration\index{configuration! $F_4$} (a nondegenerate nongrid set in $\PP^3$ 
coming from the $F_4$ root system having cardinality $c=24$), both
have unexpected cones with no components in common of degrees $a$ and $b$ with $ab=c$.
(See Definition \ref{def:D4} for the definition of the $D_4$ configuration
and Example \ref{e:F4exOf4.2b} for the $F_4$ configuration.)
The case $D_4$ has $a=3$ and $b=4$, while for $F_4$ one has $a=4$ and $b=6$.
This implies that in each case $Z$ is a nontrivial $(a,b)$-geproci set. Thus the appendix to \cite{CM} exhibits geproci sets with $12$ and $24$ elements and in addition identifies geproci subsets of $F_4$ with $16$ and $20$ elements. The next discovery came up in 2020 during one of the very few Research in Pairs groups held in person in the pandemic year. Pokora, Szemberg and Szpond showed in \cite{PSS} that points of the Klein configuration of $60$ points in $\PP^3$ form a $(6,10)$-geproci set (see Section \ref{sec: The Klein configuration}). In addition, they revealed  geproci subsets of the Klein configuration with $24, 30, 36, 42, 48$ and $54$ elements. These results were presented in the subsequent workshop on Lefschetz Properties, in Oberwolfach. This is where our collaboration began.

It is worthwhile to point out that another, very interesting, geproci set of $60$ points was discovered in several steps between 2018 and 2021.
The paper \cite{HMNT} showed that the $H_4$ configuration\index{configuration! $H_4$}, a nondegenerate nongrid set in $\PP^3$
coming from the $H_4$ root system having cardinality $c=60$, has an unexpected cone of degree 6.
It also has an unexpected cone of degree 10, and hence is $(6,10)$-geproci,
but the degree 10 cone was not found until \cite{FZ}.

In addition to announcing the existence of nontrivial geproci sets,
\cite{CM} also made a start on Polizzi's other question, by showing 
that any nontrivial $(a,b)$-geproci set with $ab<12$ is a grid.
Another result from \cite{CM} is that $(a,b)$-grids with $3\leq a\leq b$
satisfy both $C(a)$ and $C(b)$. But the results of \cite{CM} raised many questions,
some of which we  address here.

\section{Questions and Results}\label{seq: Q and R}
One question left open by \cite{CM} is whether there are, up to projective equivalence,
only finitely many nontrivial geproci sets.
We address that here first by showing any nontrivial $(3,b)$-geproci set is projectively equivalent
to the $D_4$ configuration of 12 points (see Theorem \ref{thm:classification_of_3xb}).
However we also show there is a nontrivial $(a,b)$-geproci set for every $4\leq a\leq b$
(see Theorem \ref{t. (a,b)-geproci}).

Putting Theorems \ref{t. (a,b)-geproci} and
\ref{thm:classification_of_3xb}
together with Proposition \ref{SmallProp} gives the following result:

\begin{theorem}
There is a nontrivial $(c,d)$-geproci set $Z\subset\PP^3$ with $c\leq d$ if and only if either $c=3$ and $d=4$, or $c\geq 4$
(and for any such $(c,d)$ we give a systematic way of producing a $(c,d)$-geproci set). 
\end{theorem}

Indeed, the proof of Theorem \ref{t. (a,b)-geproci} is constructive. Starting with a specific grid $Z'$, we employ a procedure
we call the {\it standard construction}\index{standard construction} in which we add a set of 
collinear points to $Z'$ to get a nongrid geproci $Z''$ (see Theorem \ref{t. geproci infinite class}), 
and then take away grid lines to get smaller sets that continue to be nontrivial geproci sets (see Theorem \ref{t. (a,b)-geproci}).
Both the $D_4$ and $F_4$ configurations arise by the standard construction starting with a $(3,3)$-grid
and a $(4,4)$-grid, respectively (see Examples \ref{e:D4exOf4.2a} and \ref{e:F4exOf4.2b}).
Indeed, the $D_4$ and $F_4$ configurations motivated us to devise the standard construction.

Most of the nontrivial geproci sets we know of are half grids. (We say a nondegenerate $(a,b)$-geproci set
$Z$ is a {\it half grid}\index{geproci!half grid} if it is not a grid but one of the curves making the projected image $\overline{Z}$
a complete intersection is a union of lines. Equivalently, for a nongrid $(a,b)$-geproci set $Z$, 
this means either there are $a$ disjoint lines each of which contains $b$ points of $Z$,
or there are $b$ disjoint lines each of which contains $a$ points of $Z$.)
We know of only three nontrivial geproci sets which are not half grids, namely
the $H_4$ configuration of 60 points \cite{FZ}, a $(5,8)$-geproci set $Z_P$ (see Section \ref{s.Penrose}) 
originally constructed by Penrose
with an application to quantum mechanics in mind, and 
a $(10,12)$-geproci set (see Remark \ref{rem:120nonhalfgrid}).
This raises the following still open questions: are there other nontrivial nonhalfgrid geproci sets?
Are there only finitely many up to projective equivalence?

Another major open question is: are there any geproci sets in $\PP^n$ for $n>3$?
In addition and more generally we can ask for sets in $\PP^n$ of codimension more than 3 whose 
projection from a general point is a complete intersection. Of course,
a finite geproci set $Z\subset\PP^3$ has codimension 3. If $S\in\PP^3$ is general, then the cone
$C_S(Z)$ over $Z$ with vertex $S$ is thus a complete intersection, so,
thinking of $\PP^3$ as a hyperplane in $\PP^4$, the cone $C_Q(Z)$ (where
$Q\in\PP^4$ is a general point) projects from a general point $P$ to a complete
intersection $C_S(Z)$. Thus we can consider $C_Q(Z)$  to be a geproci curve
in $\PP^4$, but it still has codimension 3.

Another open question is: do all nontrivial geproci sets contain a subset of 3 collinear points?
If the answer is no, can we classify those which do not?
Related to this we ask:

\begin{question}\label{q. LGP->4}
Let $X$ be a set of points in linear general position. 
If  $X$ is geproci, must $|X|=4$?
\end{question}

\begin{question}\label{q. quadric->grid}
Let $X$ be a set of points in a smooth quadric. 
If $X$ is geproci, must $X$ be a grid?
\end{question}

Another question: what nontrivial Gorenstein geproci sets are there?
(The only one we know of is $Z_P$.)
And yet another: by \cite{CM} we know all $(a,b)$-grids with $3\leq a\leq b$ satisfy both
$C(a)$ and $C(b)$. Here we show many nontrivial $(a,b)$-geproci sets
also satisfy both $C(a)$ and $C(b)$, but we do not have a proof
that all nontrivial $(a,b)$-geproci sets satisfy both $C(a)$ and $C(b)$
nor do we have any counterexamples (see Chapter~\ref{ch: unexp hypersurf}).
Thus we ask: does every nontrivial $(a,b)$-geproci set satisfy both $C(a)$ and $C(b)$?

\section{Summary}
We now summarize the contents of the chapters that follow.
In Chapter 1 we introduce some preliminaries that we will use in the rest of work.  The main focus is on cross ratio and harmonicity, which are applied in Chapter \ref{Chap.Weddle} in relation to special projections of sets of points and in Chapter \ref{chap. non iso and real} to develop a tool to identify geproci sets which cannot be realized over the reals.

In Chapter \ref{Chap.Weddle} we study Weddle varieties.
These are of interest in their own right, but we use them
in our classification of nontrivial $(3,b)$-geproci sets.
The classical case is that of the Weddle surface\index{Weddle! surface} 
\cite{WEDDLE}, a surface in $\PP^3$ of degree 4 that is the closure of the locus of points $P \in \PP^3$ 
with the property that the projection from $P$ of a fixed set of 6 general points in 
$\PP^3$ is contained in a conic. Additional work, for general sets of other cardinalities, 
was done by Emch\index{Emch} \cite{EMCH}. In this chapter we generalize these results, 
introducing the $d$-Weddle scheme\index{Weddle! scheme}. 
The main result we need later is Proposition \ref{seven pts}, but this
topic is of substantial intrinsic interest that we explore the notion 
of Weddle varieties in more depth than what we strictly need later.

One of the significant motivations for this work came from understanding
the $D_4$ configuration. It is so interesting 
and intricate that we record in Chapter \ref{Ch:D4} the facts that
we learned about $D_4$.
We also include a noncomputational proof that $D_4$ is $(3,4)$-geproci.

In Chapter \ref{chap.Geography} we prove our main results, classifying
nontrivial $(3,b)$-geproci sets, and proving the existence of nontrivial
$(a,b)$-geproci sets for each $4\leq a\leq b$.

In Chapter \ref{chap.extending} we generalize the standard construction
and study important special examples of geproci sets.

In Chapter \ref{chap. non iso and real} we consider the question of realizability over the reals.
We show, for example, that the geproci sets obtained by our standard construction are,
up to projective equivalence, all real, despite the fact that the complex numbers apparently have a central role in the construction.
We also define a notion of combinatorial equivalence.
Let us say that two finite subsets
$Z_1$ and $Z_2$ of $\PP^n$ {\em share the same
combinatorics} (or are {\it combinatorially equivalent}\index{combinatorially equivalent}) if, intuitively, 
they are bijective with the same linear dependencies; i.e., the same 
collinearities, coplanarities, etc. (see Definition \ref{DefCombEq}). 
We can then ask: when can a nongeproci set 
share the same combinatorics as a geproci set?
When must two geproci sets with the same combinatorics 
be projectively equivalent?
We do not know any example of a nongeproci set
which shares the same combinatorics as a 
nontrivial geproci set. Nor do we know any
example of two nontrivial geproci sets which
share the same combinatorics but are not projectively equivalent. 
However we do show that sets combinatorially equivalent to grids are themselves grids
(Proposition \ref{CombEqToGridIsGrid}).
We also apply this notion to 
the question: how many nontrivial $(a,b)$-geproci sets can
there be, up to projectively equivalence, for a given $(a,b)$? 
We do not know any cases where the answer is infinite, but we do
show there are at least two projectively nonequivalent nontrivial $(a,b)$-geproci
sets for each $4\leq a\leq b$ (Corollary \ref{p. two non-isomorphic}).

Chapter \ref{ch: unexp hypersurf} is mainly devoted to studying unexpectedness for geproci sets,
obtaining partial results extending the result of \cite{CM}, mentioned above, that grids
have unexpected cones. This chapter also explores some occurrences of unexpected cones
in higher dimensions in situations which are not geproci.

Chapter \ref{ch:FW} concerns questions for which we have partial results
but feel are worthy of continued attention.
For example, 
there are Cayley-Bacharach\index{Cayley-Bacharach Property} type questions. If $Z$ is $(a,b)$-geproci what is the smallest 
$r\leq ab$ such that no subset $S\subseteq Z$ of $r$ points is contained in a proper geproci subset of $Z$?
(Note by \cite{CM} that $F_4$ contains nontrivial geproci subsets of 12, 16 and 20 points, so here $r$ is
at least 21, and in fact $r=21$ in this case).
What kinds of bounds $B(a,b)$ are there in terms of $a$ and $b$ such that no subset of $r\geq B(a,b)$ points
is contained in a proper geproci subset of $Z$?
Alternatively, given a nontrivial nongrid $(a,b)$-geproci set $Z$, we can ask what is the smallest $r_Z$ such 
that there is a subset $S\subset Z$ of $r_Z$ points contained in no proper geproci subset of $Z$?
In addition, we look at the relationship between $D_4$ and $F_4$,
and we raise the question of gepro-$\mathcal P$ where $\mathcal P$ is the property of having a certain betti table.
We also make a start on the possible existence of geproci sets in higher dimensions.
We show there is no $(2,2,2)$-geproci set in $\PP^4$ (see Theorem \ref{no 222}).
This shows, in particular, that there is no obvious generalization of grids to $\PP^4$.
And we study the possibility of geproci curves in $\PP^4$.

In closing, we note that this work would not have been possible without the computational aid of \cite{CoCoA,DGPS,M2},
which we used to carry out many calculations and to run enlightening experiments.

\mainmatter

\chapter{Preliminaries}
\section{Harmonic points}\label{harmonic points}
In order to make our work more or less self-contained, in this chapter we collect some basic facts and properties from various branches of mathematics used in the core chapters. An important component of our work consists of finite sets of point with special properties, often closely related to finite sets of lines. To maintain the distinction between points and lines, we will
refer to a finite set of points as a {\it configuration} (or a configuration of points),
and we will refer to a finite set of lines 
as an {\it arrangement} (or an arrangement of lines).\index{configuration! of points}\index{arrangement! of lines}

\subsection{Harmonic property and conics}\index{points!harmonic}

\begin{definition}[Cross ratio] Recall that the {\it cross ratio}\index{cross ratio} of an ordered set of four distinct points $A,B,C,D\in\PP^1$ 
	 is
	\begin{equation}\label{eq. cross ratio} \frac {(a_1c_2-a_2c_1)(b_1d_2-b_2d_1)}{(a_1d_2-a_2d_1)(b_1c_2-b_2c_1)},\end{equation}
	where $A=[a_1:a_2],B=[b_1:b_2],C=[c_1:c_2],D=[d_1:d_2]$ with respect to some choice of coordinates on $\PP^1$ (by Remark \ref{r. cr and autom}, the cross ratio is independent of the choice of coordinates).
	Note that each factor is a maximal minor in the matrix
	$$ \begin{pmatrix} a_1 & b_1 & c_1 & d_1\\ 
	                   a_2 & b_2 & c_2 & d_2   \end{pmatrix}$$
	but only four of the six maximal minors are used.

	We say that the points are {\it harmonic} 
	for the given order if the cross ratio is $-1$,
	for example, $A=[0:1]$, $B=[1:0]$, $C=[1:1]$ and $D=[1:-1]$.
\end{definition}

 A set of $4$ distinct points in $\PP^1$ always has $4$ permutations which come from an automorphism of the line. Sets of harmonic points are exactly those for which the number of permutations coming from automorphisms of $\PP^1$ is the maximum of $8$. This is explained in detail in Remark \ref{harmperm}.

	The next proposition shows that, given three ordered points $A,B,C$  on a line, there exists a unique point $D$, 
	called the {\it projective harmonic conjugate},\index{projective harmonic conjugate} such that $A,B,C,D$ are harmonic.

\begin{proposition}[The fourth harmonic point]\label{1harm} For any ordered set $(A,B,C)$ there is a unique 
point $D$ such that the (ordered) set $(A,B,C,D)$ is harmonic.
\end{proposition}
\begin{proof} Since any $3$ points can be interchanged by an automorphism of $\PP^1$, which does not alter the cross ratio, we may assume
$A=[1:0],B=[0:1],C=[1:1]$. Furthermore, $D$ cannot coincide with $A$, so we may assume $D=[1:a]$ with some $a\neq 0$. An easy computation now proves the claim.
\end{proof}

\begin{remark}\label{GeomConstrOfHarmonicConj}
The geometric construction of the point $D$ 
given $A$, $B$ and $C$ is a very well known procedure, summarized in Figure 
\ref{fig.harmonic construction}. (In this figure,
$A$, $B$ and $C$ are given points on a line $L$.
Then $P_1$ is any point not on $L$,
$L_1$ is the line through $A$ and $P_1$, 
and $P_2$ is any point of $L_1$ other than $P_1$ or $A$. 
Everything else is now defined in terms of $A,B,C,P_1,P_2$.
So $L_2$ is the line through $P_1$ and $B$, $L_3$ is the line through
$P_2$ and $C$, $P_3$ is the point $L_2\cap L_3$ and $L_4$ is the line
through $B$ and $P_2$.
Now $L_5$ is the line through $A$ and $P_3$, and $P_4$ is the point $L_4\cap L_5$.
Finally $L_6$ is the line through $P_1$ and $P_4$ and $D$ is the point $L\cap L_6$.)
\end{remark}

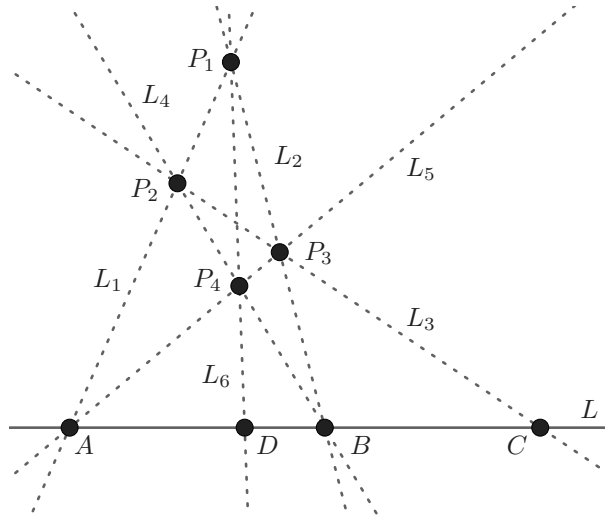
\begin{figure}[ht]
\definecolor{wrwrwr}{rgb}{0.3803921568627451,0.3803921568627451,0.3803921568627451}
\definecolor{sqsqsq}{rgb}{0.12549019607843137,0.12549019607843137,0.12549019607843137}
\begin{tikzpicture}[line cap=round,line join=round,>=triangle 45,x=1cm,y=1cm,scale=0.8]
	\clip(-5.843471074380161,-1.952854996243425) rectangle (4.052096168294518,6.487415477084898);
	\draw [line width=1pt,color=wrwrwr,domain=-14.843471074380161:14.052096168294518] plot(\x,{(-3.9882-0*\x)/7.82});
	\draw [line width=1pt,loosely dotted,color=wrwrwr,domain=-14.843471074380161:14.052096168294518] plot(\x,{(-28.0604-6.08*\x)/-2.68});
	\draw [line width=1pt,loosely dotted,color=wrwrwr,domain=-14.843471074380161:14.052096168294518] plot(\x,{(--3.6871407530895506--4.063836480266734*\x)/-2.448703656724532});
	\draw [line width=1pt,loosely dotted,color=wrwrwr,domain=-14.843471074380161:14.052096168294518] plot(\x,{(-4.4436-6.08*\x)/1.56});
	\draw [line width=1pt,loosely dotted,color=wrwrwr,domain=-14.843471074380161:14.052096168294518] plot(\x,{(--9.035593846265356-4.063836480266734*\x)/6.028703656724532});
	\draw [line width=1pt,loosely dotted,color=wrwrwr,domain=-14.843471074380161:14.052096168294518] plot(\x,{(--12.341918607426397--2.9178718428852557*\x)/3.491335513996546});
	\draw [line width=1pt, loosely dotted,color=wrwrwr,domain=-14.843471074380161:14.052096168294518] plot(\x,{(-7.26286791618962-3.7232588730246836*\x)/0.13992302505272835});
		\draw [fill=sqsqsq] (-4.84,-0.51) circle (4pt);
		\draw[color=sqsqsq] (-4.6,-0.8) node {$A$};
		\draw [fill=sqsqsq] (2.98,-0.51) circle (4pt);
		\draw[color=sqsqsq] (2.6,-0.8) node {$C$};
		\draw [fill=sqsqsq] (-2.16,5.57) circle (4pt);
		\draw[color=sqsqsq] (-2.65,5.6) node {$P_1$};
		\draw [fill=sqsqsq] (-0.6,-0.51) circle (4pt);
		\draw[color=sqsqsq] (0,-0.8) node {$B$};
		\draw [fill=sqsqsq] (-3.048703656724532,3.5538364802667343) circle (4pt);
		\draw[color=sqsqsq] (-3.59,3.45) node {$P_2$};
		\draw [fill=sqsqsq] (-1.3486644860034535,2.4078718428852554) circle (4pt);
		\draw[color=sqsqsq] (-0.7,2.4) node {$P_3$};
		\draw [fill=sqsqsq] (-2.020076974947272,1.8467411269753166) circle (4pt);
		\draw[color=sqsqsq] (-2.55,1.95) node {$P_4$};
		\draw [fill=sqsqsq] (-1.9315087719298245,-0.51) circle (4pt);
		\draw[color=sqsqsq] (-1.55,-0.8) node {$D$};
		\draw[color=sqsqsq] (-4.2,1.95) node {$L_1$};
		\draw[color=sqsqsq] (1,3.8) node {$L_5$};
		\draw[color=sqsqsq] (-1.2,4) node {$L_2$};
		\draw[color=sqsqsq] (1,1.3) node {$L_3$};
		\draw[color=sqsqsq] (3.8,-0.2) node {$L$};
		\draw[color=sqsqsq] (-3.4,5) node {$L_4$};
		\draw[color=sqsqsq] (-2.4,0.35) node {$L_6$};

\end{tikzpicture}
\caption[Harmonic conjugate points.]{Construction of the harmonic conjugate $D$, given $A, B, C$.}
\label{fig.harmonic construction}
\end{figure}	 

\begin{remark}
Harmonic sets of points sometimes appear hidden in classical constructions. 
For example, the 18 roots of the $B_3$ root system in $\CC^3$
give the $B_3$ configuration of nine points in $\PP^2=\PP(\CC^3)$.
In Figure \ref{fig.harmonic construction},
suppose we denote by $P_5$
the point $L_3\cap L_6$.
Then the $B_3$ configuration\index{configuration! $B_3$} of nine points is exactly $P_1,\ldots,P_5,A,B,C,D$.
It has the following interesting property: for a general point $P$,
there is a unique unexpected quartic plane curve $X$ containing the configuration 
such that $X$ has a singular point of multiplicity 3 at $P$
(see \cite{CHMN} and \cite{HMNT}, which explain in what sense $X$ is an unexpected curve; see also section \ref{ch: unexp hypersurf}).
The 12 points coming from the $D_4$ root system are related to harmonic sets of points; see Remark \ref{rem:onD4}(c).
Harmonic sets of points also appear in the $F_4$ configuration of 24 points in $\PP^3$ (see Example \ref{e:D4exOf4.2a},
which in addition points out that this configuration is $(4,6)$-geproci).
Indeed, $F_4$ contains 18 sets of four collinear points (see section \ref{sec.geometryF4}) and each set is harmonic.
Moreover, the $F_4$ configuration of 24 points contains the $D_4$ configuration of 12 points
(see Remark \ref{def:F4}); the 18 4-point lines of $F_4$ arise as 2-point lines of $D_4$
(see sections \ref{sec.geometryF4} and \ref{Sec: D4 in F4}). The complement $D_4'$ of $D_4$ in $F_4$ is
projectively equivalent to $D_4$ and has the same 2-point lines as does $D_4$, so the two points of $D_4$ on a given 2-point line
together with the two points of $D_4'$ on the same line give one of the 18 harmonic sets of four points in $F_4$.
Like $B_3$ and $D_4$, the $F_4$ configuration comes in the same way from a root system, in this case the $F_4$
root system (see \cite{HMNT}, which also explains
how these configurations are related to unexpected surfaces). 
\end{remark}

Properties of the cross ratio are related to automorphisms of $\PP^1$.

\begin{remark}[Invariance of cross ratio]\label{r. cr and autom} Let $f:\PP^1\to\PP^1$ be an automorphism. Then for any ordered set  $\{P_1,P_2,P_3,P_4\}$ of four distinct points, the cross ratio equals
	the cross ratio of the set  $\{f(P_1),f(P_2),f(P_3),f(P_4)\}$ (note the same ordering).
	
	Conversely, if the cross ratio of $\{P_1,P_2,P_3,P_4\}$ and $\{Q_1,Q_2,Q_3,Q_4\}$ are equal, then there exists an automorphism which sends
	$P_i$ to $Q_i$ for all $i$.
\end{remark}

\begin{example}
We can use the invariance of the cross ratio to find 
a formula for the point $D$ in terms of the points $A,B,C$
when $A=[a_1:a_2],B=[b_1:b_2],C=[c_1:c_2],D=[d_1:d_2]$ 
have cross ratio $t$. Note
$A'=[0:1]$, $B'=[1:0]$, $C'=[1:1]$ and $D'=[1:t]$ has cross ratio $t$.
The matrix taking $A'$ to $A$, $B'$ to $B$ and $C'$ to $C$
is 
$$ \begin{pmatrix} 
b_1(a_2c_1-a_1c_2) & a_1(b_1c_2-b_2c_1) \\ 
b_2(a_2c_1-a_1c_2) & a_2(b_1c_2-b_2c_1)
\end{pmatrix}.$$
This matrix must also take $D'$ to $D$ so we get
$$ \begin{pmatrix} 
d_1 \\ 
d_2
\end{pmatrix}
=
\begin{pmatrix} 
b_1(a_2c_1-a_1c_2) + ta_1(b_1c_2-b_2c_1) \\ 
b_2(a_2c_1-a_1c_2) + ta_2(b_1c_2-b_2c_1)
\end{pmatrix}
$$
\end{example}

Since the cross ratio is invariant under automorphisms of $\PP^1$, given four 
distinct points with cross ratio $t$, we can choose coordinates such that
the points in order are
$\{[0:1], [1:0], [1:1], [1:t]\}$ for some $t$ other than 0 or 1. 
Thus the points are harmonic exactly when $t=-1$, and we 
also see that the cross ratio can never be 0, 1 or $\infty$.

\begin{remark} \label{harmperm} Choose an ordered set $\{P_1,P_2,P_3,P_4\}$ of four points in $\PP^1$ and let $k$ be the cross ratio. Let $\sigma$ be a permutation
	on the set $\{1,2,3,4\}$. Then the cross ratio of $\{P_{\sigma(1)},P_{\sigma(2)},P_{\sigma(3)},P_{\sigma(4)}\}$ must have one
	of the following values:
	$$ k, \dfrac{1}{k}, 1-k, \dfrac{1}{1-k}, \dfrac {k-1}{k}, \dfrac{ k}{k-1}.$$
	If the six values are distinct, then each of them is obtained for exactly 4 of the 24 permutations of $\{1,2,3,4\}$, which means there are four permutations for which the cross ratio does not change. These four can be obtained by an
	automorphism of $\PP^1$.
	
	If the points are harmonic, then the cross ratio   of a permutation can have only the values $-1,2,1/2$ and there are exactly $8$ permutations for
	which the new ordered set is harmonic. These 8 permutations are
	$$S_h = \begin{matrix}  (1,2,3,4) & (1,2,4,3) \\ (2,1,4,3) & (2,1,3,4) \\ (3,4,1,2) & (3,4,2,1) \\ (4,3,2,1) & (4,3,1,2)
	\end{matrix}$$
	and can be obtained by automorphisms of $\PP^1$.
\end{remark}

\begin{lemma}\label{l. projectivity and harmonic points}
	Let $A,B,C,D\in \PP^1$ be four distinct points and let $\varphi$ be a projectivity on $\PP^1$ such that $\varphi(A)=A$, $\varphi(B)=B$ and $\varphi(C)=D$. Then $\varphi$ is an involution on~$\PP^1$
	if and only if
the points $A,B,C,D$ (in order) are harmonic.
\end{lemma}
\begin{proof}
If $\varphi$ is an involution then  in particular $\varphi(D)=C$ and, by Remark  \ref{r. cr and autom}, the points $A,B,C,D$ and $A,B,D,C$ have the same cross ratio. But one can check from the definition that the cross ratios are each other's inverse. So, the only possibility is that the four points are harmonic.
	
On the other hand, assume the points are harmonic.
Without loss of generality assume $A=[0:1]$, $B=[1:0]$,  $C=[1:1]$ and $D=[1:-1]$. 
Then the matrix for the map $\varphi$ is 
$$ \begin{pmatrix} -1 & 0\\ 
	                0 & 1  
\end{pmatrix}$$
which is its own inverse, so $\varphi$ is an involution.
\end{proof}

Lemma \ref{l. projectivity and harmonic points} and the construction
given in Remark \ref{GeomConstrOfHarmonicConj} both are 
manifestations of the next result, as we explain in Remark \ref{HarmonicityAndDeg2Maps}.

\begin{lemma}\label{deg 2 maps and harmonic points}
Let $f:\PP^1\to\PP^1$ be a degree 2 morphism. Then there are exactly
two points where $f$ ramifies; call them $A,B\in\PP^1$,
so $f^{-1}(f(A))=\{A\}$ and $f^{-1}(f(B))=\{B\}$. Moreover, if
$\{C,D\}$ is any nonsingleton fiber of $f$,
then $A,B,C,D$ in this order are harmonic.
(Note that transposing $A$ and $B$, or $C$ and $D$,
does not spoil the harmonicity.)
\end{lemma}

\begin{proof}
By the Hurwitz Theorem,\index{Theorem! Hurwitz}\index{Hurwitz} we know a degree 2 map $\PP^1\to\PP^1$ has 
two ramification points; thus $A$ and $B$ are the preimages of the ramification points. 
We can choose coordinates $x,y$ on each $\PP^1$ such that 
$f(A)=A=[0:1]$, $f(B)=B=[1:0]$, $f(C)=C=[1:1]$, and so $f([a:b])=[a^2,b^2]$ and
thus $D$ is $[1:-1]$. The rest is an easy calculation.
\end{proof}

\begin{remark}\label{HarmonicityAndDeg2Maps}
Both the construction given in Remark \ref{GeomConstrOfHarmonicConj}
and an involution $\varphi$ as in Lemma \ref{l. projectivity and harmonic points} 
are examples of Lemma \ref{deg 2 maps and harmonic points}.
In Remark \ref{GeomConstrOfHarmonicConj} there 
is a line $L\subset\PP^2$ and a harmonic set of points
$A,B,C,D$, and there is also a pencil
of conics with base points $P_1,P_2,P_3,P_4\in \PP^2$.
Thinking of the pencil as a $\PP^1$ in the space of conics,
this gives a map $f$ from $L=\PP^1$ to the pencil $\PP^1$, whereby
$f:p\in L\mapsto C_p$, where $C_p$ is the conic through $p$ in the pencil.
Note that this map has degree 2 (since every conic intersects $L$ twice),
the points $f(A)$ and $f(B)$ are the ramification points, and
$\{C,D\}$ is a nonsingleton fiber of $f$.
(In Remark \ref{GeomConstrOfHarmonicConj}, $L$ is the line
through the singular points of two of the three reducible conics in the pencil,
these singular points being $A$ and $B$,
and $C$ and $D$ are the points where the third reducible conic meets $L$.
But this is not essential: if $L$ is any line not containing any of the points $P_i$,
the pencil will always have two conics tangent to $L$; the points of tangency
give $A$ and $B$, and $C$ and $D$ are then the points of intersection of $L$
with any other conic in the pencil.)
Likewise, given an involution $\varphi$ as in 
Lemma \ref{l. projectivity and harmonic points}, 
modding out by the involution gives a degree 2 map 
$f:\PP^1\to\PP^1$ where $f(A)$ and $f(B)$ are points of ramification
and $\{C,D\}$ is a nonsingleton fiber.
(Conversely, a degree 2 morphism $f$
gives a 2-sheeted ramified cover of $\PP^1$, so
the deck transformation\index{deck transformation} swapping the sheets gives an involution of $\PP^1$.)
\end{remark}

The next lemma is another example of how harmonicity 
arises in well known geometrical constructions. To state it we recall the definition of a polar line and plane.

Given a smooth plane conic $\gamma$, and a point $A\not\in\gamma$, let $P_1$ and $P_2$ be the points
of $\gamma$ whose tangent lines contain $A$ (see Figure 
\ref{fig.harmonic points and polar}). Then the line through
$P_1$ and $P_2$ is called the \textit{polar line}\index{polar! line} for $A$ with respect to $\gamma$.
If $F$ is the form defining $\gamma$, then $\nabla F\cdot (a_0,a_1,a_2)=0$ defines the polar line for $A=[a_0:a_1:a_2]$.

Similarly, given a smooth quadric $\calq\subset\PP^3$, and a point $P$ not in $\calq$, the plane $H$ polar to $P$ with respect to $\calq$ is the locus of points such that for a general line $\ell$ passing through $P$, the points of intersection $\ell\cap\calq$ and the points $P$ and $\ell\cap H$ are harmonic. The \textit{polar plane}\index{polar! plane} is in fact spanned by the points of $\calq$
whose tangent planes contain $P$. Thus if $F$ is the form defining $\calq$, 
then $\nabla F\cdot (p_0,p_1,p_2,p_3)=0$ defines the polar plane for $P=[p_0:p_1:p_2:p_3]$.

\begin{lemma}\label{l. harmonic and polar}
	Let $\gamma$ be a smooth plane conic and let $A\in \PP^2$ be a point not on $\gamma$. Consider a line $\ell$ through $A$ and let $B,C$ be the points of intersection of $\ell$ with $\gamma$ (see Figure \ref{fig.harmonic points and polar}). Then, the polar line $p$ for $A$ with respect to $\gamma$ meets $\ell$ in a point $H$ such that $\{A,H,B,C\}$ are harmonic.  
\end{lemma}

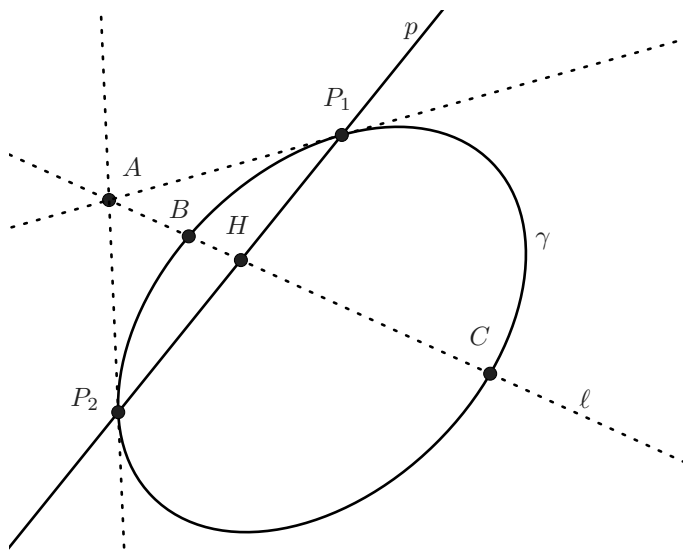
\begin{figure}[ht]
	\definecolor{sqsqsq}{rgb}{0.12549019607843137,0.12549019607843137,0.12549019607843137}
	\begin{tikzpicture}[line cap=round,line join=round,>=triangle 45,x=1cm,y=1cm,scale=0.7]
		\clip(-6.636453428298804,-4.890379762570668) rectangle (6.155830779311015,5.388444765409979);
		\draw [rotate around={44.58179005850271:(-0.696110355019287,-0.6820816242810619)},line width=1pt] (-0.696110355019287,-0.6820816242810619) ellipse (4.524050450212558cm and 3.060601062515585cm);
		\draw [line width=1pt, domain=-12.636453428298804:8.155830779311015] plot(\x,{(--3240.1661142479515--1178.4368415346632*\x)/949.5672692710345});
		\draw [line width=1pt,loosely dotted, domain=-6.636453428298804:6.155830779311015] plot(\x,{(-0.5794775215913148-0.6902492195434127*\x)/1.5136888032580606});
		\draw [fill=sqsqsq] (-4.740078132440041,1.7786732015888566) circle (3.5pt);
		\draw[color=sqsqsq] (-4.3,2.4) node {$A$};
		\draw [fill=sqsqsq] (-3.22638932918198,1.088423982045444) circle (3.5pt);
		\draw[color=sqsqsq] (-3.4,1.6) node {$B$};
		\draw [fill=sqsqsq] (2.49804410984125,-1.521944612284684) circle (3.5pt);
		\draw[color=sqsqsq] (2.3,-0.8) node {$C$};
		\draw[color=sqsqsq] (4.3,-2) node {$\ell$};
		\draw[color=sqsqsq] (3.5,1) node {$\gamma$};
		\draw [fill=sqsqsq] (-2.2363073042664725,0.6369419181556047) circle (3.5pt);
		\draw[color=sqsqsq] (-2.3,1.3) node {$H$};
		\draw [line width=1pt, loosely dotted,domain=-12.636453428298804:8.155830779311015] plot(\x,{(-18.805878921293633-4.032786778667346*\x)/0.17420035382778032});
		\draw [line width=1pt,loosely dotted,domain=-12.636453428298804:8.155830779311015] plot(\x,{(--13.725720094636887--1.2369330526110267*\x)/4.420463957988218});
		\draw [fill=sqsqsq] (-0.31961417445182283,3.0156062541998833) circle (3.5pt);
		\draw[color=sqsqsq] (-0.4,3.7) node {$P_1$};
		\draw[color=sqsqsq] (1,5) node {$p$};
		\draw [fill=sqsqsq] (-4.56587777861226,-2.2541135770784897) circle (3.5pt);
		\draw[color=sqsqsq] (-5.2,-2.0) node {$P_2$};
	\end{tikzpicture}
	\caption[Polarity and harmonicity.]{The polar line $p$ for $A$ with respect to
	the irreducible conic $\gamma$ defines a harmonic set 
	$A,H,B,C$.}
	\label{fig.harmonic points and polar}
\end{figure}
\begin{proof}
Let $P_1, P_2$	be the points intersection of the polar line with $\gamma$. 
Consider the projective automorphism $\varphi:\PP^2\to \PP^2$ that fixes $P_1$ and $P_2$ and 
transposes $B$ and $C$.

The point $H$ is the intersection of two lines preserved by $\varphi$ so it is fixed by $\varphi$. 
Without loss of generality, assume $A=[0:0:1]$, $P_1=[1:0:0]$, $P_2=[0:1:0]$ and $B=[1:1:1]$.
Then $\gamma$ is defined by $xy-z^2$ and $\ell$ is 
defined by $y-x$, hence $C=[1:1:-1]$. 
Since $\varphi$ fixes $P_1$ and $P_2$ but transposes
$B$ and $C$ an easy computation shows that
its matrix is
$$ \begin{pmatrix} 
1 & 0 & 0\\ 
0 & 1 & 0\\
0 & 0 & -1
\end{pmatrix}.$$
We see that $\varphi$ fixes $A$ and that $\varphi$
is an involution. It now follows from
Lemma \ref{l. projectivity and harmonic points} that $\{A,H,B,C\}$ is harmonic.
\end{proof}

The following result is certainly well known.
For example, it is mentioned in \cite{moore} but without any reference or proof, so it was once common knowledge. We include a proof here for completeness and for the reader's convenience.  

\begin{lemma}\label{l. harmonic and conic}
	Let $\PP^2$ contain three distinct 
	lines $\ell_1,\ell_2,\ell_3$ 
	through a point $A$.
	Let $H_i,B_i,C_i$ be three points on the line $\ell_i$, for $i=1,2,3$, such that  $\{A,H_i,B_i,C_i\}$ are harmonic, in the order given (see Figure \ref{fig. harmonic and conic}).
\begin{enumerate}
\item[(a)] Assume the six points $B_1,B_2,B_3,C_1,C_2,C_3$ lie on a 
conic $\gamma$. Then $H_1,H_2,H_3$ are collinear, where $\ell$ denotes the line containing them.
\item[(b)] Conversely, assume $H_1,H_2,H_3$ are collinear, on a line $\ell$.
Then the six points $B_1,B_2,B_3,C_1,C_2,C_3$ lie on a conic $\gamma$.
\item[(c)] In either case (a) or case (b), if $\gamma$ is irreducible, then $\ell$ is a polar line for $A$ with respect to $\gamma$, while if
$\gamma$ is reducible, then $\ell$ contains the singular point of $\gamma$.
Moreover, the nine points $H_i, B_i, C_i$ form 
a complete intersection of two cubic curves
(namely, $\ell_1\cup\ell_2\cup\ell_3$ and $\ell\cup\gamma$).
\end{enumerate}
\end{lemma}

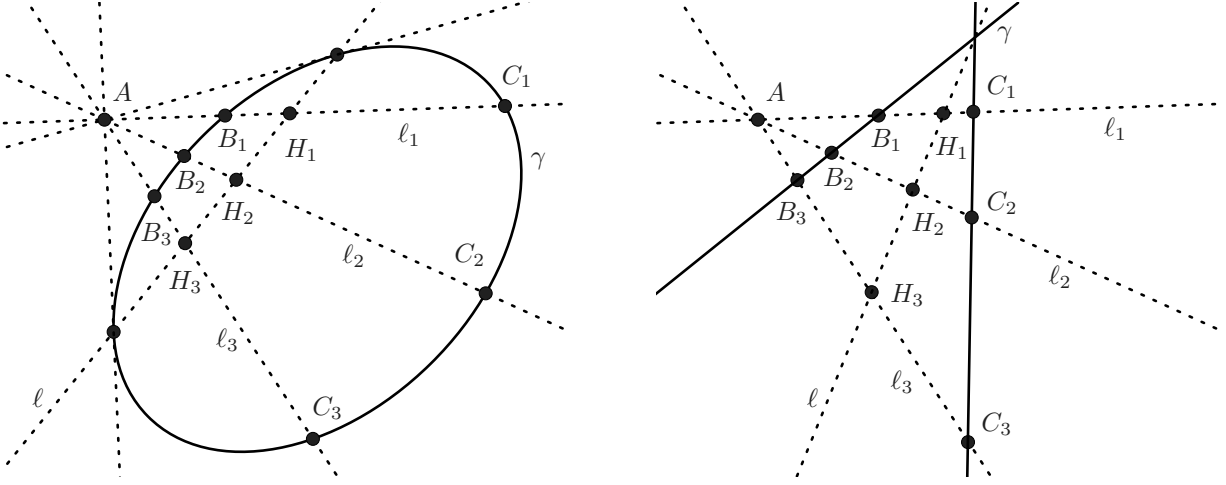
\begin{figure}[ht]
\hbox to\hsize{\hbox{
\definecolor{sqsqsq}{rgb}{0.12549019607843137,0.12549019607843137,0.12549019607843137}
\begin{tikzpicture}[line cap=round,line join=round,>=triangle 45,x=1cm,y=1cm,scale=0.7]
\clip(-6.673360310006428,-5) rectangle (4,4);
\draw [rotate around={44.58179005850271:(-0.696110355019287,-0.6820816242810619)},line width=1pt] (-0.696110355019287,-0.6820816242810619) ellipse (4.524050450212558cm and 3.060601062515585cm);
\draw [line width=1pt,loosely dotted, domain=-13.673360310006428:9.198152318364373] plot(\x,{(--3240.1661142479515--1178.4368415346632*\x)/949.5672692710345});
\draw [line width=1pt,loosely dotted, domain=-13.673360310006428:9.198152318364373] plot(\x,{(-0.5794775215912971-0.6902492195434098*\x)/1.5136888032580633});
\draw [line width=1pt,loosely dotted, domain=-13.673360310006428:9.198152318364373] plot(\x,{(--4.44404086963513--0.07813686421392263*\x)/2.290283580254286});
\draw [line width=1pt,loosely dotted, domain=-13.673360310006428:9.198152318364373] plot(\x,{(-5.209740731111874-1.4554793959205914*\x)/0.9497783652439917});
\draw [fill=sqsqsq] (-0.31961417445182283,3.0156062541998833) circle (3.5pt);
\draw[color=sqsqsq] (3.5,1) node {$\gamma$};
\draw [fill=sqsqsq] (-4.56587777861226,-2.2541135770784897) circle (3.5pt);
\draw [line width=1pt, loosely dotted,domain=-12.636453428298804:8.155830779311015] plot(\x,{(-18.805878921293633-4.032786778667346*\x)/0.17420035382778032});
\draw [line width=1pt,loosely dotted,domain=-12.636453428298804:8.155830779311015] plot(\x,{(--13.725720094636887--1.2369330526110267*\x)/4.420463957988218});
\draw [fill=sqsqsq] (-4.740078132440041,1.7786732015888566) circle (3.5pt);
\draw[color=sqsqsq] (-4.4,2.3) node {$A$};
\draw[color=sqsqsq] (-6,-3.5) node {$\ell$};
\draw [fill=sqsqsq] (-3.2263893291819774,1.0884239820454468) circle (3.5pt);
\draw[color=sqsqsq] (-3.1,0.6) node {$B_2$};
\draw [fill=sqsqsq] (2.5,-1.521944612284669) circle (3.5pt);
\draw[color=sqsqsq] (2.2,-0.8) node {$C_2$};
\draw [fill=sqsqsq] (-2.236307304266469,0.6369419181556095) circle (3.5pt);
\draw[color=sqsqsq] (-2.2,0) node {$H_2$};
\draw [fill=sqsqsq] (-2.4497945521857547,1.8568100658027793) circle (3.5pt);
\draw[color=sqsqsq] (-2.3,1.4) node {$B_1$};
\draw [fill=sqsqsq] (-3.790299767196049,0.3231938056682653) circle (3.5pt);
\draw[color=sqsqsq] (-3.75,-0.4) node {$B_3$};
\draw [fill=sqsqsq] (-1.219534555684807,1.898782448228776) circle (3.5pt);
\draw[color=sqsqsq] (-1,1.2) node {$H_1$};
\draw [fill=sqsqsq] (-0.7812276837567444,-4.288031906090976) circle (3.5pt);
\draw[color=sqsqsq] (-0.5,-3.7) node {$C_3$};
\draw [fill=sqsqsq] (2.8663971111224877,2.0381808574865685) circle (3.5pt);
\draw[color=sqsqsq] (3.1,2.6) node {$C_1$};
\draw [fill=sqsqsq] (-3.2080696429296895,-0.5690395622041159) circle (3.5pt);
\draw[color=sqsqsq] (-3.2,-1.3) node {$H_3$};
\draw[color=sqsqsq] (-2.4,-2.4) node {$\ell_3$};
\draw[color=sqsqsq] (0,-0.8) node {$\ell_2$};
\draw[color=sqsqsq] (1.05,1.5) node {$\ell_1$};
\end{tikzpicture}
}
\hbox{
\definecolor{sqsqsq}{rgb}{0.12549019607843137,0.12549019607843137,0.12549019607843137}
\begin{tikzpicture}[line cap=round,line join=round,>=triangle 45,x=1cm,y=1cm,scale=0.7]
\clip(-6.673360310006428,-5) rectangle (4,4);
\draw [line width=1pt,loosely dotted, domain=-13.673360310006428:9.198152318364373] plot(\x,{(-0.5794775215912971-0.6902492195434098*\x)/1.5136888032580633});
\draw [line width=1pt,loosely dotted, domain=-13.673360310006428:9.198152318364373] plot(\x,{(--4.44404086963513--0.07813686421392263*\x)/2.290283580254286});
\draw [line width=1pt,loosely dotted, domain=-13.673360310006428:9.198152318364373] plot(\x,{(-5.209740731111874-1.4554793959205914*\x)/0.9497783652439917});
\draw[color=sqsqsq] (-0.05,3.4) node {$\gamma$};
\draw [fill=sqsqsq] (-4.740078132440041,1.7786732015888566) circle (3.5pt);
\draw[color=sqsqsq] (-4.4,2.3) node {$A$};
\draw[color=sqsqsq] (-3.7,-3.5) node {$\ell$};
\draw [fill=sqsqsq] (-3.34,1.15) circle (3.5pt);
\draw[color=sqsqsq] (-3.2,0.65) node {$B_2$};
\draw [fill=sqsqsq] (-1.8,0.45) circle (3.5pt);
\draw[color=sqsqsq] (-1.5,-0.2) node {$H_2$};
\draw [fill=sqsqsq] (-2.4497945521857547,1.8568100658027793) circle (3.5pt);
\draw[color=sqsqsq] (-2.3,1.4) node {$B_1$};
\draw [fill=sqsqsq] (-3.99,0.63) circle (3.5pt);
\draw[color=sqsqsq] (-4.1,0) node {$B_3$};
\draw [fill=sqsqsq] (-1.219534555684807,1.898782448228776) circle (3.5pt);
\draw[color=sqsqsq] (-1.05,1.25) node {$H_1$};
\draw [fill=sqsqsq] (-0.75,-4.35) circle (3.5pt);
\draw[color=sqsqsq] (-0.2,-4) node {$C_3$};
\draw [fill=sqsqsq] (-2.58,-1.5) circle (3.5pt);
\draw[color=sqsqsq] (-1.9,-1.55) node {$H_3$};
\draw [line width=1pt, loosely dotted] (-4,-5)--(-0.35,4);
\draw [line width=1pt] (-6.7,-1.55)--(0.2,4);
\draw [line width=1pt] (-0.7812276837567444,-6)--(-0.6,4.5);
\draw [fill=sqsqsq] (-0.65,1.93) circle (3.5pt);
\draw[color=sqsqsq] (-0.1,2.4) node {$C_1$};
\draw [fill=sqsqsq] (-0.68,-0.08) circle (3.5pt);
\draw[color=sqsqsq] (-0.1,0.2) node {$C_2$};
\draw[color=sqsqsq] (-2,-3.2) node {$\ell_3$};
\draw[color=sqsqsq] (1,-1.2) node {$\ell_2$};
\draw[color=sqsqsq] (2.05,1.6) node {$\ell_1$};
\end{tikzpicture}
}}
\caption[Harmonicity and conics.]{Points as in Lemma \ref{l. harmonic and conic} for $\gamma$
irreducible (left) and reducible (right).}
\label{fig. harmonic and conic}
\end{figure}

\begin{proof}
(a, c) Assume $B_1,B_2,B_3,C_1,C_2,C_3$ lie on a conic $\gamma$.
If $\gamma$ is irreducible, then from Lemma \ref{l. harmonic and polar} the points $H_1,H_2,H_3$ lie on a line $\ell$ and that line is polar for $A$ with respect to $\gamma$. It follows that 
the nine points $H_i, B_i, C_i$ form 
a complete intersection of the two cubic curves
$\ell_1\cup\ell_2\cup\ell_3$ and $\ell\cup\gamma$.
Now say $\gamma$ is reducible; we may assume $B_1,B_2,B_3$ lie on one 
component of $\gamma$ and $C_1,C_2,C_3$ lie on the other. 
Let $P$ be the singular point of $\gamma$.
Projection from $P$ sends $A, H_1, B_1, C_1$ to a harmonic set of four points
on $\ell_2$ but $A$ is fixed and $B_1\mapsto B_2$ and $C_1\mapsto C_2$.
Since $A,H_2,B_2,C_2$ is already harmonic, $H_1$ must map to $H_2$.
The same argument applies to $A,H_3,B_3,C_3$. Thus $H_2$ and $H_3$ are on the line
$\ell$ through $P$ and $H_1$. 
It again follows that 
the nine points $H_i, B_i, C_i$ form 
a complete intersection of the two cubic curves
$\ell_1\cup\ell_2\cup\ell_3$ and $\ell\cup\gamma$.
	
(b) Now assume $H_1,H_2,H_3$ are collinear. Let $\gamma$ be the conic through $B_1,B_2,B_3,C_1,C_2$. 
First assume $\gamma$ is irreducible.
From Lemma \ref{l. harmonic and polar} the line through $H_1$ and $H_2$ 
is the line polar for $A$ 
with respect to $\gamma$. Therefore, if $C_3'$ is the second point of intersection between $\gamma$ and $\ell_3$ we get that the points $\{A,H_3,B_3, C_3'\}$ are harmonic. But $\{A,H_3,B_3, C_3\}$ are harmonic by hypothesis. Hence, $C_3=C_3'$ is also on $\gamma$.
Now assume $\gamma$ is reducible; we may assume (after swapping the labels $B_1$ and $C_1$
or $B_2$ and $C_2$, if need be) that $B_1,B_2,B_3$ are on one component of $\gamma$ and $C_1,C_2$ are on the other. Let $P$ be the singular point of $\gamma$.
Under a projection from $P$ we have $A\mapsto A$, $B_1\mapsto B_2$ and $C_1\mapsto C_2$,
so as before $H_1$ must map to $H_2$, hence $P,H_1,H_2$ are collinear,
and thus $H_3$, being collinear with $H_1$ and $H_2$ is also on this line;
i.e., the line $\ell$ through $H_1,H_2,H_3$ contains $P$, and under another projection
from $P$ we have $A\mapsto A$, $B_1\mapsto B_3$ and $H_1\mapsto H_2$,
so harmonicity ensures $C_1\mapsto C_3$; i.e., the line through $C_1$ and $C_2$ contains $C_3$, so $C_3\in\gamma$.
\end{proof}

\section{Segre Embeddings and grids}
\label{sec:Segre_Embeddings}\index{grid}
Recall that the Segre embedding \index{Segre embedding} $s:\PP^1\times\PP^1\to \PP^3$ acts by sending the point $([a:b],[c:d])$ to the point $[ac:ad:bc:bd]$.
This operation is also known as the {\it tensor product} \index{tensor product} of vectors $(a,b)\otimes (c,d)$. If one identifies the point
$[x:y:z:w]\in\PP^3$ with the matrix $\begin{pmatrix} x & y \\ z & w\end{pmatrix}$ then the tensor product is equivalent to the
row by column product
$$ (a,b)\otimes (c,d) = \begin{pmatrix} a\\ b  \end{pmatrix}   \begin{matrix} (c,d)\\ \phantom{(c,d)} \end{matrix} =  \begin{pmatrix} ac & ad \\ bc & bd  \end{pmatrix}. $$
The result of a tensor product is necessarily a matrix of rank $1$, for all $[a:b],[c:d]$ in $\PP^1$.

Recall also that a smooth quadric surface $\calq\subset \PP^3$ is the image of $\PP^1\times\PP^1$
under the Segre embedding, and that each factor of $\PP^1$ induces a ruling of $\calq$,
i.e., a fibration by lines, where each line in one ruling meets each line in the other.

We are interested in $(3,3)$-grids\index{grid! $(3,3)$} as they are the starting point to define geometrically the $D_4$ configuration, which in turn is the smallest nontrivial nongrid geproci set of points in $\PP^3$.

Let $H_b, H_m, H_t$ be the three skew lines (which we refer to as the horizontal bottom, middle and top lines), which together with the three lines $V_l,V_m,V_r$ (which we refer to as vertical left, middle and right lines) form a $(3,3)$-grid. It is well known that three skew lines in $\PP^3$ determine a unique smooth quadric surface $\calq$ on which they belong to the same ruling. Any line transversal to three skew lines is a line in the other ruling. Thus any $(3,3)$-grid is contained in a smooth quadric. Its points are therefore determined by images of $3$ points on each factor under the Segre embedding $s$. But any $3$ distinct points in $\PP^1$ are projectively equivalent. So we can take the points
$$[1:0],\; [0:1],\; [1:1]$$
on each factor. Hence we get the $9$ points with coordinates as in Figure \ref{GridOnTetrahedron}. 

\begin{figure}[ht!]
\definecolor{ffffff}{rgb}{1.,1.,1.}
\definecolor{uuuuuu}{rgb}{0.26666666666666666,0.26666666666666666,0.26666666666666666}
\definecolor{ududff}{rgb}{0.30196078431372547,0.30196078431372547,1.}
\begin{tikzpicture}[line cap=round,line join=round,>=triangle 45,x=1.5cm,y=1.5cm]
\clip(-2.9251177150754075,-0.3643893007173551) rectangle (4.096570195978351,5.0108338587789705);
\draw [line width=2pt] (-2.25376510915961,1.4988893252096052)-- (1.2680846267921186,0.4863575261234834);
\draw [line width=2pt] (1.2680846267921186,0.4863575261234834)-- (2.92,1.62);
\draw [line width=2pt] (2.92,1.62)-- (0.64,4.32);
\draw [line width=2pt] (0.64,4.32)-- (-2.25376510915961,1.4988893252096052);
\draw [line width=1.pt,dash pattern=on 3pt off 3pt] (-2.25376510915961,1.4988893252096052)-- (2.92,1.62);
\draw [line width=2.pt,color=white] (0.95,1.574)-- (1.5,1.587);
\draw [line width=2.pt,color=white] (0.35,1.56)-- (-0.1,1.549);
\draw [line width=2.pt] (-0.4928402411837456,0.9926234256665443)-- (1.78,2.97);
\draw [line width=2.pt] (2.0940423133960593,1.0531787630617417)-- (-0.8068825545798048,2.9094446626048027);
\draw [line width=2.pt] (0.877978830884413,2.1852393744871916)-- (1.1147070976400486,2.391193546136915);
\draw [line width=3pt,color=white] (0.9571298794272152,1.780674950653604)-- (1.1510884940726007,1.6565632350199406);
\draw [line width=3pt,color=white] (0.877978830884413,2.1852393744871916)-- (1.1147070976400486,2.391193546136915);
\draw [line width=1pt] (0.64,4.32)--(1.2680846267921186,0.4863575261234834);
\draw [fill=white] (-2.25376510915961,1.4988893252096052) circle (4pt);
\draw[color=black] (-2.392437942512709,1.2154904792697407) node {0010};
\draw [fill=white] (1.2680846267921186,0.4863575261234834) circle (4pt);
\draw[color=black] (1.7,0.4) node {1000};
\draw [fill=white] (2.92,1.62) circle (4pt);
\draw[color=black] (3.15,1.9) node {0100};
\draw [fill=white] (0.64,4.32) circle (4pt);
\draw[color=black] (0.8278534097981529,4.6) node {0001};
\draw [fill=white] (-0.8068825545798048,2.9094446626048027) circle (4pt);
\draw[color=black] (-1.2544402465833064,3.0677633247718528) node {0011};
\draw [fill=white] (-0.4928402411837456,0.9926234256665443) circle (4pt);
\draw[color=black] (-0.7096541155532734,0.7) node {1010};
\draw [fill=white] (2.0940423133960593,1.0531787630617417) circle (4pt);
\draw[color=black] (2.5,0.9733633099230593) node {1100};
\draw [fill=white] (1.78,2.97) circle (4pt);
\draw[color=black] (2,3.25) node {0101};
\draw [fill=white] (0.6435798794081273,1.9813117128332722) circle (4pt);
\draw[color=black] (0.08936554329077495,2.014510138113789) node {1111};
\end{tikzpicture}
\caption[A $(3,3)$-grid (in 3D perspective).]{The $(3,3)$-grid is unique up to choice of coordinates.
Here it is shown in 3D perspective relative to the coordinate tetrahedron: the 9 grid points\index{grid! points} are shown as open dots and the grid lines are bold.}
\label{GridOnTetrahedron}
\end{figure}
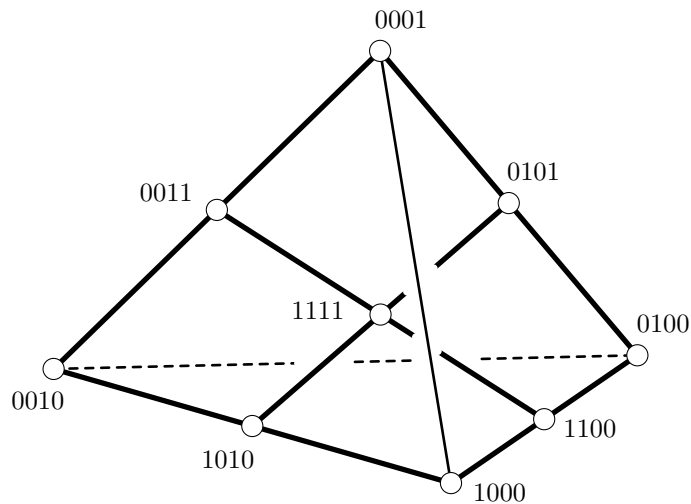

We summarize this discussion with the following Lemma.
\begin{lemma}[The uniqueness of a $(3,3)$-grid] \label{3by3GridUniqueLem}
Let $Z$ be a $(3,3)$-grid in $\PP^3$. Then coordinates $x,y,z,w$ can be chosen on $\PP^3$ such that
the 9 points of the grid are as shown in Figure \ref{GridOnTetrahedron}.
With these coordinates, the unique
quadric containing the grid is $xw-yz=0$.
\end{lemma}
We derive an immediate consequence of Lemma \ref{3by3GridUniqueLem}.
\begin{corollary}\label{cor:3x3gridUniq}
   All $(3,3)$-grids are projectively equivalent.
\end{corollary}

\begin{lemma}\label{BrianchonLem}
Let $P=[a:b:c:d]$ be a general point of $\PP^3$ and
let $H$ be the plane defined by the equation
$$dx-cy-bz+aw=0.$$ 
Let $\calq\subset \PP^3$ be the quadric $xw-yz=0$,
so that $H$ is the polar plane of $P$ with respect to this quadric.
And let $C$ be the conic plane curve
defined by intersecting $\calq$ and $H$. 
Then every line on $\calq$ projects from $P$ to a line in $H$ tangent to $C$.
\end{lemma}

\begin{proof}
Since $P$ is general and $\calq$ is smooth, Bertini's Theorem gives that $C$ is smooth.
Let $F$ be the quadric defining the cone $C_P$ over $C$ with vertex $P$
(in particular, $F$ is the unique form of degree 2 in $I(P)^2\cap (\calq,H)$).
Then a calculation shows that $(F,\calq)=(\calq,H^2)$, hence $C_P\cap \calq$ 
is a $(2,2)$ divisor class on $\calq$ set-theoretically equal to the $(1,1)$ divisor class $C$,
so $C_P$ is tangent to $\calq$ along $C$. Each line $L$ on $\calq$ is either a $(1,0)$ or $(0,1)$
divisor class, and thus meets $C$. Say $L$ meets $C$ at $q\in C$. 
Since $C_P$ is tangent to $\calq$ along $C$, the tangent plane $T_q$ to $C_P$ at $q$ 
is the same as the tangent plane to $\calq$ at $q$, and thus $T_q$ contains $L$,
and $T_q\cap H$ is the line tangent to $C$ at $q$, which is also the image of $L$ under the
projection from $P$ to $H$.
\end{proof}
We consider now the group of projective automorphisms of a $(3,3)$-grid. 

\begin{notation}
We sometimes find it quite handy to represent here, in Chapter \ref{Ch:D4} and in Chapter \ref{ch:FW} the number $-1$ by a $*$.
\end{notation}

\begin{proposition}\label{GridGroupProp}
Let $Z$ be a $(3,3)$-grid in $\PP^3$. 
Let $G=\Aut(Z)$ be the group of linear transformations of $\PP^3$ mapping $Z$ to $Z$.
Then $G\cong (S_3\times S_3)\rtimes {\mathbb Z}_2$, where the first $S_3$
corresponds to permutations of the columns of the grid, the second 
$S_3$ corresponds to permutations of the rows of the grid and ${\mathbb Z}_2$
corresponds to 
exchanging the columns and the rows.
\end{proposition}

\begin{proof}
By Lemma \ref{3by3GridUniqueLem}, it is enough to verify the claims for the grid shown in 
Figure \ref{GridOnTetrahedron}.
It is well known that any automorphism of $\PP^3$ fixing the four corner points and the middle point is the identity.

First we show that $S_3\times S_3$ is a subgroup of $G$.
Transposing the $y$ and $w$ coordinates while simultaneously
transposing the $x$ and $z$ coordinates (i.e., the permutation $(xz)(yw)$), permutes the top and bottom
rows of the grid as shown in Figure \ref{GridOnTetrahedron}.
Its matrix
$$
\begin{pmatrix}
   0 & 0 & 1 & 0\\
   0 & 0 & 0 & 1\\
   1 & 0 & 0 & 0\\
   0 & 1 & 0 & 1
\end{pmatrix}\;\mbox{ together with the matrix } \;
\left(
\begin{array}{cccc}
1  & 0  & 0 &  0\\
0  &  1 & 0 & 0 \\
1  & 0  &  * &  0 \\
0  &  1 & 0  & * \\
\end{array}
\right)
$$
transposing the bottom and middle rows, generate $S_3$ acting on the rows.

The permutation $(xy)(zw)$ permutes the outer
columns of the grid, and the matrix
\[
\left(
\begin{array}{cccc}
*  & 1  & 0 &  0\\
0  &  1 & 0 & 0 \\
0  & 0  &  * &  1 \\
0  &  0 & 0  & 1 \\
\end{array}
\right)
\]
transposes the middle and right columns, giving the 
$S_3$ acting on the columns.

The row permutations commute with the column permutations
so we get $S_3\times S_3\subseteq G$. 
Note that $S_3\times S_3$ acts transitively on the 9 grid points.
Let $p$ be the point shown as $1111$ in Figure \ref{GridOnTetrahedron}
and let $G_p$ be the subgroup fixing $p$. It is easy to see that
$G_p$ acts on the four corners of the grid as shown in 
Figure \ref{GridOnTetrahedron} since points on grid lines containing $p$ must
go to points on grid lines containing $p$.
But $G_p$ acts transitively on these four points (since 
$(S_3\times S_3)_p$ does). Now let $q$ be the point $1000$.
The subgroup $H$ of $G$ which fixes both $p$ and $q$,
contains $(yz)$. Any element $\sigma\in H$ must map the points
$1010$ and $1100$ to themselves or to each other
(since these are the only  points on grid lines
through both $p$ and $q$). By composing $\sigma$ with $(yz)$ if need be,
we can assume that $\sigma$ fixes $p,q,1010,1100$, and now it is easy to see
that $\sigma$ is the identity. Thus $H$ is the group of order 2 generated by $(yz)$,
hence $|G_p|=|H|\cdot 4=8$ and $|G|=|G_p|\cdot 9=72$ (since a finite group acting transitively 
on a finite set has order equal to the product of the cardinality of the set times the order of 
the stabilizer of an element of the set).

Since $S_3\times S_3$ is a subgroup of index 2 in $G$, it is normal.
The permutation $(xw)\in G$ swaps rows and columns, hence is an element of order 2 in $G$
which is not in $S_3\times S_3$, thus it gives a subgroup ${\mathbb Z}_2\subset G$ such that
$G\cong (S_3\times S_3)\rtimes {\mathbb Z}_2$.
\end{proof}
Now we pass to the study of collinearities between the points of the $(3,3)$-grid and incidences between the lines determined by these points. 

\begin{proposition}[Collinearities in the $(3,3)$-grid] \label{GridColinearitiesProp}
Let $Z$ be a $(3,3)$-grid in $\PP^3$. There are exactly six lines through exactly three grid points (these are the six grid lines)
and 18 lines through exactly two grid points. These 18 lines come in six sets of three concurrent lines
and the six points of concurrency come in two sets of three collinear points (the two lines through
each set of three points are skew). Whenever two of these 18 lines meet,
they meet at either a grid point or at one of these six points of concurrency.
The group $G=\Aut(Z)$ acts transitively on these six points and on the 18 lines.
\end{proposition}

\begin{proof}
There are $\binom{9}{2}=36$ pairs of grid points.
Each of the four points 1110,1101,1011, 0111 on the faces of the coordinate simplex
is a point of concurrency of three lines where each of the three lines goes through
exactly two grid points (see Figure \ref{GridOnTetrahedron2} for the case of the point 1110).

\begin{figure}[ht!]
\definecolor{ffffff}{rgb}{1.,1.,1.}
\definecolor{qqqqff}{rgb}{0.,0.,1.}
\definecolor{qqffqq}{rgb}{0.,1.,0.}
\definecolor{ffqqqq}{rgb}{0.,0.,0.}
\begin{tikzpicture}[line cap=round,line join=round,>=triangle 45,x=1.5cm,y=1.5cm]
\clip(-0.6578006202659156,-2.800701710843169) rectangle (10,4.536653624746194);
\draw [line width=2.pt] (0.,0.)-- (2.,3.);
\draw [line width=2.pt] (2.,3.)-- (4.,0.);
\draw [line width=2.pt] (2.,-2.)-- (0.,0.);
\draw [line width=1.pt,dash pattern=on 4pt off 4pt] (0.,0.)-- (3.2783505154639174,-0.7216494845360824);
\draw [line width=1.pt,dotted] (0.,0.)-- (9.,0.);
\draw[color=black] (3.3909652743723573,-1.900057868974655) node {1110};
\draw [->,line width=0.5pt] (3.12,-1.8) -- (2.44,-0.56);
\draw [fill=white,color=white] (2.57,-0.82) circle (2.0pt);
\draw [line width=1.pt,dotted] (1.008,-1.008)-- (9.,0.);
\draw [line width=1.pt,dash pattern=on 4pt off 4pt] (1.008,-1.008)-- (4.,0.);
\draw [line width=1.pt,dotted,color=black] (0.9936,1.4904)-- (9.,0.);
\draw [fill=white,color=white] (2.3560660445082555,0.) circle (3.0pt);
\draw [fill=white,color=red] (2.2059583097877464,1.2647181923326771) circle (3.0pt);
\draw [fill=white,color=white] (1.7056065292919547,-0.37544797814602765) circle (3.0pt);
\draw [fill=white,color=white] (2.9116852057434834,-0.3666515082254572) circle (3.0pt);
\draw [fill=white,color=white] (2.5329833601157303,0.) circle (3.0pt);
\draw [fill=white,color=white] (2.1,1.28) circle (8.0pt);
\draw [fill=black] (2.1,1.28) circle (.5pt);
\draw [fill=white,color=white] (2.1,0) circle (3pt);
\draw [line width=2.pt,color=black] (1.008,-1.008)-- (3.291581108829569,1.0626283367556473);
\draw [line width=2.pt,color=black] (0.9936,1.4904)-- (3.2783505154639174,-0.7216494845360824);
\draw [line width=1.pt,dash pattern=on 4pt off 4pt,color=black] (2.41927303465765,-0.5325443786982248)-- (2.,3.);
\draw [fill=white,color=white] (2.8,-1.2) circle (3.0pt);
\draw [fill=white,color=white] (2.,0.) circle (3.0pt);
\draw [fill=white,color=white] (2.,1.3030575539568348) circle (3.0pt);
\draw [fill=white,color=white] (2.,-0.6737967914438502) circle (3.0pt);
\draw [fill=white,color=white] (2.,-0.8828828828828829) circle (3.0pt);
\draw [fill=white,color=white] (2.,0.5160237667222576) circle (3.0pt);
\draw [fill=white,color=white] (2.,-0.10850783739650223) circle (3.0pt);
\draw [fill=white,color=white] (2.,-0.44025157232704404) circle (3.0pt);
\draw [line width=2.pt,color=black] (4.,0.)-- (2.,-2.);
\draw [line width=1.pt,color=black] (2.,3.)-- (2.,-2.);
\draw [fill=white] (0.,0.) circle (3.0pt);
\draw[color=ffqqqq] (-0.40991699406357224,0) node {0010};
\draw [fill=white] (2.,3.) circle (3.0pt);
\draw[color=ffqqqq] (2,3.2) node {0001};
\draw [fill=white] (4.,0.) circle (3.0pt);
\draw[color=ffqqqq] (4.266820753620636,0.28) node {0100};
\draw [fill=white] (2.,-2.) circle (3.0pt);
\draw[color=ffqqqq] (2,-2.230569370577779) node {1000};
\draw [fill=white] (9.,0.) circle (3.0pt);
\draw [fill=black] (9.,0.) circle (1.0pt);
\draw[color=ffqqqq] (9,0.2) node {$01{*}0$};
\draw [fill=white] (0.9936,1.4904) circle (3.0pt);
\draw[color=ffqqqq] (0.5,1.6) node {0011};
\draw [fill=white] (1.008,-1.008) circle (3.0pt);
\draw[color=ffqqqq] (0.5,-1.1) node {1010};
\draw [fill=white] (3.2783505154639174,-0.7216494845360824) circle (3.0pt);
\draw[color=ffqqqq] (3.3579141242120447,-0.9911512395660627) node {1100};
\draw [fill=white] (3.291581108829569,1.0626283367556473) circle (3.0pt);
\draw[color=ffqqqq] (3.5562210251739192,1.3720059968962772) node {0101};
\draw [fill=white] (2.3330966659797974,0.1935262759484402) circle (3.0pt);
\draw[color=ffqqqq] (2.7464678462462646,0.2) node {1111};
\draw [fill=ffffff] (2.41927303465765,-0.5325443786982248) circle (3.0pt);
\draw [fill=black] (2.41927303465765,-0.5325443786982248) circle (1.0pt);
\end{tikzpicture}
\caption[Points of concurrency of a $(3,3)$-grid.]{A $(3,3)$-grid
has six points of concurrency; two are shown here, as dotted circles.
(The grid points here are white and the grid lines are bold.
The six points of concurrency are where three lines (shown dashed or dotted here)
through exactly two of the grid points meet.
The other four points of concurrency are (0111), (1011), (1101) and ($100{*}$), where $*$ represents $-1$.
Thus each coordinate face has a point of concurrency, as do both coordinate lines which are not grid lines.)}
\label{GridOnTetrahedron2}
\end{figure}
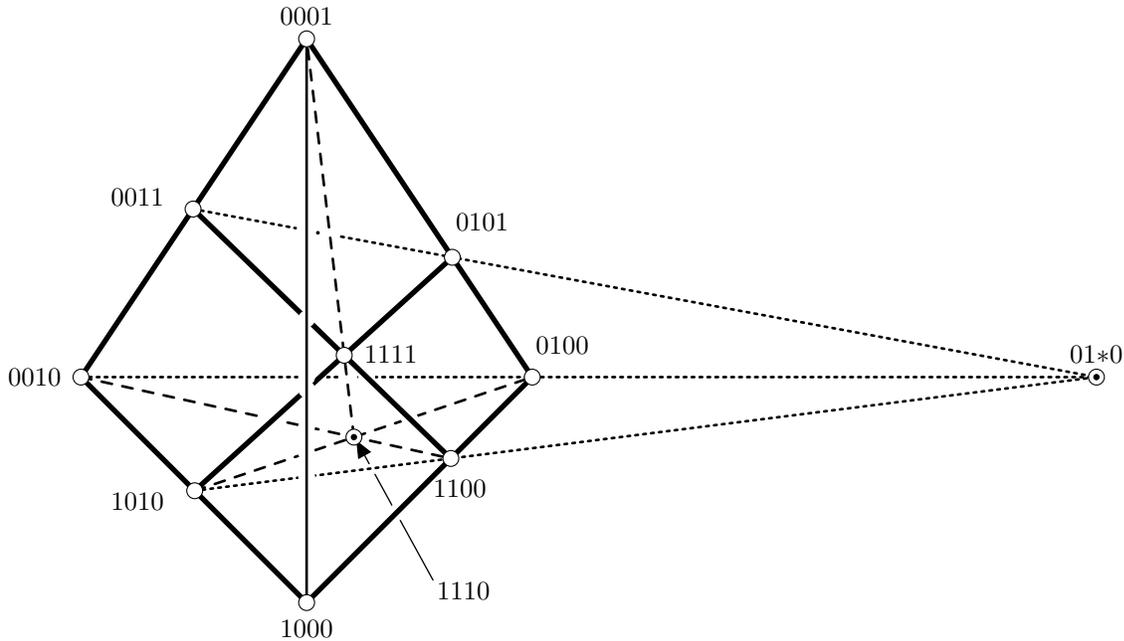

Likewise $100*$ is concurrent for the lines through the pairs
$\{0011, 1010\}$, $\{1100, 0101\}$ and $\{1000, 0001\}$, and
01$*$0 is concurrent for the lines through the pairs
$\{0011, 0101\}$, $\{1100, 1010\}$ and $\{0100, 0010\}$.
None of these 18 lines goes through a third grid point, hence they account for
18 of the 36 pairs of grid points. The six grid lines go through three grid points each,
and thus account for the remaining 18 pairs, so there are no other lines through 
exactly two grid points. 

Since $G$ is transitive on the grid points and acts on the set of 18 lines,
the number of these lines through each grid point is the same; let this number be $c$.
If we count the grid points on each of the 18 lines we get $18\cdot 2=36$.
If we count the number of these lines through grid points we thus must get 
$36=9c$, so $c=4$.
(The 18 lines are easy to visualize using Figure \ref{GridOnTetrahedron2}.
They are the lines between any two grid points which are not both on grid lines.
For example, the 4 lines through 1111 are the lines that also go through the coordinate vertices.
For 0011, they are the lines that also go through 1010, 1000, 0100 and 0101.)

Note that $G$ acts transitively on the 18 lines. This is because $G$ acts transitively on
the 9 grid points, and, for any grid point $p$, $G_p$ acts transitively on the four lines
through $p$. (For example, take $p=1111$ and note that 
$G_p$ acts transitively on the coordinate vertices, and it is the coordinate vertices
that define the 4 lines of the 18 that go through $p$.)

With the 18 lines in hand, it is easy to check that exactly 3 meet at each of the points
0111, 1011, 1101, 1110 and at $100*$ and 01$*$0 (also see Remark \ref{BrianchonRem}).
By picking one of the 18 lines, say the line $L$ through 0001 and 1111, shown as a dotted line
in Figure \ref{GridOnTetrahedron}, one can check $L$ is disjoint from 9 of the 18 lines,
and for the remaining 8, it meets 2 at 1110, 3 at 1111, and 3 at 0001. Thus whenever two of the 18 meet,
it is at either a point where 3 of the 18 are concurrent, or at a grid point (and at grid points 4 are concurrent).
Using the permutations $(xw)$ and $(yz)$ together with the element of $G$ given by 
one of the matrices above, we check that 0111, 1011, 1101, 1110, 100$*$ and 01$*$0
are part of a single orbit. Since there are no other points of triple concurrency,
these 6 must be the full orbit. 
Finally we note that 1101, 1011 and 01$*$0 are collinear and
1110, 0111 and 100$*$ are collinear (note that the difference, regarded as vectors,
of the first two points in each triple equals the third), 
but the two lines of collinearity are skew (since the determinant 
of the matrix two of whose rows are points on one line and two are points on the other
is nonzero; for example
\[
\det \left(
\begin{array}{cccc}
1  & 1  & 1 &  0\\
1  &  1 & 0 & 1 \\
1  & 0  &  1 &  1 \\
0  &  1 & 1  & 1 \\
\end{array}
\right)\neq 0).
\]
\end{proof}
Now we are in a position to define the set $D_4$ geometrically. 

\begin{definition}[$D_4$ configuration]\label{def:D4}\index{configuration! $D_4$}
A $D_4$ configuration is the set of points in a $(3,3)$-grid together with three
collinear points arising as intersection points (called {\it Brianchon points}; see Remark \ref{BrianchonRem})
of nongrid lines passing through pairs of grid points, as in Proposition \ref{GridColinearitiesProp}.
\end{definition}

\begin{example}\label{e:D4list}
In particular, the $(3,3)$-grid  
$$0010,\; 0011,\; 0001$$
$$1010,\; 1111,\; 0101$$
$$1000,\; 1100,\; 0100,$$
together with 
$$1101,\; 1011,\; 01{*}0$$
forms a $D_4$ configuration.
\end{example}

\begin{corollary}\label{cor:D4uniq}
All $D_4$ configurations are projectively equivalent.
\end{corollary}

\begin{proof}
Let $Z$ be a set of 12 points consisting of a $(3,3)$-grid and 3 collinear Brianchon points \index{points!Brianchon} for the grid.
By Corollary \ref{cor:3x3gridUniq}, we may assume the grid is 
$$0010,\; 0011,\; 0001$$
$$1010,\; 1111,\; 0101$$
$$1000,\; 1100,\; 0100.$$
Then by Proposition \ref{GridColinearitiesProp}, the remaining three points are either
$1101,\; 1011,\; 01{*}0$ or $0111,\;1110,\;100{*}$. So it is enough to show that
the 9 grid points plus the first three Brianchon points is projectively equivalent to
the 9 grid points plus the second three Brianchon points. The matrix 
$$
\begin{pmatrix}
   0 & 1 & 0 & 0\\
   1 & 0 & 0 & 0\\
   0 & 0 & 0 & 1\\
   0 & 0 & 1 & 0\\
\end{pmatrix}
$$
does this.
\end{proof}

\begin{example}\label{e:D4list2}
Another example of a $D_4$ configuration is shown in Figure \ref{3ptPerspective}; it consists of 
the eight corners of a unit cube, the center of the cube and the 3 points at infinity where 
parallel sets of four edges of the cube meet.
Thinking of the cube as having a corner at the origin and being aligned with the coordinate axes,
the projective coordinates of the corners are:
$0001, 1001, 0101, 0011, 0111, 1011, 1101, 1111$.
The center of the cube is $[1/2:1/2:1/2:1]$.
The points at infinity are $1000, 0100, 0010$.
(To get affine coordinates for the points not at infinity, just
drop the final coordinate.)
If one excludes the three points $0001, [1/2:1/2:1/2:1], 1111$ on the main diagonal of the cube, 
then the remaining 9 points form a $(3,3)$-grid where the grid lines are edges of the cube
(shown in Figure \ref{3ptPerspective} as dashed lines).
Note that each of the dropped points is collinear with three sets of two grid points.
(For example, the center of the cube is on three diagonals, in addition to the diagonal which contains all 
three dropped points.)
To visualize the 12 points of this $D_4$ configuration think of the 8 vertices of the cube, the center of the cube, and
the three centers of 3 point perspective of the cube (these are the points at infinity).
As discussed in Proposition \ref{D4factsProp}, a $D_4$ configuration has 
16 lines which contain exactly three of the 12 points and 18 lines which contain
exactly two of the 12 points. These are easy to see here.
The 16 lines consist of the 12 edges of the cube together with the four diagonals through the center of the cube, while the 18 lines 
consist of the two diagonals on each of the six faces,
the three lines through pairs of the three points at infinity, and the three lines through
the center of the cube and each point at infinity. 
\end{example}

\begin{figure}[ht!]
\begin{tikzpicture}[line cap=round,line join=round,>=triangle 45,x=1.0cm,y=1.0cm]
\clip(-4.496680397937174,0.75) rectangle (5.928312362917013,6.5);
\fill[line width=2.pt,color=black,fill=white] (-1.176350188308613,1.7059124529228467) -- (-0.8953798813051128,2.7315518487365438) -- (0.,3.230935550935551) -- (0.,2.) -- cycle;
\fill[line width=2.pt,color=black,fill=lightgray,fill opacity=0.30000000149011612] (0.,3.230935550935551) -- (0.6656059123409749,2.7359609200174266) -- (0.8751724137931034,1.7082758620689655) -- (0.,2.) -- cycle;
\fill[line width=2.pt,color=black,fill=gray,fill opacity=0.60000000149011612] (-1.176350188308613,1.7059124529228467) -- (-0.23559785204630612,1.546900694008833) -- (0.8751724137931034,1.7082758620689655) -- (0.,2.) -- cycle;
\draw [line width=0.4pt] (-4.,1.)-- (0.,2.);
\draw [line width=0.4pt] (0.,2.)-- (0.,6.);
\draw [line width=0.4pt] (0.,2.)-- (3.,1.);
\draw [line width=0.4pt,dash pattern=on 2pt off 2pt] (-4.,1.)-- (0.8751724137931034,1.7082758620689655);
\draw [line width=0.4pt,dash pattern=on 1pt off 2pt on 4pt off 4pt] (3.,1.)-- (-1.176350188308613,1.7059124529228467);
\draw [line width=0.4pt,dash pattern=on 1pt off 2pt on 4pt off 4pt] (-4.,1.)-- (0.,3.230935550935551);
\draw [line width=0.4pt,dash pattern=on 2pt off 2pt] (0.,3.230935550935551)-- (3.,1.);
\draw [line width=0.4pt,dash pattern=on 2pt off 2pt] (-1.176350188308613,1.7059124529228467)-- (0.,6.);
\draw [line width=0.4pt,dash pattern=on 1pt off 2pt on 4pt off 4pt] (0.,6.)-- (0.8751724137931034,1.7082758620689655);
\draw [line width=0.4pt,color=black] (-1.176350188308613,1.7059124529228467)-- (-0.8953798813051128,2.7315518487365438);
\draw [line width=0.4pt,color=black] (-0.8953798813051128,2.7315518487365438)-- (0.,3.230935550935551);
\draw [line width=0.4pt,color=black] (0.,3.230935550935551)-- (0.,2.);
\draw [line width=0.4pt,color=black] (0.,2.)-- (-1.176350188308613,1.7059124529228467);
\draw [line width=0.4pt,color=black] (0.,3.230935550935551)-- (0.6656059123409749,2.7359609200174266);
\draw [line width=0.4pt,color=black] (0.6656059123409749,2.7359609200174266)-- (0.8751724137931034,1.7082758620689655);
\draw [line width=0.4pt,color=black] (0.8751724137931034,1.7082758620689655)-- (0.,2.);
\draw [line width=0.4pt,color=black] (0.,2.)-- (0.,3.230935550935551);
\draw [line width=0.4pt,color=black] (-1.176350188308613,1.7059124529228467)-- (-0.23559785204630612,1.546900694008833);
\draw [line width=0.4pt,color=black] (-0.23559785204630612,1.546900694008833)-- (0.8751724137931034,1.7082758620689655);
\draw [line width=0.4pt,color=black] (0.8751724137931034,1.7082758620689655)-- (0.,2.);
\draw [line width=0.4pt,color=black] (0.,2.)-- (-1.176350188308613,1.7059124529228467);
\begin{scriptsize}
\draw [fill=black] (-4.,1.) circle (2.5pt);
\draw[color=black] (-4.137197888942202,1.3898073887022795) node {0100};
\draw [fill=black] (0.,2.) circle (2.5pt);
\draw[color=black] (-0.35,2.2345912848404637) node {0001};
\draw [fill=black] (3.,1.) circle (2.5pt);
\draw[color=black] (3.0165040400577405,1.3538591378027822) node {1000};
\draw [fill=black] (0.,6.) circle (2.5pt);
\draw[color=black] (0.4462041007436909,6.045105880187166) node {0010};
\draw [fill=black] (-1.176350188308613,1.7059124529228467) circle (2.5pt);
\draw[color=black] (-1.65,1.8391605249459944) node {0101};
\draw [fill=black] (0.8751724137931034,1.7082758620689655) circle (2.5pt);
\draw[color=black] (1.25,1.947005277644486) node {1001};
\draw [fill=black] (-0.23559785204630612,1.546900694008833) circle (2.5pt);
\draw[color=black] (-0.25,1.2) node {1101};
\draw [fill=black] (0.,3.230935550935551) circle (2.5pt);
\draw[color=black] (-0.35,3.4) node {0011};
\draw [fill=black] (-0.8953798813051128,2.7315518487365438) circle (2.5pt);
\draw[color=black] (-1.4,2.881659801031413) node {0111};
\draw [fill=black] (0.6656059123409749,2.7359609200174266) circle (2.5pt);
\draw[color=black] (1.05,2.91760805193091) node {1011};
\end{scriptsize}
\end{tikzpicture}
\caption[Representing the $D_4$ configuration as a cube in 3-point perspective.]{The $D_4$ configuration 
represented by a unit cube in 3-point perspective.
(Not visible: the back vertex point 1111, the center point $[\frac{1}{2}:\frac{1}{2}:\frac{1}{2}:1]$ 
and the orthogonal lines through the point 1111 along the three back edges. The 9 points shown, excluding the point 0001,
give a $(3,3)$-grid; the three grid lines in one ruling are shown with small dashes, each of the other three grid lines
is shown with dashes and dots.)}
\label{3ptPerspective}
\end{figure}
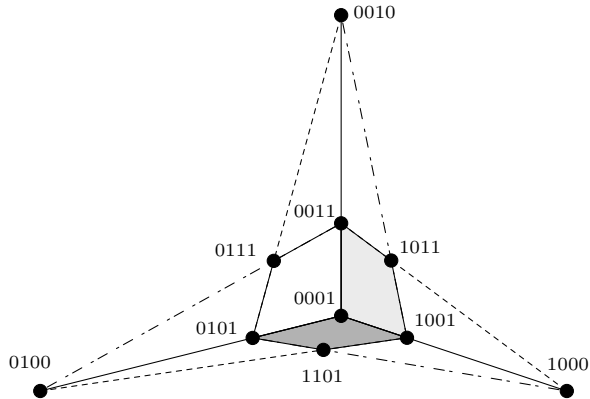

Some remarks are immediately in order. 

\begin{remark}\label{rem:onD4}\ 
\begin{itemize}
    \item[(a)] As noted in the proof of Corollary \ref{cor:D4uniq},
    it is possible to complete a fixed $(3,3)$-grid to a $D_4$ configuration in two different ways.
    In particular, to the $(3,3)$-grid given by the first 9 points listed in Example \ref{e:D4list}, we can append either
    the points $1101,\; 1011,\; 01{*}0$ as in the example, or we can append the points $1110,\; 0111,\; 100*$.
    \item[(b)] A $D_4$ configuration of points 
    contains multiple $(3,3)$-grids. Indeed, removing any subset of 3 collinear points from a $D_4$
    gives a $(3,3)$-grid. 
    \item[(c)] The urexample of a $D_4$ configuration 
    is given by the $D_4$ root system. In this guise it is best known in the following coordinates, see \cite{HMNT} (and see Figure \ref{3by3gridFig} for a $(3,3)$-grid chosen from among the following 12 points):
    $$1100,\; 1{*}00,\; 1010,\; 10{*}0,\; 1001,\; 100{*},$$
    $$0110,\; 01{*}0,\; 0101,\; 010{*},\; 0011,\; 001{*}.$$
     This $D_4$ has exactly $2$ points on each of the six edges of the coordinate tetrahedron;
    the two points on a given coordinate line together with the two coordinate vertices on that line
    give a harmonic set of four points.
    It is not hard to confirm directly that these 12 points are projectively equivalent to either set of 12 given in (a) above.
\end{itemize}
\end{remark}

\begin{remark}\label{BrianchonRem}
One way to see that the 18 lines through exactly two points of a grid meet in 6 sets of three concurrent lines
is to use Brianchon's Theorem.\index{Theorem! Brianchon} Given 6 distinct lines in the plane tangent to a smooth conic,
number the lines. Take the 6 points of intersection of lines which are consecutive (mod 6).
These 6 points in order form a hexagon;
then the theorem states that lines through opposite vertices of the hexagon are concurrent (see Figure~\ref{fig:BT}).
\definecolor{ffffff}{rgb}{0,0,0}
\definecolor{qqttcc}{rgb}{0,0,0}
\definecolor{ffqqtt}{rgb}{0,0,0}
\definecolor{uuuuuu}{rgb}{0,0,0}
\definecolor{ttttff}{rgb}{0,0,0}
\begin{figure}[ht]
\begin{center}
\begin{tikzpicture}[line cap=round,line join=round,>=triangle 45,x=2.0cm,y=2.0cm,scale=0.8]
\clip(-9.569511217222953,-6.543612310698059) rectangle (-3.6609177394627928,-3.2303464424023693);
\draw [rotate around={0.3360417744630027:(-6.798126845572247,-4.993432883898468)},line width=2.pt] (-6.798126845572247,-4.993432883898468) ellipse (2.815018897670958cm and 1.9793709914707096cm);
\draw [line width=1.pt,color=qqttcc,domain=-8.569511217222953:-3.6609177394627928] plot(\x,{(-126.39649988979318-22.026920927136047*\x)/1.542327896020538});
\draw [line width=1.pt,color=qqttcc,domain=-9.569511217222953:-3.6609177394627928] plot(\x,{(-185.3449546752288-15.157799133652802*\x)/22.69923908549066});
\draw [line width=1.pt,color=qqttcc,domain=-8.569511217222953:-4.0609177394627928] plot(\x,{(-116.90104300885878--1.2476788361142752*\x)/31.327092815790976});
\draw [line width=1.pt,color=qqttcc,domain=-9.569511217222953:-3.6609177394627928] plot(\x,{(--52.49753317428008--17.397594619681797*\x)/19.38953372025358});
\draw [line width=1.pt,color=qqttcc,domain=-9.569511217222953:-3.6609177394627928] plot(\x,{(--245.08371613966574--12.70283618533044*\x)/-25.56986626193938});
\draw [line width=1.pt,color=qqttcc,domain=-8.569511217222953:-3.6609177394627928] plot(\x,{(--93.69087782368888-10.878261305237814*\x)/-27.355101748223035});
\draw [line width=1.pt,color=ffqqtt, dotted,domain=-9.569511217222953:-4.6609177394627928] plot(\x,{(--11.877867477648163-0.6583442882810466*\x)/-3.3977442895957655});
\draw [line width=1.pt,color=ffqqtt,dotted,domain=-8.569511217222953:-4.6609177394627928] plot(\x,{(-11.2025592069763-2.1824403081407144*\x)/-0.6209486099552768});
\draw [line width=1.pt,color=ffqqtt,dotted,domain=-8.069511217222953:-4.6609177394627928] plot(\x,{(-20.134074126325896-1.521593410305858*\x)/2.1602360388163078});
\begin{scriptsize}
\draw [fill=black] (-5.41996369498286,-4.5459776553015425) circle (2.5pt);
\draw [fill=black] (-6.265768481927503,-3.9811772662442815) circle (2.5pt);
\draw [fill=black] (-7.509735878257461,-4.030721339878676) circle (2.5pt);
\draw [fill=black] (-8.817707984578625,-5.204321943582589) circle (2.5pt);
\draw [fill=black] (-6.88671709188278,-6.163617574384996) circle (2.5pt);
\draw [fill=black] (-5.349499839441153,-5.552314750184534) circle (2.5pt);
\draw [fill=ffffff] (-6.485189923251614,-4.752375059972431) circle (3pt);
\end{scriptsize}
\end{tikzpicture}
\end{center}
\caption[Tangents to a conic and Brianchon's Theorem.]{Brianchon's Theorem.}
\label{fig:BT}
\end{figure}
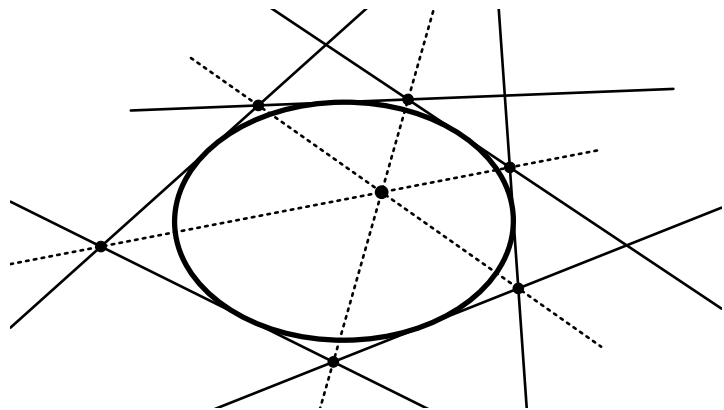

In our situation, we start with the 6 lines of a grid, numbered so that even numbered
lines are disjoint from each other.
Projecting the grid from a general point gives, by Lemma \ref{BrianchonLem},
a hexagon with an inscribed smooth conic, hence opposite vertices define 3 concurrent lines.
These lines are the images of the lines in $\PP^3$ through opposite vertices of the 
hexagon formed by the ordered grid lines, which have to be concurrent since
their images under a projection from a general point are concurrent.
There are 6 ways to order the 6 grid lines (up to cyclic permutations),
giving the 18 lines through exactly two grid points and
the 6 points of concurrency described in Proposition \ref{GridColinearitiesProp}. We will refer to these 6 points of concurrency as the {\it Brianchon points}\index{points!Brianchon} of the $(3,3)$-grid.
\end{remark}

\begin{example}\label{DiagActionEx}
An example showing how certain symmetries of the $(3,3)$-grid act on the points of concurrency
of lines through exactly two points of the grid will be useful later.
In each of the following triples, we give two points of the grid and the point of concurrency
on the line defined  by those two points:
$\{1000, 0011, 1011\}$,
$\{1111, 0100, 1011\}$,
$\{0001, 1010, 1011\}$;
$\{0010, 1111, 1101\}$,
$\{1100, 0001, 1101\}$,
$\{0101, 1000, 1101\}$;
$\{$0100, 0010, 01$*$0$\}$,
$\{$1010, 1100, 01$*$0$\}$,
$\{$0011, 0101, 01$*$0$\}$.

Let $\delta$ be the diagonal cyclic permutation on the grid shown in Figure \ref{GridOnTetrahedron},
taking the left column to the middle column to the right column,
and the bottom row to the middle row to the top row.
Thus, for example, $\delta$ takes $1000\mapsto 1111\mapsto0001\mapsto 1000$ and
$0011\mapsto 0100\mapsto 1010\mapsto 0011$.
Thus $\delta$ permutes the three concurrent lines through 1011, and hence fixes 1011.
Likewise, $\delta$ is the identity on the three collinear points of concurrency 1011, 1101 and 01$*$0.
In a similar way, we find that $\delta$ cyclically permutes the other three points of concurrency,
$0111\mapsto 1110\mapsto$ 100$*\mapsto 0111$.

As another example, note that $\tau=(xw)(yz)$ 
(the permutation which transposes the first and last coordinates, and transposes the middle two coordinates)
preserves the grid in Figure \ref{GridOnTetrahedron} (it induces a 180 degree rotation about the center of the grid)
but it transposes 1011 and 1101 while fixing 01$*$0.
\end{example}

\section{Intersection of quadric cones in \texorpdfstring{$\PP^3$}{P3}}

We recall in this section some remarks on the intersection of singular quadric surfaces. The goal is to point out results that are classical in the precise form that we will need in the arguments of the next chapters.

Quadric surfaces in $\PP^3$ correspond to points in a projective space $\PP^9$. There is a hypersurface $S$ of degree $4$ in $\PP^9$ whose points correspond to singular quadrics (cones). $S$ is defined by the determinant of the $4\times 4$
matrix associated to a quadric.The singular locus of $S$ corresponds to the set of quadrics which split in a union of  planes.

\begin{remark}\label{r. 4singcones} A pencil of quadrics corresponds to a line $\ell$ in $\PP^9$, which can meet the hypersurface $S$ in at most $4$ points, unless it lies inside $S$. Thus, either the pencil is entirely composed by singular quadrics, or it contains at most $4$ singular quadrics.
\end{remark}

There are examples of pencils of quadrics $\ell$ that are entirely composed by singular quadrics. Besides the case of pencils generated by two cones with a commone vertex $P$, there are examples of pencils of cones whose vertices move.
For one, consider the pencil generated by the quadrics of equations $x^2+y^2-z^2$ and $yw+zw$ in the coordinates of $\PP^3$. Note that the vertices of the cones in the pencil
move along a line contained in the base locus.

On the other hand, we have:

\begin{remark} \label{r. Bertinicones} Let $C,C'\subset\PP^3$ be two cones. Assume $C$ is irreducible, and call $P$  its vertex.
If $P\notin C'$,
then the general element of the pencil spanned by $C,C'$ is smooth.

Indeed $C,C'$ cannot share a vertex. Hence, if the pencil contains infinitely many cones, the vertices must lie in the base locus $B$ of the pencil, by Bertini's theorem. Since $P$ is the unique singular point of $C$, and $P\notin B$, then
$C$ is smooth along $B$. Then the general quadric in the pencil is smooth along $B$, and the claim follows.
\end{remark}

Next we consider pencils generated by two cones whose intersection contains a conic $\sigma$.
 
\begin{proposition} \label{p. 2cones}
Let $\sigma$ be an irreducible conic in $\PP^3$ and let $P,P'$ be points of $\PP^3$ which do not lie in the plane of $\sigma$. Assume that the line $L$ spanned by $P,P'$ does not meet $\sigma$. Call $C,C'$ the cones over $\sigma$ with vertices $P,P'$ resp.

Then the intersection $C\cap C'$ is the union of $\sigma$ with another irreducible conic $\sigma'$, which misses  $L$. 
\end{proposition}
\begin{proof}
Since $L$ does not meet $\sigma$, then the projection from a point $Q\in L$,
$Q\neq P$, cannot send $P$ to $\sigma$. Thus $L\setminus \{P\}$ does not meet $C$. Similarly 
$L\setminus \{P'\}$ does not meet $C'$. Hence $L$ misses $C\cap C'$. 

The intersection $C\cap C'$ consists of a curve  $\sigma\cup\sigma'$. Since $\sigma'$ is directly linked to $\sigma$ by a complete intersection of two quadrics, then it is a conic. Since all lines in $C$ contain $P$, which cannot lie in $\sigma'$, then $\sigma'$ cannot be reducible, and the claim follows.
 \end{proof} 

\begin{remark} \label{r. 3singcones}
With the notation of Proposition \ref{p. 2cones}, the pencil $\ell$ of quadrics spanned by $C,C'$ is not contained in the hypersurface $S$ which parameterizes singular quadrics in $\PP^9$, by Remark \ref{r. Bertinicones}. Thus, by Remark \ref{r. 4singcones}, $\ell$ contains at most 4  singular quadrics.

Indeed one can see that there are exactly three singular quadrics in $\ell$. Namely the union $H$
of the planes spanned by $\sigma$ and $\sigma'$ contains the base locus of the pencil, thus it is a quadric of the pencils. Hence the line $\ell\in\PP^9$ which represents the pencil meets the hypersurface $S$ of singular quadrics in a finite set containing the points corresponding to $C,C',H$. Since $H$ is a singular point of $S$, there are no other intersections. 
\end{remark}


\section{Arrangements of lines}
We collect here some information on arrangements of lines used in Chapter \ref{chap. non iso and real}. We begin recalling the well known Sylvester-Gallai Theorem and its dual version, see \cite[Chapter 10]{proofs}.
\begin{theorem}[Sylvester-Gallai Theorem]\label{thm:SG}\index{Theorem! Sylvester-Gallai}  
In any configuration of a finite number of points in the real projective plane there is a line passing through exactly $2$ configuration points, unless all points are collinear.
\end{theorem}

\begin{theorem}[Sylvester-Gallai dual version]\label{thm:SG dual}\index{Theorem! Sylvester-Gallai dual}
For any arrangement of a finite number of lines in the real projective plane, there is a point where exactly $2$ arrangement lines intersect, unless all lines are concurrent.
\end{theorem}
It is well known that the Sylvester-Gallai Theorem fails in the complex projective plane. A counterexample is given by the following arrangement.
\begin{definition}[Fermat arrangement]\label{def:Fermat arrangement}\index{arrangement! Fermat}
   The Fermat arrangement in $\PP^2$ consists of zeroes of all linear factors of a polynomial
   $$F_n=(x^n-y^n)(y^n-z^n)(z^n-x^n).$$
\end{definition}
In the Fermat arrangement for $n\neq 3$ there are $3$ points where $n$ lines meet, $n^2$ points where $3$ lines meet and no other intersection points among the arrangement lines occur. For $n=3$ we obtain an arrangement known as the dual Hesse arrangement\index{arrangement! Hesse}. In this arrangement there are $12$ points, where exactly $3$ lines intersect and these are the only intersection points of arrangement lines.

There are certain line arrangements involving combinatorial symmetries. We recall that a $(p_q,n_k)$
point-line configuration consists of $p$ lines and $n$ points distributed so that on every line there are $k$ points and through each point $q$ lines pass. For example the dual Hesse configuration is a $(9_4,12_3)$ configuration.


\section{Basic algebraic and geometric definitions used in this book} \label{basic facts}

Finally, in this section we recall some basic definitions and facts that will be used without further comment in this book.

Recall $K$ is a field, typically the complex numbers.
The ring $R=K[\PP^n]=K[x_0,\ldots,x_n]$ will have the standard grading. 
Given a graded module $M$, often $M=I$ or $M=R/I$ for a homogeneous ideal $I\subseteq R$,
we have its Hilbert function\index{Hilbert function} $h_M(t)=\dim_K [M]_t$, where 
$[M]_t$ is the homogeneous component of $M$ of degree $t$. In case $I=I(X)$ is the ideal of  a projective subscheme $X\subseteq\PP^n$, we write
$h_X$ for $h_{R/I}$.

Also, given a homogeneous ideal $I$ of $R$,  
the minimal graded free resolution\index{minimal graded free resolution} of $R/I$ has the form
\[0 \rightarrow \bigoplus_{j \in \mathbb{N}} R(-j)^{\beta_{p,j}(R/I)}
\rightarrow \cdots \rightarrow \bigoplus_{j \in \mathbb{N}} R(-j)^{\beta_{1,j}(R/I)}
\rightarrow R \rightarrow R/I \rightarrow 0\]
where $R(-j)$ is the ring $R$ with its grading shifted by
$j$, and $\beta_{i,j}(R/I)=\dim_{K} {\rm Tor}_i^R(R/I, K)_j$ is called the $(i,j)$-th {\em graded Betti number}\index{Betti numbers} of $R/I$.
We denote by $\beta_{R/I}$ the Betti table of  the standard $K$-algebra $R/I$; its   $(i,j)$-th entry is the number $\beta_{i, i+j}(R/I)$.
When $I = I(X)$ is the saturated ideal of a scheme $X$ we write $\beta_X$ instead of $\beta_{R/I(X)}$.

We say that the ring $R/I$ is {\it Cohen-Macaulay}\index{ring!Cohen-Macaulay} if its depth is equal to its Krull dimension. If $X$ is a subscheme of $\PP^n$ and $I = I(X)$ then $R/I$ is Cohen-Macaulay if and only if $p$ (in the above resolution) is equal to the codimension of $X$. If this holds, we say that $X$ is {\it arithmetically Cohen-Macaulay (ACM)}\index{arithmetically!Cohen-Macaulay}. Throughout this book we will be in the setting of $I = I(X)$ for some subscheme $X$, so we can take the latter homological criterion as our working definition of a Cohen-Macaulay ring.
In particular, if $X$ is a zero-dimensional scheme in $\PP^n$ (e.g., a finite set of points in $\PP^n$), then $X$ is ACM and $p=n$. (Note that non-ACM varieties can also have $p=n$. For example, for $X$ the disjoint union of two lines in $\PP^3$, we have $p=3$ and $X$ is not ACM since the codimension is $2 \neq p$.) The Cohen-Macaulay property is preserved under proper hyperplane sections (meaning that the hyperplane contains no component of $X$), including passing from the coordinate ring of a finite set of points to the Artinian reduction.

In the special case where $R/I(X)$ is Cohen-Macaulay and $\beta_{p,j}$ is nonzero for only one $j$, and for that $j$ we have $\beta_{p,j} = 1$ (i.e., the last free module has rank 1), we say that $R/I(X)$ is a {\it Gorenstein ring}\index{ring!Gorenstein} and  $X$ is {\it arithmetically Gorenstein}\index{arithmetically!Gorenstein} (sometimes abbreviated AG). Specializing further, if  the number of minimal generators (i.e., the rank $\sum_j \beta_{1,j}$ of the  first free module) of $I(X)$ is equal to the codimension of $X$, we say that $R/I(X)$ is a {\it complete intersection}. We also say that $X$ is a complete intersection. \index{complete intersection} This forces the last rank to be 1, so $R/I(X)$ is Gorenstein.  But being Gorenstein does not imply being a complete intersection in codimension $\geq 3$ (although it does in codimension 2).

An arithmetically Gorenstein set of points $X$ has the so-called {\it Cayley-Bacharach (CB) Property}. \index{Cayley-Bacharach Property} This means that the Hilbert function of $X \backslash \{P\}$ remains the same no matter which $P \in X$ we remove. Of course Cayley-Bacharach does not imply arithmetically Gorenstein, but see \cite{DGO1985} for a result in this direction.

For a point $P \in \PP^n$ and an integer $m \geq 1$ let $mP$ denote the scheme defined by the saturated ideal $I_P^m=I(P)^m$. If $X \subset \PP^n$ is a closed subscheme,  we say that $X$ {\it admits an unexpected hypersurface of degree $d$ and multiplicity $m$}\index{unexpected!hypersurface} if, for a {\it general} point $P \in \PP^n$, we have
\[
\dim_K [I(X) \cap I_P^m]_d = \dim_K [I(X \cup mP)]_d > \dim_K [I(X)]_d - \binom{m-1+n}{n}.
\]
In the case where $d=m$ we say that $X$ {\it admits an unexpected cone of degree $d$}. \index{unexpected!cone} Typically for us, $X$ will be a finite set of points.

Unexpected hypersurfaces are related to the {\it Lefschetz properties}\index{Lefschetz properties} (see for instance \cite{HMNT}). We omit the details but recall the definitions. Let $M$ be a finitely generated graded module (e.g., an algebra $R/I$ where $I$ is a homogeneous ideal). Let $L$ be a linear form. Then multiplication by $L$ induces a homomorphism $\times L: [M]_{t}\to [M]_{t+1}$ from any homogeneous component $[M]_{t}$ to the next $[M]_{t+1}$. Similarly, for any $k \geq 1$, multiplication by $L^k$ induces a homomorphism $\times L^k: [M]_t\to[M]_{t+k}$. We say that $M$ has the {\it Weak Lefschetz Property (WLP)} \index{Lefschetz Property!Weak} if $\times L$ has maximal rank in all degrees when $L$ is general. We say that $M$ has the {\it Strong Lefschetz Property (SLP)} \index{Lefschetz Property!Strong} if $\times L^k$ has maximal rank in all degrees when $L$ is general, for all $k \geq 1$.  SLP implies WLP but not conversely.

\chapter{Weddle and Weddle-like varieties}\label{Chap.Weddle}

The main theme of this work is the study of finite sets $Z\subset\PP^3$ of 
points whose projection to $\PP^2$ from a general point is 
a complete intersection of a curve of degree $a$ with a curve of degree $b\geq a$ 
(i.e., we study {\it $(a,b)$-geproci}\index{geproci} sets).
An $(a,b)$-geproci set $Z$ has the property that
there are surfaces of degrees $a$ and $b$ containing $Z$,
where both surfaces are cones whose vertex is a general point.
However, given a finite set $Z\subset\PP^3$ and a degree $d$,
it is often true that there is not even one degree $d$ cone which contains $Z$
when the vertex is a general point.
In such cases there still can be a nonempty locus of points
occurring as the vertex of a degree $d$ cone containing $Z$.
Studying such vertex loci is of interest in its own right \cite{WEDDLE, EMCH, moore}
but will also be useful in our study of the geproci property
and essential for our classification of $(3,b)$-geproci sets in Section \ref{sec:small_a_b}. 
The classical example of such a vertex locus is for the case of
$d=2$ with $Z$ being a set of 6 points in $\PP^3$ in 
linear general position. In this case the vertex locus is a surface
now known as the Weddle surface, first studied in 1850 by 
T.\ Weddle\index{Weddle} \cite{WEDDLE}. 
Later work by A.\ Emch\index{Emch} \cite{EMCH} 
extended the notion to analogous vertex locus varieties for 
larger $d$ for sets $Z\subset\PP^3$ of other cardinalities. 

In this chapter we take a new approach to studying vertex loci (including giving
a new proof that the classical Weddle surface has degree 4), and we 
generalize the work of Emch.  
For example, we introduce the notion of the {\it $d$-Weddle locus}\index{Weddle! locus} 
and the {\it $d$-Weddle scheme}\index{Weddle! scheme} associated to any finite set of 
points in projective space and we compute these objects in several situations,
with a focus on points in $\PP^3$.

In particular, we find situations where a finite set of points $Z\subset \PP^3$ fails to be geproci, but there is a locus of points such that the projection of $Z$ from a point of this locus is a complete intersection. This is a special case of the more general phenomenon of $d$-Weddle loci; most of our work in this chapter is in the greater generality of $d$-Weddle loci, but we remark along the way when this special projection property holds. For example, a set $Z$ of 6 general points 
in $\PP^3$ is not a geproci set, but in
Proposition \ref{proci(6)} 
we describe the explicit locus of points from 
which the projection of $Z$ is a complete intersection 
of a cubic and a smooth conic. 
Similarly, in Theorem \ref{no 222} we show that there are no sets of 8 
points in $\PP^4$ whose general projection is a complete intersection of 
quadrics in $\PP^3$, but in Remark \ref{proj8fromP4} we show that for a 
general set $Z$ of 8 points in $\PP^4$ there is a somewhat mysterious 
arithmetically Cohen-Macaulay\index{arithmetically! Cohen-Macaulay, ACM} 
(ACM) curve $C$ of degree 7 in $\PP^4$ such that the projection 
of $Z$ from any point of $C$ {\it is} a complete intersection of quadrics.  See also Remark \ref{projCI}.


\section{The \texorpdfstring{$d$}{d}-Weddle locus for a finite set of points in projective space}

Let $Z$ be a set $P_1,\ldots,P_r\in\PP^n$ of distinct points.
Let $P$ be a point not in $Z$ and let $d$ be a positive integer.
Then we have the graded ideal\index{graded ideal} $I=I(Z)\cap I(P)^d\subset R=\field[\PP^n]=\field[x_0,\ldots,x_n]$ 
in the homogeneous coordinate ring\index{homogeneous! coordinate ring} $R$ of $\PP^n$
(so $R$ is a polynomial ring in variables $x_i$ over the ground field $\field$, 
with the standard grading\index{standard grading}
by degree). The homogeneous component\index{homogeneous! component} $[I]_t$ of $I$
in degree $t$ is the $\field$-vector space span of all forms of degree $t$
that vanish on $Z$ and vanish to order $d$ (or more) at $P$.
Let $\delta(Z,P,d,t)=\dim_\field [I]_t$. For fixed $Z$, $t$ and $d$, 
$\delta(Z,P,d,t)$ achieves its minimum as a function of $P$
on a nonempty open set of points $P$ (i.e., when $P$ is general); denote this
minimum by $\delta(Z,d,t)$. 

\begin{definition}
The {\it $d$-Weddle locus}\index{locus! $d$-Weddle}\index{Weddle! $d$ locus} of $Z$ is the closure of the
set of points $P \in \PP^n\setminus Z$ (if any) for which $\delta(Z,P,d,d)>\delta(Z,d,d)$.
\end{definition}

Geometrically, since $\delta(Z,P,d,d)$ is the dimension of a family of cones with vertex $P$
containing $Z$, then the $d$-Weddle locus is related to projections. 
Let $Z$ be a finite set of points in $\PP^n$ and let $H\cong\PP^{n-1}$
be a general hyperplane. Let $Z_P$ be the projection of $Z$ to $H$ from $P$.
There is a bijection between the family of cones of degree $d$ containing $Z$ with vertex $P$
and the family of hypersurfaces in $H$ of degree $d$ containing $Z_p$.
Then the dimension of $I(Z_P)\subset \field[H]$ in degree $d$ achieves its minimum
$\delta(Z,P,d,d)$ for a general point~$P$. Hence $\dim_\field [I(Z_P)]_d\geq\delta(Z,P,d,d)$
and the $d$-Weddle locus is the closure of the points $P\not\in Z$ 
for which $\dim_\field [I(Z_P)]_d>\delta(Z,P,d,d)$.
The $d$-Weddle locus gives a 
measure of the extent of such unexpected 
behavior in degree $d$ (in this case, degree 2). 
This is closely analogous to the so-called non-Lefschetz 
locus\index{locus! non-Lefschetz} \cite{BMMN}.

\begin{example} Let $Z\subset\PP^3$ be a set of 6 points 
in linear general position\index{linear general position} (LGP\index{LGP}; see 
Remark \ref{ProjEqVsCombEq}(f) for the precise definition). 
Then the $2$-Weddle locus is the classical Weddle surface, i.e., 
the closure of the locus of points 
$P\not\in Z$ in $\PP^3$ that are the vertices of  quadric cones 
in $\PP^3$ containing $Z$. Equivalently, the classical 
Weddle surface\index{Weddle! surface, classical} is 
the closure of the locus of points $P\not\in Z$ 
from which $Z$ projects to a set $Z_P\subset\PP^2$ contained in a conic. 
The general projection does not have this property and indeed 
for a point $P$ of the classical Weddle surface we have
$\dim [I(Z_P)]_2=1 > 0=\delta(Z,P,2,2)$. 
\end{example}

In Example \ref{ClassicalWeddleResult} we give a new proof that the classical 
Weddle surface has degree 4. We will also consider sets $Z$ of 6 points, not 
necessarily in LGP, for which the surface still has degree 4 and see that it 
can come in very different forms depending on the geometry of $Z$. In the 
remaining sections we will consider finite sets of points in $\PP^3$
of cardinality other than 6, as was done by 
Emch (but one of his results has a minor mistake which we correct). However, we will go beyond Emch 
by considering finite sets $Z\subset\PP^n$ of various cardinalities with $n\geq3$. 


\section{The \texorpdfstring{$d$}{d}-Weddle scheme and two approaches to finding it}
\index{Weddle! scheme}

In this section we give two different approaches to finding Weddle loci which lead to 
equivalent scheme structures, which allow us to Weddle schemes.
One approach is in terms of an interpolation matrix,
where what is of interest is the kernel of the matrix. 
The other is in terms of a matrix coming from Macaulay duality, where what is of interest is the cokernel of the matrix.\index{Macaulay duality}
The interpolation matrix is conceptually simpler to define but 
initially leads to larger (and thus less computationally efficient) matrices, while the Macaulay duality matrix is  sometimes conceptually easier to apply but less easy to write down explicitly, so having both is useful (see, e.g., 
Example \ref{BigWeddleExample} and Remark \ref{compare matrices}).

The following notation will be helpful.
Given a monomial $M=x_0^{i_0}\cdots x_n^{i_n}$, let
\begin{equation}\label{notation1}
\begin{array}{lcl}
e_M &=&\displaystyle i_0!\cdots i_n!\\[4pt]
c_M &=&\displaystyle\frac{(i_0+\cdots+i_n)!}{i_0!\cdots i_n!}=\frac{d!}{e_M}.
\end{array}
\end{equation}

\subsection{Interpolation matrices and $d$-Weddle schemes}\index{interpolation matrix}
Let $Z\subset\PP^n$ be a finite set of distinct points $P_1,\ldots,P_r$.
Let $s_1,\ldots,s_r$ be positive integers.
Then $I=I(P_1)^{s_1}\cap\cdots\cap I(P_r)^{s_r}$ is the graded ideal 
generated by all forms that vanish to order at least $s_i$ at each point $P_i$.
The ideal $I$ is saturated\index{saturation} and defines a subscheme
$X\subset\PP^n$ known as a {\it fat point subscheme}\index{points!fat},
denoted $X=s_1P_1+\cdots+s_rP_r$, so we can write $I=I(X)$.
Since our ground field has characteristic 0, $[I]_t$ is given for each degree $t$ by 
the kernel of a matrix $\Lambda(X,t)$ with entries in $\field$, known as 
the {\it interpolation matrix} for $X$ in degree $t$. We now 
explain this in more detail.

For a single point $P_1$, the vector space $[I(P_1)]_t$ is the span of all forms
$F\in [R]_t$ which vanish at $P_1$; i.e., such that $F(P_1)=0$.
If we enumerate the monomials of degree $t$ as $M_1,\ldots,M_N$, 
so $N=\binom{n+t}{n}$, then $F=\sum a_jM_j$ for some coefficients $a_j\in \field$.
Thus $F\in[I(P_1)]_t$ if and only if the coefficient vector
$(a_1,\ldots,a_N)$ has inner product 0 with 
the vector $(M_1(P_1),\ldots,M_N(P_1))$
of monomials of degree $t$ evaluated at $P_1$.
(We make sense of the value $M_j(P_1)$ by fixing a representative 
$(P_{01},\ldots,P_{n1})$ of the point $P_1=(P_{01}:\ldots:P_{n1})$.)
Thus the interpolation matrix $\Lambda(P_1,t)=(M_1(P_1),\ldots,M_N(P_1))$
has a single row and $N$ columns. (Since this vector cannot be zero,
the kernel in this case has dimension $N-1$.)

Now consider $X=2P_1$. Now a form $F$ is in $I(X)$ if and only if it 
is singular at $P_1$.
This means $F\in I(X)$ if and only if $(\nabla(F))(P_1)=0$;
i.e., the gradient of $F$ evaluated at $P_1$ is zero.
The gradient $\nabla(F)$ of $F$ is 
$$\nabla(F)=\left(\frac{\partial F}{\partial x_0},\ldots,\frac{\partial F}{\partial x_n}\right).$$
Thus $F=\sum a_jM_j$ is in $[I(X)]_t$ if and only if 
$(\nabla(F))(P_1)=\sum a_j(\nabla(M_j))(P_1)=0$. So with respect to 
the partials $\frac{\partial }{\partial x_i}$ and the monomials
$M_j$, the matrix
$\Lambda(2P_1,t)$ is the $(n+1)\times\binom{n+t}{n}$-matrix 
whose entries are 
$$\Lambda(2P_1,t)_{ij}=\frac{\partial M_j}{\partial x_i}(P_1).$$

More generally, given a monomial $m=x_0^{i_0}\cdots x_n^{i_n}$,
let $\partial_m$ denote the differential operator 
$$\partial_m=\frac{\partial^{i_0}}{\partial x_0^{i_0}}
\cdots\frac{\partial^{i_n}}{\partial x_n^{i_n}}.$$
(We can extend this to arbitrary forms by linearity:
given an enumeration $m_i$ of the monomials of degree $k$ 
and a form $F=\sum a_im_i$, we denote by $\partial_F$
the operator $\sum a_i\partial_{m_i}$.)

We now define the $k$-th order gradient $\nabla_k$ as
the vector whose components are the $\partial_{m_i}$.
Thus $\nabla_1=\nabla$ and we define $\nabla_0=1$.
Generalizing the case of multiplicity 1 and 2, for $X=(k+1)P_1$
we have that $F\in [I(X)]_t$ if and only if $(\nabla_k(F))(P_1)=0$, so 
$\Lambda((k+1)P_1,t)$ is the $\binom{n+k}{n}\times\binom{n+t}{n}$-matrix 
whose entries are 
$$\Lambda((k+1)P_1,t)_{ij}=\frac{\partial M_j}{\partial m_i}(P_1)=\partial_{m_i}M_j(P_1).
$$

Thus the interpolation matrix $\Lambda(X,t)$ for $X=s_1P_1+\cdots+s_rP_r$
is a $(\sum_i\binom{n+s_i-1}{n})\times \binom{n+t}{n}$ block matrix
whose blocks are $\Lambda(s_iP_i,t)$, arranged vertically; i.e.,
$$\Lambda(X,t)=
\left(
\begin{array}{c}
\Lambda(s_1P_1,t)\\
\vdots\\
\Lambda(s_rP_r,t)\\
\end{array}
\right).$$

For a fixed $Z=P_1+\cdots+P_r$ and an additional fat point $dP$, 
$\Lambda(Z+dP,t)$ is an $(r+\binom{n+d-1}{n})\times \binom{n+t}{n}$ matrix. 
For $t\geq d$, its entries are either
scalars, or monomials of degree $t-d+1$ in the coordinates of $P$.
If we regard the coordinate variables as being the coordinates of $Q$,
so $Q=[x_0:\ldots:x_n]$, then the entries of $\Lambda(Z+dQ,t)$
are monomials in the variables $x_i$, 
and we can specialize the matrix $\Lambda(Z+dQ,t)$ to $\Lambda(Z+dP,t)$
for a particular point $P=[p_0:\ldots:p_n]$ by plugging each $p_i$ 
into the variable $x_i$.
For fixed $Z$, $d$ and $t$, 
the rank of $\Lambda(Z+dP,t)$ achieves its maximum $\rho(Z,d,t)$ 
when $P$ is general for which we have 
$$\rho(Z,d,t)=\binom{n+d-1}{n}-\delta(Z,d,t).$$

The matrix relevant to the $d$-Weddle locus is 
$$\Lambda(Z+dQ,d)=
\left(
\begin{array}{c}
\Lambda(P_1,d)\\
\vdots\\
\Lambda(P_r,d)\\
\Lambda(dQ,d)\\
\end{array}
\right)=
\left(
\begin{array}{c}
\Lambda(Z,d)\\
\Lambda(dQ,d)\\
\end{array}
\right),$$
where $\Lambda(Z,d)$ is the submatrix consisting of
the first $r$ rows of $\Lambda(Z+dQ,d)$ and thus
corresponds to the points $P_i$ of $Z=P_1+\cdots+P_r$, 
while $\Lambda(dQ,d)$ consists of the remaining $\binom{n+d-1}{n}$ rows and 
corresponds to $dQ$.
With respect to an enumeration $m_i$ for the monomials of degree $d-1$,
the rows of $\Lambda(dQ,d)$ correspond to the monomials $m_i$.
With respect to an enumeration $M_j$ for the monomials of degree $d$,
the columns of $\Lambda(Z+dQ,d)$ correspond to the monomials $M_j$.
In particular, the entry for the row for $P_i$ and the column for $M_j$ is 
$M_j(P_i)$.
The entry for the row for $m_i$ and the column for $M_j$ is 
$(\partial_{m_i}M_j)(Q)$. This is 0 if $m_i\not| M_j$, but
if $m_i|M_j$, then for some variable $x_{j_i}$ we have $m_ix_{j_i}=M_j=x_0^{i_0}\cdots x_n^{i_n}$,
so the entry is $e_{M_j}x_{j_i}$ where we recall that $e_{M_j}=i_0!\cdots i_n!$.
(See Example \ref{BigWeddleExample}.)

Thus the $d$-Weddle locus is the closure of the locus of points $P\not\in Z$
such that $\rank(\Lambda(Z+dP,d))<\rho(Z,d,d)$.
This locus is defined by the ideal $I_{\rho(Z,d,d)}(\Lambda(Z+dQ,d))$
of $\rho(Z,d,d)\times\rho(Z,d,d)$ minors\index{minors! ideal of} of $\Lambda(Z+dQ,d)$.
Note that the largest $s$ such that $I_s(\Lambda(Z+dQ,d))\neq(0)$
is $s=\rho(Z,d,d)$, and the $d$-Weddle locus for $Z$ is
the zero locus of $I_s(\Lambda(Z+dP,d))$ for this $s$.
We call this ideal the {\it $d$-Weddle ideal}\index{Weddle! ideal} for $Z$.

\begin{definition}
Let $Z\subset\PP^n$ be a finite set of points.
The $d$-Weddle scheme for $Z$ is the scheme defined by
the saturation of the $d$-Weddle ideal.
\end{definition}

\begin{remark}\label{r. ReducedIntMat}
Computing ideals of minors is more efficient for smaller matrices, so it can be useful to reduce the size of 
the interpolation matrix.
It is easy to check that row and column operations do not change ideals of minors. Row operations applied to $\Lambda(Z,d)$
can be used to bring $\Lambda(Z,d)$ 
first into block form as follows
$$\Lambda_{Z,d}'=
\left(
\begin{array}{c|c}
id_\alpha & *\\
\hline
0 & 0\\
\end{array}
\right)
$$
where $id_\alpha$ is the $\alpha\times\alpha$ identity 
matrix with
$\alpha$ being the rank of $\Lambda(Z,d)$
(so the number of zero rows above is $r-\alpha$),
and where the 0's represent 0 matrices of appropriate sizes
and $*$ represents a submatrix whose entries 
are not of interest (since they will eventually be zeroed out
with column operations). Row operations applied to
$$\left(
\begin{array}{c}
\Lambda_{Z,d}'\\
\Lambda(dQ,d)\\
\end{array}
\right)$$
reduce it to the form
$$\left(
\begin{array}{c|c}
id_\alpha & *\\
\hline
0 & 0\\
\hline
0 & \Lambda_{Z+dQ,d}'\\
\end{array}
\right)
$$
and now column operations give
$$\left(
\begin{array}{c|c}
id_\alpha & 0\\
\hline
0 & 0\\
\hline
0 & \Lambda_{Z+dQ,d}'\\
\end{array}
\right),
$$
where $\Lambda_{Z+dQ,d}'$ is a 
$\binom{n+d-1}{n}\times (\binom{n+d}{n}-\alpha)$
matrix of linear forms (as shown in Example \ref{BigWeddleExample}).
Since the ideal of $q\times q$ minors is the same for this last matrix and for $\Lambda(Z+dQ,d)$ for every $q$,
we see $I_q(\Lambda(Z+dQ,d))$ is $(1)$ for $q\leq\alpha$.
If for $q=\alpha+1$ we have $I_q(\Lambda(Z+dQ,d))=(0)$,
then $\Lambda_{Z+dQ,d}'$ is a zero matrix and the $d$-Weddle
locus is empty. Otherwise, for $q>\alpha$ we have
$$I_q(\Lambda(Z+dQ,d))=
I_{q-\alpha}(\Lambda_{Z+dQ,d}')+
I_{q-\alpha+1}(\Lambda_{Z+dQ,d}')+\cdots.$$
Thus for the largest $q$ such that $I_{q-\alpha}(\Lambda_{Z+dQ,d}')$
is nonzero we have
$I_q(\Lambda(Z+dQ,d))=I_{q-\alpha}(\Lambda_{Z+dQ,d}')$
so we see that this $q$ is $q=\rho(Z,d,d)$, hence
$$I_{\rho(Z,d,d)}(\Lambda(Z+dQ,d))=
I_{\rho(Z,d,d)-\alpha}(\Lambda_{Z+dQ,d}')$$
both define the $d$-Weddle scheme. 
\end{remark}

\subsection{Macaulay duality and $d$-Weddle schemes}
The $d$-Weddle scheme can also be obtained using 
Macaulay duality\index{Macaulay duality} 
(also called inverse systems\index{inverse systems}; see \cite{EI}),
as we now explain.

Consider the polynomial rings $R=\field[x_0,\ldots,x_n]$ and
$R^*=\field[\partial_{x_0},\ldots,\partial_{x_n}]$, where formally we think of
the differential operators $\partial_{x_i}$ as independent indeterminates. 
Thus given a point $P=[p_0:\ldots:p_n]\in \PP^n$,
we get the dual form $L_P=p_0x_0+\cdots+p_nx_n\in [R]_1$ 
and the element $\partial_{L_P}=\sum p_i\partial_{x_i}\in [R^*]_1$.

Macaulay duality comes from regarding $R^*$ as acting on $R$.
Given a point $P=[p_0:\ldots:p_n]\in \PP^n$ and $0\leq k\leq t$, the annihilator
of $[I(P)^k]_t$ under this action is $[(\partial_{L_P}^{t-k+1})]_t$, hence we have 
$$[I(P)^k]_t\cong [R^*/(\partial_{L_P}^{t-k+1})]_t,$$
where $\cong$ denotes $\field$-vector space isomorphism.
More generally, given points $P_i\in\PP^n$ and a point $P\in\PP^n$, and
integers $0\leq k_i\leq t$ and $0\leq d\leq t$, we have
$$[I(P_1)^{k_1}\cap\cdots\cap I(P_r)^{k_r}]_t\cong [R^*/(\partial_{L_{P_1}}^{t-k_1+1},\ldots,\partial_{L_{P_r}}^{t-k_r+1})]_t$$
and
$$[I(P_1)^{k_1}\cap\cdots\cap I(P_r)^{k_r}\cap I(P)^{d}]_t\cong [R^*/(\partial_{L_{P_1}}^{t-k_1+1},\ldots,\partial_{L_{P_r}}^{t-k_r+1},\partial_{L_P}^{t-d+1})]_t.$$

We clearly now have the exact sequence
\begin{equation}\label{OrigMacDualSeq}
\left[\frac{R^*}{(\partial_{L_{P_1}}^t,\ldots,\partial_{L_{P_r}}^t)}\right]_{d-1} \xrightarrow{\times\partial_{L_P}^{t-d+1}} 
\left[\frac{R^*}{(\partial_{L_{P_1}}^t,\ldots,\partial_{L_{P_r}}^t)}\right]_t
\to 
\left[\frac{R^*}{(\partial_{L_{P_1}}^t,\ldots,\partial_{L_{P_r}}^t,\partial_{L_P}^{t-d+1})}\right]_t
\to 0
\end{equation}
(where $\times\partial_{L_P}^{t-d+1}$ denotes the map given by multiplication by $\partial_{L_P}^{t-d+1}$), where
we have
$$[R]_{d-1}\cong [R^*]_{d-1}=[R^*/(\partial_{L_{P_1}}^t,\ldots,\partial_{L_{P_r}}^t)]_{d-1},$$
$$[R^*/(\partial_{L_{P_1}}^t,\ldots,\partial_{L_{P_r}}^t)]_t\cong [I(P_1)\cap\cdots\cap I(P_r)]_t$$ 
and
$$[R^*/(\partial_{L_{P_1}}^t,\ldots,\partial_{L_{P_r}}^t,\partial_{L_P}^{t-d+1})]_t\cong [I(P_1)\cap\cdots\cap I(P_r)\cap I(P)^d]_t.$$
In particular, as a vector space, $[I(P_1)\cap\cdots\cap I(P_r)\cap I(P)^d]_t$ is isomorphic to the cokernel of the map $\times\partial_{L_P}^{t-d+1}$.

For the $d$-Weddle locus we want $t=d$, in which case the sequence \eqref{OrigMacDualSeq} is
\begin{equation}
\left[\frac{R^*}{(\partial_{L_{P_1}}^d,\ldots,\partial_{L_{P_r}}^d)}\right]_{d-1} \xrightarrow{\times\partial_{L_P}} 
\left[\frac{R^*}{(\partial_{L_{P_1}}^d,\ldots,\partial_{L_{P_r}}^d)}\right]_d
\to 
\left[\frac{R^*}{(\partial_{L_{P_1}}^d,\ldots,\partial_{L_{P_r}}^d,\partial_{L_P})}\right]_d
\to 0.\tag{\ref*{OrigMacDualSeq}$^\prime$} \label{OrigMacDualSeqPrime}
\end{equation}
In preparation for comparing what this sequence gives to what the interpolation matrix gives, 
it is helpful to rewrite it as the exact sequence 
\begin{equation}\label{NewMacDualSeq}
([R^*]_0)^r \oplus[R^*]_{d-1}\xrightarrow{D \oplus (\times\partial_{L_P})} [R^*]_d\to   
[R^*/(\partial_{L_{P_1}}^d,\ldots,\partial_{L_{P_r}}^d,\partial_{L_P})]_d\to 0
\end{equation}
where
$$([R^*]_0)^r \xrightarrow{D} [R^*]_d$$
is the map $v=(a_1,\ldots,a_r)\in ([R^*]_0)^r\mapsto D(v)=a_1\partial_{L_{P_1}}^d+\cdots+a_r\partial_{L_{P_r}}^d$,
and
$$[R^*]_{d-1} \xrightarrow{\times\partial_{L_P}} [R^*]_d$$
is the multiplication map 
$w\in [R^*]_{d-1}\mapsto (\times\partial_{L_P})(w)=w\partial_{L_P}$,
hence
$$(D \oplus (\times\partial_{L_P}))(v\oplus w)=D(v)+(\times\partial_{L_P})(w).$$

So now $[I(P_1)\cap\cdots\cap I(P_r)\cap I(P)^d]_d$ is isomorphic to the vector space cokernel of the map $D \oplus (\times\partial_{L_P})$.
If we regard $[R^*]_{d-1}$ as being the sum $\oplus_m [R^*]_0$ over all monomials $m$ of degree $d-1$ and
$[R^*]_d$ as being the sum $\oplus_M [R^*]_0$ over all monomials $M$ of degree $d$, then 
$$([R^*]_0)^r\oplus [R^*]_{d-1} \xrightarrow{D \oplus (\times\partial_{L_P})} [R^*]_d$$
can (in terms of the bases of monomials $m$ and $M$) be written as a matrix map $T=T(Z,dP)$
\begin{equation}\label{MatrixT}
([R^*]_0)^r\bigoplus \oplus_m [R^*]_0 \xrightarrow{T=[T_1 | T_2]} \oplus_M [R^*]_0,
\end{equation}
where $T_1$ is the matrix for $D$ giving the map
$([R^*]_0)^r\xrightarrow{T_1} \oplus_M [R^*]_0$ and
$T_2$ is the matrix for $\times\partial_{L_P}$ giving the map
$[R^*]_{d-1}\xrightarrow{T_2} \oplus_M [R^*]_0$.

The entries of $T_1$ are scalars given in terms of the coordinates of the points
$P_i=[p_{0i}:\ldots:p_{ni}]$.
The entry of $T_1$ corresponding to the point $P_i$ and the monomial $M$ is
$(T_1)_{M,i}=c_MM(P_i)$, where
$c_M$ comes (via the Binomial Theorem\index{Theorem! Binomial}) from the expansion $\partial_{L_{P_i}}^d=(p_{0i}\partial_{x_0}+\cdots+p_{ni}\partial_{x_n})^d=\sum_Mc_MM(P_i)\partial_M$.
Recall that $c_M$ is the binomial coefficient corresponding  to the monomial $M$
(hence for $M=x_0^{i_0}\cdots x_n^{i_n}$, we have $c_M=\frac{(i_0+\cdots+i_n)!}{i_0!\cdots i_n!}=d!/e_M$).
Finally, $M(P_i)$ is the value of $M$ at $P_i$ (so plug $p_{ji}$ into each variable $x_j$
appearing in $M$).

The entry of $T_2$ corresponding to the monomials $m$ and $M$ is
$(T_2)_{M,m}=0$ unless $mx_i=M$, and then $(T_2)_{M,m}=p_i$,
since $\times\partial_{L_p}: \partial_m\mapsto \partial_{mL_p}=\sum_i p_i\partial_{mx_i}$.

Note that we have 
\begin{align*}
\binom{d+n}{n}-\rank\Lambda(Z+dP,d)=&\dim \ker \Lambda(Z+dP,d)\\
=\dim\coker T(Z+dP) =&\binom{d+n}{n}-\rank T(Z+dP,d)
\end{align*}
since both the kernel and cokernel are isomorphic to
$[I(P_1)\cap\cdots\cap I(P_r)\cap I(P)^d]_d$,
and hence the interpolation matrix $\Lambda(Z+dP,d)$ has the same rank as
the Macaulay duality matrix $T(Z+dP)$.

The rank of $T(Z+dP)$ achieves its maximum $\tau(Z,d)$ 
when $P$ is general; indeed $\tau(Z,d)=\rho(Z,d,d)$.
Thus the $d$-Weddle locus is the closure of the locus of points $P\not\in Z$
such that $\rank(\tau(Z+dP))<\tau(Z,d)$.
This locus is defined by the ideal $I_{\tau(Z,d)}(T(Z,dQ))$
of $\tau(Z,d)\times\tau(Z,d)$ minors of $T(Z,dQ)$
for $Q=[x_0:\ldots:x_n]$.
Note that the largest $s$ such that $I_s(T(Z,dQ))\neq(0)$
is $s=\tau(Z,d)$, and the $d$-Weddle locus for $Z$ is
the zero locus of $I_s(T(Z,dQ))$ for this $s$.

What is true but not obvious is that 
$I_s(T(Z,dQ))$ is equal to the $d$-Weddle ideal defined above; i.e.,
it is not obvious that
$I_{\tau(Z,d)}(T(Z,dQ))=I_{\rho(Z,d)}(\Lambda(Z+dQ,d))$.
We will see this below, and also that the matrix $T'(Z,dQ)$ for $\times\delta_{L_Q}$
in sequence \eqref{OrigMacDualSeqPrime} again gives the same ideal.

Since $T'(Z,dQ)$ gives a map to the quotient space
$[R^*/(\partial_{L_{P_1}}^d,\ldots,\partial_{L_{P_r}}^d)]_d$,
we need to make some choices to write $T'(Z,dQ)$ down explicitly.
After reordering the points $P_i$ if need be,
we may assume $\partial_{L_{P_1}}^d,\ldots,\partial_{L_{P_\alpha}}^d$
give a basis for the vector space 
$\langle\partial_{L_{P_1}}^d,\ldots,\partial_{L_{P_r}}^d\rangle$
spanned by $\partial_{L_{P_1}}^d,\ldots,\partial_{L_{P_r}}^d$,
where $\alpha=\dim \langle\partial_{L_{P_1}}^d,\ldots,\partial_{L_{P_r}}^d\rangle$.
Likewise we may assume that the elements $\partial_{M_j}$ for the degree $d$ monomials $M_1,\ldots,M_\beta$, together with
$\partial_{L_{P_1}}^d,\ldots,\partial_{L_{P_\alpha}}^d$, give a basis for
$[R^*]_d$, where $\beta=\binom{d+n}{n}-\alpha$. 

We now want to write down $T'(Z,dQ)$ in terms of the monomial basis $\partial_{m_i}$
for $[R^*]_{d-1}$ and the basis $\partial_{M_1},\ldots,\partial_{M_\beta}$ modulo 
$\langle\partial_{L_{P_1}}^d,\ldots,\partial_{L_{P_r}}^d\rangle$ 
for $[R^*/(\partial_{L_{P_1}}^d,\ldots,\partial_{L_{P_r}}^d)]_d$.
To do so we first write down $T(Z,dQ)$ in terms of the basis 
$\partial_{m_i}$ for $[R^*]_{d-1}$ and the basis 
$$\partial_{L_{P_1}}^d,\ldots,\partial_{L_{P_\alpha}}^d,\partial_{M_1},\ldots,\partial_{M_\beta}$$ 
for $[R^*]_d$. This is just $S^{-1}T(Z,dQ)$, where $S$ is the matrix
whose columns are 
$$\partial_{L_{P_1}}^d,\ldots,\partial_{L_{P_\alpha}}^d,M_1,\ldots,M_\beta$$
written in terms of the original monomial basis $\partial_{M_i}$.

Thus $T(Z,dQ)$ in terms of the basis
$\partial_{L_{P_1}}^d,\ldots,\partial_{L_{P_\alpha}}^d,M_1,\ldots,M_\beta$
is the block matrix
\begin{equation}\label{BlockMatrix}
S^{-1}T(Z,dQ)=\left(
\begin{array}{ccccc}
id_\alpha & | & * & | & **\\
\hline
0 &| & 0 & | & T'(Z,dQ)\\
\end{array}
\right),
\end{equation}
where $id_\alpha$
is the $\alpha\times\alpha$ identity matrix, $T'(Z,dQ)$ is the 
$\beta\times(\binom{d+n-1}{n}-r)$ matrix representing the map $\times\partial_{L_Q}$
in sequence \eqref{OrigMacDualSeqPrime}, $*$ is an $\alpha\times(r-\alpha)$ matrix and 
$**$ is an $\alpha\times(\binom{n+d-1}{n}-r)$ matrix, but the entries
of $*$ and $**$ will not really concern us. 

Since $T$ and $S^{-1}T$ represent the same map, just with respect to different bases,
we have $I_q(T)=I_q(S^{-1}T)$ for the ideals of $q\times q$ minors for all $q$.
By column reduction (i.e., row reduction applied to the transpose), we obtain 
from $S^{-1}T$ the matrix
$$T''(Z,dQ)=\left(
\begin{array}{ccccc}
id_\alpha & | & 0 & | & 0\\
\hline
0 &| & 0 & | & T'(Z,dQ)\\
\end{array}
\right)
,$$
so (suppressing $Z$ and $dQ$) we also have $I_q(T)=I_q(S^{-1}T)=I_q(T'')$ for every $q$.

For $q\leq \alpha$, these are the unit ideal $(1)$.
For $q>\alpha$, we have $I_q(T'') = I_{q-\alpha}(T')+\cdots+ I_q(T')$.
If for $q=\alpha+1$ we have $I_q(T)=I_q(T'')=(0)$, then 
$T'$ is the zero matrix, hence $\partial_{L_Pm}$ is
in $\langle \partial_{L_{P_1}}^d,\ldots,\partial_{L_{P_r}}^d\rangle$
for all points $P$. This happens if and only if
$\langle \partial_{L_{P_1}}^d,\ldots,\partial_{L_{P_r}}^d\rangle=[R^*]_d$,
which in turn means $[I(P_1)\cap\cdots \cap I(P_r)]_d=(0)$ and hence the $d$-Weddle locus 
is empty. Suppose $T'$ is not the 0 matrix. Then for some $q>\alpha$ we have
$I_{q-\alpha}(T')\neq0$ but $I_{q-\alpha+1}(T')=(0)$, hence
$I_{q+1}(T)=I_{q+1}(T'')=I_{q-\alpha+1}(T')=(0)$,
so the zero locus defined by each of the ideals 
$I_q(T)=I_q(T')=I_q(T'')=I_{q-\alpha}(T')\neq0$ is the 
$d$-Weddle locus.

In fact we now will show that $I_q(T(Z,dQ))=I_q(\Lambda(Z+dQ,d))$
for all $q$, so the matrices $\Lambda(Z+dQ,d)$, $T(Z,dQ)$ and $T'(Z,dQ)$
all define the same $d$-Weddle ideal and so can be used to
obtain the $d$-Weddle scheme.

In preparation for the proof that
$I_q(T(Z,dQ))=I_q(\Lambda(Z+dQ,d))$, 
we make some observations.
First, note that $T(Z,dQ)$ and the transpose $N=\Lambda(Z+dQ,d)^t$
of $\Lambda(Z+dQ,d)$ 
have the same size; both are $\binom{n+d}{n}\times(r+\binom{n+d-1}{n})$
matrices. For both, the entries of the first $r$ columns are 
scalars and the entries of the remaining $\binom{n+d-1}{n}$ 
columns are scalar multiples of the variables $x_i$.

We can compare the entries of the two matrices.
Recall $T_1$ is the submatrix of $T=T(Z,dQ)$ consisting of the first $r$
columns of $T$, and $T_2$ is the submatrix 
consisting of the remaining columns of $T$.
Likewise, let $N_1$ be the submatrix of $N$ 
consisting of the first $r$ columns of $N$, and 
let $N_2$ be the submatrix consisting of the remaining columns of $N$.

The entry $(N_1)_{ij}$ for row $i$ and column $j$ of $N_1$ is
the value $M_i(P_j)$ of the monomial $M_i$ at $P_j$.
The entry $(T_1)_{ij}$ is $c_{M_i}M_i(P_j)=d!M_i(P_j)/e_{M_i}$.
The entry $(N_2)_{ij}$ for row $i$ and column $j$ of $N_2$ is
$\partial_{m_j}M_i(Q)$. This is 0 if $m_j\not| M_i$ and it is
$e_{M_i}x_{k_{ij}}$ if $m_j|M_i$ where $x_{k_{ij}}=M_i/m_j$.
The entry $(T_2)_{ij}$ is 0 if $m_j\not| M_i$ and it is
$x_{k_{ij}}$ if $m_j|M_i$ where $x_{k_{ij}}=M_i/m_j$.
Thus 
$$(T_1)_{ij}=c_{M_i}(N_1)_{ij}=d!(N_1)_{ij}/e_{M_i}$$ 
and 
$$(N_2)_{ij}=e_{M_i}(T_2)_{ij}.$$

Since $N$ and $T$ are matrices of the same size,
we can speak of corresponding minors.
A minor of $T$ is the determinant of some
square submatrix $A$ of $T$.
The corresponding minor for $N$ is the determinant
of the matrix $B$ whose entries occupy the same locations in
$N$ as do those of $A$ in $T$.

\begin{proposition}\label{MacDuality=Interp}
Given a finite set of points $Z\subset\PP^n$ and a degree $d$,
let $A$ be a minor of $T=T(Z,dQ)$, 
coming from a given choice of $s$ rows and $s$ columns of $T$.
Assume that the rows correspond to $\partial_{M_{i_j}}$ for monomials $M_{i_1},\ldots,M_{i_s}$, and that
$j$ of the chosen columns come from $T_1$.
Let $B$ be the corresponding minor of $N=(\Lambda(Z+dP,d))^t$. 
Then 
$$B=\frac{e_{M_{i_1}}\cdots e_{M_{i_s}}}{(d!)^j}A$$
and thus $I_s(T(Z,dQ))=I_s(\Lambda(Z+dP,d))$.
\end{proposition}

\begin{proof}
The minor $A$ is a signed sum of products of $s$ entries of $T$,
where no two of the $s$ entries come from the same row or column.
Let $\pi$ be one of these products.
Let $j$ be the number of entries in $\pi$ which come from $T_1$; let's say those entries 
were chosen from rows corresponding to $M_{i_1},\ldots, M_{i_j}$.
The other entries come from $T_2$ in rows corresponding to 
$M_{i_{j+1}},\ldots, M_{i_s}$.

The corresponding product for $N$ is 
$$\frac{e_{M_{i_{j+1}}}\cdots e_{M_{i_s}}}{c_{M_{i_1}}\cdots c_{M_{i_j}}}\pi=
\frac{e_{M_{i_{j+1}}}\cdots e_{M_{i_s}}}{\frac{(d!)^j}{e_{M_{i_1}}\cdots e_{M_{i_j}}}}\pi=
\frac{e_{M_{i_1}}\cdots e_{M_{i_s}}}{(d!)^j}\pi$$
using the notation introduced in display \eqref{notation1}.
Thus all terms in the minor $A$ are multiplied by the same factor
$\frac{e_{M_{i_1}}\cdots e_{M_{i_s}}}{(d!)^j}$ to get the terms for $B$.
\end{proof}

\begin{example}\label{BigWeddleExample}
Consider the case of $n=3$ and degree $d=3$
with $Z$ consisting of the points 
$P_1=[1:0:0:0]$,
$P_2=[0:1:0:0]$,
$P_3=[0:0:1:0]$,
$P_4=[0:0:0:1]$,
$P_5=[1:1:1:1]$,
$P_6=[2:3:5:7]$.
We will give the matrices $N=\Lambda(Z+3Q,3)^t$ and $T=T(Z,3Q)$ with respect to the following
bases of $[R]_d$ and $[R]_{d-1}$.
For $[R]_3$ we have
$M_1=x_0^3$, 
$M_2=x_0^2x_1$, 
$M_3=x_0^2x_2$, 
$M_4=x_0^2x_3$, 
$M_5=x_0x_1^2$, 
$M_6=x_0x_1x_2$,
$M_7=x_0x_1x_3$, 
$M_8=x_0x_2^2$, 
$M_9=x_0x_2x_3$, 
$M_{10}=x_0x_3^2$, 
$M_{11}=x_1^3$, 
$M_{12}=x_1^2x_2$,
$M_{13}=x_1^2x_3$, 
$M_{14}=x_1x_2^2$, 
$M_{15}=x_1x_2x_3$, 
$M_{16}=x_1x_3^2$, 
$M_{17}=x_2^3$, 
$M_{18}=x_2^2x_3$, 
$M_{19}=x_2x_3^2$,
$M_{20}=x_3^3$, and
for $[R]_2$ we have 
$m_1=x_0^2$, 
$m_2=x_0x_1$, 
$m_3=x_0x_2$, 
$m_4=x_0x_3$, 
$m_5=x_1^2$, 
$m_6=x_1x_2$,
$m_7=x_1x_3$, 
$m_8=x_2^2$, 
$m_9=x_2x_3$, 
$m_{10}=x_3^2$.

For $N$, row $i$ corresponds to the basis element $M_i$,
column $j$, for $1\leq j\leq 6$, corresponds to the point $P_j$,
column $j$, for $6< j\leq 16$, corresponds to the basis element
$m_{j-6}$.
Entry $N_{i,j}$, for $1\leq j\leq 6$, is $M_i(P_j)$, while for
$6< j\leq 16$, it is
$\partial_{m_{j-6}}M_i$.
So for example
$N_{2,6}=M_2(P_6)=x_0^2x_1([2:3:5:7])=2^23=12$, and
$N_{11,11}=\partial_{m_5}M_{11}=\partial_{x_1^2}(x_1^3)=6x_1$.
Here is the entire matrix (with $N_{2,6}$ and $N_{11,11}$ shown boxed):
$$\tiny
N=
\left(
\begin{array}{cccccccccccccccc}
1 & 0 & 0 & 0 & 1 & 8 & 6x_0 & 0 & 0 & 0 & 0 & 0 & 0 & 0 & 0 & 0\\
0 & 0 & 0 & 0 & 1 & \fbox{12} & 2x_1 & 2x_0 & 0 & 0 & 0 & 0 & 0 & 0 & 0 & 0\\
0 & 0 & 0 & 0 & 1 & 20 & 2x_2 & 0 & 2x_0 & 0 & 0 & 0 & 0 & 0 & 0 & 0\\
0 & 0 & 0 & 0 & 1 & 28 & 2x_3 & 0 & 0 & 2x_0 & 0 & 0 & 0 & 0 & 0 & 0\\
0 & 0 & 0 & 0 & 1 & 18 & 0 & 2x_1 & 0 & 0 & 2x_0 & 0 & 0 & 0 & 0 & 0\\
0 & 0 & 0 & 0 & 1 & 30 & 0 & x_2 & x_1 & 0 & 0 & x_0 & 0 & 0 & 0 & 0\\
0 & 0 & 0 & 0 & 1 & 42 & 0 & x_3 & 0 & x_1 & 0 & 0 & x_0 & 0 & 0 & 0\\
0 & 0 & 0 & 0 & 1 & 50 & 0 & 0 & 2x_2 & 0 & 0 & 0 & 0 & 2x_0 & 0 & 0\\
0 & 0 & 0 & 0 & 1 & 70 & 0 & 0 & x_3 & x_2 & 0 & 0 & 0 & 0 & x_0 & 0\\
0 & 0 & 0 & 0 & 1 & 98 & 0 & 0 & 0 & 2x_3 & 0 & 0 & 0 & 0 & 0 & 2x_0\\
0 & 1 & 0 & 0 & 1 & 27 & 0 & 0 & 0 & 0 & \fbox{$6x_1$} & 0 & 0 & 0 & 0 & 0\\
0 & 0 & 0 & 0 & 1 & 45 & 0 & 0 & 0 & 0 & 2x_2 & 2x_1 & 0 & 0 & 0 & 0\\
0 & 0 & 0 & 0 & 1 & 63 & 0 & 0 & 0 & 0 & 2x_3 & 0 & 2x_1 & 0 & 0 & 0\\
0 & 0 & 0 & 0 & 1 & 75 & 0 & 0 & 0 & 0 & 0 & 2x_2 & 0 & 2x_1 & 0 & 0\\
0 & 0 & 0 & 0 & 1 & 105 & 0 & 0 & 0 & 0 & 0 & x_3 & x_2 & 0 & x_1 & 0\\
0 & 0 & 0 & 0 & 1 & 147 & 0 & 0 & 0 & 0 & 0 & 0 & 2x_3 & 0 & 0 & 2x_1\\
0 & 0 & 1 & 0 & 1 & 125 & 0 & 0 & 0 & 0 & 0 & 0 & 0 & 6x_2 & 0 & 0\\
0 & 0 & 0 & 0 & 1 & 175 & 0 & 0 & 0 & 0 & 0 & 0 & 0 & 2x_3 & 2x_2 & 0\\
0 & 0 & 0 & 0 & 1 & 245 & 0 & 0 & 0 & 0 & 0 & 0 & 0 & 0 & 2x_3 & 2x_2\\
0 & 0 & 0 & 1 & 1 & 343 & 0 & 0 & 0 & 0 & 0 & 0 & 0 & 0 & 0 & 6x_3\\
\end{array}
\right).$$
And here for reference is the transpose of $\Lambda_{Z+3Q,3}'$:
$$
\frac{1}{4}
\left(
\begin{array}{cccccccccc}
8x_1-16x_2+8x_3 &  8x_0 &  -16x_0 &  8x_0 &  0 &  0 &  0 &  0 &  0 &  0\\
-2x_1-6x_2 &  -2x_0+8x_1 &  -6x_0 &  0 &  8x_0 &  0 &  0 &  0 &  0 &  0\\
10x_1-18x_2 &  10x_0+4x_2 &  -18x_0+4x_1 &  0 &  0 &  4x_0 &  0 &  0 &  0 &  0\\
22x_1-30x_2 &  22x_0+4x_3 &  -30x_0 &  4x_1 &  0 &  0 &  4x_0 &  0 &  0 &  0\\
30x_1-38x_2 &  30x_0 &  -38x_0+8x_2 &  0 &  0 &  0 &  0 &  8x_0 &  0 &  0\\
50x_1-58x_2 &  50x_0 &  -58x_0+4x_3 &  4x_2 &  0 &  0 &  0 &  0 &  4x_0 &  0\\
78x_1-86x_2 &  78x_0 &  -86x_0 &  8x_3 &  0 &  0 &  0 &  0 &  0 &  8x_0\\
25x_1-33x_2 &  25x_0 &  -33x_0 &  0 &  8x_2 &  8x_1 &  0 &  0 &  0 &  0\\
43x_1-51x_2 &  43x_0 &  -51x_0 &  0 &  8x_3 &  0 &  8x_1 &  0 &  0 &  0\\
55x_1-63x_2 &  55x_0 &  -63x_0 &  0 &  0 &  8x_2 &  0 &  8x_1 &  0 &  0\\
85x_1-93x_2 &  85x_0 &  -93x_0 &  0 &  0 &  4x_3 &  4x_2 &  0 &  4x_1 &  0\\
127x_1-135x_2 &  127x_0 &  -135x_0 &  0 &  0 &  0 &  8x_3 &  0 &  0 &  8x_1\\
155x_1-163x_2 &  155x_0 &  -163x_0 &  0 &  0 &  0 &  0 &  8x_3 &  8x_2 &  0\\
225x_1-233x_2 &  225x_0 &  -233x_0 &  0 &  0 &  0 &  0 &  0 &  8x_3 &  8x_2\\\end{array}
\right).$$

For $T$, we have 
$T_{i,j}=c_{M_i}M_i(P_j)$ 
for $1\leq j\leq 6$, and
$T_{i,j}=0$ if $m_{j-6}$ does not divide $M_i$, and it is
$T_{i,j}=M_i/m_{j-6}$ if $m_{j-6} | M_i$.
So for example
$$T_{2,6}=c_{M_2}M_2(P_6)=\frac{3!}{2!1!}(x_0^2x_1)([2:3:5:7])=36,$$
and $T_{11,11}=M_{11}/m_5=x_1^3/(x_1^2)=x_1$.
Here is the entire matrix (with $T_{2,6}$ and $T_{11,11}$ shown boxed):
$$\tiny
T=\left(
\begin{array}{cccccccccccccccc}
1 & 0 & 0 & 0 & 1 & 8 & x_0 & 0 & 0 & 0 & 0 & 0 & 0 & 0 & 0 & 0\\ 
0 & 0 & 0 & 0 & 3 & \fbox{36} & x_1 & x_0 & 0 & 0 & 0 & 0 & 0 & 0 & 0 & 0\\ 
0 & 0 & 0 & 0 & 3 & 60 & x_2 & 0 & x_0 & 0 & 0 & 0 & 0 & 0 & 0 & 0\\ 
0 & 0 & 0 & 0 & 3 & 84 & x_3 & 0 & 0 & x_0 & 0 & 0 & 0 & 0 & 0 & 0\\ 
0 & 0 & 0 & 0 & 3 & 54 & 0 & x_1 & 0 & 0 & x_0 & 0 & 0 & 0 & 0 & 0\\ 
0 & 0 & 0 & 0 & 6 & 180 & 0 & x_2 & x_1 & 0 & 0 & x_0 & 0 & 0 & 0 & 0\\ 
0 & 0 & 0 & 0 & 6 & 252 & 0 & x_3 & 0 & x_1 & 0 & 0 & x_0 & 0 & 0 & 0\\ 
0 & 0 & 0 & 0 & 3 & 150 & 0 & 0 & x_2 & 0 & 0 & 0 & 0 & x_0 & 0 & 0\\ 
0 & 0 & 0 & 0 & 6 & 420 & 0 & 0 & x_3 & x_2 & 0 & 0 & 0 & 0 & x_0 & 0\\ 
0 & 0 & 0 & 0 & 3 & 294 & 0 & 0 & 0 & x_3 & 0 & 0 & 0 & 0 & 0 & x_0\\ 
0 & 1 & 0 & 0 & 1 & 27 & 0 & 0 & 0 & 0 & \fbox{$x_1$} & 0 & 0 & 0 & 0 & 0\\ 
0 & 0 & 0 & 0 & 3 & 135 & 0 & 0 & 0 & 0 & x_2 & x_1 & 0 & 0 & 0 & 0\\ 
0 & 0 & 0 & 0 & 3 & 189 & 0 & 0 & 0 & 0 & x_3 & 0 & x_1 & 0 & 0 & 0\\ 
0 & 0 & 0 & 0 & 3 & 225 & 0 & 0 & 0 & 0 & 0 & x_2 & 0 & x_1 & 0 & 0\\ 
0 & 0 & 0 & 0 & 6 & 630 & 0 & 0 & 0 & 0 & 0 & x_3 & x_2 & 0 & x_1 & 0\\ 
0 & 0 & 0 & 0 & 3 & 441 & 0 & 0 & 0 & 0 & 0 & 0 & x_3 & 0 & 0 & x_1\\ 
0 & 0 & 1 & 0 & 1 & 125 & 0 & 0 & 0 & 0 & 0 & 0 & 0 & x_2 & 0 & 0\\ 
0 & 0 & 0 & 0 & 3 & 525 & 0 & 0 & 0 & 0 & 0 & 0 & 0 & x_3 & x_2 & 0\\ 
0 & 0 & 0 & 0 & 3 & 735 & 0 & 0 & 0 & 0 & 0 & 0 & 0 & 0 & x_3 & x_2\\ 
0 & 0 & 0 & 1 & 1 & 343 & 0 & 0 & 0 & 0 & 0 & 0 & 0 & 0 & 0 & x_3\\
\end{array}
\right).$$
If we take the $16\times16$ submatrix obtained from $T$ by deleting
rows 2, 12, 18, 19, its determinant is some nonzero polynomial $A$ of degree 10; 
this is a minor of $T$. By direct calculation we find that the minor $B$ of $N$ 
obtained by deleting the same rows is 
$$B=(64/9)A=\frac{1}{6^6}\left(\prod_{\substack{i=1,\ldots,20\\ i\neq 2,12,18,19}}e_{M_i} \right)A,$$
exactly as asserted by Proposition \ref{MacDuality=Interp}.

To find $T'$, we need to find the change of basis matrix $S$.
The first $\alpha$ columns of $S$ should be a basis for the column space of the
first $r$ columns of $T$; here $\alpha=r=6$, so we just take the first $r$ columns 
of $T$ for the first $r$ columns of $S$. The remaining $\binom{d+n}{n}-r=14$
columns of $S$ should be unit vectors making $S$ invertible.
It is not hard by inspection to see that the following choice of unit vector
columns gives an invertible matrix:
$$
S=
\left(
\begin{array}{cccccccccccccccccccc}
1 & 0 & 0 & 0 & 1 & 8 & 0 & 0 & 0 & 0 & 0 & 0 & 0 & 0 & 0 & 0 & 0 & 0 & 0 & 0\\
0 & 0 & 0 & 0 & 3 & 36 & 0 & 0 & 0 & 0 & 0 & 0 & 0 & 0 & 0 & 0 & 0 & 0 & 0 & 0\\
0 & 0 & 0 & 0 & 3 & 60 & 0 & 0 & 0 & 0 & 0 & 0 & 0 & 0 & 0 & 0 & 0 & 0 & 0 & 0\\
0 & 0 & 0 & 0 & 3 & 84 & 1 & 0 & 0 & 0 & 0 & 0 & 0 & 0 & 0 & 0 & 0 & 0 & 0 & 0\\
0 & 0 & 0 & 0 & 3 & 54 & 0 & 1 & 0 & 0 & 0 & 0 & 0 & 0 & 0 & 0 & 0 & 0 & 0 & 0\\
0 & 0 & 0 & 0 & 6 & 180 & 0 & 0 & 1 & 0 & 0 & 0 & 0 & 0 & 0 & 0 & 0 & 0 & 0 & 0\\
0 & 0 & 0 & 0 & 6 & 252 & 0 & 0 & 0 & 1 & 0 & 0 & 0 & 0 & 0 & 0 & 0 & 0 & 0 & 0\\
0 & 0 & 0 & 0 & 3 & 150 & 0 & 0 & 0 & 0 & 1 & 0 & 0 & 0 & 0 & 0 & 0 & 0 & 0 & 0\\
0 & 0 & 0 & 0 & 6 & 420 & 0 & 0 & 0 & 0 & 0 & 1 & 0 & 0 & 0 & 0 & 0 & 0 & 0 & 0\\
0 & 0 & 0 & 0 & 3 & 294 & 0 & 0 & 0 & 0 & 0 & 0 & 1 & 0 & 0 & 0 & 0 & 0 & 0 & 0\\
0 & 1 & 0 & 0 & 1 & 27 & 0 & 0 & 0 & 0 & 0 & 0 & 0 & 0 & 0 & 0 & 0 & 0 & 0 & 0\\
0 & 0 & 0 & 0 & 3 & 135 & 0 & 0 & 0 & 0 & 0 & 0 & 0 & 1 & 0 & 0 & 0 & 0 & 0 & 0\\
0 & 0 & 0 & 0 & 3 & 189 & 0 & 0 & 0 & 0 & 0 & 0 & 0 & 0 & 1 & 0 & 0 & 0 & 0 & 0\\
0 & 0 & 0 & 0 & 3 & 225 & 0 & 0 & 0 & 0 & 0 & 0 & 0 & 0 & 0 & 1 & 0 & 0 & 0 & 0\\
0 & 0 & 0 & 0 & 6 & 630 & 0 & 0 & 0 & 0 & 0 & 0 & 0 & 0 & 0 & 0 & 1 & 0 & 0 & 0\\
0 & 0 & 0 & 0 & 3 & 441 & 0 & 0 & 0 & 0 & 0 & 0 & 0 & 0 & 0 & 0 & 0 & 1 & 0 & 0\\
0 & 0 & 1 & 0 & 1 & 125 & 0 & 0 & 0 & 0 & 0 & 0 & 0 & 0 & 0 & 0 & 0 & 0 & 0 & 0\\
0 & 0 & 0 & 0 & 3 & 525 & 0 & 0 & 0 & 0 & 0 & 0 & 0 & 0 & 0 & 0 & 0 & 0 & 1 & 0\\
0 & 0 & 0 & 0 & 3 & 735 & 0 & 0 & 0 & 0 & 0 & 0 & 0 & 0 & 0 & 0 & 0 & 0 & 0 & 1\\
0 & 0 & 0 & 1 & 1 & 343 & 0 & 0 & 0 & 0 & 0 & 0 & 0 & 0 & 0 & 0 & 0 & 0 & 0 & 0\\
\end{array}
\right).$$
Now $S^{-1}T$ is the block matrix shown in equation \eqref{BlockMatrix} with
$T'$ being the block in the lower right corner. We get
$$
T'=
\left(\begin{array}{cccccccccc}
x_1-2x_2+x_3 & x_0 & -2x_0 & x_0 & 0 & 0 & 0 & 0 & 0 & 0\\
-\frac{1}{4}x_1-\frac{3}{4}x_2 & -\frac{1}{4}x_0+x_1 & -\frac{3}{4}x_0 & 0 & x_0 & 0 & 0 & 0 & 0 & 0\\
\frac{5}{2}x_1-\frac{9}{2}x_2 & \frac{5}{2}x_0+x_2 & -\frac{9}{2}x_0+x_1 & 0 & 0 & x_0 & 0 & 0 & 0 & 0\\
\frac{11}{2}x_1-\frac{15}{2}x_2 & \frac{11}{2}x_0+x_3 & -\frac{15}{2}x_0 & x_1 & 0 & 0 & x_0 & 0 & 0 & 0\\
\frac{15}{4}x_1-\frac{19}{4}x_2 & \frac{15}{4}x_0 & -\frac{19}{4}x_0+x_2 & 0 & 0 & 0 & 0 & x_0 & 0 & 0\\
\frac{25}{2}x_1-\frac{29}{2}x_2 & \frac{25}{2}x_0 & -\frac{29}{2}x_0+x_3 & x_2 & 0 & 0 & 0 & 0 & x_0 & 0\\
\frac{39}{4}x_1-\frac{43}{4}x_2 & \frac{39}{4}x_0 & -\frac{43}{4}x_0 & x_3 & 0 & 0 & 0 & 0 & 0 & x_0\\
\frac{25}{8}x_1-\frac{33}{8}x_2 & \frac{25}{8}x_0 & -\frac{33}{8}x_0 & 0 & x_2 & x_1 & 0 & 0 & 0 & 0\\
\frac{43}{8}x_1-\frac{51}{8}x_2 & \frac{43}{8}x_0 & -\frac{51}{8}x_0 & 0 & x_3 & 0 & x_1 & 0 & 0 & 0\\
\frac{55}{8}x_1-\frac{63}{8}x_2 & \frac{55}{8}x_0 & -\frac{63}{8}x_0 & 0 & 0 & x_2 & 0 & x_1 & 0 & 0\\
\frac{85}{4}x_1-\frac{93}{4}x_2 & \frac{85}{4}x_0 & -\frac{93}{4}x_0 & 0 & 0 & x_3 & x_2 & 0 & x_1 & 0\\
\frac{127}{8}x_1-\frac{135}{8}x_2 & \frac{127}{8}x_0 & -\frac{135}{8}x_0 & 0 & 0 & 0 & x_3 & 0 & 0 & x_1\\
\frac{155}{8}x_1-\frac{163}{8}x_2 & \frac{155}{8}x_0 & -\frac{163}{8}x_0 & 0 & 0 & 0 & 0 & x_3 & x_2 & 0\\
\frac{225}{8}x_1-\frac{233}{8}x_2 & \frac{225}{8}x_0 & -\frac{233}{8}x_0 & 0 & 0 & 0 & 0 & 0 & x_3 & x_2\\
\end{array}
\right).$$
One can check by direct computation that the ideals of maximal minors for the matrices
$N,T,T'$ and $\Lambda_{Z+3Q,3}'$ are indeed equal and nonzero (indeed, the rows of $T'$ are nonscalar multiples of the rows of
$(\Lambda_{Z+3Q,3}')^t$).
Thus any of these matrices can be used
to find the 3-Weddle scheme. It is interesting to mention that by direct computation
we find the 3-Weddle ideal is not saturated, and its saturation is not radical,
so the 3-Weddle scheme is not reduced and thus is not equal to the 3-Weddle locus.
In fact, the scheme defined by the 3-Weddle ideal consists of the union of the 15 lines together with embedded components at each of the six points of $Z$. More precisely, a primary decomposition for the ideal of the 3-Weddle scheme is given by the intersection of the ideals of the 15 lines with the cubes of the ideals of the six points.
\end{example}

We now illustrate the Macaulay duality 
method in the case that $Z\subset \PP^3$ consists of 
6 points in LGP, in particular recovering the fact, due to Weddle, 
that the $2$-Weddle scheme (and locus) of a set of 6 points in LGP 
is a surface of degree 4. More general results about sets of 6 points will be described in Section \ref{six pt section}. Indeed, 
we will see that the geometry of this quartic surface can have unexpected subtleties.

\begin{example}\label{ClassicalWeddleResult}
Six points in LGP in $\PP^3$ necessarily impose independent conditions on 
forms of degree 2, since five points in LGP can, up to change of coordinates, 
be taken to be the four coordinate vertices and the point $[1:1:1:1]$.
It is easy to see that these five points impose independent conditions on quadrics,
but the ideal of these five points is generated by quadrics, so any sixth
distinct point will impose an additional condition. 

Now, looking at the sequence \eqref{OrigMacDualSeqPrime}
with $d=2$, we see that $\times\partial_{L_P}$ is multiplication by a linear form
from one vector space of dimension 4 to another 
vector space of dimension 4, so $\times\partial_{L_P}$ is represented by a $4\times4$ matrix $J$ of linear forms
in the coordinates of $P$ (namely $J=T'(Z,2P)$ to be explicit, as given in Equation \eqref{BlockMatrix}). 
Furthermore, not all projections of $Z$ lie on a conic (for example, project from a 
general point $P$ in the plane of 3 of the points of $Z$, hence $P$ is not in the plane of the other 3 points of $Z$). Thus the determinant of $J$ does not vanish identically,
so the $2$-Weddle scheme is defined by $\det(J)$.
Hence the $2$-Weddle scheme has degree 4 and 
all of its components have codimension 1; i.e., it is a quartic surface. 
To see that the $2$-Weddle locus is also a quartic surface, we note
that there are lines which intersect the $2$-Weddle locus in at least 4 points.
Thus the degree of the $2$-Weddle locus is at least 4, but its degree cannot be
more than the degree of the $2$-Weddle scheme, so
the $2$-Weddle locus is also a quartic surface.
In particular, let $L$ be a general line in the plane defined by 3 of the points of $Z$.
There are 4 points $P\in L$ from which the projection of $Z$ is contained
in a conic, namely the points of intersection of $L$ with the lines through 
pairs of the 3 points, and the point of intersection of $L$ with the plane defined by the
other three points of $Z$. 
\end{example}

Next we do a similar example for $\PP^4$ and 
we then give a result for $\PP^n$.

\begin{example}\label{10ptsInP4}
Let $X$ be a set of 10 general points in $\PP^4$. What is the locus of points $P$ in $\PP^4$ so that projection from $P$ gives a set of 10 points on a quadric surface in $\PP^3$? For a general point $P$, the projection of $X$ is a general set of 10 points in
$\PP^3$, which therefore do not lie on a quadric. Thus the locus we are asking for is
the 2-Weddle locus for $X$. We will also find the 2-Weddle scheme.

We have
\[
\dim [I(P_1) \cap \dots \cap I(P_{10})]_1 = 0, \ \ \dim [I(P_1) \cap \dots \cap I(P_{10})]_2 = 15 - 10 = 5.
\]
We want to know for which $P$ is it true that
\[
 \dim[I(P_1) \cap \dots \cap I(P_{10}) \cap  I(P^2)]_2 \geq 1.
\]
Since for any distinct points $P_i$ there is a canonical vector space isomorphism
$$[R^*/((\partial_{L_{P_1}})^{t-k_1+1},\ldots,(\partial_{L_{P_r}})^{t-k_r+1})]_t
\cong [R/((L_{P_1})^{t-k_1+1},\ldots,(L_{P_r})^{t-k_r+1})]_t,$$
we can apply Macaulay duality while working in $R$, which is notationally
simpler.
So, by Macaulay duality for any $t$ we have
\[
\dim [I(P_1) \cap \dots \cap I(P_{10})]_t = \dim [R/(L_{P_1}^t, \dots, L_{P_{10}}^t)]_t 
\]
and 
\[ 
\dim [I(P_1) \cap \dots \cap I(P_{10}) \cap I(P)^2]_2 = \dim [R/(L_{P_1}^2, \dots, L_{P_{10}}^2, L_P)]_2.
\]
In the exact sequence
\[
[R/(L_{P_1}^2, \dots, L_{P_{10}}^2)]_1 \stackrel{\times L_P}{\longrightarrow} [R/(L_{P_1}^2, \dots, L_{P_{10}}^2)]_2 \rightarrow [R/(L_{P_1}^2, \dots, L_{P_{10}}^2, L_P)]_2 \rightarrow 0,
\]
the first two terms both have dimension 5. Since for most points $P$ in $\PP^4$ the 
projection from $P$ to $\PP^3$ does not lie on a quadric surface, 
for most linear forms $L_P$ the map 
by $\times L_P$ in the above sequence will be surjective. 
Thus the ideal $I_5(T'(X,5Q))$ is not $(0)$,
so the 2-Weddle locus is the degeneracy locus given by $I_5(T'(X,5Q))$, and since
$T'(X,5Q)$ is a $5\times 5$ matrix, $I_5(T'(X,5Q))=(\det(T'(X,5Q)))$.
Since $T'(X,5Q)$ is a matrix of linear forms, we see the 2-Weddle scheme
is a hypersurface of degree 5. 
\end{example}

More generally, we have the following result.

\begin{theorem}\label{WeddleSchemeSigmaThm}
Fix a hyperplane $\PP^n$ inside $\PP^{n+1}$. Let $Z$ be a general set of $\binom{d+n}{n}$ points in $\PP^{n+1}$. For a point $P \in \PP^{n+1}$ 
let $\pi_P$ be the projection to $\PP^n$. 
Let $\Sigma$ be the $d$-Weddle scheme of points in $\PP^{n+1}$ for which $\pi_P(Z)$ lies on a hypersurface in $\PP^n$ of degree $d$. Then $\Sigma$ is defined by a form
of degree $\binom{d+n}{n+1}$.
\end{theorem}

\begin{proof}
We have the same set-up and notation as before. 
Let $r = \binom{d+n}{n}$ and let $Z$ consist
of the points $P_1,\ldots,P_r$. We wind up with the exact sequence
\[
[R/(L_{P_1}^d, \dots, L_{P_r}^d)]_{d-1} \stackrel{\times L_P}{\longrightarrow} [R/(L_{P_1}^d, \dots, L_{P_r}^d)]_d \rightarrow [R/(L_{P_1}^d,\dots,L_{P_r}^d, L_P)]_d \rightarrow 0.
\]
The first vector space has dimension $\binom{d-1+n+1}{n+1} = \binom{d+n}{n+1}$. The second has dimension
\[
\binom{d+n+1}{n+1} - r = \binom{d+n+1}{n+1} - \binom{d+n}{n} = \binom{d+n}{n+1}.
\]
Thus the multiplication is represented by a square $\binom{d+n}{n+1} \times \binom{d+n}{n+1}$ matrix of linear forms.
Note, as in the previous example, the determinant of this matrix is not 0, 
so the result follows as before. 
\end{proof}

In Example \ref{ClassicalWeddleResult}, the $d$-Weddle scheme was reduced,
but in general (as in Example \ref{BigWeddleExample}) 
the $d$-Weddle scheme need not be reduced 
(even, unlike Example \ref{BigWeddleExample}, if it is a hypersurface) or equidimensional.
For examples of these phenomena, see Subsection \ref{non-reduced}, Example \ref{nonred sup on pt}.
Thus it can happen that the $d$-Weddle locus may have smaller degree than the $d$-Weddle scheme.

In some cases, such as 
Example \ref{ClassicalWeddleResult},
Example \ref{10ptsInP4} and
Theorem \ref{WeddleSchemeSigmaThm},
the value of the degree of the $d$-Weddle scheme is 
a consequence of the following fact
taken from \cite{J1}, which studies a disguised version of a similar problem.
Let $A$ be a $q \times p$ matrix of linear forms in $m+1$ variables, and let $Y_r$ be the subscheme of $\PP^m$ defined by the vanishing of the $(r+1) \times (r+1)$ minors of $A$. If $Y_r$ is not empty, then $Y_r$ has codimension at most $(p-r)(q-r)$ (cf. \cite{EN} page 189, \cite{BV} Theorem (2.1) or \cite{ACGH} page 83 and Proposition (4.1)). 
For a generic matrix this is achieved (see \cite{HE} for references) and in this case $Y_r$ is ACM \cite{HE}. We will refer to $(p-r)(q-r)$ as the  {\it expected codimension}\index{expected codimension}. Note that if this maximum codimension is achieved, the degree of $Y_r$ (as a scheme) is forced:

\begin{lemma}\label{J1 lemma}
Assume $Y_r  \neq \emptyset$ has the expected codimension $(p-r)(q-r)$. 
Then by \cite[Lemma 1.4]{J1}
\[
\deg Y_r = \prod_{i=0}^{p-r-1} \left [ \binom{q+i}{r}  /  \binom{r+i}{r} \right ],
\]
and by \cite{HE} $Y_r$ is ACM\index{arithmetically! Cohen-Macaulay, ACM}.
\end{lemma}

\begin{remark}
Sometimes one is interested in situations where 
the codimension is not the expected one, so Lemma \ref{J1 lemma}
does not apply. 
We give such an example in subsection \ref{not max minors},
which also has the point of interest that the minors needed are nonmaximal.
\end{remark}


\section{Different sets of six points and their Weddle schemes and loci} \label{six pt section}

In this section we describe some interesting behavior of Weddle schemes. It is a pleasant surprise that such a variety of phenomena can be produced using only sets of 6 points. 

For example, when a set of points is not geproci, it can still be of interest to find the locus of points from which it projects to a complete intersection. The first interesting case, Proposition \ref{proci(6)}, involves 6 points.

\begin{proposition} \label{proci(6)}
Let $Z$ be a set of 6 points in linear general position (LGP) in $\PP^3$. Let $\mathcal W(Z)$ be the Weddle surface for $Z$. Let $X$ be the union of 15 lines joining pairs of points of $Z$. The locus of points $P$ for which the projection $\pi_P(Z)$ is a complete intersection is $\mathcal W(Z) \backslash X$.
\end{proposition}   

\begin{proof}
It is well known (and obvious) that $X \subset \mathcal W(Z)$. If $P \in X\setminus Z$ then $\pi_P(Z)$ consists of five points, which by LGP cannot be a complete intersection. We have to show that projection from any other point of $\mathcal W(Z)$ does give a complete intersection of type $(2,3)$. 

Note first that for $P \in \mathcal W(Z) \backslash X$, $|\pi_P(Z)| = 6$. A set of 6 distinct points in $\PP^2$ on a conic fails to be a complete intersection of type $(2,3)$ if and only if it contains at least four points on a line. But the latter is possible for $\pi_P(Z)$ if and only if at least four of the points of $Z$ lie on a plane (and $P$ also lies on that plane). This is excluded by the LGP assumption. 
\end{proof} 

\begin{remark}
We show below that the Weddle surface $\mathcal W(Z)$ from Proposition \ref{proci(6)} contains at least 10 lines in addition to the 15
mentioned in the proposition. Projection from a point $P$ of one of these 10 lines still gives a complete intersection, 
but in this  case a complete intersection of a cubic with a reducible conic.
\end{remark} 

\subsection{A Weddle surface with more than 25 lines}	\label{>25 lines}

Let $Z=\{P_1,\ldots,P_6\}\subseteq \PP^3$ be a set of 6 points in LGP. 
By Example \ref{ClassicalWeddleResult}, the 2-Weddle surface $\mathcal W(Z)$
is a quartic surface. Since the points of $Z$ are in LGP,
there are 15 distinct lines through pairs of points of $Z$.
In addition, LGP ensures that each subset of three of the six points of $Z$ is 
contained in a unique plane. Thus, given two planes $\Pi_1, \Pi_2$ with $Z\subset \Pi_1\cup\Pi_2$, each intersection $Z\cap \Pi_i$ has exactly three points of $Z$, so
$Z\cap \Pi_1$ is disjoint from $Z\cap \Pi_2$. Thus no point of $Z$ is on the line
$\Pi_1\cap\Pi_2$, and no other plane containing this line contains a point of $Z$.
Thus there are 10 distinct lines which arise as intersections of two planes whose union
contains $Z$, and these lines are distinct from the 15 lines 
through pairs of points of $Z$. These 25 lines are contained in $\mathcal W(Z)$
since projecting from a general point $O$ of any such line, the image of $Z$ is 
contained in a conic, and so $O$ (and hence the line) is contained in 
$\mathcal W(Z)$. This is classical -- see \cite{EMCH}.

However it is also known that there can be more than 25 lines in $\mathcal W(Z)$ if the points are not general, but still in LGP.
This case occurs for instance when the points in $Z$ are pairs of a skew involution, see \cite[Theorem 1]{moore}. 
For the reader's convenience we include a statement and a proof of such a result, namely
Proposition \ref{p.Weddle >25}, which gives a construction having two lines
in addition to the 25 mentioned above, and which shows precisely
how to pick the points so that they are in LGP.
In preparation for stating the proposition, recall that three pairwise disjoint lines in $\PP^3$ are contained in a 
unique smooth quadric surface and that the three lines are members of the same 
ruling on that quadric. 
Any line meeting all three of the lines is a member of the other ruling. 

\begin{proposition}\label{p.Weddle >25}
Let $\ell_1,\ell_2, \ell_3$ be three skew lines in $\PP^3$, contained in the smooth quadric $\mathcal Q$.
Let  $m_1, m_2$ be distinct lines intersecting each of $\ell_1,\ell_2, \ell_3$. 
Let $P_{ij}$ be the point of intersection of $\ell_i$ and $m_j$. 
For each $i=1,2,3$, 
pick any distinct points $A_i, B_i$ on $\ell_i$ distinct from the points $P_{ij}$.
Let
$Z=\{A_1,B_1,A_2,B_2,A_3,B_3\}$. Furthermore,
let $R_i$ be the ruling line of $\mathcal Q$ through $A_i$ other than $\ell_i$ and let 
$R'_i$ be the ruling line of $\mathcal Q$ through $B_i$ other than $\ell_i$.
\begin{enumerate}
\item[(1)] The following are equivalent:
\begin{enumerate}
\item[(a)] the points of $Z$ are in LGP;
\item[(b)] no four of the points of $Z$ are coplanar; and
\item[(c)] $A_2, B_2\not\in R_1\cup R'_1$ and
$A_3, B_3\not\in R_1\cup R_2\cup R'_1\cup R'_2$.

\end{enumerate}

\item[(2)]

Now assume that $B_i$, $i=1,2,3$, is the point on $\ell_i$ such that 
$P_{i1},P_{i2},A_i,B_i$, in this order, are harmonic (see Section \ref{harmonic points}).\index{points!harmonic} If the points of $Z$ are in LGP, then $\mathcal W(Z)$
is a quartic surface which contains at least 27 lines,
including the 25 classical ones and 
the lines $m_1$ and $m_2$.
\end{enumerate}
\end{proposition}

\begin{proof}
In the following figure, $\circ$ represents the $P_{ij}$ and $\bullet$ represents the $A_i$ and the $B_j$.

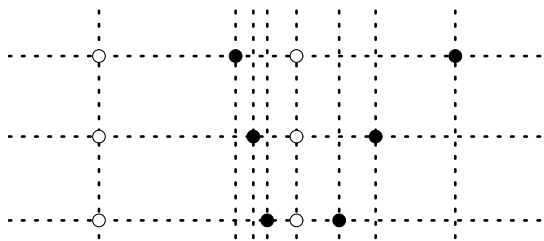
\begin{figure}[ht!]
    \centering
\begin{tikzpicture}[line cap=round,line join=round,>=triangle 45,x=1cm,y=1cm,scale=0.4]
\clip(-13.961357795587496,-7.362792578839625) rectangle (3.7393825057934395,1.4829751952773793);
\draw [line width=1.0pt,loosely dotted,domain=-13.961357795587496:3.7393825057934395] plot(\x,{(-7.333740186383283-0*\x)/6.565703282102932});
\draw [line width=1.0pt,loosely dotted,domain=-13.961357795587496:3.7393825057934395] plot(\x,{(-24.91194700892617-0*\x)/6.565703282102932});
\draw [line width=1.0pt,loosely dotted,domain=-13.961357795587496:3.7393825057934395] plot(\x,{(-43.22257911574169-0*\x)/6.565703282102932});
\draw [line width=1.0pt,loosely dotted] (-10.961657445889542,-9.362792578839625) -- (-10.961657445889542,0.4829751952773793);
\draw [line width=1.0pt,loosely dotted] (-6.419848136667854,-9.362792578839625) -- (-6.419848136667854,0.4829751952773793);
\draw [line width=1.0pt,loosely dotted] (-5.830209735119775,-9.362792578839625) -- (-5.830209735119775,0.4829751952773793);
\draw [line width=1.0pt,loosely dotted] (-5.3680607176901995,-9.362792578839625) -- (-5.3680607176901995,0.4829751952773793);
\draw [line width=1.0pt,loosely dotted] (-4.395954163786609,-9.362792578839625) -- (-4.395954163786609,0.4829751952773793);
\draw [line width=1.0pt,loosely dotted] (-2.9776347654682586,-9.362792578839625) -- (-2.9776347654682586,0.4829751952773793);
\draw [line width=1.0pt,loosely dotted] (-1.7664856163424733,-9.362792578839625) -- (-1.7664856163424733,0.4829751952773793);
\draw [line width=1.0pt,loosely dotted] (0.8815415123341171,-9.362792578839625) -- (0.8815415123341171,0.4829751952773793);
\draw [fill=white] (-10.961657445889541,-1.116977096173977) circle (6pt);
\draw [fill=white] (-4.395954163786609,-1.116977096173977) circle (6pt);
\draw [fill=black] (-6.419848136667854,-1.116977096173977) circle (6pt);
\draw [fill=black] (0.8815415123341171,-1.116977096173977) circle (6pt);
\draw [fill=white] (-10.961657445889541,-3.7942541626625425) circle (6pt);
\draw [fill=white] (-10.961657445889541,-6.5830844402548) circle (6pt);
\draw [fill=black] (-5.830209735119775,-3.7942541626625425) circle (6pt);
\draw [fill=black] (-1.7664856163424733,-3.7942541626625425) circle (6pt);
\draw [fill=white] (-4.395954163786609,-3.7942541626625425) circle (6pt);
\draw [fill=white] (-4.395954163786609,-6.583084440254801) circle (6pt);
\draw [fill=black] (-5.3680607176901995,-6.5830844402548) circle (6pt);
\draw [fill=black] (-2.9776347654682586,-6.5830844402548) circle (6pt);
\end{tikzpicture}
    \caption[Harmonic points on a quadric ruling.]{Three sets of harmonic points on the rulings of a quadric.}
    \label{fig:moore1}
\end{figure}

(1) Note that a finite set of points in $\PP^3$ are in LGP
if and only if no three are collinear and no four are coplanar.
Thus (a) clearly implies (b), and since three points of six being collinear
implies four are coplanar, (b) also implies (a).
So we must show (b) and (c) are equivalent.

First assume (b). We must show (c) holds.
We argue contrapositively.
Let $1\leq i < j\leq 3$ and let $k$ be such that $\{i,j,k\}=\{1,2,3\}$.
If $A_j\in R_i$, then $\ell_k\cup R_i$ defines a plane that contains
$A_k,B_k,A_i,A_j$, contrary to assumption, while
if $A_j\in R'_i$, then $\ell_k\cup R'_i$ defines a plane that contains
$A_k,B_k,A_j,B_i$, contrary to assumption. The same argument holds for the $B_j$. 

Finally assume (c). We must show (b) holds. 
Assuming four of the points are coplanar
(contained in a plane $\Pi$ say), we will derive a contradiction.
The four points must be three $A$'s and a $B$, three $B$'s and an $A$, or two of each.
If it is three $A$'s and a $B$, or three $B$'s and an $A$, then 
$\Pi$ contains an $A_i$ and a $B_i$ on the same line $\ell_i$, hence
$\Pi$ contains the line $\ell_i$. This is a line in a ruling on $\mathcal Q$,
so $\Pi\cap \mathcal Q$ is $\ell_i\cup L$ where $L$ is a line in the other ruling.
In the case of three $A$'s and a $B$, $L$ must contain the two $A$ points
not on $\ell_i$. 
These two points are $A_j, A_k$ for some $1\leq j<k\leq 3$, 
so $L=R_j=R_k$ and we have $A_k\in R_j$, contrary to assumption.
The case of three $B$'s and an $A$ is completely analogous. 
%

Now say $\Pi$ contains two $A$'s and two $B$'s.
Then $\Pi$ must contain $A_i$ and $B_i$ (and hence $\ell_i$) for some $i$,
so $\Pi\cap \mathcal Q=\ell_i\cup L$ for a line $L$ in the other ruling.
Since no plane contains skew lines, the remaining two points
cannot be $A_j$ and $B_j$ for any $j$, so they are
$A_j$ and $B_k$ such that
$\{i,j,k\}=\{1,2,3\}$. In particular, 
$L$ contains $A_j$ and $B_k$, so $R_j=L=R'_k$, and hence $R'_j=R_k$.

If $k<j$, then $A_j\in R'_k$, contrary to assumption.
If $j<k$, then $A_k\in R'_j$, contrary to assumption.

(2) 
We now check that $m_1,m_2\subset\mathcal W(Z)$.
Let $O$ be a general point of $m_1$.
Projecting from $O$, the points $P_{11},P_{21}, P_{31}$ collide. Since, after the projection, $P_{12},P_{22}, P_{32}$ are on a line and the harmonic property is preserved, the images of the 6 points 
$A_1,B_1,A_2,B_2,A_3,B_3$ lie on a conic by Lemma \ref{l. harmonic and conic}, 
and hence $O$ (and thus $m_1$) is contained in $\mathcal W(Z)$. 
The same argument holds for $m_2$, applied to the quadruple
$P_{i2}, P_{i1}, A_i, B_i$ on each line $\ell_i$.
Just note that the cross-ratio of $P_{i2}, P_{i1}, A_i, B_i$ is the reciprocal
of the cross ratio for $P_{i1},P_{i2},A_i,B_i$, but the latter are harmonic
so have cross ratio $-1$, hence $P_{i2}, P_{i1}, A_i, B_i$ also
have cross ratio $-1$ and so are harmonic.

Finally, we check that neither $m_1$ nor $m_2$ can be among the 25 lines
mentioned earlier. First, no point of $Z$, by construction, is on
$m_1$ or $m_2$, so neither line can be among the lines through pairs of points of $Z$.
And second, both $m_1$ and $m_2$ are ruling lines in the ruling of $\mathcal Q$ which does 
not include the $\ell_i$. So any plane containing $m_i$ meets $\mathcal Q$ in a line in the ruling
which includes the $\ell_i$, and each such line contains either 2 points of $Z$
(if the line is one of the $\ell_i$) or it contains no points of $Z$.
In particular, no plane containing $m_i$ can contain 3 points of $Z$,
so neither line $m_i$ can occur as the intersection of two planes whose union
contains $Z$.
\end{proof}

 	The next result goes in the opposite direction. It excludes the Weddle surface $\mathcal W(Z)$ of $Z$ containing any
 	line through only one of the points in a set $Z$
 	of six points, under a small assumption of generality on the points.
	
\begin{lemma}\label{l.line in Weddle}
	Let $Z=\{P_1,\ldots,P_6\}\subseteq \PP^3$ be a set of distinct points in LGP.
	Let $L$ be a line in $\PP^3$ containing only one of the points in $Z$.  
	Assume that no plane through $L$ meets $Z$ in three points. Then $L$ is not contained in $\mathcal W(Z).$
\end{lemma}	
\begin{proof}
Let $L$ be a line containing $P_1$ and such that $P_2,\ldots,P_6\notin L.$
We assume by way of contradiction that $L\subseteq \mathcal W(Z)$, i.e., 
for any $Q\in L$ the projection from $Q$ of $Z$ lies a conic. 

Consider the planes \[\alpha=\langle P_2,P_3,P_4\rangle\ \ \text{and} \ \ \beta=\langle P_2,P_5,P_6\rangle\] 
and let \[P=L\cap \beta.\]

The projection $\pi_P$ from the point $P$ to the plane $\alpha$
maps   $P_2,P_5,P_6$ to three distinct collinear points.
Note that the images of the points $P_2,P_5,P_6$ are collinear since $P$ is in the plane they span. The images of the points $P_2,P_5,P_6$ are distinct since $P$ is not on any
line through two of the points. Indeed, if $P$ were say on the line through $P_2$ and $P_5$, then there would be a plane $\Pi$ containing $L$ (and hence $P_1$) and also containing $P_2$ and $P_5$, so $|\Pi\cap Z|=3$
contrary to assumption.
So, the conic through  $\pi_P(Z)$ splits into the union of two lines. This implies that $\pi_P(P_1)=L\cap \alpha$, $\pi_P(P_3)=P_3$, $\pi_P(P_4)=P_4$ are also collinear. But  then 
there is a plane containing $L$ (and hence $P_1$)
and also $P_3, P_4$, contrary to assumption.
\end{proof}

As an application of Lemma \ref{l.line in Weddle} we get the following result which will be useful later.

\begin{proposition} \label{seven pts}
	Let $X=\{P_1,\ldots,P_7\}\subseteq \PP^3$ be a set of seven points in LGP.
	Then the  2-Weddle locus of $X$  is a curve which does not contain any line joining two points of $X$.
\end{proposition}
\begin{proof}
Let $\mathcal C$ be the set of cones in $[I(X)]_2$ and let $\mathfrak c$ be the set of the vertices of the cones in $\mathcal C$, i.e., $\mathfrak c $ is the 2-Weddle locus of $X$.  	
From \cite{maroscia} the $h$-vector of $X$ is $(1,3,3)$.
Thus $\PP([I(X)]_2)\cong \PP^2$. Since the family of singular quadrics in $\PP^3$ is a hypersurface and contains no planes,  $\mathcal C$ has codimension 1 in it. So $\mathfrak c$ is a curve in $\PP^3$.

We need to exclude that  $\mathfrak c$ contains any line through two points of $X.$
Assume, by contradiction, that the line $L=\langle P_1, P_2\rangle$ is a component of $\mathfrak c$.
So, after projecting from a point in $L,$ the point $P_1$ collides with $P_2$ and $P_2,\ldots,P_7$ are on a conic.
Hence, the line $L$ is contained in the Weddle surface of $P_2,\ldots,P_7$ and this contradicts Lemma~\ref{l.line in Weddle}.
\end{proof}

We will see later  that the $2$-Weddle  curve obtained in Proposition \ref{seven pts} is ACM of degree 6 and genus 3. 


\subsection{A totally reducible Weddle surface} \label{totally reducible}

In subsection \ref{>25 lines} we saw that there exist sets of six points for which the Weddle surface does not behave in the expected way, even though LGP holds. Here we note that if we do not have LGP, even stranger behavior can occur, while still giving a surface of degree 4. 


\begin{proposition}\label{p. reducible Weddle surface and harmonic}
Let $L_i$ be three noncoplanar lines concurrent at a point $O$.
Let $Z=\{P_1,P_2,P_3, Q_1,Q_2,Q_3\}$ be a set of six points in $\PP^3$ away from $O$,
distributed in pairs $P_i,Q_i$ on the lines $L_i$.
 Then the Weddle surface $\mathcal W(Z)$  consists of four planes: the three planes generated by pairs of the lines $L_i$ and the plane spanned by $H_1,H_2,H_3$, where $H_i$ is the point on $L_i$ such that $(P_i,Q_i,O,H_i)$ are harmonic,\index{points!harmonic} for $i=1,2,3.$ 
 
     
\end{proposition}

\begin{proof}
The projection of $Z$ from a general point is not contained in a conic since $Z$ is not geproci (cf. \cite{CM} Proposition 4.3). 
Since the $h$-vector of $Z$ is $(1,3,2)$, we see that the 2-Weddle scheme is a surface of  degree 4. 
It is clear that the projection from a general point in the plane containing two of the lines maps $Z$ into a conic.   Furthermore, a general point in the plane spanned by $H_1,H_2,H_3$ maps the points $H_1,H_2,H_3$ into a line and thus, by Lemma \ref{l. harmonic and conic}, $Z$ into a conic. Thus, we are done.
\end{proof}

In the following example we exhibit the interpolation matrix for sets as in Proposition~\ref{p. reducible Weddle surface and harmonic}.
\begin{example}\label{e. reducible Weddle surface}
 We may assume that
   $$O=[0:0:0:1],\; P_1=[1:0:0:0],\; P_2=[0:1:0:0],\; P_3=[0:0:1:0]$$
   and
   $$Q_1=[a:0:0:1],\; Q_2=[0:b:0:1],\; Q_3=[0:0:c:1]$$
   for some nonzero numbers $a,b,c$.
   Then the interpolation matrix defining the Weddle surface has the form
$$\begin{pmatrix}
1 & 0 & 0 & 0 & 0 & 0 & 0 & 0 & 0 & 0\\
0 & 1 & 0 & 0 & 0 & 0 & 0 & 0 & 0 & 0\\
0 & 0 & 1 & 0 & 0 & 0 & 0 & 0 & 0 & 0\\
a^2 & 0 & 0 & 1 & 0 & 0 & a & 0 & 0 & 0\\
0 & b^2 & 0 & 1 & 0 & 0 & 0 & 0 & b & 0\\
0 & 0 & c^2 & 1 & 0 & 0 & 0 & 0 & 0 & c\\
2x & 0 & 0 & 0 & y & z & w & 0 & 0 & 0\\
0 & 2y & 0 & 0 & x & 0 & 0 & z & w & 0\\
0 & 0 & 2z & 0 & 0 & x & 0 & y & 0 & w\\
0 & 0 & 0 & 2w & 0 & 0 & x & 0 & y & z
\end{pmatrix}.$$   
Computing its determinant we get
\[
2xyz(bcx+acy+abz-2abcw), 
\]
and we note that the plane
$$bcx+acy+abz-2abcw=0$$
intersects the lines $L_1, L_2, L_3$ in points
$$H_1=[2a:0:0:1],\;
H_2=[0:2b:0:1],\;
H_3=[0:0:2c:1]$$
such that $(P_i,Q_i,O,H_i)$ are harmonic as proved in Proposition \ref{p. reducible Weddle surface and harmonic}.
\end{example}


\subsection{A nonreduced  equidimensional Weddle scheme} \label{non-reduced}

Now we let $Z$ consist of five general points $Z_1$ in a plane $H$, together with a general point $P \in \PP^3$. Let $\mathcal Q$ be the quadric cone with vertex $P$ over the conic $C$ in $H$ containing the five points. 

The projection from a point of $\mathcal Q$ maps $Z$ to the projection of $C$, which is again a conic. The projection from any point of $H$ sends $Z$ to a set containing five collinear points. On the other hand, the projection from any point not on either $\mathcal Q$ or $H$ sends $Z$ to six points not on a conic. Hence  the 2-Weddle scheme is a proper subscheme of $\PP^3$ supported on $H$ and $\mathcal Q$. 

Since the $h$-vector of the five points on $H$ is $(1,2,2)$, adding $P$ makes the $h$-vector of $Z$ to be $(1,3,2)$. Then exactly as before, the 2-Weddle scheme is a quartic surface determined by the determinant of a $4 \times 4$ matrix of linear forms.  It follows that as a scheme, it must have a double structure on $H$, so it is not reduced.

\subsection{Nonmaximal minors} \label{not max minors}

In this subsection we give another interesting phenomenon arising from 6 points, namely a situation where the expected codimension of the ideal of maximal minors
\index{minors! maximal} is not obtained, and in fact where we need to look at 
nonmaximal minors.\index{minors! nonmaximal} 

Let $L_1,L_2$ be disjoint lines in $\PP^3$ and let $P_1,P_2,P_3$ be points on $L_1$ and $Q_1,Q_2,Q_3$ points on $L_2$. Let $Z = \{P_1,P_2,P_3,Q_1,Q_2,Q_3\}$. The $h$-vector of $Z$ is $(1,3,2)$, so as  in the sequence \eqref{OrigMacDualSeqPrime}
we obtain a $4 \times 4$ Macaulay  duality matrix of linear forms. 

Thus one would expect the 2-Weddle locus to be a surface of degree 4. But notice that $Z$ is a $(2,3)$-grid, so it is geproci, thus its projection from a general point lies on a unique conic. Hence by definition the 2-Weddle locus is the set of points from which the projection of $Z$ lies on a {\it pencil} of conics. It is not hard to check that such points of projection must lie on either $L_1$ or $L_2$ (thus collapsing three points and obtaining an image under projection consisting of three points on a line plus one more point), so the 2-Weddle locus is $L_1 \cup L_2$. 

What about the 2-Weddle scheme? The matrix obtained by \eqref{OrigMacDualSeqPrime} is  a $4 \times 4$ matrix of linear forms, but the geometry just described means that the determinant of this matrix must be zero. We thus look at the $3 \times 3$ minors. The degree formula in Lemma \ref{J1 lemma} does not apply since the scheme defined by these nonmaximal minors does not have the expected codimension. 

We illustrate what is happening here by means of the following example.

\begin{example}
Consider the $(2,3)$-grid 
\[
Z = \{[1:0:0:0], [0:1:0:0], [1:1:0:0], [0:0:1:0], [0:0:0:1],  [0:0:1:1] \}.
\]
The Macaulay duality matrix $B$   is
\[
\left [
\begin{array}{cccccc}
      z & 0 & x & 0 \\ 
      w & 0 & 0 & x \\
      0 & z & y & 0 \\
      0 & w & 0 & y 
\end{array}
\right ] .
\]    
We get $\det(B)=0 $ as expected.

So we look at the ideal of $3 \times 3$ minors of $B$, which is:
{\small
\[
(xzw,xw^2,-yzw,-yw^2,-xz^2,-xzw,yz^2,yzw,-xyz,-xyw,
y^2z,y^2w,x^2z,x^2w,-xyz,-xyw).
\]  }
This has primary decomposition
\[
(y,x) \cap (w,z) \cap (w^2,z^2,y^2,x^2,yzw,xzw,xyw,xyz).
\]

Note that $(w^2,z^2,y^2,x^2,yzw,xzw,xyw,xyz) $ 
is primary for the irrelevant ideal.
Thus the saturation of $I$ is $(yw, xw, yz, xz)$, 
which has primary decomposition
$(x, y) \cap  (z, w)$, 
so the 2-Weddle scheme consists of the two lines, $x=y=0$ and $w=z=0$,
which are grid lines for the (2,3)-grid $Z$.
Hence the 2-Weddle scheme is the same as the 2-Weddle locus.
\end{example}

\begin{remark} \label{compare matrices}
In the last example, 
the interpolation matrix is
\[
N = 
\left [
\begin{array}{cccccccccccccc}
1 & 0 & 0 & 0 & 0 & 0 & 0 & 0 & 0 & 0 \\ 
      0 & 0 & 0 & 0 & 1 & 0 & 0 & 0 & 0 & 0 \\
      0 & 0 & 0 & 0 & 0 & 0 & 0 & 1 & 0 & 0 \\
      0 & 0 & 0 & 0 & 0 & 0 & 0 & 0 & 0 & 1 \\
      1 & 1 & 0 & 0 & 1 & 0 & 0 & 0 & 0 & 0 \\
      0 & 0 & 0 & 0 & 0 & 0 & 0 & 1 & 1 & 1 \\
      2x & y & z & w & 0 & 0 & 0 & 0 & 0 & 0 \\
      0 & x & 0 & 0 & 2y & z & w & 0 & 0 & 0 \\
      0 & 0 & x & 0 & 0 & y & 0 & 2z & w & 0 \\
      0 & 0 & 0 & x & 0 & 0 & y & 0 & z & 2w
\end{array}
\right ].
\]
After row and column reduction (as in Remark \ref{r. ReducedIntMat})
this becomes
\[
\left [
\begin{array}{cccccc}
      z & w & 0 & 0 \\ 
      0 & 0 & z & w \\
      x & 0 & y & 0 \\
      0 & x & 0 & y 
\end{array}
\right ]
\]
which is the transpose of the Macaulay duality matrix above.
We have $\det N = 0$, which we expect since $Z$ is geproci,
so we need to look at subminors.
The ideal of $9 \times 9$ minors of $N$ gives the same ideal 
as $I$ above.
\end{remark}


\section{\texorpdfstring{$d$}{d}-Weddle loci for some general sets of   points in \texorpdfstring{$\PP^3$}{P3}}

 For five general points in $\PP^3$ it is clear that the only way a projection can lie on more than one conic is for one of the points to ``disappear", i.e., for $P$ to lie on a line joining two of the points. Hence the 2-Weddle locus is a curve of degree $\binom{5}{2} = 10$. 
 For six points the 2-Weddle locus is the classical Weddle surface, which has degree 4. For seven points, Emch claims that it is classically known that the 2-Weddle locus is a curve of degree 6 and genus 3 (we will confirm this). Emch omits the case of 8 and 9 points. For 10 points he shows (\cite[page 273]{EMCH}) that the 3-Weddle locus is a surface of degree 10. For 11 points he claims 
the 3-Weddle locus is a curve of degree 52 (\cite[page 275]{EMCH}), but this is an error -- our proof gives a curve of degree 45, and we have confirmed this using symbolic computation software. Emch goes on to prove a more general result (\cite[page 276]{EMCH}), and we already gave a new proof of a more general result.
Indeed, Theorem \ref{WeddleSchemeSigmaThm}  settled the case of a general set of $\binom{d+n-1}{n-1}$ points in $\PP^n$, so in $\PP^3$ the case of $\binom{d+2}{2}$ general points has been disposed of. 

There is, however, one remaining small point that will be needed shortly. We know that the $d$-Weddle locus of a general set of $\binom{d+2}{2}$ general points in $\PP^3$ is a surface of degree $\binom{d+2}{3}$. This is the locus of points from which the projection to $\PP^2$ lies on a curve of degree $d$, i.e., the locus for which the map $\times \partial_{L_P}$ in \eqref{OrigMacDualSeqPrime} fails to be surjective. But we can also ask: does the general projection from (a component of) this surface lie on a unique curve of degree $d$, or is it possible that it lies on a pencil of such curves? We believe that there is {\it no} point from which there is a pencil, but anyway we have the following result.

\begin{theorem}\label{t. not a pencil}
Let $Z$ be a general set of $\binom{d+2}{2}$ points in $\PP^{3}$. For a point $P \in \PP^{3}$ 
let $\pi_P$ be the projection to a general plane $H$. 
Let $\Sigma$ be the $d$-Weddle scheme of points in $\PP^{3}$ for which $\pi_P(Z)$ lies on a curve of degree $d$ in $H$. Then $\Sigma$ is defined by a form
of degree $\binom{d+2}{3}$. Furthermore, 
 for a general point $Q$ in any component of  $\Sigma$, the projection  $\pi_Q(Z)$ does not lie on a pencil of curves of degree $d$ in $H$.
\end{theorem}
\begin{proof}
The first part of the theorem is a special case of Theorem \ref{WeddleSchemeSigmaThm}. 

To prove the rest of the statement, assume by contradiction $\Sigma_0$ to be a component of $\Sigma$ such that, for the projection from a general point $Q\in \Sigma_0$, the set of points  $\pi_Q(Z)$ lies on a pencil of curves of degree $d$ in~$H$. 

Take $P_1\in Z$ and consider $Z_1=Z\setminus \{P_1\}$.  The set $\pi_Q(Z_1)$ sits on a pencil of curves of degree $d$ so $\Sigma_0$ is contained in the $d$-Weddle locus of $Z_1$, which cannot be $\PP^3$ since $Z_1$ is a general set of degree $\binom{d+2}{2}-1$. 
Moreover, notice that this Weddle locus does not depend on $P_1$. Since $P_1$ is general in $\PP^3$, if the linear system of curves of degree $d$ in $H$ containing $\pi_Q(Z_1)$ is a pencil, then $\pi_Q(P_1)$ cannot lie in the base locus of the pencil. This contradicts that $\pi_Q(Z)$ lies in a pencil of curves of degree $d$.
Thus, $\pi_Q(Z_1)$ lies in a 2-dimensional linear system of curves of degree $d$. 

Now we repeat the argument. Take $P_2\in Z_1$ and consider $Z_2=Z\setminus \{P_1,P_2\}=Z_1\setminus \{P_2\}$. 
The set $\pi_Q(Z_2)$ sits in a 2-dimensional linear system of curves of degree $d$. So, $\Sigma_0$ is also contained in the $d$-Weddle locus of $Z_2$, which is a general set of degree $\binom{d+2}{2}-2$. 
So, as above, the linear system of curves of degree $d$ in $H$ containing $\pi_Q(Z_2)$ has dimension at least~3. 

Continuing in this way, the linear system of curves of degree $d$ in $H$ containing one point has dimension at least $\binom{d+2}{2}$, which is a contradiction since the complete linear system has dimension $\binom{d+2}{2}-1$.  
\end{proof}

Now we apply this to the case of general sets of $\binom{d+2}{2} \pm 1$ points.


\subsection{$d$-Weddle varieties for $\binom{d+2}{2} +1$ general points}

Now suppose we have a set $Z$ of  $\binom{d+2}{2}+1$ general points in $\PP^3$. This includes the case of seven general points.
We again look for the locus of points that are cones of degree $d$ containing $Z$. 
Emch also considers this situation (\cite[page 278]{EMCH}) and claims that the $d$-Weddle locus is a curve of degree
\[
\frac{1}{72} ( 2d^6 + 12 d^5 + 17 d^4 - 66 d^3 - 271 d^2 + 954d - 648).
\]
In the case $d=2$ this correctly gives a curve of degree 6 (which we will see is in fact ACM and has genus 3). In the case $d=3$, so for a general set of 11 points, this formula gives that the 3-Weddle locus is a curve of degree 52. As noted above,  this claim is false, as is Emch's formula mentioned above. Instead,  the computer gives a curve of degree 45 (again ACM) for 11 points.  We  prove this and give the correct formula in Proposition \ref{binom d+2 2 + 1}.

So consider a general set of $N+1 = \binom{d+2}{2}+1$ points with ideals $\mP_1,\dots,\mP_{N+1}$. Then
\[
\dim [\mP_1 \cap \dots \cap \mP_{N+1}]_d = \binom{d+3}{3} - \binom{d+2}{2} -1 = \binom{d+2}{3}-1.
\]
We want to know for which $\mP$  it is true that
\[
 \dim[\mP_1 \cap \dots \cap \mP_{N+1} \cap  \mP^d]_d \geq 1.
\]
Again using \eqref{OrigMacDualSeqPrime}, we have the exact sequence
\[
[R/(L_1^d, \dots, L_{N+1}^d)]_{d-1} \stackrel{\times \ell}{\longrightarrow} [R/(L_1^d, \dots, L_{N+1}^d)]_d \rightarrow [R/(L_1^d,\dots,L_{N+1}^d, \ell)]_d \rightarrow 0.
\]
The first vector space has dimension $\binom{d+2}{3}$ and the second has dimension 
\[
\binom{d+3}{3} - N -1 = \binom{d+2}{3} -1.
\]
This gives a $\left ( \binom{d+2}{3}-1 \right ) \times \binom{d+2}{3}$ matrix of linear forms, and the $d$-Weddle locus is the variety defined by the ideal of maximal minors of this matrix. Because of the size of the matrix, the maximal minors define a locus of codimension one or two. When the codimension is 2, by Lemma  \ref{J1 lemma} the degree is known and the scheme is ACM. 

In the case $d=2$, Proposition \ref{seven pts} tells us that this locus is a curve, which has the expected codimension 2 in $\PP^3$. So we know that the curve has degree $\binom{4}{2} = 6$ and genus 3, and it is ACM. 

In the case $d=3$, we have a $9 \times 10$ matrix of linear forms. The points $Z_1 = \{P_1,\dots,P_{10} \}$ give rise to a surface $W_1$ of degree 10, and the points $Z_2 = \{ P_2, \dots, P_{11} \}$ give rise to a surface $W_2$ of degree 10. Clearly a general point of $W_1$ does not lie on the $3$-Weddle surface of $Z_2$ and vice versa, and by Theorem \ref{t. not a pencil} and the generality of the points these surfaces do not share  a component. Hence the 3-Weddle variety for $Z$ has codimension at least two, hence exactly two. So by Lemma \ref{J1 lemma}, it is an ACM curve of degree $\binom{10}{2} = 45$.

By the same reasoning, we have the following result.

\begin{proposition}\label{binom d+2 2 + 1} 
The $d$-Weddle locus for a general set of $\binom{d+2}{2}+1$ points in $\PP^3$ is an ACM curve of degree
\[
\binom{ \binom{d+2}{3} }{2} = \frac{1}{72} d^6 + \frac{1}{12} d^5 + \frac{13}{72} d^4 + \frac{1}{12} d^3 - \frac{7}{36} d^2 - \frac{1}{6} d.
\]
\end{proposition}


\subsection{$d$-Weddle varieties for $\binom{d+2}{2} -1$ general points}

Consider the analogous question for $\binom{d+2}{2}-1$ points in $\PP^3$. For example, consider a general set of 5 or 9 points in $\PP^3$. It is obvious that the general projection lies on a plane curve of degree $d$, so the question is to find the locus where the projection lies on a {\it pencil} of curves of degree $d$, i.e., for which $\dim [\mP_1 \cap \dots \cap \mP_N \cap \mP^d]_d \geq 2$. 

In the case of 5 points, it is easy to see that the 2-Weddle {\it locus} is precisely the union of lines joining two of the points, so the locus is a curve of degree $\binom{5}{2} = 10$. Then applying Proposition \ref{binom -1} we get that this is also the 2-Weddle scheme for the 5 points.

The union of lines joining two points is  no longer  the full Weddle scheme 
when $|Z| = \binom{d+2}{2}-1$ for $d \geq 3$.
However, it is necessarily part of the scheme -- indeed, for example for 9 general points, projecting from a point on a line joining two points of $Z$ gives 8 points of $\PP^2$, which certainly lie on a pencil of cubics.  So the 36 such lines form a reduced proper subscheme of the curve of degree 55 obtained in Proposition \ref{binom -1}.



\begin{proposition}\label{binom -1}
Let $Z$ be a set of $\binom{d+2}{2}-1$ general points in $\PP^3$. The expected codimension of the $d$-Weddle scheme is 2. If the $d$-Weddle scheme of $Z$ has the expected codimension then it is an ACM curve of degree
\[
\binom{\binom{d+2}{3}+1}{2} .
\]
This curve includes the $\displaystyle \binom{\binom{d+2}{2}-1}{2}$ lines joining two points of $Z$. In particular, it is never irreducible.

In particular, this is true for $d=2, 3, 4$.
\end{proposition}

\begin{proof}
We again set $N = \binom{d+2}{2}$. 
By \eqref{OrigMacDualSeqPrime} we have the exact sequence
\[
[R/(L_1^d, \dots, L_{N-1}^d)]_{d-1} \stackrel{\times \ell}{\longrightarrow} [R/(L_1^d, \dots, L_{N-1}^d)]_d \rightarrow [R/(L_1^d,\dots,L_{N-1}^d, \ell)]_d \rightarrow 0.
\]
The first vector space has dimension $\binom{d+2}{3}$, and the second vector space has dimension $\binom{d+2}{3}+1$. Thus we have a $[\binom{d+2}{3}+1 ] \times \binom{d+2}{3}$ matrix of linear forms. By assumption, the ideal of maximal minors defines a scheme of codimension 2 in $\mathbb P^3$, so 
Lemma \ref{J1 lemma} gives the ACMness and the degree of the $d$-Weddle scheme.

The fact that it includes the lines joining two points of $Z$ is obvious, as mentioned above. Similarly, it is not hard to check the case $d=2$. The other two cases were verified by computer using random points, and the result holds for general points by semicontinuity. (We include these two cases for the sake of Remark \ref{projCI}.)
\end{proof}

\begin{conjecture}
We believe that a set of $\binom{d+2}{2}-1$ general points in $\PP^3$ has $d$-Weddle scheme of the expected codimension, namely 2.
\end{conjecture}

\begin{remark} \label{projCI}
From Proposition \ref{binom -1} and Proposition \ref{binom d+2 2 + 1} we get that for a general set $Z$ of 9 or 16 points, there is a curve, the general point of which projects $Z$ to a complete intersection. This is not the full $d$-Weddle scheme ($d = 3,4$) since it does not include the lines joining two points of $Z$.
\end{remark}


\subsection{$d$-Weddle varieties for $\binom{d+2}{2} \pm 2$ general points}

We first give a broad situation where there is only a finite number of points from which the projection of our set $Z$ lies on an unexpectedly large linear system, applying the previous section.

\begin{proposition}\label{general set plus 2}

The $d$-Weddle locus for a general set, $Z$, of $\binom{d+2}{2} +2$ points in $\PP^3$ is a zero-dimensional scheme of degree $\binom{\binom{d+2}{3}}{3}$.

\end{proposition}

\begin{proof}

We continue the same notation, but now we take $N+2 = \binom{d+2}{2}+2$ general points. The sequence \eqref{OrigMacDualSeqPrime} gives
\[
[R/(L_1^d, \dots, L_{N+2}^d)]_{d-1} \stackrel{\times \ell}{\longrightarrow} [R/(L_1^d, \dots, L_{N+2}^d)]_d \rightarrow [R/(L_1^d,\dots,L_{N+2}^d, \ell)]_d \rightarrow 0.
\]
Now the first vector space has dimension $\binom{d+2}{3}$ and the second has dimension 
\[
\binom{d+3}{3}-(N+2) = \binom{d+3}{3}-\binom{d+2}{2}-2 = \binom{d+2}{3}-2.
\]
Since these differ by two, the expected codimension of Lemma \ref{J1 lemma} is 3. If this is the correct codimension, then the degree (using  Lemma \ref{J1 lemma}) is the one claimed by the statement.

We have seen in Proposition \ref{binom d+2 2 + 1} that the $d$-Weddle locus of any subset, $Z'$, of $\binom{d+2}{2}+1$ of the points of $Z$ has $d$-Weddle locus that is an ACM curve, $C$. Projecting from any point, $P$, of any component of $C$ sends $Z'$ to a set lying on a unique curve, $D$, of degree $d$. Clearly when we add a general  point to $Z'$ to form $Z$, projecting from $P$  does not send the new point to a point of $D$. So the locus must be zero-dimensional, and we are done.
\end{proof}

Now notice that if we instead look at a general set of $N-2 = \binom{d+2}{2}-2$ points, the general projection lies on a pencil of plane curves of degree $d$, so the $d$-Weddle locus is the set of points in $\PP^3$ from which the projection lies on a 2 (projective) dimensional linear system. In the exact sequence \eqref{OrigMacDualSeqPrime} above, the second vector space now has dimension $\binom{d+2}{3}+2$. Again the dimensions differ by 2, so as in Lemma \ref{J1 lemma} we would again expect a zero-dimensional $d$-Weddle locus. However, this is not the case.

\begin{proposition}\label{general set minus 2}
Let $Z$ be a general set of $N - 2$ points in $\PP^3$. Then the $d$-Weddle locus does not have the expected codimension.
\end{proposition}

\begin{proof}
It is clear that the 1-skeleton\index{skeleton} of the $N-2$ points (i.e., the lines through pairs of the points) is contained in the locus. Since we expect codimension 3, we are done. 
\end{proof}

However, now something unusual happens, giving a locus that is not unmixed and that combines different Weddle loci in the same example.
We illustrate this with a careful analysis for the case of 8 general points. 


\subsection{The curious behavior of 8 general points}

Note that $8 = \binom{d+2}{2} +2$ for $d=2$, but also $8 = \binom{d+2}{2} -2$ for $d=3$. 

Let $Z$ be a general set of 8 points in $\PP^3$. What are the 2-Weddle scheme and the 3-Weddle scheme of $Z$? We will show that the 2-Weddle scheme is contained in the 3-Weddle scheme, and that the latter is not unmixed.

For a general projection $\pi_P$, it is clear that the image $\pi_P(Z)$ lies on a pencil of cubics and does not lie on any conics. Also, by Proposition \ref{general set plus 2}, the 2-Weddle scheme is a zero-dimensional scheme of degree 4. What about the 3-Weddle scheme? Note that the 28 lines forming the 1-skeleton of $Z$ are contained in this 3-Weddle locus, so
(as noted in Proposition \ref{general set minus 2})
it does not have the expected codimension. 

Let us look for a moment instead at the 2-Weddle scheme of $Z$, with the idea of relating it to the 3-Weddle locus. The 2-Weddle locus is the set of the points of projection, $P$, such that $\pi_P(Z)$ is a set of 8 points in the plane lying on a conic. By Proposition \ref{general set plus 2}, the 2-Weddle locus of $Z$ is  the support of a zero-dimensional scheme of degree 4. 
In Example \ref{nonred sup on pt} we will show that this zero-dimensional scheme can be nonreduced and supported at a point, even for points in LGP. However, for general sets of points this is not the case. Using random points, we have checked by computer that  this zero-dimensional scheme consists of four distinct points, disjoint from the 28 lines joining two points of $Z$. In the following remark  we give a computer-free argument for this claim.

\begin{remark} 
A quadric can be defined by the matrix of a quadratic form. The quadric cones are the  singular quadrics,
these being the ones where the matrix of the quadratic form is singular. (The entries of the matrix 
are the second partials of the form defining the quadric. The rows are the coefficients of the first 
partials, hence a nonzero element of the kernel of the matrix is a point where the quadric is singular
and thus gives a vertex of the quadric cone.)

The linear system of quadrics through 8 general points in $\PP^3$ corresponds to a pencil of  $4\times4$ matrices. Taking the determinant gives a form of degree 4
in the two variables defining the pencil. Thus this vanishes for four members of the pencil,
so the pencil of quadrics has 4 cones. We now argue that they are distinct.  
A general set of 8 points determines a general  linear pencil of quadrics in the $\PP^9$ of all quadrics. The cones in the pencil correspond to the intersection of this line with the degree 4 hypersurface of singular quadrics. Thus we get four distinct cones. 

Since the 8 points are general, they are not contained in two planes. Thus every quadric in the pencil is irreducible, so each cone has a unique vertex, and by B{\'e}zout's Theorem these vertices are distinct. 
To see that the vertices are disjoint from the 28 lines joining pairs of points, note that none of the vertices can be one of the 8 points since otherwise that point would be distinguished. Similarly, none of the vertices can be on a line through two of the 8 points since then that line would be distinguished among the 28. 
\end{remark}

Now we bring in the 3-Weddle locus. Notice that any point in the 2-Weddle locus of $Z$ must also lie in the 3-Weddle locus of $Z$. Indeed, if $P$ is a point on the 2-Weddle locus of $Z$, the projection of $Z$ from $P$ is a set of 8 distinct points on a conic.  By generality, no projection of $Z$ contains 4 or more collinear points
so the eight points cannot be contained in two lines, hence
this conic must be smooth. 
Thus the projection $\pi_P(Z)$ must in fact be a complete intersection, and $\dim [I({\pi_P (Z)})]_3 = 3$ so $P$ lies in the $3$-Weddle locus. Thus the 3-Weddle locus has a curve component of degree 28 plus a zero-dimensional scheme of degree 4, so it is not unmixed. 
Computer calculations using random points
suggest that the 3-Weddle locus contains nothing else.

Finally, we give the promised example to show that even for 8 points in LGP, the 2-Weddle scheme can be nonreduced of degree 4.

\begin{example} \label{nonred sup on pt}
Let $Z\subset\PP^3$ be a set of 8 points constructed as follows. Let $C$ be a twisted cubic curve and choose 7 points of $Z$, say $P_1,\dots,P_7$, to lie on $C$. Then at a general point of $C$ let $\lambda$ be the tangent line and choose the last point, $P_8$, of $Z$ to be a general point of $\lambda$. The set $Z$ is in LGP.  What is the 2-Weddle scheme of $Z$? As above,  the 2-Weddle scheme is a zero-dimensional scheme of degree 4.

Let $P$ be a point of projection in the $2$-Weddle locus of $Z$. By Proposition \ref{seven pts}, $P$ cannot lie on a line joining two points of $\{ P_1,\dots,P_7 \}$. 
If $P$ does not lie on either $C$ or $\lambda$ then the projection from $P$ is a set of 8 points on an irreducible cubic, so it cannot lie on a conic. In order for the projection to lie on a conic, we need $P$ to lie on $C$ (so the projection of $C$ becomes a conic) and also to lie on the tangent line $\lambda$ (so that the projection of $P_8$ lies on that conic). Thus the 2-Weddle locus is a single point, while the 2-Weddle scheme is a scheme of degree 4 supported at that point.
\end{example}

\section{Connections to Lefschetz Properties}\index{Lefschetz properties}

We end this chapter with some immediate observations connecting our work in this book on the $d$-Weddle scheme and geprociness with objects of interest in the study of the Lefschetz Properties. This is strongly motivated by Subsection \ref{not max minors}.

Given a graded artinian ring $A=\oplus_{i\geq0} A_i$ (such as $A=R/(L_1^d,\dots,L_r^d)$ for linear forms $L_j$ spanning 
$[R]_1$), recall that $A$ has the {\it Weak Lefschetz Property} (WLP)\index{Lefschetz Property! Weak, WLP}\index{WLP} 
if there is a linear form $L$ such that the map $A_t\to A_{t+1}$ given by
multiplication $\times L$ has  maximal rank, for every $t$.
And $A$ has the {\it Strong Lefschetz Property} (SLP)\index{Lefschetz Property! Strong, SLP}\index{SLP} 
if there is a linear form $L$ for which the map $A_t\to A_{t+s}$ given by
multiplication $\times L^s$ has maximal rank, for all $t$ and $s$. See \cite{MN-Tour}. Either of these is an open condition, by semicontinuity, so to show that either of these properties fails we have to find a $t$ (and $s$ for SLP) for which maximal rank fails for a general $L \in [R]_1$. When WLP does hold, one can still study the {\it non-Lefschetz locus} \index{locus! non-Lefschetz} \cite{BMMN}, which is the set of linear forms (with a natural scheme structure) for which maximal rank fails, even if it does hold for the general $L$.

Our situation is that we have a set of points $Z = \{P_1,\ldots, P_r\}$ and the dual linear forms $L_1,\dots,L_r$. We have another point $P$, which is dual to the linear form $L$. 
The bridge to Lefschetz theory is provided by the exact sequence \eqref{OrigMacDualSeqPrime}, which for convenience we recall here with the current notation). 

\[
[R/(L_1^d,\dots,L_r^d)]_{d-1} \stackrel{\times L}{\longrightarrow} [R/(L_1^d,\dots,L_r^d)]_d \rightarrow [R/(L_1^d,\dots,L_r^d, L)]_d \rightarrow 0.
\]

Many papers have been written at this point about ideals generated by powers of linear forms, with respect to the Lefschetz properties (for instance \cite{DIV}, \cite{SS}, \cite{HSS}, \cite{MMN}, \cite{MN-L2}, \cite{MR-SLP} to name a few). Furthermore, \cite[Proposition 2.17]{HMNT} connected the question of unexpected hypersurfaces for the points in a precise way with the failure of the Strong Lefschetz Property for the corresponding quotient by the ideal of powers of linear forms. Most of these  papers used  Macaulay duality as an important tool, so it is not new that the above sequence is related to Lefschetz properties. We will just make the connection between Weddle schemes and the non-Lefschetz locus, and between geprociness and the failure of the Weak Lefschetz Property.\index{Lefschetz Property! Weak, WLP} 

We have noted that the dimension of the last vector space in the above exact sequence is equal to $\dim [I(P_1) \cap \dots \cap I(P_r) \cap I(P)^d]_d$, so this is larger than expected if and only if the rank of $\times L$ is smaller than expected. This immediately gives that for the ideal of uniform powers of linear forms, the quotient fails the Weak Lefschetz Property from degree $d-1$ to degree $d$ if and only if there is an unexpected cone of degree $d$ (see Chapter \ref{ch: unexp hypersurf} for the definitions and relevant facts about unexpected hypersurfaces -- here we are simply making basic observations). 

Furthermore, in \cite{BMMN}  the authors studied the  non-Lefschetz locus\index{locus! non-Lefschetz} for a graded ring $R/I$ and gave it a scheme structure in precisely the same way that we have given to the $d$-Weddle scheme in this chapter. Thus the $d$-Weddle scheme, whose support gives the locus of vertices of projections sending a set $Z$ of points in $\PP^3$ to points in $\PP^2$ lying on a plane curve of unexpectedly small degree, is nothing but the non-Lefschetz locus (with its scheme structure) for the multipication from degree $d-1$ to degree $d$ of the quotient of the dual polynomial  ring  by the ideal generated by the corresponding powers of linear forms.

We have from \cite[Theorem 3.5]{CM} that $(a,b)$-grids have unexpected cones in degree $a$, and also in degree $b$ provided $a,b \geq 3$. 
We believe that any $(a,b)$-geproci has unexpected cones of degrees $a$ and $b$ under the same assumptions, and we give some results in this direction in Chapter \ref{ch: unexp hypersurf}. 
In any case we have the following result.

\begin{corollary}
Let $Z$ be an $(a,b)$-grid with $b \geq a\geq 2$ and $b \geq 3$, and let $L_1,\dots,L_{ab}$ be the dual linear forms. Then $R/(L_1^a,\dots,L_{ab}^a)$ fails the Weak Lefschetz Property from degree $a-1$ to degree $a$, and if $b \geq a \geq 3$ then  $R/(L_1^b,\dots,L_{ab}^b)$ fails the Weak Lefschetz Property from degree $b-1$ to degree~$b$.
\end{corollary}

More generally, Chapter \ref{ch: unexp hypersurf} gives a number of results in the direction of showing that $(a,b)$-geproci sets of points admit unexpected cones of degree $a$ and almost always also of degree $b$. Any such result applies here to give a failure of the Weak Lefschetz Property. In particular, in Remark \ref{cone-summary} we speculate that this may fail only when $Z$ is degenerate or is a $(2,b)$-grid.

\chapter{Geometry of the \texorpdfstring{$D_4$}{D4} configuration}\label{Ch:D4}
In this chapter we study in greater detail the $D_4$ configuration of points, see Definition \ref{def:D4}. In particular we study various flats (lines, planes) determined by points of the $D_4$ configuration.
\section{The rise of \texorpdfstring{$D_4$}{D4}}

We will see in Chapter \ref{chap.Geography} that the smallest cardinality of a nontrivial nongrid geproci\index{geproci! nontrivial} set of points is $12$ and that any such set is equivalent to the $D_4$ configuration of points. It is also contained in many nontrivial nongrid geproci sets. We construct some of nonstandard geproci sets taking the $D_4$ configuration as the departure point.
Thus it is reasonable to study its geometry in great detail and this is exactly what this Chapter is about.

Moreover, the construction of the $D_4$ configuration of points starting with a $(3,3)$-grid\index{grid! $(3,3)$} is one of the main motivations for what we call the ``standard construction" in Chapter \ref{chap.Geography}. This provides additional motivation to study the geometry and symmetries both of $D_4$ and of $(3,3)$-grids extending our preliminary work in Section \ref{sec:Segre_Embeddings}. 

In this section $Z$ will always refer to the 12 point set coming from $D_4$.
We introduced the  configuration already in Definition \ref{def:D4} working with explicit coordinates. Here, we want to take a more conceptual approach along the lines of Remark \ref{rem:onD4} (c). This approach allows for revealing of additional features of the $D_4$ configuration. We keep our convention of writing $*$ rather than $-1$.

Regarding $S_4$ as the group of $4\times4$ permutation matrices
acting on the coordinates of points in $\PP^3$, 
$Z$ (as given in Remark \ref{rem:onD4}(c)) is the union of the two $S_4$-orbits under permutations of
the coordinates of 00$*$1 and 0011. Call points in the first orbit null points \index{null points of $D_4$}  
and those in the second nonnull.\index{nonnull points of $D_4$}
Note that the six null points lie in the plane $x+y+z+w=0$, whereas the six nonnull points are nondegenerate.
Let $A$ be the quotient $A=H/\{\pm I\}$, where $H$ is the group of $4\time4$ matrices having 4 nonzero entries,
where each nonzero entry is either 1 or $*$, and every row and column has
exactly one nonzero entry. Clearly $|A|=24\cdot 16/2=192$.
Then $A$ acts on $Z$, and the action is transitive.
It is easy to see that $A$ acts transitively on $Z$.

Now we are in a position to prove a result which, in particular, justifies our claim in Remark \ref{rem:onD4} (b).
\begin{proposition}[Collinearities and grids in $D_4$]\label{D4factsProp}
There are exactly 16 lines through exactly 3 points of $Z$ (call these the 3-point lines).
Each point of $Z$ is on four 3-point lines and every pair of 3-point lines is either disjoint or meets in a 
point of $Z$ (thus the 12 points of $Z$ and the 16 3-point lines form a $(12_4,16_3)$
point-line arrangement, a representation of which is shown in Figure \ref{D4Fig2}). 
Moreover, each 3-point line is disjoint from 6 3-point lines,
these 6 are the grid lines of a $(3,3)$-grid, and every subset of $Z$ forming 
the 9 points of a grid arises this way. Thus $Z$ has exactly 16 $(3,3)$-grids.
\end{proposition}

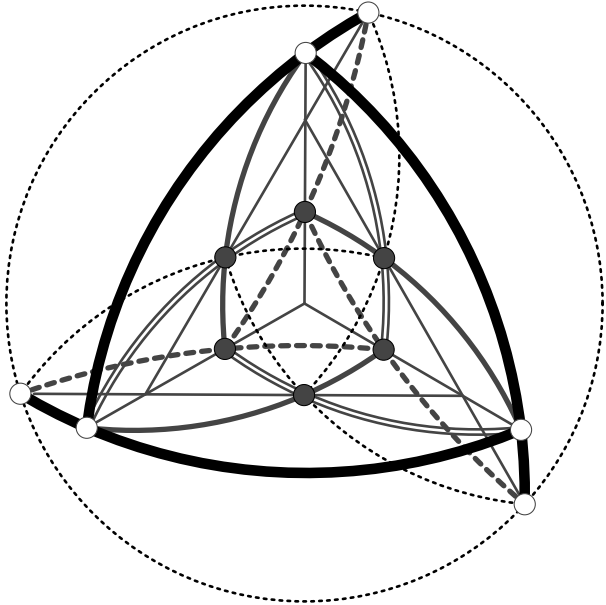
\begin{figure}[ht!]
\definecolor{ffqqff}{rgb}{0.26666666666666666,0.26666666666666666,0.26666666666666666}
\definecolor{qqqqff}{rgb}{0.26666666666666666,0.26666666666666666,0.26666666666666666}
\definecolor{qqffqq}{rgb}{0.26666666666666666,0.26666666666666666,0.26666666666666666}
\definecolor{uuuuuu}{rgb}{0.26666666666666666,0.26666666666666666,0.26666666666666666}
\definecolor{ffqqqq}{rgb}{0.26666666666666666,0.26666666666666666,0.26666666666666666}
\begin{tikzpicture}[line cap=round,line join=round,>=triangle 45,x=1.0cm,y=1.0cm]
\clip(-5.229995463038163,-2.9609516110909535) rectangle (6.579611387108048,6.079506046607177);

\draw [shift={(-2.004330830239797,-2.592439645708379)},line width=2pt,color=qqffqq]  plot[domain=0.41151273756358037:1.1504612622861772,variable=\t]({1.*5.650742093257234*cos(\t r)+0.*5.650742093257234*sin(\t r)},{0.*5.650742093257234*cos(\t r)+1.*5.650742093257234*sin(\t r)});
\draw [shift={(4.858527562611111,1.3265846411658468)},line width=2pt,color=qqffqq]  plot[domain=2.505907839956776:3.244856364679372,variable=\t]({1.*5.65074209325724*cos(\t r)+0.*5.65074209325724*sin(\t r)},{0.*5.65074209325724*cos(\t r)+1.*5.65074209325724*sin(\t r)});
\draw [shift={(-1.966876224295619,5.310482208512857)},line width=2pt,color=qqffqq]  plot[domain=4.600302942349971:5.339251467072568,variable=\t]({1.*5.65074209325723*cos(\t r)+0.*5.65074209325723*sin(\t r)},{0.*5.65074209325723*cos(\t r)+1.*5.65074209325723*sin(\t r)});

\draw [shift={(2.595877835623589,5.288857781688598)},line width=3pt,color=qqqqff]  plot[domain=4.076047891677878:4.814996416400475,variable=\t]({1.*5.650742093257232*cos(\t r)+0.*5.650742093257232*sin(\t r)},{0.*5.650742093257232*cos(\t r)+1.*5.650742093257232*sin(\t r)});
\draw [shift={(2.5584232296794163,-2.614064072532644)},line width=3pt,color=qqqqff]  plot[domain=1.9816527892846834:2.7206013140072796,variable=\t]({1.*5.6507420932572385*cos(\t r)+0.*5.6507420932572385*sin(\t r)},{0.*5.6507420932572385*cos(\t r)+1.*5.6507420932572385*sin(\t r)});
\draw [shift={(-4.266980557227314,1.3698334948143738)},line width=3pt,color=qqqqff]  plot[domain=-0.11274231310851324:0.6262062116140841,variable=\t]({1.*5.650742093257235*cos(\t r)+0.*5.650742093257235*sin(\t r)},{0.*5.650742093257235*cos(\t r)+1.*5.650742093257235*sin(\t r)});

\draw [shift={(2.595877835623589,5.288857781688598)},line width=1pt,color=white]  plot[domain=4.076047891677878:4.814996416400475,variable=\t]({1.*5.650742093257232*cos(\t r)+0.*5.650742093257232*sin(\t r)},{0.*5.650742093257232*cos(\t r)+1.*5.650742093257232*sin(\t r)});
\draw [shift={(2.5584232296794163,-2.614064072532644)},line width=1pt,color=white]  plot[domain=1.9816527892846834:2.7206013140072796,variable=\t]({1.*5.6507420932572385*cos(\t r)+0.*5.6507420932572385*sin(\t r)},{0.*5.6507420932572385*cos(\t r)+1.*5.6507420932572385*sin(\t r)});
\draw [shift={(-4.266980557227314,1.3698334948143738)},line width=1pt,color=white]  plot[domain=-0.11274231310851324:0.6262062116140841,variable=\t]({1.*5.650742093257235*cos(\t r)+0.*5.650742093257235*sin(\t r)},{0.*5.650742093257235*cos(\t r)+1.*5.650742093257235*sin(\t r)});

\draw [line width=1.pt,color=ffqqqq] (0.2957735026918962,1.3482090679901106)-- (-2.598920151018907,-0.30480110394799387);
\draw [line width=1.pt,color=ffqqqq] (0.2957735026918962,1.3482090679901106)-- (0.3115715279348178,4.681592394246314);
\draw [line width=1.pt,color=ffqqqq] (0.2957735026918962,1.3482090679901106)-- (3.1746691311597792,-0.3321640863279878);
\draw [line width=1.pt,dash pattern=on 1pt off 2pt,color=black] (0.2957735026918962,1.3482090679901106) circle (3.96299094205599cm);

\draw [line width=1.pt,color=ffqqqq,domain=-1.82:1.2] plot(\x,{(--3.4655231911783--1.8223136019851656*\x)/1.0636602540378448});
\draw [line width=1.pt,color=ffqqqq,domain=-3.3:2.4] plot(\x,{(-0.2772--0.01*\x)/-2.11});
\draw [line width=1.pt,color=ffqqqq,domain=0.3:3.2] plot(\x,{(--4.523113414279856-1.8323136019851658*\x)/1.0463397459621553});
\draw [shift={(-9.786216342860852,7.232932679033489)},line width=2.pt,color=ffqqff,dashed]  plot[domain=5.659784774992629:6.101005028397384,variable=\t]({1.*11.114824665765795*cos(\t r)+0.*11.114824665765795*sin(\t r)},{0.*11.114824665765795*cos(\t r)+1.*11.114824665765795*sin(\t r)});
\draw [shift={(0.2404482840546112,-10.325412064476998)},line width=2.pt,color=ffqqff,dashed]  plot[domain=1.4709945702062388:1.9122148236109935,variable=\t]({1.*11.11482466576578*cos(\t r)+0.*11.11482466576578*sin(\t r)},{0.*11.11482466576578*cos(\t r)+1.*11.11482466576578*sin(\t r)});
\draw [shift={(10.433088566881926,7.137106589413849)},line width=2.pt,color=ffqqff,dashed]  plot[domain=3.565389672599434:4.006609926004188,variable=\t]({1.*11.114824665765788*cos(\t r)+0.*11.114824665765788*sin(\t r)},{0.*11.114824665765788*cos(\t r)+1.*11.114824665765788*sin(\t r)});

\draw [shift={(4.718690590602634,-1.2333889117377246)},line width=4.pt]  plot[domain=2.076619948272229:3.0153696150788445,variable=\t]({1.*7.3762932617047365*cos(\t r)+0.*7.3762932617047365*sin(\t r)},{0.*7.3762932617047365*cos(\t r)+1.*7.3762932617047365*sin(\t r)});
\draw [shift={(0.32004439153941744,6.469366614817032)},line width=4.pt]  plot[domain=4.171015050665425:5.109764717472039,variable=\t]({1.*7.376293261704749*cos(\t r)+0.*7.376293261704749*sin(\t r)},{0.*7.376293261704749*cos(\t r)+1.*7.376293261704749*sin(\t r)});
\draw [shift={(-4.151414474066371,-1.1913504991089645)},line width=4.pt]  plot[domain=-0.017775154120966086:0.9209745126856483,variable=\t]({1.*7.3762932617047445*cos(\t r)+0.*7.3762932617047445*sin(\t r)},{0.*7.3762932617047445*cos(\t r)+1.*7.3762932617047445*sin(\t r)});

\draw [shift={(-3.0648897791590204,3.30978360500457)},line width=1.pt,dash pattern=on 1pt off 2pt,color=black]  plot[domain=-0.7586081051645435:0.4258384521306778,variable=\t]({1.*4.62237051792276*cos(\t r)+0.*4.62237051792276*sin(\t r)},{0.*4.62237051792276*cos(\t r)+1.*4.62237051792276*sin(\t r)});
\draw [shift={(0.2773317631461349,-2.5429979761655943)},line width=1.pt,dash pattern=on 1pt off 2pt,color=black]  plot[domain=1.3357869972286525:2.5202335545238737,variable=\t]({1.*4.622370517922758*cos(\t r)+0.*4.622370517922758*sin(\t r)},{0.*4.622370517922758*cos(\t r)+1.*4.622370517922758*sin(\t r)});
\draw [shift={(3.674878524088573,3.277841575131356)},line width=1.pt,dash pattern=on 1pt off 2pt,color=black]  plot[domain=3.4301820996218475:4.614628656917069,variable=\t]({1.*4.622370517922755*cos(\t r)+0.*4.622370517922755*sin(\t r)},{0.*4.622370517922755*cos(\t r)+1.*4.622370517922755*sin(\t r)});
\begin{scriptsize}
\draw [fill=uuuuuu] (0.29,0.13) circle (4pt);
\draw [fill=uuuuuu] (-0.7563397459621553,1.9623136019851657) circle (4pt);
\draw [fill=uuuuuu] (1.3536602540378446,1.9523136019851657) circle (4pt);
\draw [fill=uuuuuu] (-0.762113248654052,0.7441045339950554) circle (4pt);
\draw [fill=uuuuuu] (0.30154700538379287,2.566418135980221) circle (4pt);
\draw [fill=uuuuuu] (1.347886751345948,0.7341045339950554) circle (4pt);
\draw [color=uuuuuu,fill=white] (0.3115715279348178,4.681592394246314) circle (4pt);
\draw [color=uuuuuu,fill=white] (3.1746691311597792,-0.3321640863279878) circle (4pt);
\draw [color=uuuuuu,fill=white] (-2.598920151018907,-0.30480110394799387) circle (4pt);
\draw [color=uuuuuu,fill=white] (1.1446693152601979,5.219213215006983) circle (4pt);
\draw [color=uuuuuu,fill=white] (-3.481062333063777,0.14787233333205582) circle (4pt);
\draw [color=uuuuuu,fill=white] (3.223713525879269,-1.3224583443687075) circle (4pt);
\end{scriptsize}
\end{tikzpicture}
\caption[The line-point configuration determined by $D_4$.]{Schematic representation of the $(12_4,16_3)$ structure of
the 12 points of $Z$ and its sixteen 3-point lines
in relation to the coordinate simplex viewed from above.
The 12 points of $Z$ are shown explicitly as dots:
the 6 open dots are the null points, the nonnull points are the filled in dots.
The 16 3-point lines are shown as curved lines.
The 6 straight lines are thin and represent the coordinate axes of $\PP^3$.
When two 3-point lines intersect, it is at a point of $Z$.
The 4 dotted lines are in the coordinate plane forming the base of the coordinate simplex.
Each of the other coordinate planes 
contains 4 3-point lines, one each which is bold, medium thickness, dashed or doubled.
The 3 doubled and 3 medium lines show the $(3,3)$-grid whose lines are disjoint from
the 3-point line represented by the dotted circle.}
\label{D4Fig2}
\end{figure}

\begin{proof}
Note that there are 2 points of $Z$ on each of the 6 coordinate axis (and none of them are coordinate vertices).
Thus each coordinate plane contains 6 points of $Z$.
As we can see from Figure \ref{D4Fig1}, each plane contains 4 3-point lines (giving, so far, 16 3-point lines), and
3 2-point lines (the latter are the coordinate axes in the plane).
Since the points of $Z$ are on coordinate axes, the only 2-point lines not
contained in a coordinate plane are the 4 lines between points on opposite axes
(each axis is defined by setting two variables to 0; the opposite axis sets the other two variables to 0).
There are three pairs of opposite axes, giving 12 2-point lines. Together with the coordinate axes 
themselves, there are thus 18 2-point lines. 
Moreover, there are $\binom{12}{2}=66$ pairs of points of $Z$.
Counting the number of pairs of points 
on lines defined by these 66 pairs, we get (so far) $16\cdot 3=48$ pairs from the 16 3-point lines we have identified
so far plus the 18 pairs from the 2-point lines, for $48+18=66$ total. Thus we have accounted for
all pairs, hence there are no lines with more than 3 points of $Z$ and no additional 3-point lines.
Using the fact that $A$ is transitive, we see that it is enough to check any
point of $Z$ to confirm that each point of $Z$ is on 4 of the 16 3-point lines.
But each point of $Z$ is on two coordinate planes, and (as we can see in Figure \ref{D4Fig1})
in each plane it is on 2 3-point lines, for the given total of 4 lines per point.
Since all of the 3-point lines are in coordinate planes, any point of intersection of two such 
lines must be in a coordinate plane. From Figure \ref{D4Fig1} we see that all such intersections occur at points of $Z$
when the two lines are in the same plane. If the lines are in different planes then their intersection is at a point of $Z$
since when a 3-point line intersects a coordinate axis, it does so at a point of $Z$.
Thus the 12 points of $Z$ and the 16 lines form a $(12_4,16_3)$
point-line arrangement.

The action of $A$ is transitive on 3-point lines. This is because it is transitive on the points of $Z$,
and given any point $p\in Z$, the stabilizer $A_p$ is transitive on 
the four 3-point lines containing $p$. It is enough to confirm this for a specific choice of $p$, say $p=1100$.
Then the four lines are defined by the collinear triples of points 
$\{$1100,1010,01$*$0$\}$,
$\{$1100,0110,10$*$0$\}$,
$\{$1100,1001,010$*\}$,
$\{$1100,0101,100$*\}$.
The group $A_p$ includes all transformation that involve the transposition $(xy)$ on the
first two coordinates, the transposition $(zw)$ on the second two coordinates, and flipping the sign of either of
the last two coordinates. This shows the points (other than $p$) of $Z$ on the four 3-point lines 
through $p$ form an orbit under $A_p$.

Thus to confirm that each 3-point line is disjoint from 6 3-point lines,
and these 6 are the grid lines of a $(3,3)$-grid, it is enough to do for any specific
3-point line $L$. Take $L$ to be the line through
1$*$00, 10$*$0, 01$*$0. Each of these 3 points is on 4 3-point lines, one of which is $L$.
So there are 6 lines left. As shown in Figure \ref{D4Minus3}, these 6 are disjoint from $L$
and form a grid.

Now take a $(3,3)$-grid $\Gamma$ in $Z$. Let $L_1,L_2,L_3$ be three of the grid lines in the same ruling
(i.e., they are disjoint). There are 6 3-point lines disjoint from $L_1$. Since $L_2$ and $L_3$ are two of them,
and the 6 form a grid, there is one 3-point line among the 6 disjoint from $L_2$ and $L_3$; call it $M$.
Thus $M$ is disjoint from $L_1,L_2,L_3$. But there are 6 3-point lines disjoint from $M$ and they form a grid.
Thus $\Gamma$ is the grid of 6 3-point lines disjoint from $M$.
\end{proof}

\begin{figure}[ht!]
\definecolor{ffqqqq}{rgb}{1.,0.,0.}
\definecolor{uuuuuu}{rgb}{0.26666666666666666,0.26666666666666666,0.26666666666666666}
\definecolor{xdxdff}{rgb}{0.49019607843137253,0.49019607843137253,1.}
\definecolor{ududff}{rgb}{0.30196078431372547,0.30196078431372547,1.}
\begin{tikzpicture}[line cap=round,line join=round,>=triangle 45,x=.75cm,y=.75cm]
\clip(-4.3,-2.58) rectangle (7.3,6.3);
\draw [line width=1.pt,domain=-4.3:7.3] plot(\x,{(--11.9652--4.46*\x)/1.94});
\draw [line width=1.pt,domain=-4.3:7.3] plot(\x,{(--5.3592-4.1*\x)/1.28});
\draw [line width=1.pt,domain=-4.3:7.3] plot(\x,{(-4.537058206456756--1.9400374651255479*\x)/-2.2195352730171383});
\draw [line width=1.pt,domain=-4.3:7.3] plot(\x,{(--6.8618311356282815--1.9401516999946962*\x)/2.726684347318729});
\draw [line width=1pt,dash pattern=on 1pt off 2pt,color=black] (-.3,.55) -- (-.37,6.3);
\draw [line width=1.pt,dash pattern=on 1pt off 2pt,color=black,domain=-4.3:7.3] plot(\x,{(--3.6616--0.36*\x)/3.22});
\draw [line width=1.pt,dash pattern=on 1pt off 2pt,color=black,domain=-4.3:7.3] plot(\x,{(--5.0217588196585545-0.3598857651308518*\x)/1.7262196203358675});
\draw (-0.1,5.42) node[anchor=north west] {$x+z-y=$};
\draw (2.3,4.44) node[anchor=north west] {$y+z-x=w=0$};
\draw (2.7,0.26) node[anchor=north west] {$x+y-z=w=0$};
\draw (2.24,1.5) node[anchor=north west] {$x=w=0$};
\draw (2.76,2.9) node[anchor=north west] {$z=w=0$};
\draw (-1.8,0.68) node[anchor=north west] {$y=w=0$};
\draw (0.02,5.) node[anchor=north west] {$w=0$};
\draw (-3.1,-0.5) node[anchor=north west] {$x+y+z=$};
\draw (-3.12,-1.) node[anchor=north west] {$w=0$};
\draw [fill=white] (-2.3,0.88) circle (2.5pt);
\draw [fill=white] (-1.2995352730171383,3.180037465125548) circle (2.5pt);
\draw [fill=white] (-0.36,5.34) circle (2.5pt);
\draw [fill=black] (0.92,1.24) circle (2.5pt);
\draw [fill=black] (0.42668434731872906,2.820151699994696) circle (2.5pt);
\draw [fill=black] (-0.29792823980984157,2.3045590740504807) circle (2.5pt);
\end{tikzpicture}
\caption[$3$-point lines of $D_4$ in a coordinate plane.]{The 6 points of $Z=D_4$ and the four 
3 point lines on coordinate plane $w=0$ (the 3 coordinate lines are shown dashed,
the null points are shown as open dots and the nonnull points are shown as filled in dots), for $Z$ as given in Remark \ref{rem:onD4}(c).}
\label{D4Fig1}
\end{figure}
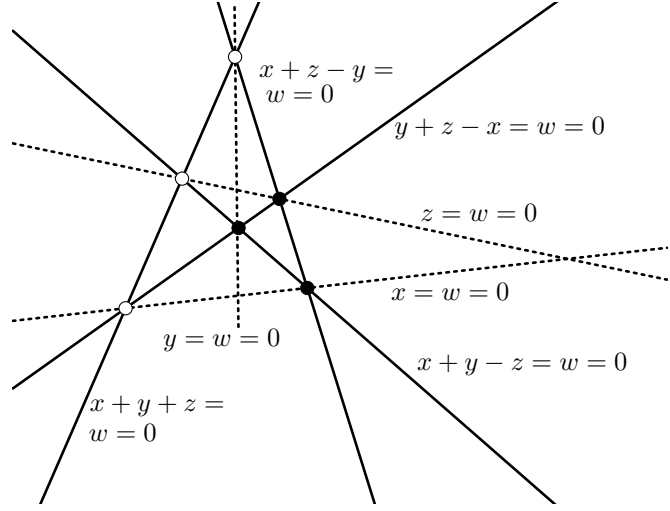

\begin{proposition}[Projective automorphisms of $D_4$]\label{symGrpD4}
The group $B$ of projective symmetries of $Z$ has order $24^2=576$.
It is the group $A$ extended by an element of order 3 corresponding
to $\delta$ from Example \ref{DiagActionEx} acting on a $(3,3)$-subgrid of $Z$.
\end{proposition}

\begin{proof}
First note that $B$ is transitive on $Z$ since $A\subset B$ is.
Let $p=1100\in Z$. Then $B_p$ is transitive on the points other than $p$ 
on 3-point lines through $p$, since $A_p$
is by the proof of Proposition \ref{D4factsProp}.
Pick one of the lines through $p$. The subgroup fixing two points on the line fixes the third.
Let $L$ be the 3-point line containing 1$*$00, 10$*$0, 01$*$0.
Consider the subgroup $B_L$ of $B$ fixing the points 1$*$00, 10$*$0, 01$*$0,
and hence acts as the identity on $L$. 
Thus $B_L$ is a subgroup of the group $G$ of symmetries of the grid defined by the 6 lines disjoint from $L$;
it is the subgroup fixing three collinear points of the 6 points of concurrency of the grid.
By Proposition \ref{GridGroupProp}, the group $G$ of symmetries of the grid 
acts transitively on these 6 points, and there is an element of $G$ fixing one of them
but transposing the other two collinear with it (this is $\tau$ from Example \ref{DiagActionEx}).
Thus $6\cdot 2\cdot |B_L|=G$, so $B_L=6$, hence $B=|Z|\cdot 8\cdot 6=24^2=576$.
Note that $A$ has order $192=576/3$ and that $\delta\not\in A$ but $\delta\in B$,
hence $B$ is the group obtained from $A$ and $\delta$.
\end{proof}

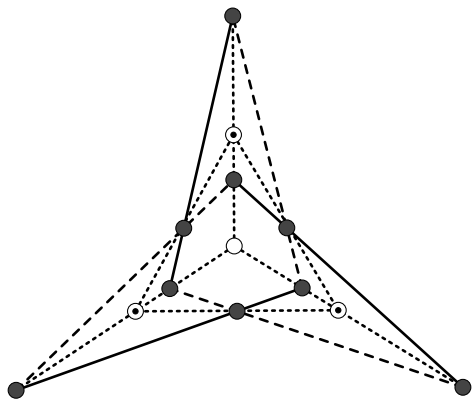
\begin{figure}[ht!]
\definecolor{uuuuuu}{rgb}{0.26666666666666666,0.26666666666666666,0.26666666666666666}
\definecolor{xdxdff}{rgb}{1,1,1}
\definecolor{ududff}{rgb}{0.26666666666666666,0.26666666666666666,0.26666666666666666}
\begin{tikzpicture}[line cap=round,line join=round,>=triangle 45,x=1.0cm,y=1.0cm]
\clip(-4.3,0) rectangle (7,6.3);
\draw [line width=1.pt,dotted] (1.3,2.72)-- (-1.6,0.8);
\draw [line width=1.pt,dotted] (1.3,2.72)-- (1.28,5.78);
\draw [line width=1.pt,dotted] (1.3,2.72)-- (4.34,0.84);
\draw [line width=1.pt,dashed] (-1.6,0.8)-- (1.2942486117043999,3.5999624092268263);
\draw [line width=1.pt] (0.4435097679078306,2.1529443980631155)-- (1.28,5.78);
\draw [line width=1.pt] (1.2942486117043999,3.5999624092268263)-- (4.34,0.84);
\draw [line width=1.pt,dashed] (2.202983536508735,2.1615759708432822)-- (1.28,5.78);
\draw [line width=1.pt] (-1.6,0.8)-- (2.202983536508735,2.1615759708432822);
\draw [line width=1.pt,dashed] (0.4435097679078306,2.1529443980631155)-- (4.34,0.84);
\draw [line width=1.pt,dotted] (-0.013012135842068645,1.8506954135114577)-- (1.2903314822725331,4.199283212302435);
\draw [line width=1.pt,dotted] (1.2903314822725331,4.199283212302435)-- (2.682943018159048,1.864758922980588);
\draw [line width=1.pt,dotted] (2.682943018159048,1.864758922980588)-- (-0.013012135842068645,1.8506954135114577);
\begin{scriptsize}
\draw [fill=xdxdff] (1.3,2.72) circle (3pt);
\draw [fill=ududff] (-1.6,0.8) circle (3pt);
\draw [fill=ududff] (1.28,5.78) circle (3pt);
\draw [fill=ududff] (4.34,0.84) circle (3pt);
\draw [fill=uuuuuu] (0.4435097679078306,2.1529443980631155) circle (3pt);
\draw [fill=uuuuuu] (1.2942486117043999,3.5999624092268263) circle (3pt);
\draw [fill=uuuuuu] (2.202983536508735,2.1615759708432822) circle (3pt);
\draw [fill=uuuuuu] (0.6287471035526917,2.9561410047445027) circle (3pt);
\draw [fill=uuuuuu] (1.9989750471066698,2.9613618727637654) circle (3pt);
\draw [fill=uuuuuu] (1.3375014699734586,1.8517088431848634) circle (3pt);
\draw [fill=xdxdff] (-0.013012135842068645,1.8506954135114577) circle (3pt);
\draw [fill=xdxdff] (1.2903314822725331,4.199283212302435) circle (3pt);
\draw [fill=xdxdff] (2.682943018159048,1.864758922980588) circle (3pt);
\draw [fill=black] (-0.013012135842068645,1.8506954135114577) circle (1pt);
\draw [fill=black] (1.2903314822725331,4.199283212302435) circle (1pt);
\draw [fill=black] (2.682943018159048,1.864758922980588) circle (1pt);
\end{scriptsize}
\end{tikzpicture}
\caption[A $(3,3)$-grid in the $D_4$ configuration.]{A $(3,3)$-grid in $Z$. The open dot and the three dotted dots are the coordinate vertices, the dotted lines are the
coordinate axes. The black dots are 9 of the 12 points of $Z$; not shown are the three additional
points of $Z$ in the plane containing the three dotted dots. The grid lines in one ruling are shown dashed;
the grid lines in the other ruling are shown solid.}
\label{D4Minus3}
\end{figure}

Now we show that there is essentially one way to extend a $(3,3)$-grid to a $D_4$ configuration, as already done in Definition \ref{def:D4}.
\begin{corollary}
For each $(3,3)$-grid $Y$, there are exactly two distinct sets $S_1$ and $S_2$ of $3$ points
such that, for either $i=1$ or $i=2$, $Y\cup S_i$ is projectively equivalent to $Z$.
\end{corollary}

\begin{proof}
The 18 2-point lines for the grid $Y$ have 6 points of concurrency for $Y$.
They occur in two collinear subsets of 3. Call one collinear subset $S_1$ and the other $S_2$.
Let $X$ be a grid contained in $Z$. There is a linear automorphism $\alpha$ of $\PP^3$ taking 
$X$ to $Y$ by Lemma \ref{3by3GridUniqueLem}. 
The image of the remaining 3 points of $Z$ under $\alpha$ will be collinear points of 
concurrency for the 2-point lines of $Y$, so must be either $S_1$ or $S_2$; say it is $S_1$.
There is an automorphism $\beta$ in the group $G$ of
projective symmetries of $Y$ such that $\beta$ takes $S_1$ to $S_2$. 
Thus $Y\cup S_1$ is projectively equivalent to $Z$ via $\alpha$, while
$Y\cup S_2$ is projectively equivalent to $Z$ via $\beta\alpha$.
Finally, $S_1$ and $S_2$ are the only possibilities since
every linear map from a grid in $Z$ to $Y$ takes the remaining 3 points of $Z$ to either $S_1$ or $S_2$. 
\end{proof}

\begin{lemma}[Coplanarities in $D_4$]\label{CoplanarLem}
The maximum number of coplanar points of $Z$ is 6
and any plane with more than 3 points must have 6,
and those 6 points occur in four 3-point lines.
Moreover, there are 12 6-point planes and every point of $Z$ is on 6 such planes 
(thus the 12 points of $Z$ and the 12 6-point planes form a $(12_6,12_6)$
point-plane arrangement).
The 12 points dual to the 12 6-point planes are projectively equivalent to $Z$.
\end{lemma}

\begin{proof}
Let $I$ be a plane containing 4 or more points of $Z$.
Let $p\in I\cap Z$.
Each point of $Z$ is on 3 2-point lines and these lines are not coplanar.
Thus there must be a 3-point line $L$ in $I$ through $p$ and another line $J$,
either 3-point or 2-point, in $I$ through $p$.
Since $B$ acts transitively on $Z$ and $B_p$ is transitive on the 4 3-point lines through $p$,
we can assume that $p=1100$ and $L$ is the line through $p$ and 1010 and 01$*$0.
Now $J$ is either the line through $p$ and 0011 or 001$*$, or it is the line
through $p$, 0110, 10$*$0, or through $p$, 1001 and 010$*$ or through $p$, 0101 and 100$*$.
Checking each case, the planes we obtain, and the points in $Z$ in those planes, are thus 
%
$w$: 1100, 1010, 01$*$0, 0110,  1$*$00, 10$*$0;
$x-y-z+w$: 1100, 1010, 01$*$0, 100$*$, 0101, 0011;
$x-y-z-w$: 1100, 1010, 01$*$0, 1001, 010$*$, 001$*$.
And there are four 3-point lines in each plane. 

We also see that each plane is defined by two 3-point lines through $p$.
Since there are 4 such lines, the six 6-point planes are just the
planes defined by combinations of 2 3-point lines through $p$.

Counting the planes at each point of $Z$ we get $6\cdot 12=72$, and since there are
6 points per plane, we find that there are $72/6=12$ planes.

The points dual to the 12 planes are not on a quadric but contain a $(3,3)$-grid
with three collinear points of concurrency, hence are projectively equivalent to $Z$.
\end{proof}

\section{The geprociness of \texorpdfstring{$D_4$}{D4}}
In this section we establish that $D_4$ is a $(3,4)$-geproci set. We provide two alternative proofs for this fundamental fact. 
\begin{lemma}\label{l.D4 in P2}
Let $H_1,H_2,H_3$, $V_1,V_2,V_3$ be six different lines in $\PP^2$.	 
Set $P_{ij}=H_i\cap V_j$, $1\le i,j\le 3$  and $A_1= P_{11}P_{22} \cap P_{13}P_{32}$ and $A_2= P_{22}P_{33} \cap P_{21}P_{13}$.
Then the eleven points  $P_{ij}$, $1\le i,j\le 3$, $A_1$ and $A_2$ are on a cubic.
\end{lemma}
\begin{proof}
Consider nine the points  $A_1$, $A_2$ and
\[\begin{array}{ccc}
P_{11}&   & P_{13}\\
P_{21}& P_{22}&    \\
P_{31}& P_{32}& P_{33}\\
\end{array}
\]

and note that the following two cubics contain these nine points
\[\begin{array}{rl}
	C_{1}=& V_1 \cup P_{13}P_{32}\cup P_{22}P_{33}\\
	C_{2}=& H_3 \cup P_{11}P_{22}\cup P_{21}P_{13}\\
\end{array}
\]

Thus, there is a cubic $C$ vanishing at the following ten points:
$A_1$, $A_2$ and
\[\begin{array}{ccc}
P_{11}&P_{23} & P_{13}\\
P_{21}& P_{22}&    \\
P_{31}& P_{23}& P_{33}\\
\end{array}
\]

Since the $h$-vector of the $8$ points in the above table is $(1,2,3,2)$, the cubic $C$ is in the pencil generated by $H_1\cup H_2\cup H_3$ and $V_1\cup V_2\cup V_3$, and it also contains the point $P_{23}$.  
\end{proof}

\begin{proposition}[$D_4$ is a $(3,4)$-geproci set]\label{prop:D4_is_geproci}\index{configuration! $D_4$}\index{geproci! $D_4$}
A general projection of the $D_4$ configuration is a complete intersection of degrees $(3,4)$.	
\end{proposition}
\begin{proof}
Here we use the coordinates from Definition \ref{def:D4}:

\[\begin{array}{lll}
[1:0:0:0], & [0:1:0:0], & [1:1:0:0],   \\
{[0:0:1:0]}, & [0:0:0:1], & [0:0:1:1],    \\
{[1:0:1:0]}, & [0:1:0:1], & [1:1:1:1],    \\
{[1:1:1:0]}, & [0:1:1:1], & [1:0:0:-1].    
\end{array}
\] 

Note that the general projection of the first nine points and any two of the last row in the  above table, satisfies the conditions in Lemma \ref{l.D4 in P2}. Thus, by B{\'e}zout's Theorem, the general projection of the $D_4$  configuration is contained in a unique irreducible cubic curve.
Since the whole configuration is also contained in four lines, we are done.
\end{proof}

Here is an alternate argument, using the group law on a cubic curve, a similar result can be found in \cite[Section 7]{dolgachev2004}. 

Consider projection of the $D_4$ arrangement of 12 points to a general plane.
We can see that there is always a cubic through the projected images of the 12 points of the $D_4$
using the group law on the smooth points of an irreducible cubic. 

The 12 points of the $D_4$ arrangement contains a $(3,3)$-grid; see Figure \ref{D4Fig2-2}.
The 6 grid lines project to two cubics $B$ and $G$ defining a pencil of cubics through the images
of the 9 grid points (in Figure \ref{D4Fig2-2} $B$ is the image of the lines represented  by the blue arcs
and $G$ is the image of the lines represented  by the green arcs).
The only lines through 3 of the grid points are the 6 lines of the grid,
hence there are no additional lines through any 3 of the images
of the 9 grid points. One component of any reducible cubic in a pencil of cubics through 9 points
is always a line through three of the 9 points.
Thus the pencil defined by the images of the 6 grid lines has only 2 reducible cubics, namely $B$ and $G$. 

Hereafter regard Figure \ref{D4Fig2-2} as showing the projection of the $D_4$ arrangement into the plane.
We wish to see why there is a cubic in the pencil defined by $B$ and $G$ that contains all 12 points
(labelled $a\ldots,l$ in Figure \ref{D4Fig2-2}). 
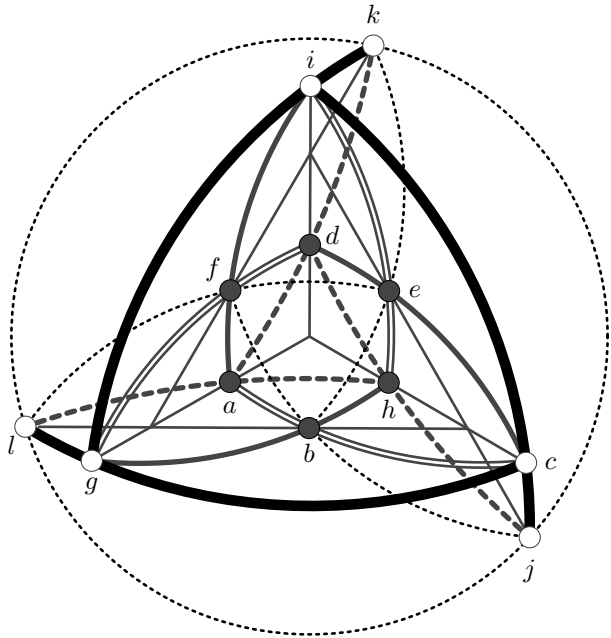
\begin{figure}[ht!]
\definecolor{ffqqff}{rgb}{0.26666666666666666,0.26666666666666666,0.26666666666666666}
\definecolor{qqqqff}{rgb}{0.26666666666666666,0.26666666666666666,0.26666666666666666}
\definecolor{qqffqq}{rgb}{0.26666666666666666,0.26666666666666666,0.26666666666666666}
\definecolor{uuuuuu}{rgb}{0.26666666666666666,0.26666666666666666,0.26666666666666666}
\definecolor{ffqqqq}{rgb}{0.26666666666666666,0.26666666666666666,0.26666666666666666}
\begin{tikzpicture}[line cap=round,line join=round,>=triangle 45,x=1.0cm,y=1.0cm]
\clip(-5.229995463038163,-2.9609516110909535) rectangle (6.579611387108048,6.079506046607177);

\draw [shift={(-2.004330830239797,-2.592439645708379)},line width=2pt,color=qqffqq]  plot[domain=0.41151273756358037:1.1504612622861772,variable=\t]({1.*5.650742093257234*cos(\t r)+0.*5.650742093257234*sin(\t r)},{0.*5.650742093257234*cos(\t r)+1.*5.650742093257234*sin(\t r)});
\draw [shift={(4.858527562611111,1.3265846411658468)},line width=2pt,color=qqffqq]  plot[domain=2.505907839956776:3.244856364679372,variable=\t]({1.*5.65074209325724*cos(\t r)+0.*5.65074209325724*sin(\t r)},{0.*5.65074209325724*cos(\t r)+1.*5.65074209325724*sin(\t r)});
\draw [shift={(-1.966876224295619,5.310482208512857)},line width=2pt,color=qqffqq]  plot[domain=4.600302942349971:5.339251467072568,variable=\t]({1.*5.65074209325723*cos(\t r)+0.*5.65074209325723*sin(\t r)},{0.*5.65074209325723*cos(\t r)+1.*5.65074209325723*sin(\t r)});

\draw [shift={(2.595877835623589,5.288857781688598)},line width=3pt,color=qqqqff]  plot[domain=4.076047891677878:4.814996416400475,variable=\t]({1.*5.650742093257232*cos(\t r)+0.*5.650742093257232*sin(\t r)},{0.*5.650742093257232*cos(\t r)+1.*5.650742093257232*sin(\t r)});
\draw [shift={(2.5584232296794163,-2.614064072532644)},line width=3pt,color=qqqqff]  plot[domain=1.9816527892846834:2.7206013140072796,variable=\t]({1.*5.6507420932572385*cos(\t r)+0.*5.6507420932572385*sin(\t r)},{0.*5.6507420932572385*cos(\t r)+1.*5.6507420932572385*sin(\t r)});
\draw [shift={(-4.266980557227314,1.3698334948143738)},line width=3pt,color=qqqqff]  plot[domain=-0.11274231310851324:0.6262062116140841,variable=\t]({1.*5.650742093257235*cos(\t r)+0.*5.650742093257235*sin(\t r)},{0.*5.650742093257235*cos(\t r)+1.*5.650742093257235*sin(\t r)});

\draw [shift={(2.595877835623589,5.288857781688598)},line width=1pt,color=white]  plot[domain=4.076047891677878:4.814996416400475,variable=\t]({1.*5.650742093257232*cos(\t r)+0.*5.650742093257232*sin(\t r)},{0.*5.650742093257232*cos(\t r)+1.*5.650742093257232*sin(\t r)});
\draw [shift={(2.5584232296794163,-2.614064072532644)},line width=1pt,color=white]  plot[domain=1.9816527892846834:2.7206013140072796,variable=\t]({1.*5.6507420932572385*cos(\t r)+0.*5.6507420932572385*sin(\t r)},{0.*5.6507420932572385*cos(\t r)+1.*5.6507420932572385*sin(\t r)});
\draw [shift={(-4.266980557227314,1.3698334948143738)},line width=1pt,color=white]  plot[domain=-0.11274231310851324:0.6262062116140841,variable=\t]({1.*5.650742093257235*cos(\t r)+0.*5.650742093257235*sin(\t r)},{0.*5.650742093257235*cos(\t r)+1.*5.650742093257235*sin(\t r)});

\draw [line width=1.pt,color=ffqqqq] (0.2957735026918962,1.3482090679901106)-- (-2.598920151018907,-0.30480110394799387);
\draw [line width=1.pt,color=ffqqqq] (0.2957735026918962,1.3482090679901106)-- (0.3115715279348178,4.681592394246314);
\draw [line width=1.pt,color=ffqqqq] (0.2957735026918962,1.3482090679901106)-- (3.1746691311597792,-0.3321640863279878);
\draw [line width=1.pt,dash pattern=on 1pt off 2pt,color=black] (0.2957735026918962,1.3482090679901106) circle (3.96299094205599cm);

\draw [line width=1.pt,color=ffqqqq,domain=-1.82:1.2] plot(\x,{(--3.4655231911783--1.8223136019851656*\x)/1.0636602540378448});
\draw [line width=1.pt,color=ffqqqq,domain=-3.3:2.4] plot(\x,{(-0.2772--0.01*\x)/-2.11});
\draw [line width=1.pt,color=ffqqqq,domain=0.3:3.2] plot(\x,{(--4.523113414279856-1.8323136019851658*\x)/1.0463397459621553});
\draw [shift={(-9.786216342860852,7.232932679033489)},line width=2.pt,color=ffqqff,dashed]  plot[domain=5.659784774992629:6.101005028397384,variable=\t]({1.*11.114824665765795*cos(\t r)+0.*11.114824665765795*sin(\t r)},{0.*11.114824665765795*cos(\t r)+1.*11.114824665765795*sin(\t r)});
\draw [shift={(0.2404482840546112,-10.325412064476998)},line width=2.pt,color=ffqqff,dashed]  plot[domain=1.4709945702062388:1.9122148236109935,variable=\t]({1.*11.11482466576578*cos(\t r)+0.*11.11482466576578*sin(\t r)},{0.*11.11482466576578*cos(\t r)+1.*11.11482466576578*sin(\t r)});
\draw [shift={(10.433088566881926,7.137106589413849)},line width=2.pt,color=ffqqff,dashed]  plot[domain=3.565389672599434:4.006609926004188,variable=\t]({1.*11.114824665765788*cos(\t r)+0.*11.114824665765788*sin(\t r)},{0.*11.114824665765788*cos(\t r)+1.*11.114824665765788*sin(\t r)});

\draw [shift={(4.718690590602634,-1.2333889117377246)},line width=4.pt]  plot[domain=2.076619948272229:3.0153696150788445,variable=\t]({1.*7.3762932617047365*cos(\t r)+0.*7.3762932617047365*sin(\t r)},{0.*7.3762932617047365*cos(\t r)+1.*7.3762932617047365*sin(\t r)});
\draw [shift={(0.32004439153941744,6.469366614817032)},line width=4.pt]  plot[domain=4.171015050665425:5.109764717472039,variable=\t]({1.*7.376293261704749*cos(\t r)+0.*7.376293261704749*sin(\t r)},{0.*7.376293261704749*cos(\t r)+1.*7.376293261704749*sin(\t r)});
\draw [shift={(-4.151414474066371,-1.1913504991089645)},line width=4.pt]  plot[domain=-0.017775154120966086:0.9209745126856483,variable=\t]({1.*7.3762932617047445*cos(\t r)+0.*7.3762932617047445*sin(\t r)},{0.*7.3762932617047445*cos(\t r)+1.*7.3762932617047445*sin(\t r)});

\draw [shift={(-3.0648897791590204,3.30978360500457)},line width=1.pt,dash pattern=on 1pt off 2pt,color=black]  plot[domain=-0.7586081051645435:0.4258384521306778,variable=\t]({1.*4.62237051792276*cos(\t r)+0.*4.62237051792276*sin(\t r)},{0.*4.62237051792276*cos(\t r)+1.*4.62237051792276*sin(\t r)});
\draw [shift={(0.2773317631461349,-2.5429979761655943)},line width=1.pt,dash pattern=on 1pt off 2pt,color=black]  plot[domain=1.3357869972286525:2.5202335545238737,variable=\t]({1.*4.622370517922758*cos(\t r)+0.*4.622370517922758*sin(\t r)},{0.*4.622370517922758*cos(\t r)+1.*4.622370517922758*sin(\t r)});
\draw [shift={(3.674878524088573,3.277841575131356)},line width=1.pt,dash pattern=on 1pt off 2pt,color=black]  plot[domain=3.4301820996218475:4.614628656917069,variable=\t]({1.*4.622370517922755*cos(\t r)+0.*4.622370517922755*sin(\t r)},{0.*4.622370517922755*cos(\t r)+1.*4.622370517922755*sin(\t r)});
\draw [fill=uuuuuu] (0.29,0.13) circle (4pt);
\draw [fill=uuuuuu] (-0.7563397459621553,1.9623136019851657) circle (4pt);
\draw [fill=uuuuuu] (1.3536602540378446,1.9523136019851657) circle (4pt);
\draw [fill=uuuuuu] (-0.762113248654052,0.7441045339950554) circle (4pt);
\draw [fill=uuuuuu] (0.30154700538379287,2.566418135980221) circle (4pt);
\draw [fill=uuuuuu] (1.347886751345948,0.7341045339950554) circle (4pt);
\draw [color=uuuuuu,fill=white] (0.3115715279348178,4.681592394246314) circle (4pt);
\draw [color=uuuuuu,fill=white] (3.1746691311597792,-0.3321640863279878) circle (4pt);
\draw [color=uuuuuu,fill=white] (-2.598920151018907,-0.30480110394799387) circle (4pt);
\draw [color=uuuuuu,fill=white] (1.1446693152601979,5.219213215006983) circle (4pt);
\draw [color=uuuuuu,fill=white] (-3.481062333063777,0.14787233333205582) circle (4pt);
\draw [color=uuuuuu,fill=white] (3.223713525879269,-1.3224583443687075) circle (4pt);
\draw (-0.762113248654052,0.4) node {$a$};
\draw (0.29,-.2) node {$b$};
\draw (3.5,-0.3321640863279878) node {$c$};
\draw (0.6,2.7) node {$d$};
\draw (1.7,1.9523136019851657) node {$e$};
\draw (-0.7563397459621553,1.9623136019851657) node[anchor=south east] {$f$};
\draw (-2.598920151018907,-0.4) node[anchor=north] {$g$};
\draw (1.347886751345948,0.6341045339950554) node[anchor=north] {$h$};
\draw (0.3115715279348178,4.8) node[anchor=south] {$i$};
\draw (3.223713525879269,-1.5) node[anchor=north] {$j$};
\draw (1.1446693152601979,5.4) node[anchor=south] {$k$};
\draw (-3.481062333063777,0.14787233333205582) node[anchor=north east] {$l$};
\end{tikzpicture}
\caption[Schematic representation of 
the 12 points of the $D_4$ configuration.]{Schematic representation of the 
the 12 points of the $D_4$ configuration and one of the grids it contains.
The 12 dots labelled $a, b,\ldots,l$ are the 12 points; the 3 doubled arcs and the three medium bold arcs
represent the 6 grid lines.
The points labelled $a,b,\ldots,i$ are the 9 grid points.}
\label{D4Fig2-2}
\end{figure}
Every cubic in the pencil contains $a,\ldots,i$.
Pick one of the remaining 3 points, say $j$. 
There is a cubic $C$ in the pencil such that $j\in C$, hence $a,\ldots,j\in C$
and these are smooth points of $C$ since each point is on a line which has three distinct points of $C$ on it.
Also, any three collinear smooth points of  $C$ add up to 0 in the group law on $C$.

We want to show that $k\in C$. Note that $k$ is where the line through $b$ and $e$ and the line through
$a$ and $d$ cross. Thus $k$ is on $C$ if and only if $a+d=b+e$ in the group law on $C$
(because then the point of $C$ collinear with $a$ and $d$ is the same point as the
point of $C$ collinear with $b$ and $e$, and $k$ is the only point on both lines).

Thus we want to show that $a+d-b-e=0$ in the group law.
But $c+d+e=0$ since these points are collinear smooth points of $C$, and likewise $a+b+c=0$.
Substituting $-e=c+d$ into $a+d-b-e$ gives $a+d-b-e=a+d-b+c+d=a+2d-b+c$, and
substituting in $c=-a-b$ now gives $a+2d-b+c=a+2d-b-a-b=2(d-b)$.
Thus $k$ is on $C$ if $2(d-b)=0$.

But $d,h,j$ are collinear points of $C$ so $d+h+j=0$, and likewise $b+g+h=0$ and $b+f+j=0$,
so $h=-b-g$ and $j=-b-f$. Substituting these into $d+h+j=0$ gives
$0=d-b-g-b-f=d-2b-g-f$. But $d,f,g$ are collinear so $d+f+g=0$ or $-f=d+g$.
Substituting this in gives $0=d-2b-g-f=d-2b-g-d+g=2(d-b)$, as we wanted.

Thus $k\in C$. By the 3 fold cyclic symmetry of Figure \ref{D4Fig2-2}, we now see that $l\in C$ too.

If we take the quartic $\calq$ given by either $B$ or $G$ with the line $L$ through $j,k,l$, we see that the 12
points $a,\ldots,l$ are the intersection of $\calq$ and $C$.
\medskip\medskip

We now consider the general projection of $D_4$ and establish and exclude some collinearities. It is convenient to change back to coordinates of points in $D_4$ as in Remark \ref{rem:onD4}, i.e.,
${D_4}\subset \PP^3$ is the set of 12 points $[a:b:c:d]$ where 
two coordinates are $0$ and the other two are taken from from the set $\{-1,1\}$.
\begin{remark}\label{3x3GridRem}
Let $L$ be a line through any three points of ${D_4}$; for example.
$p_1=[0:1:-1:0], p_2=[1:0:-1:0], p_3=[1:-1:0:0]$. The remaining 9 points
of ${D_4}$ give a $3\times 3$ grid, shown diagrammatically in Figure \ref{3by3gridFig}, 
where $H=\{H_1,H_2,H_3\}$ (these lines are shown as solid) 
and $V=\{V_1,V_2,V_3\}$ (these lines are shown as dotted) are the grid lines 
(so $H$ and $V$ each consist of 3 disjoint lines, but each
of the 3 lines in $H$ meets each of the 3 lines in $V$ in 
exactly one point, giving the 9 points of the grid).

\begin{figure}
\hskip.75in\begin{tikzpicture}[line cap=round,line join=round,>=triangle 45,x=.9cm,y=.9cm]
\clip(-6.24,-1) rectangle (7.127619047619046,8);
\draw [line width=1.pt,dotted,domain=-2.173557454166223:7.127619047619046] plot(\x,{(--15.908623301238265--6.867532923536022*\x)/2.1735574541662244});
\draw [line width=1.pt,domain=-2.6439639420810233:7.127619047619046] plot(\x,{(--3.9774038724013487-1.5695582784667867*\x)/7.116520429334922});
\draw [line width=1.pt,domain=-6.239999999999998:0.46637273040552907] plot(\x,{(-22.047881897604483-5.4017043880717335*\x)/-4.938929217659425});
\draw [line width=1.pt,dotted,domain=-0.3511189295271149:7.127619047619046] plot(\x,{(--21.53337721895241-5.275651364327068*\x)/4.823675416781013});
\draw [line width=1.pt,domain=-6.239999999999998:2.1526546232405837] plot(\x,{(-15.755632055256672--6.801488716008047*\x)/-2.1526546232405828});
\draw [line width=1.pt,dotted,domain=-6.239999999999998:2.5486025948986732] plot(\x,{(-3.9241066753614913-1.5485262033955407*\x)/-7.021159082152569});
\begin{scriptsize}
\draw [fill=black] (-2.,1.) circle (2.5pt);
\draw[color=black] (-1.4921904761904758,1.23) node {0101};
\draw [fill=black] (2.,1.) circle (2.5pt);
\draw[color=black] (1.4921904761904758,1.23) node {1001};
\draw [fill=black] (0.,4.464101615137755) circle (2.5pt);
\draw[color=black] (0,3.788285714285709) node {0011};
\draw [fill=black] (1.3820061658979723,2.9526021702821668) circle (2.5pt);

\draw [line width=1.pt,color=lightgray] (1.4,1.8) -- (-1.35,2.27);
\draw [fill=black] (0.8,1.9) circle (2.5pt);
\draw [fill=black] (-0.8,2.2) circle (2.5pt);
\draw [fill=black] (0,2.05) circle (2.5pt);

\draw [fill=black] (0.8,1.6) node {10{$*$}0};
\draw [fill=black] (-0.8,2.5) node {01{$*$}0};
\draw [fill=black] (0.3,2.3) node {1{$*$}00};

\draw[color=black] (1.7805714285714282,3.258190476190471) node {1010};
\draw [fill=black] (-1.3820061658979694,2.952602170282168) circle (2.5pt);
\draw[color=black] (-2,3.235142857142852) node {0110};
\draw [fill=black] (0.,0.5588972745734192) circle (2.5pt);
\draw[color=black] (0,0.10066666666666214) node {1100};
\draw [fill=black] (-4.472556487253896,-0.4275311535025909) circle (2.5pt);
\draw[color=black] (-4.096571428571428,-0.7981904761904806) node {010{$*$}};
\draw [fill=black] (0.,7.319163922142935) circle (2.5pt);
\draw[color=black] (-.6,7.3) node {001{$*$}};
\draw [fill=black] (4.4725564872538985,-0.42753115350259463) circle (2.5pt);
\draw[color=black] (4.915047619047618,-0.037619047619052104) node {100{$*$}};
\draw[color=black] (-1.2501904761904759,4.95219047619047) node {$V_1$};
\draw[color=black] (1.769047619047619,-0.049142857142861596) node {$H_3$};
\draw[color=black] (-2.9,1.9) node {$H_1$};
\draw[color=black] (2.875333333333333,1.8) node {$V_2$};
\draw[color=black] (1.1006666666666667,4.906095238095232) node {$H_2$};
\draw[color=black] (-1.7,-0.1) node {$V_3$};
\end{scriptsize}
\end{tikzpicture}
\caption[A $(3,3)$-grid in a ${D_4}$.]{A $(3,3)$-grid in ${D_4}$ (as given in Remark \ref{rem:onD4}(c)) together with the remaining 3 points 
(shown on the gray line). Here $*$ represents $-1$.}\label{3by3gridFig}
\end{figure}
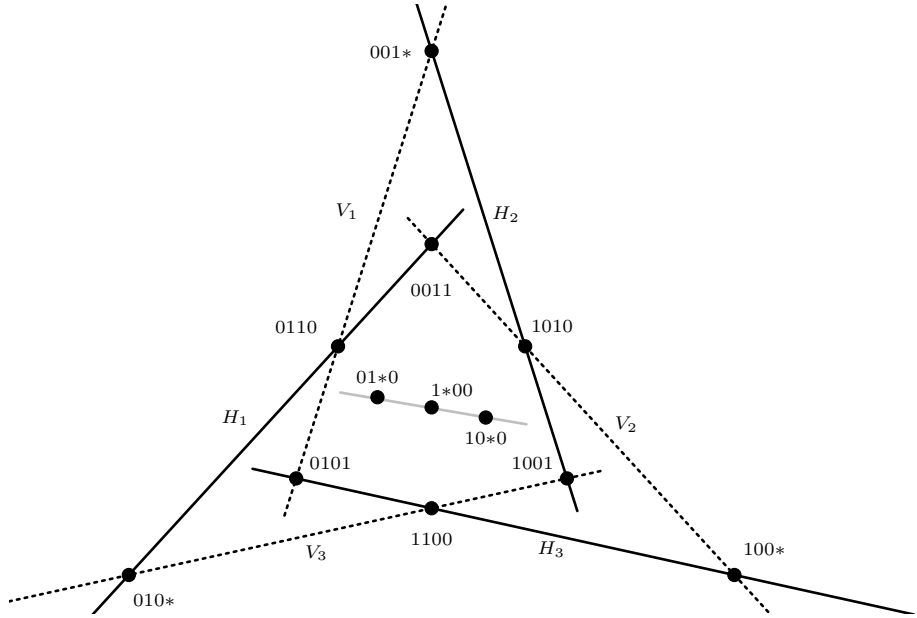

Let $\Pi$ be a general plane and $p$ a point not in $\Pi$.
Let $\pi_p$ be projection from $p$ to $\Pi$ and
let $h$, $v$ and $l$ be the images of $H$, $V$ and $L$.
Note that ${D_4}$ is a half grid: it is contained in 4 pairwise skew lines
(namely $H\cup L$ or $V\cup L$). Thus when $P$ is a general point, $\pi_P({D_4})$ 
is contained in a quartic consisting of 4 lines (such as $h\cup l$ or $v\cup l$).
But ${D_4}$ does not contain any sets of 4 (or more) collinear points,
so while, for $P$ general, $\pi_P({D_4})$ is contained in a unique cubic,
this cubic does not consist of lines (in fact it is irreducible, as shown in Figure \ref{3by3CI-Fig},
where $\pi_P(Z)$ consists of the 12 black points, $v$ consists of the three dotted lines, 
$l$ is the dashed line, $h$ consists of the three solid
lines and the cubic is shown as a solid black curve).

\begin{figure}
\hbox{\hskip.5in\raise.15in\hbox{\begin{tikzpicture}[line cap=round,line join=round,>=triangle 45,x=.75cm,y=.75cm]
\clip(-4.496680397937174,0.75) rectangle (5.928312362917013,6.5);
\fill[line width=2.pt,color=black,fill=white] (-1.176350188308613,1.7059124529228467) -- (-0.8953798813051128,2.7315518487365438) -- (0.,3.230935550935551) -- (0.,2.) -- cycle;
\fill[line width=2.pt,color=black,fill=lightgray,fill opacity=0.30000000149011612] (0.,3.230935550935551) -- (0.6656059123409749,2.7359609200174266) -- (0.8751724137931034,1.7082758620689655) -- (0.,2.) -- cycle;
\fill[line width=2.pt,color=black,fill=gray,fill opacity=0.60000000149011612] (-1.176350188308613,1.7059124529228467) -- (-0.23559785204630612,1.546900694008833) -- (0.8751724137931034,1.7082758620689655) -- (0.,2.) -- cycle;
\draw [dotted,line width=0.4pt] (-4.,1.)-- (0.,2.);
\draw [dotted,line width=0.4pt] (0.,2.)-- (0.,6.);
\draw [dotted,line width=0.4pt] (0.,2.)-- (3.,1.);
\draw [line width=0.4pt,color=black] (-4.,1.)-- (0.8751724137931034,1.7082758620689655);
\draw [line width=0.4pt,dash pattern=on 2pt off 4pt] (3.,1.)-- (-1.176350188308613,1.7059124529228467);
\draw [line width=0.4pt,dash pattern=on 2pt off 4pt] (-4.,1.)-- (0.,3.230935550935551);
\draw [line width=0.4pt,color=black] (0.,3.230935550935551)-- (3.,1.);
\draw [line width=0.4pt,color=black] (-1.176350188308613,1.7059124529228467)-- (0.,6.);
\draw [line width=0.4pt,dash pattern=on 2pt off 4pt] (0.,6.)-- (0.8751724137931034,1.7082758620689655);
\draw [line width=0.4pt,color=black] (-1.176350188308613,1.7059124529228467)-- (-0.8953798813051128,2.7315518487365438);
\draw [line width=0.4pt,color=black] (-0.8953798813051128,2.7315518487365438)-- (0.,3.230935550935551);
\draw [line width=0.4pt,color=black] (0.,3.230935550935551)-- (0.,2.);
\draw [line width=0.4pt,color=black] (0.,2.)-- (-1.176350188308613,1.7059124529228467);
\draw [line width=0.4pt,color=black] (0.,3.230935550935551)-- (0.6656059123409749,2.7359609200174266);
\draw [line width=0.4pt,color=black] (0.6656059123409749,2.7359609200174266)-- (0.8751724137931034,1.7082758620689655);
\draw [line width=0.4pt,color=black] (0.8751724137931034,1.7082758620689655)-- (0.,2.);
\draw [line width=0.4pt,color=black] (0.,2.)-- (0.,3.230935550935551);
\draw [line width=0.4pt,color=black] (-1.176350188308613,1.7059124529228467)-- (-0.23559785204630612,1.546900694008833);
\draw [line width=0.4pt,color=black] (-0.23559785204630612,1.546900694008833)-- (0.8751724137931034,1.7082758620689655);
\draw [line width=0.4pt,color=black] (0.8751724137931034,1.7082758620689655)-- (0.,2.);
\draw [line width=0.4pt,color=black] (0.,2.)-- (-1.176350188308613,1.7059124529228467);
\begin{scriptsize}
\draw [fill=black] (-4.,1.) circle (1.5pt);
\draw [fill=black] (0.,2.) circle (1.5pt);
\draw [fill=black] (3.,1.) circle (1.5pt);
\draw [fill=black] (0.,6.) circle (1.5pt);
\draw [fill=black] (-1.176350188308613,1.7059124529228467) circle (1.5pt);
\draw [fill=black] (0.8751724137931034,1.7082758620689655) circle (1.5pt);
\draw [fill=black] (-0.23559785204630612,1.546900694008833) circle (1.5pt);
\draw [fill=black] (0.,3.230935550935551) circle (1.5pt);
\draw [fill=black] (-0.8953798813051128,2.7315518487365438) circle (1.5pt);
\draw [fill=black] (0.6656059123409749,2.7359609200174266) circle (1.5pt);
\end{scriptsize}
\end{tikzpicture}}
\hskip-1in\raise.75in\hbox{$\xrightarrow{project\ to\ plane}$}
\hbox{\definecolor{uuuuuu}{rgb}{0.26666666666666666,0.26666666666666666,0.26666666666666666}
\framebox{\begin{tikzpicture}[rotate=90,line cap=round,line join=round,>=triangle 45,x=.75cm,y=.75cm]
\clip(0,0) rectangle (7,6);
\draw [line width=1.2pt,dash pattern=on 2pt off 6pt,domain=-3.58:10.02] plot(\x,{(--30.-0.*\x)/10.});
\draw [line width=1.2pt,dash pattern=on 2pt off 6pt,domain=-3.58:10.02] plot(\x,{(--36.--4.*\x)/10.});
\draw [line width=1.2pt,dash pattern=on 2pt off 6pt,domain=-3.58:10.02] plot(\x,{(--26.-3.*\x)/11.});
\draw [line width=1pt,color=black] (4.,-1.6) -- (4.,7.28);
\draw [line width=1pt,color=black] (5.,-1.6) -- (5.,7.28);
\draw [line width=1pt,domain=-3.58:10.02,color=black] plot(\x,{(-18.9744--7.02*\x)/-0.48});
\draw [line width=1pt,domain=-3.58:10.02,color=lightgray] plot(\x,{(-16.14--7.5*\x)/6.06});
\begin{scriptsize}
\draw [fill=uuuuuu] (2.391347753743761,4.5565391014975045) circle (2.0pt);
\draw [fill=uuuuuu] (2.4977777777777783,3.) circle (2.0pt);
\draw [fill=uuuuuu] (2.5895803642121935,1.6573871733966743) circle (2.0pt);
\draw [fill=uuuuuu] (4.,5.2) circle (2.0pt);
\draw [fill=uuuuuu] (5.,5.6) circle (2.0pt);
\draw [fill=uuuuuu] (4.,3.) circle (2.0pt);
\draw [fill=uuuuuu] (5.,3.) circle (2.0pt);
\draw [fill=uuuuuu] (4.,1.2727272727272727) circle (2.0pt);
\draw [fill=uuuuuu] (5.,1.) circle (2.0pt);
\end{scriptsize}
\draw [line width=1pt] (0.097546956225292,7.00644496440457) -- (0.029874521645215,7.08032983112054);
\draw [line width=1pt] (0.207909884163657,6.8860530517775) -- (0.08136192326907,7.02410648979939);
\draw [line width=1pt] (0.261280492872658,6.82789059961892) -- (0.191553650563331,6.90389085997573);
\draw [line width=1pt] (0.308505522062576,6.7764512228906) -- (0.250648504554545,6.83947602117211);
\draw [line width=1pt] (0.336622621222295,6.74583908767053) -- (0.300689693274599,6.78496367139075);
\draw [line width=1pt] (0.375500556373535,6.70352663189195) -- (0.329821633667846,6.75324245489769);
\draw [line width=1pt] (0.414053523390327,6.66158850175123) -- (0.367843111853309,6.71185880945072);
\draw [line width=1pt] (0.44110071751508,6.63217939720874) -- (0.407393741057439,6.66883208468232);
\draw [line width=1pt] (0.45928169872892,6.61241663619216) -- (0.436383630454036,6.63730807434421);
\draw [line width=1pt] (0.47839345608353,6.591646982151) -- (0.455462623671694,6.61656764457336);
\draw [line width=1pt] (0.497462743567779,6.57092889842741) -- (0.474575263365705,6.59579595907336);
\draw [line width=1pt] (0.516889160630778,6.54982845262451) -- (0.49361602563459,6.57510776265003);
\draw [line width=1pt] (0.535586063110087,6.52952590215727) -- (0.51307370267846,6.55397225812385);
\draw [line width=1pt] (0.533710033074118,6.53156309848782) -- (0.533710033074118,6.53156309848782);
\draw [line width=1pt] (0.533710033074118,6.53156309848782) -- (0.533710033074118,6.53156309848782);
\draw [line width=1pt] (0.533710033074118,6.53156309848782) -- (0.533710033074118,6.53156309848782);
\draw [line width=1pt] (0.533710033074118,6.53156309848782) -- (0.533710033074118,6.53156309848782);
\draw [line width=1pt] (0.555094488003596,6.5083479313509) -- (0.531765991716893,6.53367356822754);
\draw [line width=1pt] (0.638188075159361,6.41820740006808) -- (0.54541975314519,6.51884485572903);
\draw [line width=1pt] (0.715242198087624,6.33472911056113) -- (0.622749671073685,6.43494519508334);
\draw [line width=1pt] (0.792847551677464,6.250764952788) -- (0.699778754396371,6.35147276895162);
\draw [line width=1pt] (0.838898949078576,6.2010007763892) -- (0.780200261251809,6.26444422483912);
\draw [line width=1pt] (0.875540649256142,6.1614347526015) -- (0.830231641078183,6.21036527386535);
\draw [line width=1pt] (0.913679806650042,6.1202817680119) -- (0.867954452385973,6.16962416222454);
\draw [line width=1pt] (0.920873131287026,6.11252536356518) -- (0.908869017659037,6.12547256789003);
\draw [line width=1pt] (0.930038416483547,6.10264225707049) -- (0.918948640484797,6.11460084636696);
\draw [line width=1pt] (0.940164107739177,6.09172574818078) -- (0.92810973764224,6.10472181145105);
\draw [line width=1pt] (0.939159576897765,6.0928087534533) -- (0.939159576897765,6.0928087534533);
\draw [line width=1pt] (0.939159576897765,6.0928087534533) -- (0.939159576897765,6.0928087534533);
\draw [line width=1pt] (0.939159576897765,6.0928087534533) -- (0.939159576897765,6.0928087534533);
\draw [line width=1pt] (0.939159576897765,6.0928087534533) -- (0.939159576897765,6.0928087534533);
\draw [line width=1pt] (0.939159576897765,6.0928087534533) -- (0.939159576897765,6.0928087534533);
\draw [line width=1pt] (1.00228585023023,6.02480167677782) -- (0.933420824776632,6.09899121496925);
\draw [line width=1pt] (1.05981927587439,5.96291050184002) -- (0.990795081948613,6.03717265069866);
\draw [line width=1pt] (1.08598666085086,5.93479373367711) -- (1.05116549597419,5.97221767611471);
\draw [line width=1pt] (1.12572764025106,5.8921221585714) -- (1.07920828409842,5.94207514436287);
\draw [line width=1pt] (1.17443900054521,5.83988120980433) -- (1.11707030239226,5.90141251625855);
\draw [line width=1pt] (1.21260201501886,5.79900587598394) -- (1.16575429939734,5.84919090437474);
\draw [line width=1pt] (1.20869803871707,5.80318796168317) -- (1.20869803871707,5.80318796168317);
\draw [line width=1pt] (1.2921657463543,5.71393045805465) -- (1.2011100652955,5.81130228019485);
\draw [line width=1pt] (1.40136898619027,5.59753381581877) -- (1.27396038990932,5.73336395481611);
\draw [line width=1pt] (1.4436695286872,5.55259987923616) -- (1.38594088266501,5.61396691359877);
\draw [line width=1pt] (1.4814480312604,5.51252103773221) -- (1.43498706063307,5.56182223158767);
\draw [line width=1pt] (1.49833102032321,5.49463343607289) -- (1.47568948947039,5.51862911005174);
\draw [line width=1pt] (1.50681684237662,5.4856470713482) -- (1.49550126096809,5.49763180322776);
\draw [line width=1pt] (1.62376219608039,5.36213706882766) -- (1.49515675736641,5.49796477447548);
\draw [line width=1pt] (1.77361869933305,5.20506846183693) -- (1.59844747408343,5.38876400634025);
\draw [line width=1pt] (1.87821875761042,5.09649427227493) -- (1.74818494628514,5.23163843766105);
\draw [line width=1pt] (1.92111839685392,5.05228763787718) -- (1.86249753483144,5.1127988904371);
\draw [line width=1pt] (1.94878375997828,5.02386089954213) -- (1.91327419456784,5.06037290977672);
\draw [line width=1pt] (1.94582462952741,5.02690356706168) -- (1.94582462952741,5.02690356706168);
\draw [line width=1pt] (2.02205641154534,4.94895263806192) -- (1.93889446752578,5.03399001515257);
\draw [line width=1pt] (2.13603099890238,4.83389276595658) -- (2.00413490869292,4.96714329707974);
\draw [line width=1pt] (2.22439455533765,4.74636951801377) -- (2.11600739466195,4.85396310950803);
\draw [line width=1pt] (2.32728680796609,4.64681366532145) -- (2.20518733594637,4.76520128566709);
\draw [line width=1pt] (2.3996612830605,4.57904278574314) -- (2.30960735822844,4.66373716531453);
\draw [line width=1pt] (2.3921567893245,4.58610065070742) -- (2.3921567893245,4.58610065070742);
\draw [line width=1pt] (2.43100303334533,4.5504950370103) -- (2.38862531259533,4.58933752467989);
\draw [line width=1pt] (2.47743199245923,4.50903204083452) -- (2.42292969881225,4.55779553554169);
\draw [line width=1pt] (2.54733700727068,4.44937886115848) -- (2.46612223714482,4.51888810214208);
\draw [line width=1pt] (2.60705309148059,4.40220532400163) -- (2.53452511142197,4.45998638644397);
\draw [line width=1pt] (2.65127667388523,4.37034580900959) -- (2.59643931307484,4.41035446740475);
\draw [line width=1pt] (2.65863515932773,4.36550701361243) -- (2.64562250604406,4.37442285026343);
\draw [line width=1pt] (2.76476061891453,4.3052) -- (2.64780442179405,4.37180000000001);
\draw [line width=1pt] (2.84849857308455,4.287925) -- (2.74651569606994,4.312825);
\draw [line width=1pt] (2.939,4.301) -- (2.831,4.289);
\draw [line width=1pt] (3.04200730876112,4.35401876979223) -- (2.91981751738535,4.29508920274616);
\draw [line width=1pt] (3.07111717219536,4.37919363077659) -- (3.02825279468748,4.34637291269855);
\draw [line width=1pt] (3.12999808257641,4.43045642148964) -- (3.06186760056909,4.37154967543194);
\draw [line width=1pt] (3.26553183302939,4.56173846191291) -- (3.11148315235274,4.41316653180955);
\draw [line width=1pt] (3.27280285077503,4.56936737746907) -- (3.25086640590009,4.54753838503477);
\draw [line width=1pt] (3.27097481370212,4.56754829476621) -- (3.27097481370212,4.56754829476621);
\draw [line width=1pt] (3.44878238692607,4.74294517052339) -- (3.25481048886358,4.55160312424283);
\draw [line width=1pt] (3.62500544834842,4.90587277661894) -- (3.41512829970018,4.7107388384892);
\draw [line width=1pt] (3.81010544112427,5.060684353144) -- (3.58909843549169,4.87405954801396);
\draw [line width=1pt] (3.864208988927,5.10174245855545) -- (3.78509539081197,5.03998590673114);
\draw [line width=1pt] (3.9064489978241,5.13331432365891) -- (3.85317684283499,5.09325805701657);
\draw [line width=1pt] (3.97565560415021,5.18320500508193) -- (3.89531456497726,5.12513732838024);
\draw [line width=1pt] (3.99414582069611,5.19597387230789) -- (3.9666709445394,5.17676531927032);
\draw [line width=1pt] (4.14303678627048,5.29363976409698) -- (3.97811256235692,5.1853489227782);
\draw [line width=1pt] (4.25357984070237,5.3578666231099) -- (4.1179943064209,5.27795633679409);
\draw [line width=1pt] (4.31155951137266,5.38870418382044) -- (4.23598300388857,5.34779863701659);
\draw [line width=1pt] (4.32434548714985,5.39519912121387) -- (4.30352655834891,5.38439504889341);
\draw [line width=1pt] (4.35755071366615,5.41195011412851) -- (4.31943420030283,5.39269411528341);
\draw [line width=1pt] (4.42684442764666,5.44494971888157) -- (4.34778614754398,5.40719960471049);
\draw [line width=1pt] (4.4476317424804,5.45416131294231) -- (4.41776755537971,5.44068047267868);
\draw [line width=1pt] (4.44514306022201,5.45303790958701) -- (4.44514306022201,5.45303790958701);
\draw [line width=1pt] (4.46401882247003,5.46132737732956) -- (4.44342708183583,5.45228432161042);
\draw [line width=1pt] (4.57461042749734,5.50591565956361) -- (4.45209306377353,5.45645180115199);
\draw [line width=1pt] (4.69297851937339,5.54501566900094) -- (4.55271174971553,5.49786439885007);
\draw [line width=1pt] (4.83763633787212,5.58022520583052) -- (4.66707628408643,5.53752832291181);
\draw [line width=1pt] (4.93740684915966,5.59504817309738) -- (4.81306083195637,5.57499612854092);
\draw [line width=1pt] (4.99616757945258,5.59994777553092) -- (4.92076078120545,5.59277984155283);
\draw [line width=1pt] (5.04332470738826,5.60187443801676) -- (4.98502540434505,5.5991209940342);
\draw [line width=1pt] (5.01877354501755,5.60098380264128) -- (5.04025669459985,5.60170509177976);
\draw [line width=1pt] (4.94079523429932,5.5952074294867) -- (5.02781549595396,5.60157449921338);
\draw [line width=1pt] (5.32870388867241,5.57201533584976) -- (4.91344174405219,5.59789462615612);
\draw [line width=1pt] (5.35202643666208,5.55685375453823) -- (5.2888325529806,5.57574632417866);
\draw [line width=1pt] (5.42326871927919,5.52803613704678) -- (5.33980496699857,5.56119104427749);
\draw [line width=1pt] (5.48463221324215,5.49510935182291) -- (5.41010260598432,5.53404356363356);
\draw [line width=1pt] (5.48798702162201,5.49261438855687) -- (5.47755181182054,5.49887564046625);
\draw [line width=1pt] (5.48711742080522,5.49313615954932) -- (5.48711742080522,5.49313615954932);
\draw [line width=1pt] (5.52321515804118,5.46962502017608) -- (5.48383580832922,5.49527353585598);
\draw [line width=1pt] (5.61364434845761,5.39438765717739) -- (5.51141438166587,5.47879646369232);
\draw [line width=1pt] (5.66340410177779,5.33750908423948) -- (5.59982710117471,5.40723196440038);
\draw [line width=1pt] (5.76826228671991,5.16435850122952) -- (5.64809181218277,5.35958848998229);
\draw [line width=1pt] (5.83249682276125,4.94684036165249) -- (5.75149819484914,5.20188105835041);
\draw [line width=1pt] (5.84372481431657,4.84872636974609) -- (5.82411258462785,4.97894533334379);
\draw [line width=1pt] (5.84460969621935,4.83314657249389) -- (5.84186144053553,4.86198080255063);
\draw [line width=1pt] (5.84438067491236,4.83554942499862) -- (5.84438067491236,4.83554942499862);
\draw [line width=1pt] (5.84979831902828,4.68202322943718) -- (5.84388816181091,4.84950635186784);
\draw [line width=1pt] (5.8355555476591,4.49546719145305) -- (5.85055582940572,4.71420860765671);
\draw [line width=1pt] (5.80504751290306,4.3244788238196) -- (5.83969266734116,4.53089717180187);
\draw [line width=1pt] (5.79159037158525,4.27025366443336) -- (5.80942044888087,4.34817368812583);
\draw [line width=1pt] (5.75748653575601,4.14678795339138) -- (5.79631163641478,4.28856145850012);
\draw [line width=1pt] (5.74498329599008,4.10839084927841) -- (5.7621527485219,4.16316709968426);
\draw [line width=1pt] (5.73701358388464,4.0845201364081) -- (5.74726867459346,4.11554057321261);
\draw [line width=1pt] (5.73004421347253,4.06421237906609) -- (5.73857944398654,4.08918633587596);
\draw [line width=1pt] (5.63229963288681,3.82058099763459) -- (5.7397060144816,4.08863104617894);
\draw [line width=1pt] (5.50130952467495,3.57374183204324) -- (5.65397204105105,3.86738910801056);
\draw [line width=1pt] (5.44334227879364,3.48425756526725) -- (5.52045768488017,3.60857197229263);
\draw [line width=1pt] (5.41814597243528,3.44733792357418) -- (5.45264334356135,3.49891520605984);
\draw [line width=1pt] (5.41080047146956,3.43680263619397) -- (5.42194986989817,3.45298452083471);
\draw [line width=1pt] (5.41172958800528,3.4381511265807) -- (5.41172958800528,3.4381511265807);
\draw [line width=1pt] (5.39425756718696,3.41327826211337) -- (5.41331795353422,3.44041229607773);
\draw [line width=1pt] (5.30665446760295,3.29747761727578) -- (5.40395424772617,3.42627232382264);
\draw [line width=1pt] (5.1152217379577,3.09136948597454) -- (5.33290287758191,3.32792332980743);
\draw [line width=1pt] (4.98811376873737,2.98995179969035) -- (5.14656620239811,3.12209417053064);
\draw [line width=1pt] (4.95060057098645,2.96562357742807) -- (5.00592882613843,3.00417639906331);
\draw [line width=1pt] (4.89423830487475,2.93108600259192) -- (4.96075425473769,2.97226815892547);
\draw [line width=1pt] (4.82065985745297,2.89325573296129) -- (4.90697415917336,2.93826895040685);
\draw [line width=1pt] (4.75395926378603,2.86638657695949) -- (4.83457030248818,2.89979049418378);
\draw [line width=1pt] (4.72051949148875,2.85574728463102) -- (4.76432751933143,2.87039050510065);
\draw [line width=1pt] (4.71831825825764,2.85520082961357) -- (4.72470215158637,2.85712816422075);
\draw [line width=1pt] (4.68350641203829,2.84571426576222) -- (4.72206332548929,2.85623845674617);
\draw [line width=1pt] (4.47447377570831,2.82516148956895) -- (4.70601455292747,2.84853944459651);
\draw [line width=1pt] (4.42699693747106,2.83290006863478) -- (4.49983901347707,2.82658325102001);
\draw [line width=1pt] (4.2806116122065,2.86622409925283) -- (4.44692670122293,2.82929635515907);
\draw [line width=1pt] (4.21951193663295,2.88938870122779) -- (4.30128568171468,2.86076115870113);
\draw [line width=1pt] (4.1496989629712,2.91962596716638) -- (4.23329254742781,2.88403735500367);
\draw [line width=1pt] (4.13343428200564,2.92761768053563) -- (4.15877698710049,2.91566411939074);
\draw [line width=1pt] (4.11787727464984,2.93526891329502) -- (4.13715243768298,2.92583542654433);
\draw [line width=1pt] (4.09734253292795,2.94568857230168) -- (4.12149635690029,2.93346408186254);
\draw [line width=1pt] (4.06133945175662,2.96482959126333) -- (4.10281134248647,2.94283716235615);
\draw [line width=1pt] (4.0479884986189,2.97226431931187) -- (4.06632334665367,2.96215439517645);
\draw [line width=1pt] (4.0134076613321,2.99202998289222) -- (4.05279901546632,2.96954835679226);
\draw [line width=1pt] (3.98197189146693,3.01082256362164) -- (4.01984649078659,2.98827778227137);
\draw [line width=1pt] (3.95165340060814,3.02960961774724) -- (3.98817126330133,3.00706512403292);
\draw [line width=1pt] (3.88573816324472,3.07239354992164) -- (3.96096550061328,3.02367066993917);
\draw [line width=1pt] (3.82489163414789,3.1141981727329) -- (3.89810851474157,3.06416377694039);
\draw [line width=1pt] (3.74393299932851,3.17247282079389) -- (3.83890759009453,3.10435189601895);
\draw [line width=1pt] (3.71206006108558,3.1962617399991) -- (3.75546459287478,3.16411738043205);
\draw [line width=1pt] (3.67633572272649,3.22319595691016) -- (3.71925359473543,3.19089096031927);
\draw [line width=1pt] (3.66525806103638,3.23164202693067) -- (3.68124440760822,3.21949131449094);
\draw [line width=1pt] (3.66659025658403,3.23062946756069) -- (3.66659025658403,3.23062946756069);
\draw [line width=1pt] (3.60052863199355,3.28144621169587) -- (3.67259585881953,3.2260097635484);
\draw [line width=1pt] (3.5030452381545,3.35777765708277) -- (3.61594232478128,3.26946731228365);
\draw [line width=1pt] (3.40931489035265,3.4313852533428) -- (3.52182955037528,3.34305784425921);
\draw [line width=1pt] (3.38041476160956,3.45376526815486) -- (3.42217078024044,3.42132094207956);
\draw [line width=1pt] (3.33129945388212,3.49130445453767) -- (3.3886757912785,3.44740313065867);
\draw [line width=1pt] (3.31085849533816,3.5065782454541) -- (3.33837375351306,3.48592489864718);
\draw [line width=1pt] (3.28837592726742,3.52314885968546) -- (3.31540375226958,3.50319424899607);
\draw [line width=1pt] (3.2906282460176,3.52148597546134) -- (3.2906282460176,3.52148597546134);
\draw [line width=1pt] (3.27547650778883,3.53249667307805) -- (3.29200567676567,3.52048500295073);
\draw [line width=1pt] (3.2768539385369,3.53149570056744) -- (3.2768539385369,3.53149570056744);
\draw [line width=1pt] (3.2768539385369,3.53149570056744) -- (3.2768539385369,3.53149570056744);
\draw [line width=1pt] (3.2623289798208,3.54190779614396) -- (3.27817438932927,3.53054914642412);
\draw [line width=1pt] (3.24730556734157,3.55249012156182) -- (3.26513523636513,3.53991316204053);
\draw [line width=1pt] (3.21887895361415,3.57194965197405) -- (3.25151068395529,3.54957771338604);
\draw [line width=1pt] (3.22159826447591,3.57008532375838) -- (3.22159826447591,3.57008532375838);
\draw [line width=1pt] (3.18782166405607,3.5921333152626) -- (3.22466886451408,3.56808096089436);
\draw [line width=1pt] (3.17011287091646,3.6029867565818) -- (3.19278129983768,3.5889600611092);
\draw [line width=1pt] (3.15574297350518,3.61144024696452) -- (3.17347999149214,3.60094310332223);
\draw [line width=1pt] (3.14244819989523,3.61893477112177) -- (3.15856404546854,3.60980464080093);
\draw [line width=1pt] (3.04062088129737,3.64768260730716) -- (3.15317030572897,3.61549131962121);
\draw [line width=1pt] (2.918,3.6505) -- (3.062,3.6445);
\draw [line width=1pt] (2.85776250132272,3.62408158939763) -- (2.9365670453343,3.65235621914567);
\draw [line width=1pt] (2.86432954665702,3.6264378085433) -- (2.86432954665702,3.6264378085433);
\draw [line width=1pt] (2.88829494519846,3.64248731284471) -- (2.86215087406234,3.62497876269772);
\draw [line width=1pt] (2.89321462621117,3.64512217330009) -- (2.88547096773038,3.64065609369904);
\draw [line width=1pt] (2.81581653687583,3.58455064269379) -- (2.89954684719795,3.65022266884602);
\draw [line width=1pt] (2.82279406273601,3.59002331153981) -- (2.82279406273601,3.59002331153981);
\draw [line width=1pt] (2.7904329377207,3.55434876213147) -- (2.82573598319195,3.59326645239511);
\draw [line width=1pt] (2.74734055736885,3.49689799882711) -- (2.79755979461371,3.56310953063765);
\draw [line width=1pt] (2.74447282791989,3.49211746858228) -- (2.75216664525011,3.50335182265033);
\draw [line width=1pt] (2.72884751633327,3.46830810651505) -- (2.74659274873051,3.49530326095822);
\draw [line width=1pt] (2.73032628569971,3.47055770271865) -- (2.73032628569971,3.47055770271865);
\draw [line width=1pt] (2.73032628569971,3.47055770271865) -- (2.73032628569971,3.47055770271865);
\draw [line width=1pt] (2.71447937540663,3.4451247173324) -- (2.73176691390817,3.47286979229922);
\draw [line width=1pt] (2.67731821610427,3.38000054734836) -- (2.71942925702516,3.4535673759643);
\draw [line width=1pt] (2.65412593602414,3.33483688836003) -- (2.68325488164981,3.39079422803966);
\draw [line width=1pt] (2.61960404462865,3.26185061303048) -- (2.65991237575333,3.34655903517905);
\draw [line width=1pt] (2.614580988948,3.25013996719299) -- (2.62372507979277,3.2706159828474);
\draw [line width=1pt] (2.58854210613839,3.18868012461252) -- (2.6177794410984,3.25758868157798);
\draw [line width=1pt] (2.58005949904824,3.16720134332293) -- (2.59197119177931,3.19689715536298);
\draw [line width=1pt] (2.57360868777507,3.150582461657) -- (2.58172881759408,3.17141177002348);
\draw [line width=1pt] (2.56667898598693,3.13233118717871) -- (2.57497685428481,3.15413524191562);
\draw [line width=1pt] (2.55070416290058,3.08854133561943) -- (2.56888559429458,3.13829426956927);
\draw [line width=1pt] (2.54494710929349,3.07194003896154) -- (2.55288038880977,3.09457353840195);
\draw [line width=1pt] (2.53843976743103,3.0528243527209) -- (2.5462598930552,3.075735419478);
\draw [line width=1pt] (2.53253006118543,3.03499894713528) -- (2.53968793396465,3.05652766838842);
\draw [line width=1pt] (2.5331265505837,3.03679300723971) -- (2.5331265505837,3.03679300723971);
\draw [line width=1pt] (2.5331265505837,3.03679300723971) -- (2.5331265505837,3.03679300723971);
\draw [line width=1pt] (2.46447009114738,2.78838315112415) -- (2.53936804689609,3.05937572143203);
\draw [line width=1pt] (2.44181102894067,2.66246496309365) -- (2.47333891096151,2.82446594733673);
\draw [line width=1pt] (2.42765929654907,2.55381922555178) -- (2.44596372115999,2.68706921052865);
\draw [line width=1pt] (2.4242286394563,2.51511951614284) -- (2.42963521306759,2.56945101595049);
\draw [line width=1pt] (2.4207518231549,2.46822470317003) -- (2.42503622035745,2.52432190821379);
\draw [line width=1pt] (2.41859605928715,2.42561424962904) -- (2.42133729234311,2.47719812667774);
\draw [line width=1pt] (2.41791816525818,2.40441249573448) -- (2.41890688902215,2.43223112516933);
\draw [line width=1pt] (2.41724422223285,2.32952723225799) -- (2.41806931678448,2.41374921327187);
\draw [line width=1pt] (2.4184064679378,2.28864993965854) -- (2.417213572128,2.34089989349556);
\draw [line width=1pt] (2.41942122063311,2.26705230192571) -- (2.41820577261915,2.29536335707398);
\draw [line width=1pt] (2.42081409037768,2.24342898453808) -- (2.41918410083688,2.27177360851988);
\draw [line width=1pt] (2.42270962354919,2.2179641117231) -- (2.42049358831293,2.24832075697415);
\draw [line width=1pt] (2.4281337423042,2.16347835551574) -- (2.42201506409544,2.22567705731023);
\draw [line width=1pt] (2.43138263486671,2.13870714258077) -- (2.42728214496136,2.17138471140024);
\draw [line width=1pt] (2.43839185768536,2.09283885477887) -- (2.43037266098271,2.1458476750009);
\draw [line width=1pt] (2.44322037959352,2.06647086886598) -- (2.43722388326619,2.10005492806386);
\draw [line width=1pt] (2.44470731210212,2.05898760109695) -- (2.44254006788116,2.07020426222661);
\draw [line width=1pt] (2.45604487188586,2.00680775619791) -- (2.44347960264715,2.06475091982683);
\draw [line width=1pt] (2.4647755272614,1.97296139976374) -- (2.45410887873911,2.01515225802183);
\draw [line width=1pt] (2.47329486704992,1.94321650725324) -- (2.46303134650587,1.97950101346998);
\draw [line width=1pt] (2.49603600988403,1.87478925026204) -- (2.47029444310645,1.95273575845396);
\draw [line width=1pt] (2.50392903173088,1.85474416591865) -- (2.49297832000908,1.88369757685616);
\draw [line width=1pt] (2.50595498525828,1.84973625160412) -- (2.50274933489004,1.85783155912338);
\draw [line width=1pt] (2.52536522897912,1.8041818364201) -- (2.50389899488654,1.8546134991226);
\draw [line width=1pt] (2.5423402183494,1.76899054670732) -- (2.52187057230068,1.81196574118512);
\draw [line width=1pt] (2.5450632536108,1.76383475026231) -- (2.54023179277576,1.77336609133666);
\draw [line width=1pt] (2.57115722469552,1.71612639922627) -- (2.54225185070901,1.76903835863598);
\draw [line width=1pt] (2.60735657696769,1.65929030194556) -- (2.56523861321746,1.72610349528904);
\draw [line width=1pt] (2.62158771616952,1.6397226506804) -- (2.60223393124478,1.6671431060009);
\draw [line width=1pt] (2.68571542640626,1.56154667576855) -- (2.61399848933665,1.64932232615606);
\draw [line width=1pt] (2.82704510135886,1.43857045236631) -- (2.66634755258606,1.58070593702216);
\draw [line width=1pt] (2.89001878798852,1.40138559986906) -- (2.80671135268591,1.45487230119841);
\draw [line width=1pt] (2.93361887812156,1.3785524351354) -- (2.87848174022164,1.40832376951115);
\draw [line width=1pt] (2.98754735850314,1.35379097175258) -- (2.92370382191415,1.38350996220436);
\draw [line width=1pt] (3.0288350781736,1.33752508367354) -- (2.97798997157046,1.35797141525538);
\draw [line width=1pt] (3.0640568899924,1.32519851655007) -- (3.02101081286251,1.34050443810129);
\draw [line width=1pt] (3.06046971689824,1.32647401001267) -- (3.06046971689824,1.32647401001267);
\draw [line width=1pt] (3.09386520721946,1.31580905412605) -- (3.05743376323268,1.32744355145691);
\draw [line width=1pt] (3.13406458434788,1.30449380607513) -- (3.08689876893626,1.31789539461531);
\draw [line width=1pt] (3.16271353978528,1.2974100495242) -- (3.12717233245251,1.30635611017432);
\draw [line width=1pt] (3.28223804895583,1.27449714968946) -- (3.1486166564668,1.30030631865919);
\draw [line width=1pt] (3.38378231613097,1.26336424657763) -- (3.26085935262273,1.27785551987869);
\draw [line width=1pt] (3.44203517307609,1.25976333187587) -- (3.36731178699885,1.26500899094152);
\draw [line width=1pt] (3.5581987129357,1.25658373760597) -- (3.42468181617274,1.26052926399728);
\draw [line width=1pt] (3.44789018514946,1.2593868173797) -- (3.55608886121054,1.25668759638938);
\draw [line width=1pt] (3.79045181407951,1.26202443553346) -- (3.42658446216137,1.25890165018478);
\draw [line width=1pt] (3.8291291280247,1.26499857389049) -- (3.75385684445557,1.26147016974203);
\draw [line width=1pt] (3.82285643772727,1.26470454021145) -- (3.82285643772727,1.26470454021145);
\draw [line width=1pt] (3.86729103736463,1.26689540910122) -- (3.81881692866933,1.26450537031238);
\draw [line width=1pt] (3.92578382391027,1.26971175078155) -- (3.85756677416091,1.26642210178584);
\draw [line width=1pt] (3.96071505398516,1.2712382685915) -- (3.91640670756261,1.26927391743558);
\draw [line width=1pt] (3.99442524013652,1.27254331991895) -- (3.9536224601148,1.27094105018392);
\draw [line width=1pt] (4.0118974991532,1.2731265050121) -- (3.98912750931484,1.2723446422073);
\draw [line width=1pt] (4.11152282201113,1.27508836496338) -- (4.00077065254444,1.27287707567064);
\draw [line width=1pt] (4.15982626326751,1.27463233660405) -- (4.0970632210363,1.27492879578761);
\draw [line width=1pt] (4.17703494582948,1.2742032443177) -- (4.15255610646813,1.27469829582858);
\draw [line width=1pt] (4.19407979864118,1.27368823657528) -- (4.17326006472284,1.27429506788618);
\draw [line width=1pt] (4.23323722133438,1.27197728622323) -- (4.18862732985831,1.27389894399919);
\draw [line width=1pt] (4.28967416151575,1.26798381183386) -- (4.22405114572916,1.27251502551099);
\draw [line width=1pt] (4.39963007047251,1.25433255835128) -- (4.27371244108454,1.26963676339383);
\draw [line width=1pt] (4.46503260869715,1.24142174617921) -- (4.38223732796227,1.25689755991624);
\draw [line width=1pt] (4.49609612331065,1.23400155108437) -- (4.45468180912002,1.24350320152756);
\draw [line width=1pt] (4.51414956878334,1.22936394057232) -- (4.49068996334126,1.23528693844394);
\draw [line width=1pt] (4.51219460166317,1.22985752372829) -- (4.51219460166317,1.22985752372829);
\draw [line width=1pt] (4.57948170334683,1.21049726615409) -- (4.5060775924192,1.23161754714413);
\draw [line width=1pt] (4.65488621598432,1.18389860711974) -- (4.56595364665909,1.21483535161086);
\draw [line width=1pt] (4.73177259620602,1.15149225358919) -- (4.63981176602551,1.18965707057626);
\draw [line width=1pt] (4.77314122716888,1.13172755330782) -- (4.71965173610208,1.15675857334087);
\draw [line width=1pt] (4.85296734549131,1.09008408661834) -- (4.76102162631531,1.1377888702826);
\draw [line width=1pt] (4.91472356534951,1.05418009779789) -- (4.83899444194274,1.09768488411695);
\draw [line width=1pt] (4.93441496039609,1.04201905507345) -- (4.90604897276284,1.05924062771093);
\draw [line width=1pt] (4.9802103861926,1.01305558517007) -- (4.92767301372066,1.04621769530443);
\draw [line width=1pt] (4.98832032596015,1.00769862809129) -- (4.97469699417083,1.01655731855308);
\draw [line width=1pt] (5.00251435171764,0.998338958130377) -- (4.98579147527408,1.00935484267517);
\draw [line width=1pt] (5.00112077868068,0.99925694850911) -- (5.00112077868068,0.99925694850911);
\draw [line width=1pt] (5.05458017698234,0.962959046569798) -- (4.99626083338053,1.00255675777632);
\draw [line width=1pt] (5.16176213428896,0.885142901000659) -- (5.03953460417248,0.973633033549404);
\draw [line width=1pt] (5.34677512483743,0.737424616440074) -- (5.13383117786487,0.906616393465143);
\draw [line width=1pt] (5.55639315407506,0.553958638614198) -- (5.30836039972741,0.769484412335615);
\draw [line width=1pt] (5.71629194292943,0.405510179419111) -- (5.51930846832942,0.587047205242971);
\draw [line width=1pt] (5.79935000973386,0.326197146725244) -- (5.69083362098357,0.429223821102541);
\draw [line width=1pt] (5.82140624580076,0.304911665315085) -- (5.78747977111411,0.337498251796831);
\draw [line width=1pt] (5.80715281831306,0.31862018449) -- (5.81961778696449,0.30662785325207);
\draw [line width=1pt] (5.81961778696449,0.30662785325207) -- (5.80715281831306,0.31862018449);
\draw [line width=1pt] (5.94163526973306,0.188003507286011) -- (5.80739210956267,0.318502096634251);
\draw [line width=1pt] (6.03534641706921,0.095311029934825) -- (5.92091215086883,0.208293604258687);
\draw [line width=1pt] (6.0906997519821,0.040017644273955) -- (6.01991118060437,0.110608844478891);
\draw [line width=1pt] (6.1713614936898,-0.041100926187257) -- (6.07693154170161,0.053809441607241);
\draw [line width=1pt] (6.22352420307346,-0.093946513659084) -- (6.15803488811054,-0.027668566617591);
\draw [line width=1pt] (6.22922197191342,-0.099748101268946) -- (6.21705265000047,-0.087393828690779);
\draw [line width=1pt] (6.30574806769408,-0.17772379267115) -- (6.2211587521231,-0.09153628636164);
\draw [line width=1pt] (6.43759126197132,-0.313200504817001) -- (6.2860723849806,-0.157572500084299);
\draw [line width=1pt] (6.56479290565688,-0.445138649298539) -- (6.41225303281893,-0.287058127615706);
\draw [line width=1pt] (6.73149605997252,-0.619486148191965) -- (6.53577081227928,-0.414917920155243);
\draw [line width=1pt] (6.77911155087362,-0.669627748674331) -- (6.70937417464576,-0.596330709235684);
\draw [line width=1pt] (6.79471219283675,-0.686064739161808) -- (6.77135354921989,-0.661470109590138);
\draw [line width=1pt] (6.79276563920201,-0.684015186697502) -- (6.79276563920201,-0.684015186697502);
\draw [line width=1pt] (6.78178956072095,-0.67245588117585) -- (6.79376346451847,-0.685066032654016);
\draw [line width=1pt] (6.78278738603741,-0.673506727132364) -- (6.78278738603741,-0.673506727132364);
\draw [line width=1pt] (6.79376346451847,-0.685066032654016) -- (6.78178956072095,-0.67245588117585);
\draw [line width=1pt] (6.83547957491963,-0.729046665619525) -- (6.78888255413677,-0.679921415886409);
\draw [line width=1pt] (6.87767676147114,-0.773615945424813) -- (6.82740737425287,-0.720528981116034);
\draw [line width=1pt] (6.9119896647878,-0.809915036692693) -- (6.8699874623316,-0.76548994037239);
\draw [line width=1pt] (6.96187800370506,-0.862771773751234) -- (6.90363597920839,-0.801071233658253);
\draw [line width=1pt] (6.9991042797572,-0.902279599518836) -- (6.95319906729153,-0.853571013218454);
\draw [line width=1pt] (7.01665486881386,-0.920925053756723) -- (6.9933355705279,-0.896156504924448);
\draw [line width=1pt] (7.02569424375609,-0.930532110278611) -- (7.01371317124767,-0.917799998724526);
\draw [line width=1pt] (7.03567475511112,-0.941142203240349) -- (7.02369773613214,-0.928410091686263);
\draw [line width=1pt] (7.04565193105224,-0.951752296202086) -- (7.03367891920929,-0.939020184648001);
\draw [line width=1pt] (7.04465418006533,-0.950691286905912) -- (7.04465418006533,-0.950691286905912);
\draw [line width=1pt] (7.03367891920929,-0.939020184648001) -- (7.04565193105224,-0.951752296202086);
\draw [line width=1pt] (7.0346766701962,-0.940081193944175) -- (7.0346766701962,-0.940081193944175);
\draw [line width=1pt] (7.05383789140216,-0.960459771431952) -- (7.03293474099566,-0.938228595990741);
\draw [line width=1pt] (7.08692201307416,-0.995675574885867) -- (7.04892995757685,-0.955237318805123);
\draw [line width=1pt] (7.12676265340163,-1.0381328544821) -- (7.07984631345372,-0.98813961709705);
\draw [line width=1pt] (7.14726918545405,-1.06000874979629) -- (7.1206333014016,-1.03159929696399);
\draw [line width=1pt] (7.15274238362969,-1.06584997476011) -- (7.14435017797877,-1.05689505181483);
\draw [line width=1pt] (7.15204303315878,-1.06510373118134) -- (7.15204303315878,-1.06510373118134);
\draw [line width=1pt] (7.15204303315878,-1.06510373118134) -- (7.15204303315878,-1.06510373118134);
\draw [line width=1pt] (7.11993395093069,-1.03085305338521) -- (7.15496204063406,-1.06821742916281);
\draw [line width=1pt] (7.12285295840597,-1.03396675136668) -- (7.12285295840597,-1.03396675136668);
\draw [line width=1pt] (7.11416918747525,-1.02470899609805) -- (7.12364239212695,-1.03480836548201);
\draw [line width=1pt] (7.10264381711673,-1.01242567589361) -- (7.11607814883981,-1.02674378606063);
\draw [fill=black] (5.84253870683464,4.567498399547822) circle (2.0pt);
\draw [fill=white] (3.194613621644016,1.2903633931237821) circle (2.0pt);
\draw [fill=white] (4.439871520953261,2.8315241595956193) circle (2.0pt);
\end{tikzpicture}}}}
\caption[A general projection of $D_4$ as a complete intersection.]{The projection $\overline{Z}$ of $Z={D_4}$ from a general point to a plane,
showing the cubic curve and two choices of a quartic making $\overline{Z}$ a complete intersection
(one quartic consists of the three solid lines and the gray line, the other consists of
the three dashed lines and the gray line; the white dots are the images of the two points of $Z$ not visible at left).}\label{3by3CI-Fig}
\end{figure}
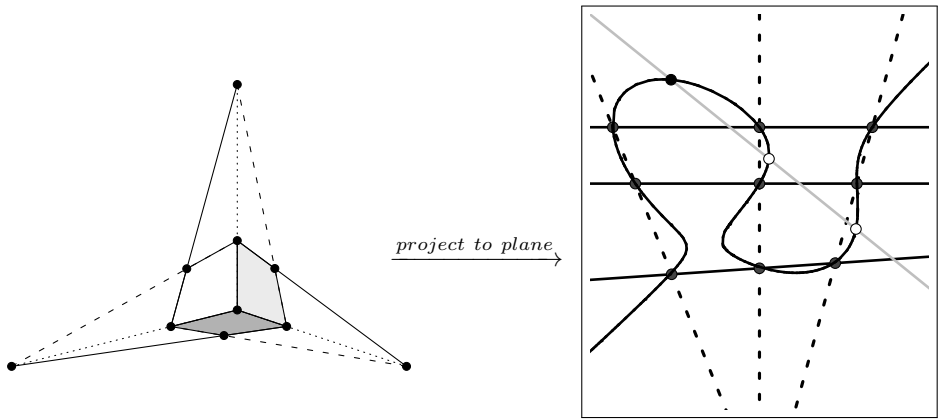

This raises the question of whether there is any point $P$
such that $\pi_P({D_4})$ is the complete intersection of 
the 12 points of intersection of a cubic 
of 3 lines with a quartic of 4 lines.

However, the answer is no, since for this to be possible there would need to be 
three planes that intersect at $p$ and which each contain 4 points of ${D_4}$.
But an exhaustive check shows that every plane contains either
no points of ${D_4}$, 1 point, 2 points, 3 points or 6 points,
but there are no planes containing exactly 4 points.
(Note that a general plane through a point of ${D_4}$
contains no other points of ${D_4}$. Also a general plane through a coordinate line,
such as $x=0,y=0$, contains only two points of ${D_4}$, and 
the points $[0:0:1:-1], [0:0:1:1], [1:1:0:0]$ define the plane $x-y=0$, which
contains no other points of ${D_4}$. Likewise a general plane
through a line, such as $x-y+z=w=0$, containing three points of ${D_4}$ contains no other points 
of ${D_4}$. And finally, there are 12 planes containing 6 points, including the coordinate planes,
such as $x=0$; see Lemma \ref{CoplanarLem}.)
\end{remark}
We turn now to the possible moduli of $D_4$ configurations. These considerations motivate our approach in the Standard Construction in Theorem \ref{t. geproci infinite class}.
\begin{remark}\label{D4moduli}
The 9 points of a $(3,3)$-grid are contained in exactly two
12 point sets, each obtained by adjoining to the 9 grid points
one or the other of two sets of 3 collinear points (these 6 points are known as Brianchon points;\index{points!Brianchon}
see Remark \ref{BrianchonRem}). A $D_4$ configuration arises this way, 
so this shows that a $(3,3)$-grid is not enough to determine a unique $D_4$ configuration 
$Z$ of 12 points containing the grid.

This raises the question of what does suffice to pin down a given $D_4$ configuration.
We now show that any 3 pairwise skew lines $H_1,H_2,H_3$ and
any point (call it $p_1$) not on the unique quadric $\calq$ containing $\cup H_i$ uniquely
determine a set $Z$ projectively equivalent to~${D_4}$. 

Note that the set of all 12 point sets projectively equivalent to
${D_4}$ is a 15 dimensional family of subsets of $\PP^3$
(since the group $PGL_4$ of automorphisms of $\PP^3$ is 15 dimensional,
and the subgroup of $PGL_4$ preserving ${D_4}$ is finite, see Proposition \ref{symGrpD4}).
The parameter space of sets of 3 skew lines $H_1,H_2,H_3$ plus a point $p_1$ 
not on $\calq$ is also 15 dimensional (since 
each $H_i$ is chosen from an open subset of the triple product $G(1,3)^3$
of the Grassmannian $G(1,3)$ of lines in $\PP^3$, which is $3(4)=12$ dimensional,
and $p_1$ is chosen from an open subset of $\PP^3$, 
giving a parameter space for $\{H_1,H_2,H_3,p_1\}$ of dimension $15$,
as necessary in order for every subset $Z$ projectively equivalent to
${D_4}$ to occur).

The planes $A_i=\langle p_1,H_i\rangle$ spanned by the lines $H_i$ and the point $p_1$ each determine 2 points:
we have $\{a_{ij_1}\}=A_i\cap H_{j_1}$ and $\{a_{ij_2}\}=A_i\cap H_{j_2}$, where $\{i, j_1,j_2\}=\{1,2,3\}$. 
So for example, $A_1$ contains $H_1$, meets $H_2$ at $a_{12}$ and meets $H_3$ at $a_{13}$.
(Here is one place we use $p_1\not\in \calq$. If $p_1\not\in \calq$, then $p_1\not\in H_i$.
Note that if $p_1\in H_i$, then $A_i=H_i$; i.e., $A_i$ is not a plane.)

Since $A_i$ contains $H_i$ and $H_i$ is a line in a ruling on $\calq$, 
the line through the points $a_{ij_1}$ and $a_{ij_2}$ is a line
(call it $V_i$) in the other ruling on $\calq$. Thus the lines $H_1,H_2,H_3$ and $V_1,V_2,V_3$
determine a $(3,3)$-grid. If we define $a_{ii}$ to be the intersection of $H_i$ and $V_i$,
then the grid points are the 9 points $\{a_{ij}\}$. 
(Here is another place we use $p_1\not\in \calq$. If $p_1\in \calq$ but $p_1\not\in \cup H_i$,
then $p_1\in V_1=V_2=V_3$, so we do not get a grid.)

Note for $1\leq i<k\leq 3$ that $A_i\cap A_k =\langle p_1,a_{ik},a_{ki}\rangle$.
Thus we see that $p_1$ is on three lines through pairs of the 9 grid points;
i.e., $p_1$ is a Brianchon point\index{points!Brianchon} for the grid, see Remark \ref{BrianchonRem}. The grid determines 2 other
Brianchon points which together with $p_1$ give a set of 3 collinear points.
These 3 and the 9 grid points then give a set $Z$ of 12 points 
projectively equivalent to
${D_4}$ (as noted at the end of the proof of Proposition \ref{GridColinearitiesProp}; see also Remark \ref{BrianchonRem}).
\end{remark}

\begin{remark}
We saw in Remark \ref{D4moduli} how to recover all 12 points of ${D_4}$
just from three disjoint lines $H_1$, $H_2$ and $H_3$, each through 3 of the 12 points, and one point $p_1$ of  
${D_4}$ not on those lines. We know there is a line through $p_1$
that contains the other two points (call them $p_2,p_3$) of ${D_4}$ not on any of the $H_i$. 
Suppose we are given $p_1,p_2,p_3$ and $H_1, H_2,H_3$. Then certainly we can recover
all 12 points of ${D_4}$, but here we give a simple geometric construction for doing so,
and for doing a little more. For each permutation $\sigma$ of $\{1,2,3\}$ we get 3 pairs of planes,
namely $B_{\sigma,i}=\langle p_i,H_{\sigma(i)}\rangle$. It turns out that $\cap_i B_{\sigma,i}$
gives one point each for three choices of $\sigma$ and it gives one line each for the other
three choices of $\sigma$. The lines are $V_1,V_2,V_3$ and the points are three of the six
Brianchon points for the grid coming from the $H_i$ and $V_j$ (the other three Brianchon points
are $p_1,p_2,p_3$). This is easy to check directly but somewhat tedious so we
do not include the details.
\end{remark}

\chapter{The geography of geproci sets: a complete numerical classification}\label{chap.Geography}

   Our research in this chapter is motivated by the following problem. We recall that trivial geproci sets are either degenerate complete intersections or grids.
   


   
\begin{question}[Geography problem] \label{a,b question}
For which pairs of positive integers $a \leq b$ does there exist a nontrivial $(a,b)$-geproci set of points in $\PP^3$?
\end{question}

\noindent This question was asked, in a slightly less specific way as Question 7.1 in the Appendix to~\cite{CM}. Surprisingly, the answer to Question \ref{a,b question} is that nontrivial $(a,b)$-geproci sets exist for almost all values of $a$ and $b$. Let us now turn to the details.
\section{A standard construction} \label{sec:standard construction}
 In this section we study in detail $(a,b)$-geproci sets for $a\geq 4$. We include $a=3$ in the discussion because what we call below the standard construction works also in this case. In Section \ref{sec:small_a_b} we show  that for $a\leq 2$ there are no nongrid $(a,b)$-geproci sets. Moreover we give the full classification of nondegenerate $(3,b)$-geproci with $b\geq 4$.

From now on we assume that $a\geq 3$. For all such $a$ and for 
\begin{itemize}
    \item $b=a+1$ for $a$ odd;
    \item $b=a+1$ or $b=a+2$ for $a$ even;
\end{itemize}
in Theorem \ref{t. geproci infinite class} we present a general construction of nongrid $(a,b)$-geproci sets, which we will call the \textit{standard construction}\index{standard construction}\index{geproci! standard construction}.

Before turning to the details, we need to introduce some more notation. For the sake of clarity, it is convenient to work  with an integer $n$ rather than with $a$.

Let $u$ be a primitive $n$-th root of unity. Let 
\[
A = \{ [1:1], [1:u], [1:u^2], \dots, [1:u^{n-1}] \} \subset \PP^1. 
\]
Let $X \subset \PP^3$ be the image of all possible Segre products of $A$ with itself (see Section \ref{sec:Segre_Embeddings}), so $X$ is the set of the following $n^2$ points:
\[
\begin{array}{lllcl}
[1:1:1:1], & [1:u:1:u], & [1: u^2: 1: u^2], & \dots & [1: u^{n-1}: 1: u^{n-1}], \\ 

 [1: 1: u: u], &  [1:u: u: u^2], & [1: u^2:u:u^3], & \dots & [1: u^{n-1}:u: 1], \\
 
 [1:1:u^2:u^2], & [1:u:u^2:u^3], & [1:u^2:u^2:u^4], & \dots & [1:u^{n-1}: u^2: u] \\
 
  \vdots & \vdots & \vdots && \vdots \\
 
 [1:1:u^{n-1}: u^{n-1}], & [1:u:u^{n-1}:1], & [1:u^2:u^{n-1}:u], & \dots & [1:u^{n-1}: u^{n-1}: u^{n-2}].
\end{array}
\]
Using the variables $x,y,z,w$ in that order, these points all lie on the quadric $\mathcal Q$ defined by $xw-yz=0$. This is an $(n, n)$-grid, hence obviously geproci. Thinking of $\mathcal Q$ as $\PP^1\times \PP^1$ the ``vertical" lines $L_i$ ($0 \leq i \leq n-1$) and ``horizontal" lines $M_i$ ($0 \leq i \leq n-1$) defined by these points are 
\[
\begin{array}{lllllll}
L_0 \colon  (x-y, z-w), &&& M_0 \colon (x-z, y-w), \\
L_1\colon (ux-y, uz-w), &&& M_1\colon  (ux-z, uy-w), \\
L_2\colon  (u^2x-y, u^2z-w), &&& M_2\colon (u^2x-z, u^2y-w), \\
\hspace{.9in} \vdots &&& \hspace{.9in} \vdots \\
L_{n-1}\colon (u^{n-1}x-y, u^{n-1}z-w), &&& M_{n-1}\colon (u^{n-1}x-z, u^{n-1}y-w).
\end{array}
\]

\vspace{.2in}

Let us now consider two sets of collinear points \[Y_1=\{[1:0:0:-1], [1:0:0:-u],\ldots, [1:0:0:-u^{n-1}]\}\] and \[Y_2=\{[0:1:-1:0], [0:1:-u:0],\ldots, [0:1:-u^{n-1}:0]\}.\] 
Note that the points in $Y_1$ and $Y_2$ are not in $\mathcal Q$. 
We call $\ell_1$ the line containing $Y_1$ that is defined by $y=z=0$ and $\ell_2$ the line defined by $x=w=0$, containing $Y_2$. 

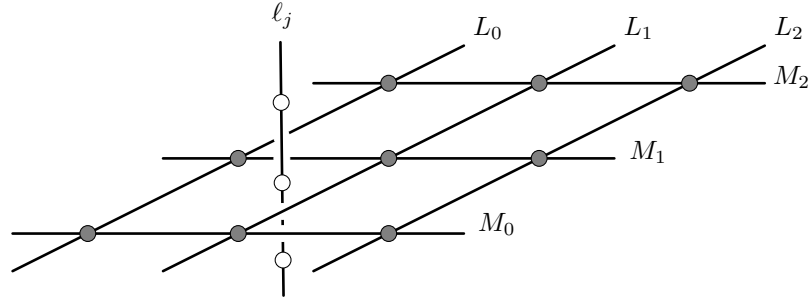
\begin{figure}
\definecolor{ffqqqq}{rgb}{1.,0.,0.}
\definecolor{ffffqq}{rgb}{1.,1.,0.}
\definecolor{ududff}{rgb}{0.30196078431372547,0.30196078431372547,1.}
\begin{tikzpicture}[line cap=round,line join=round,>=triangle 45,x=1.0cm,y=1.0cm]
\clip(-3.702788181883609,0) rectangle (13.66116786499033,5);
\draw [line width=1.pt] (2.,2.)-- (8.,2.);
\draw [line width=1.pt] (0.,0.5)-- (6.,3.5);
\draw [fill=white,color=white] (3.5763303932616,2.2881651966308003) circle (3pt);
\draw [fill=white,color=white] (3.57984460297661,2.) circle (3pt);
\draw [line width=1.pt] (3.602104592194438,0.17468088413812813)-- (3.560991344547092,3.545967191220497);
\draw [fill=white,color=white] (3.5884516053828124,1.2942258026914062) circle (3pt);
\draw [fill=white,color=white] (3.5920397249278295,1.) circle (3pt);
\draw [line width=1.pt] (0.,1.)-- (6.,1.);
\draw [line width=1.pt] (4.,3.)-- (10.,3.);
\draw [line width=1.pt] (2.,0.5)-- (8.,3.5);
\draw [line width=1.pt] (4.,0.5)-- (10.,3.5);
\draw (3.324029118956605,4.2) node[anchor=north west] {$\ell_j$};
\draw (6.065706389515648,1.393594243209811) node[anchor=north west] {$M_0$};
\draw (8.076269721258946,2.35) node[anchor=north west] {$M_1$};
\draw (9.985289450388946,3.35) node[anchor=north west] {$M_2$};
\draw (6,4) node[anchor=north west] {$L_0$};
\draw (8,4) node[anchor=north west] {$L_1$};
\draw (10,4) node[anchor=north west] {$L_2$};
\begin{scriptsize}
\draw [fill=gray] (1.,1.) circle (3pt);
\draw [fill=gray] (3.,2.) circle (3pt);
\draw [fill=gray] (5.,3.) circle (3pt);
\draw [fill=gray] (3.,1.) circle (3pt);
\draw [fill=gray] (5.,2.) circle (3pt);
\draw [fill=gray] (7.,3.) circle (3pt);
\draw [fill=gray] (5.,1.) circle (3pt);
\draw [fill=gray] (7.,2.) circle (3pt);
\draw [fill=gray] (9.,3.) circle (3pt);
\draw [fill=white] (3.5707698783183823,2.7441274219746608) circle (3pt);
\draw [fill=white] (3.583806904431091,1.6750912807325562) circle (3pt);
\draw [fill=white] (3.5963426240297354,0.6471622736437241) circle (3pt);
\end{scriptsize}
\end{tikzpicture}
\caption[The standard construction from Theorem \ref{t. geproci infinite class}(a).]{The standard construction from Theorem \ref{t. geproci infinite class}(a) for $n=3$ with the $(n,n)$-grid $X$ (the gray dots) and the set $Y_j$ (the white dots) for either $j=1$ or $j=2$.}\label{Fig: StConst n=3, i}
\end{figure}

\begin{figure}
\hbox{\hskip-1.35in\hbox{
\definecolor{ffqqqq}{rgb}{1.,0.,0.}
\definecolor{ffffqq}{rgb}{1.,1.,0.}
\definecolor{ududff}{rgb}{0.30196078431372547,0.30196078431372547,1.}
\begin{tikzpicture}[line cap=round,line join=round,>=triangle 45,x=1.0cm,y=1.0cm]
\clip(-3.702788181883609,0) rectangle (15,5);

\draw [line width=1.pt] (6,4)-- (14,4);
\draw [line width=1.pt] (6,0.5)-- (14,4.5);

\draw [line width=1.pt] (2.,2.)-- (10,2.);
\draw [line width=1.pt] (0.,0.5)-- (8,4.5);
\draw [fill=white,color=white] (3.5763303932616,2.2881651966308003) circle (3pt);
\draw [fill=white,color=white] (3.57984460297661,2.) circle (3pt);

\draw [line width=1.pt] (3.6,0.17468088413812813)-- (3.6,1+3.545967191220497);
\draw [fill=white,color=white] (3.5884516053828124,1.2942258026914062) circle (3pt);
\draw [fill=white,color=white] (3.5920397249278295,1.) circle (3pt);

\draw [line width=1.pt] (2.,0.5)-- (10,4.5);
\draw [line width=1.pt] (4.,3.)-- (12,3.);

\draw [fill=white,color=white] (2+3.5763303932616,1+2.2881651966308003) circle (3pt);
\draw [fill=white,color=white] (2+3.57984460297661,1+2.) circle (3pt);
\draw [fill=white,color=white] (2+3.5763303932616,2.2881651966308003) circle (3pt);
\draw [fill=white,color=white] (2+3.57984460297661,2.) circle (3pt);
\draw [line width=1.pt] (2+3.58,0.17468088413812813)-- (2+3.58,1+3.545967191220497);
\draw [fill=white,color=white] (2+3.5884516053828124,1.2942258026914062) circle (3pt);
\draw [fill=white,color=white] (2+3.5920397249278295,1.) circle (3pt);

\draw [line width=1.pt] (0.,1.)-- (8,1.);
\draw [line width=1.pt] (4.,0.5)-- (12,4.5);
\draw (3.7,4.8) node {$\ell_1$};
\draw (5.7,4.8) node {$\ell_2$};
\draw (2+6.065706389515648,1.393594243209811) node[anchor=north west] {$M_0$};
\draw (2+8.076269721258946,2.35) node[anchor=north west] {$M_1$};
\draw (2+9.985289450388946,3.35) node[anchor=north west] {$M_2$};
\draw (2+11.985289450388946,4.35) node[anchor=north west] {$M_3$};
\draw (8,5) node[anchor=north west] {$L_0$};
\draw (10,5) node[anchor=north west] {$L_1$};
\draw (12,5) node[anchor=north west] {$L_2$};
\draw (14,5) node[anchor=north west] {$L_3$};
\begin{scriptsize}
\draw [fill=gray] (1.,1.) circle (3pt);
\draw [fill=gray] (3.,2.) circle (3pt);
\draw [fill=gray] (5.,3.) circle (3pt);
\draw [fill=gray] (3.,1.) circle (3pt);
\draw [fill=gray] (5.,2.) circle (3pt);
\draw [fill=gray] (7.,3.) circle (3pt);
\draw [fill=gray] (5.,1.) circle (3pt);
\draw [fill=gray] (7.,2.) circle (3pt);
\draw [fill=gray] (9.,3.) circle (3pt);

\draw [fill=gray] (7,4) circle (3pt);
\draw [fill=gray] (9,4) circle (3pt);
\draw [fill=gray] (11,4) circle (3pt);
\draw [fill=gray] (13,4) circle (3pt);
\draw [fill=gray] (7,1) circle (3pt);
\draw [fill=gray] (9,2) circle (3pt);
\draw [fill=gray] (11,3) circle (3pt);

\draw [fill=white] (3.5963426240297354,3.6471622736437241) circle (3pt);
\draw [fill=white] (3.5963426240297354,2.7441274219746608) circle (3pt);
\draw [fill=white] (3.5963426240297354,1.6750912807325562) circle (3pt);
\draw [fill=white] (3.5963426240297354,0.6471622736437241) circle (3pt);

\draw [fill=white] (2+3.583806904431091,3.7441274219746608) circle (3pt);
\draw [fill=white] (2+3.583806904431091,2.7441274219746608) circle (3pt);
\draw [fill=white] (2+3.583806904431091,1.6750912807325562) circle (3pt);
\draw [fill=white] (2+3.583806904431091,0.6471622736437241) circle (3pt);
\end{scriptsize}
\end{tikzpicture}
}}
\caption[The standard construction from Theorem \ref{t. geproci infinite class}(b).]{The standard construction from Theorem \ref{t. geproci infinite class}(b) for $n=4$ with the $(n,n)$-grid $X$ (the gray dots) and the sets $Y_1$ (the white dots on $\ell_1$) and $Y_2$ (the white dots on $\ell_2$).}\label{Fig: StConst n=3, ii}
\end{figure}
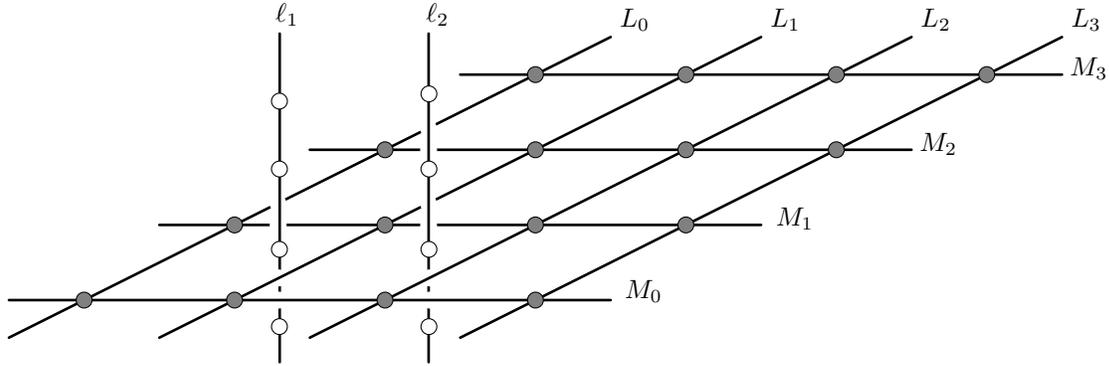

\begin{theorem}[Standard construction]\label{t. geproci infinite class}\index{standard construction}
Let $n \geq 3$ be a positive integer.  
Then  
\begin{itemize}
    \item[(a)] both $X \cup Y_1$ and $X \cup Y_2$ are $(n,n+1)$-geproci sets; and,
    \item[(b)] for $n$ even, $X \cup Y_1 \cup Y_2$ is an $(n,n+2)$-geproci set.
\end{itemize}
Moreover, none of the geproci sets constructed in this way is a grid.
\end{theorem}
\begin{proof}
Let $P = [a:b:c:d]$ be a general point in $\PP^3$. We will now construct a pencil of cones of degree $n$ with vertex at $P$, containing the grid $X$. While this pencil depends on $P$, a key point of our proof will be to show that each of these pencils contains an element vanishing at the set $Y_1$ and an element vanishing at the set $Y_2$ even though both sets do not depend  on $P$. Finally  we will show the claimed geprociness.
\medskip


\setcounter{step}{0}


\stepcounter{step}
\noindent {\textbf{Step \thestep}}. {\it Construction of two cones, ${F}$ and ${G}$ of degree $n$ containing $X$ and passing through the general point $P = [a:b:c:d]$. These cones span our pencil.}
\medskip

These cones are simply the union of the planes spanned by $P$ and each of the lines $L_i$, and the union of the planes spanned by $P$ and each of the lines $M_i$. Their equations (after a calculation) are:
\[
\begin{array}{lcl}
{F}([x:y:z:w]) & = & \displaystyle \prod_{i=0}^{n-1} \left [ (u^ic-d)(u^i x-y) + (b-u^ia)(u^iz-w) \right ], \\ \\
{G}([x:y:z:w]) & = & \displaystyle \prod_{i=0}^{n-1} \left [ (u^ib-d)(u^ix-z) + (c-u^ia)(u^iy-w) \right ].
\end{array}
\]

\noindent Depending on the parity of $n$ and possibly the choice of $Y_1$ or $Y_2$, either $F - G$ or $F + G$ will be the special element in the pencil mentioned above.
\medskip


\stepcounter{step}
\noindent {\textbf{Step \thestep}}. {\it Restricting $F$ and $G$ to the line $\ell_1$.}
\medskip

Clearly this gives on $\ell_1$ the following forms of degree $n$
\[
\begin{array}{lcl}
 F_1([x:w]) & = & \displaystyle \prod_{i=0}^{n-1} \left [ (u^ic-d)(u^i x) - (b-u^ia)w \right ]\! , \\ \\
 G_1([x:w]) & = & \displaystyle \prod_{i=0}^{n-1} \left [  (u^ib-d)(u^ix) - (c-u^ia)w \right ]\!.
\end{array}
\]
\medskip


\stepcounter{step} 
\noindent {\textbf{Step \thestep}}. {\it Restricting $F$ and $G$ to the line $\ell_2$.}
\medskip

We obtain the following two forms
\[
\begin{array}{lcl}
 F_2([y:z]) & = & \displaystyle \prod_{i=0}^{n-1} \left [ (u^ic-d)(-y) + (b-u^ia)(u^iz) \right ]\!, \\ \\
 G_2([y:z]) & = & \displaystyle \prod_{i=0}^{n-1} \left [  (u^ib-d)(-z) + (c-u^ia)(u^iy) \right ]\!.
\end{array}
\]
\medskip

We now split our argument in two cases depending on the parity of $n$.

\begin{center}
    \textbf{Case $n$ odd.}
\end{center}


\stepcounter{step}   
\noindent {\textbf{Step \thestep}}. {\it We claim that for any $0 \leq \alpha \leq n-1$ the point $[1:-u^\alpha]$  is a zero of $F_1-G_1$. }
\medskip

To begin with we observe for $n$ odd that 
$$\prod_{j=0}^{n-1} u^{j} = 1.$$
Then we have the following calculation.
{
\begin{equation} \label{calculation}
\begin{array}{l}
(F_1-G_1)([1:-u^\alpha])  = \\ 
\displaystyle \prod_{i=0}^{n-1} [(u^i c -d)(u^i) - (b-u^i a)(-u^\alpha)] -  \prod_{i=0}^{n-1} [(u^i b -d)(u^i) - (c-u^i a)(-u^\alpha)] = \\
\displaystyle \prod_{i=0}^{n-1} [ c u^{2i} - du^i + (b - a u^i)(u^\alpha)] - \prod_{i=0}^{n-1} [b u^{2i}  - du^i  + (c-au^i)(u^\alpha) ] = \\
\displaystyle \prod_{i=0}^{n-1} [cu^{2i} - du^i - au^{i+\alpha} + bu^\alpha] - \prod_{i=0}^{n-1} [bu^{2i} - du^i - au^{i+\alpha} + cu^\alpha] = \\
\displaystyle \prod_{i=0}^{n-1} [(cu^{2i} +bu^\alpha )  - (du^i + au^{i+\alpha})] -  \prod_{i=0}^{n-1} [(bu^{2i} + cu^\alpha) - (du^i + au^{i+\alpha})] = \\
\displaystyle \prod_{i=0}^{n-1} [(cu^{2i} +bu^\alpha )  - (du^i + au^{i+\alpha})] - \prod_{j=0}^{n-1} u^j   \cdot \prod_{i=0}^{n-1} [(bu^{2i} + cu^\alpha) - (du^i + au^{i+\alpha})].
\end{array}
\end{equation}
}
Using the fact that $n$ is odd, we may rearrange the factors in the product $\prod_{j=0}^{n-1} u^j$ so that  $u^j = u^{-2i+\alpha}$. Continuing the calculation, the last line of \eqref{calculation} becomes
\begin{equation} \label{calculation2}
\begin{array}{l}
\displaystyle \prod_{i=0}^{n-1} [(cu^{2i} +bu^\alpha )  - (du^i + au^{i+\alpha})] - \prod_{i=0}^{n-1} [(bu^\alpha + cu^{2\alpha - 2i}) - (du^{\alpha -i} + au^{-i+2\alpha})] = \\
\displaystyle \prod_{i=0}^{n-1} [(cu^{2i} +bu^\alpha )  - (du^i + au^{i+\alpha})] - \prod_{i=0}^{n-1} [(bu^\alpha + cu^{2(\alpha - i)}) - (du^{\alpha -i} + au^{(\alpha-i)+\alpha})].
\end{array}
\end{equation}
We now set $k=\alpha -i$ and finally get
{
\begin{equation} \label{calculation3}
\begin{array}{l}
(F_1-G_1)([1:-u^\alpha]) = \\   
\displaystyle \prod_{i=0}^{n-1} [(cu^{2i} +bu^\alpha )  - (du^i + au^{i+\alpha})] - \prod_{k=0}^{n-1} [(bu^\alpha + cu^{2k}) - (du^{k} + au^{k+\alpha})] = 0.
\end{array}
\end{equation}
}
In the next step we repeat Step 4, \emph{mutatis mutandis}, for the line $\ell_2$.
\medskip


\stepcounter{step}
\noindent {\textbf{Step \thestep}}.  {\it We claim that for any $0 \leq \alpha \leq n-1$ the point $[1:-u^\alpha]$  is a zero of $F_2+G_2$.}
\medskip

The proof amounts to the following calculations.
\begin{equation} \label{calculation4}
\begin{array}{l}
(F_2+G_2)([1:-u^\alpha])  = \\ 
\displaystyle \prod_{i=0}^{n-1} [(u^i c -d)(-1) + (b-u^i a)(-u^{i+\alpha})] +  \prod_{i=0}^{n-1} [(u^i b -d)(u^\alpha) + (c-u^i a)(u^i)] = \\
\displaystyle \prod_{i=0}^{n-1} [ -cu^i +d - bu^{i+\alpha} + au^{2i+\alpha} ] + \prod_{i=0}^{n-1} [bu^{i+\alpha} - du^\alpha + cu^i - au^{2i}]=\\
\displaystyle  (u^\alpha)^n  \cdot \prod_{i=0}^{n-1} [ -cu^i +d - bu^{i+\alpha} + au^{2i+\alpha} ] + (-1)^n \prod_{i=0}^{n-1} [-bu^{i+\alpha} + du^\alpha - cu^i + au^{2i}] = \\
\displaystyle  \prod_{i=0}^{n-1} [ -cu^{i+\alpha} +du^\alpha - bu^{i+2\alpha} + au^{2i+2\alpha} ] - \prod_{i=0}^{n-1} [-bu^{i+\alpha} + du^\alpha - cu^i + au^{2i}]. 
\end{array}
\end{equation}
\noindent  We now set $k=i+\alpha$ (i.e., $i = k-\alpha$) and get
{ 
\begin{equation} \label{calculation5}
\begin{array}{l}
(F_2+G_2)([1:-u^\alpha])  =\\   \displaystyle  \prod_{k=0}^{n-1} [ -cu^{k} +du^\alpha - bu^{k+\alpha} + au^{2k} ] - \prod_{i=0}^{n-1} [-bu^{i+\alpha} + du^\alpha - cu^i + au^{2i}] 
 =  0.
\end{array}
\end{equation}
}
This concludes Step 5.
\medskip

It follows from Step 4 that for $n$ odd $Y_1$ is the intersection of $\ell_1$ with the pencil element $(F - G)$. 
By Step 5 in the same way $Y_2$ is the intersection of $\ell_2$ with the pencil element $(F + G)$.
Both sets are independent of $P$.
\medskip

\begin{center}
    \textbf{Case $n$ even}
\end{center}
   In the next step we study the remaining cases of $n$ even.
   
   \medskip
   
\stepcounter{step}
\noindent {\textbf{Step \thestep}}. {\it Vanishing of distinguished pencil elements. }
\medskip

Now we want to show that the sets of points $Y_1$ and $Y_2$ are simultaneously cut out on $\ell_1$ and on $\ell_2$ respectively by the same pencil element, namely $(F - G)$. That is, the points $[1:-u^\alpha]$ are zeros of $F_1 - G_1$ (in the variables $x,w$) and also zeros of $F_2 - G_2$ (in the variables $y,z$).

Notice that since $n$ is even, we  have $$\prod_{i=0}^{n-1}u^i = -1.$$ 
We want to show that the following two expressions 
 { 
 \[
 \begin{array}{l}
 F_1([1:-u^\alpha])  = 
 \displaystyle \prod_{i=0}^{n-1} [(cu^{2i} - du^i) + (b u^\alpha - a u^{i+\alpha})],\\
 G_1([1:-u^\alpha])  = 
 \displaystyle \prod_{i=0}^{n-1} [(bu^{2i} - du^i )+ ( cu^\alpha - au^{i+\alpha}) ]
 \end{array} 
 \]
 }
 are equal for $0 \leq \alpha \leq n-1$. We will leave $F_1$ alone and focus on $G_1$. 
{
\[
\begin{array}{rcl}
G_1([1: - u^\alpha]) & = & \displaystyle \left ( \prod_{i=0}^{n-1} u^\alpha \right )\cdot \left ( 
\prod_{i=0}^{n-1} u^{-i} \right )^2 \cdot 
\prod_{i=0}^{n-1} [(bu^{2i} - du^i )+ ( cu^\alpha - au^{i+\alpha}) ] \\

& = & \displaystyle \left ( \prod_{i=0}^{n-1} u^\alpha \right ) \cdot \prod_{i=0}^{n-1} [(b - du^{-i}) + (c u^{\alpha - 2i} - au^{\alpha - i} )] \\
& = & \displaystyle \prod_{i=0}^{n-1} [ bu^\alpha - du^{\alpha - i} + c u^{2(\alpha -i)} - a u^{\alpha + (\alpha-i)}].
\end{array}
\]
}
Substituting $k = \alpha-i$ and comparing with $F_1([1:-u^\alpha])$ gives the desired result for $Y_1$.
\medskip

Turning to $Y_2$ we want to show that the following two expressions 
 {
 \[
 \begin{array}{l}
 F_2([1:-u^\alpha])  = 
 \displaystyle \prod_{i=0}^{n-1} [-cu^i +d  - b u^{i+\alpha} + a u^{2i+\alpha})] \mbox{ and }\\ 
 G_2([1:-u^\alpha])  = 
 \displaystyle \prod_{i=0}^{n-1} [(bu^{i+\alpha} - du^\alpha +  cu^i - au^{2i} ] 
 \end{array} 
 \]
 }
are equal for $0 \leq \alpha \leq n-1$. We will leave $F_2$ alone and focus on $G_2$.  
 {
\[
\begin{array}{rcl}
G_2([1:-u^\alpha]) & = & \displaystyle \left ( \prod_{i=0}^{n-1} (-u^{-\alpha} ) \right ) 
\cdot \prod_{i=0}^{n-1}[ bu^{i+\alpha}  - du^\alpha +  cu^i - au^{2i} ] \\
& = & \displaystyle \prod_{i=0}^{n-1} [ -bu^i + d - cu^{i-\alpha} + au^{2i-\alpha}]. 
\end{array}
\]
}
Now we make the substitution $k=i-\alpha$ and note $2i-\alpha = 2k+\alpha$. Then
\[
\begin{array}{rcl}
G_2([1:-u^\alpha]) & = & \displaystyle \prod_{k=0}^{n-1} [ -bu^{k+\alpha}+ d - cu^{k} + au^{2k+\alpha}]. 
\end{array}
\]
Comparing this with $F_2([1:-u^\alpha])$ gives the result.
\medskip

We have shown that  the same pencil element $(F - G)$ simultaneously cuts out the set $Y_1$ on $\ell_1$ and the set $Y_2$ on $\ell_2$, and we are done with Step 6 and the case $n$ is even.
\medskip

In the final short two steps we will verify the geprociness claimed in the Theorem.
\medskip

\stepcounter{step}
\noindent {\textbf{Step \thestep}}. {\it Let $Z_1 = X \cup Y_1$ and $Z_2 = X \cup Y_2$. Then both $Z_1$ and $Z_2$ are geproci. }
\medskip

Let $H$ be a general plane. Since $(F - G)$ is a cone with vertex $P$ that contains  $Z_1$, the restriction of this cone  to $H$ is a plane curve $C$ on $H$ of degree $n$ containing the projection of $Z_1$. But the image of the $n+1$ lines ($n$ from one ruling of the grid and one not from the grid) gives a plane curve of degree $n+1$ with no component in common with $C$, and the intersection is the projection of $Z_1$, hence $Z_1$ is geproci.

For $Z_2$ the argument is the same but we need to work with $F-G$ for $n$ even and $F+G$ for $n$ odd.
\medskip


\stepcounter{step}
\noindent {\textbf{Step \thestep}}. {\it Assume $n$ is even. Let $Z = X \cup Y_1 \cup Y_2$. Then $Z$ is geproci.}
\medskip

The proof is the same as that of Step 7, using $(F - G)$. 
\end{proof}

\begin{example}\label{e:D4exOf4.2a}
The $D_4$ configuration\index{configuration! $D_4$} of points motivated Theorem \ref{t. geproci infinite class}(a).
Indeed, the $D_4$ configuration, as given in Example \ref{e:D4list},
is projectively equivalent to the $(3,4)$-geproci set $X\cup Y_1$
given by Theorem \ref{t. geproci infinite class}(a) for $n=3$.
The matrix taking the former to the latter is
$$
\left(
\begin{array}{cccc}
1 & u^2 & u & 1 \\
1 & u & u & u^2 \\
u^2 & u & u & 1 \\
u^2 & 1 & u & u^2 \\
\end{array}\right).
$$
\end{example}

\begin{example}\label{e:F4exOf4.2b}
The $F_4$ configuration\index{configuration! $F_4$} of 24 points given by the $F_4$ root system is $(4,6)$-geproci \cite{CM}.
It is also projectively equivalent to the configuration of 24 points given by Theorem \ref{t. geproci infinite class}(b) for $n=4$.
This in fact is what led us to Theorem \ref{t. geproci infinite class}(b). 
The $F_4$ configuration, as given in \cite{HMNT}, consists of the points
$$\begin{array}{lllll}
\ &\ [1:1:1:1], & [1:0:1:0], & [1:-1:1:-1],& [0:1:0:1],  \\
\ &\ [1:1:0:0], & [1:0:0:0], & [1:-1:0:0],& [0:1:0:0],  \\
\ &\ [1:1:-1:-1], & [1:0:-1:0], & [1:-1:-1:1],& [0:1:0:-1],  \\
\ &\ [0:0:1:1], & [0:0:1:0], & [0:0:1:-1],& [0:0:0:1], \\
\ &\ [1:0:0:-1], & [-1:1:1:1], & [0:1:1:0],& [1:1:1:-1],\\
\ &\ [0:1:-1:0], &[1:-1:1:1], &[1:0:0:1], &[1:1:-1:1].\\
\end{array}$$
It is easy to check that the top four rows above
give a $(4,4)$-grid. The bottom two rows correspond to $Y_1$ and $Y_2$.
The matrix taking these 24 points to those given by Theorem \ref{t. geproci infinite class}(b) for $n=4$ is
 $$
\left(
\begin{array}{cccc}
-i & 1 & 1 & i \\
1 & -i & i & 1 \\
1 & i & -i & 1 \\
i & 1 & 1 & -i \\
\end{array}\right).
$$
(In this case the root of unity $u$ in the theorem is $u=i$.)
As an aside, we note that the 12 points above having exactly two nonzero coordinates
give a $D_4$ configuration, as do the complementary 12 points (see section \ref{Sec: D4 in F4}).
We also note that the sets of 4 collinear points given by each row and column of the $(4,4)$-grid above,
and also by the bottom two rows above (which correspond to $Y_1$ and $Y_2$), are harmonic sets of points.\index{points!harmonic}
\end{example}

Now we move on to Lemma \ref{lem subset} which allows us to exhibit geproci subsets of sets coming from the standard construction. This result is \cite[Proposition 2.11]{CM} adjusted to our current situation. To make the statement simpler when we do not know which 
of $a$ or $b$ might be smaller, we introduce the notation 
$\{a,b\}$-geproci, meaning $(a,b)$-geproci if $a\leq b$ and $(b,a)$-geproci if $b<a$.
(Thus we replace the ordered pair $(a,b)$ or $(b,a)$ with the set $\{a,b\}$.)\index{geproci! $\{a,b\}$}

\begin{lemma}[Splitting a geproci set]\label{lem subset}\index{geproci! splitting}  If an $\{a,b\}$-geproci set $Z$ contains a $\{c,b\}$-geproci subset $Z'$, whose general projection shares with the general projection of $Z$ a minimal generator of degree $b$, then the residual set $Z''=Z\setminus Z'$ is $\{a-c,b\}$-geproci.
\end{lemma}
\begin{proof} Let $\pi$ be a general projection to a plane. Then $W=\pi(Z'')$ is residual to $X=\pi(Z')$ in the complete intersection $Y=\pi(Z)$.
Let $I(Y)=(A,B)$ where $A$ and $B$ are forms with $\deg A=a$ and $\deg B=b$, so
$I(X)=(B,C)$ where $C$ is a form with $\deg C=c$. Then $X\subset Y$ implies $(A,B) = I(Y) \subset I(X) = (B,C)$
Thus $A\in (B,C)$, so $A = FB + GC$ for some forms $F,G$. Thus $I(Y) = (B, FB+GC) = (B,GC)$.
Because $Y$ is a finite point set, $A$ and $GC$ have no common factors of degree 1 or more.
So $W$, the complement $X$ in $Y$, has $I(W)=(B,G)$, hence $W$ is the complete intersection of
a curve of degree $b$ with a curve of degree $a-c$; i.e., $Z''$ is $\{a-c,b\}$-geproci.
\end{proof}

We are ready to derive our main result from Theorem \ref{t. geproci infinite class}.

\begin{theorem}[Geography of nongrid geproci sets]\label{t. (a,b)-geproci}\index{geproci! geography}
Fix integers $a,b$, where $4 \leq a \leq b$. Then there exists an $(a,b)$-geproci set of points in $\PP^3$ that is nontrivial. 
\end{theorem}

\begin{proof}
For $n = b$ construct a $(b,b+1)$-geproci set $Z$ as in Theorem \ref{t. geproci infinite class}(a) using the points $Y_1$ on $\ell_1$ but not $Y_2$. In particular, $Z$ consists of a $(b,b)$-grid $X$ plus one additional set of $b$ collinear points, $Y_1$, lying on a line off the quadric $\mathcal{Q}$ containing the grid. 
The removal from $Z$ of the $(j,b)$-geproci set $Z'$ consisting of the points of
$X$ on $1\leq j\leq b-3$ grid rows, one at a time (all chosen from among the lines $L_i$, 
or all chosen from among the lines $M_i$,
but always keeping the set $Y_1$) gives a residual set $Z''$ which is $(b+1-j,b)$-geproci by Lemma \ref{lem subset}.
Hence removing $j=b+1-a$ grid rows we arrive at a complete intersection of type $(a,b)$. Since $a \geq 4$ and we have not touched $Y_1$, we are assured that our geproci set does not lie on the quadric $\mathcal{Q}$ which contains the grid, and we are done.
\end{proof}

\begin{remark}\label{r. 4.2(b) geprocis}
In the proof of Theorem \ref{t. (a,b)-geproci}, by removing one grid line of $X$ at a time from the
$(n,n+1)$-geproci set $Z=X\cup Y_j$, $j=1,2$, constructed in Theorem \ref{t. geproci infinite class}(a),
we obtained a sequence of $Z_0,Z_1,\ldots,Z_{n-3}$ of nontrivial geproci sets, 
where we remove $i$ grid lines to obtain $Z_i$. 
(One can remove a selection of the grid lines $L_j$, 
or a selection of the grid lines $M_j$, but not a mixture of lines $L_j$ and $M_j$.)
Note $Z_0$ is $(n,n+1)$-geproci and,
for $i>0$, $Z_i$ is $(n+1-i,n)$-geproci.
In the same way, by removing one grid line at a time from the
$(n,n+2)$-geproci set constructed in Theorem \ref{t. geproci infinite class}(b),
we obtain a sequence of $Z_0,Z_1,\ldots,Z_{n-3}$ of nontrivial geproci sets, 
where we remove $i$ grid lines to obtain $Z_i$. Note $Z_0$ is $(n,n+2)$-geproci,
$Z_1$ is $(n,n+1)$-geproci, and, for $i>1$, $Z_i$ is $(n+2-i,n)$-geproci.
\end{remark}

To complement Theorem \ref{t. (a,b)-geproci} above, it is of interest to study what happens when $a\le 3.$ We do this in the next section.

\section{Classification of nondegenerate \texorpdfstring{$(\lowercase{a},\lowercase{b})$}{(a,b)}-geproci sets for \texorpdfstring{$\lowercase{a}\leq 3$}{a}} \label{sec:small_a_b}
Theorem \ref{t. (a,b)-geproci} provides nontrivial $(a,b)$-geproci sets of points for all $4\leq a\leq b$. In this section we will explain briefly why there are no such sets for $a=1$ and $a=2$ and then we will turn to the most intriguing case of $a=3$.
\medskip

To begin with, it is clear that a $(1,b)$-geproci set is a degenerate grid.
\medskip

If $Z$ lies on two skew lines, with the same number of points on 
each line, then $Z$ is clearly a $(2,b)$-grid, and hence it is
$(2,b)$-geproci. The converse is subsumed by the statement of Proposition \ref{SmallProp}. The part dealing with $(3,3)$-grid is taken from \cite{CM}. 

\begin{proposition}[Geproci sets with small $a$ or $b$]\label{SmallProp}
Let $Z\subset \PP^3$ be nondegenerate $(a,b)$-geproci set with $a\leq b$
and either $a=2$ or $b=3$. Then $Z$ is a grid.
\end{proposition}
\begin{proof}
The case of a nondegenerate set of $4$ points is clear, so we can assume that $|Z|\geq6$. 

Let $a=2$ and let $P$ be a general point.
Clearly $Z$ lies on a quadric cone with vertex at $P$ (the cone over the curve of degree $2$ containing the projection of $Z$ from $P$).

If $\dim[I(Z)]_2\leq 4$, then by \cite[Proposition 4.3]{CM}, 
$Z$ is contained in two lines.
For the general projection to be a complete intersection, 
the points must divide evenly on the two lines. Thus $Z$ is a $(2,b)$-grid. 

Now, assume that $\dim[I(Z)]_2\geq 5$.
This means $Z$ has $h$-vector $(1,3,1, \ldots)$.
Since $|Z|\geq 6$, the $h$-vector goes on and does not stop at $(1,3,1)$. 
But the $h$-vector is a Hilbert function, so by Macaulay's Theorem the value is at most 1 in all degrees 
bigger than 2 until it becomes 0. So there is at least one more 1. 
By maximal growth \cite{BGM} this forces all but (exactly) two of the points to lie on a line.
But such a set cannot project to a $(2,b)$ complete intersection with $b>2$, so we have a contradiction.

The only case left is when $a=b=3$, and then the result holds by \cite[Theorem 5.12]{CM}.
\end{proof}

\begin{remark}
Thus every nondegenerate $(a,b)$-geproci set $Z$ with $a\leq b$ and 
$ab<12$ is a grid. Indeed, for $a=1$ it consists of $b$ points on a line and for $a=2$ or $a=b=3$ Proposition \ref{SmallProp} applies.
Therefore the smallest number of points for
which there could exist a nontrivial geproci set is 12 with $(a,b)=(3,4)$. Such a set indeed exists. It is given by the $D_4$ root system \cite[Appendix]{CM} and proved in Chapter \ref{Ch:D4}. 
\end{remark}

\subsection{Geproci sets with $a=3$.}
In this part we provide the full classification of nondegenerate $(3,b)$-geproci sets. Our main result here is the following statement.
\begin{theorem}[Classification of $(3,b)$-geproci sets]\label{thm:classification_of_3xb}
   If $Z\subset \PP^3$ is a nondegenerate $(3,b)$-geproci set with $b\geq 3$, then it is either
   \begin{itemize}
       \item[a)] a $(3,b)$-grid, or
       \item[b)] the $D_4$ configuration of $12$ points.\index{configuration! $D_4$}
   \end{itemize}
\end{theorem}
The proof of this statement takes the rest of this section. We begin with the outline of our strategy. First we show that $Z$ cannot be in linear general position. Thus there are either $\geq 4$ coplanar points in $Z$ or $\geq 3$ points in $Z$ are aligned. We show that the first condition implies the second. Therefore, we are reduced to the case that $Z$ contains either $4$ or more points in a line or the maximal number of collinear points in $Z$ is exactly $3$. We show that the collinearity of at least $4$ points forces $Z$ to be a grid. If there are at most (and hence exactly) $3$ collinear points, then working in the spirit of Lemma \ref{lem subset} we can remove step by step such triples from $Z$ until we are down to a $(3,3)$-grid. Working back we conclude that $Z$ is either the $D_4$ configuration or a grid. We now turn to the details.

When $b=3$, $Z$ is a grid by \cite[Theorem 5.12]{CM}, so hereafter we may assume that $b\geq 4$.

\begin{theorem}[Linear General Position]\label{thm:LGP}
   If $Z$ is a $(3,b)$-geproci set, then the points in $Z$ are not in linear general position.
\end{theorem}
\begin{proof}
   We split the proof into two cases depending on whether $Z$ is contained in a twisted cubic or~not.
   
   \textbf{Case 1. $Z$ is not contained in a twisted cubic.}

We show that there is a point $P\in \PP^3$ such that the projection of $Z$ from $P$ does not lie on any cubic curves. Hence, by semicontinuity, we are done.
	
	Since $Z$ does not lie on a twisted cubic, it contains seven points, say  $P_1,\ldots, P_7\in Z$, which are not contained in a twisted cubic. Indeed, six points in LGP lie on a unique twisted cubic and at least one of the remaining points lies outside, see for instance \cite[Theorem 1]{EH1992}.

 From Proposition \ref{seven pts} there is a curve $\gamma$ such that the projection from a general point $A$ in $\gamma$ 
of $P_1,\ldots, P_7$ consists of seven distinct points on a conic, call it $\sigma_A$.
	By LGP, $\sigma_A$ is irreducible, hence every cubic curve containing the projection of $Z$ from $A$ has $\sigma_A$ as a component.
	
	For $A$ general in an irreducible component $\gamma'$ of $\gamma$, we prove that $\sigma_A$ contains at most 8 points of the projection of $Z$.
	Indeed, assume by contradiction that there are at least 9 points in $Z$ which project to $\sigma_A$. By continuity, the same 9 points of $Z$ project to $\sigma_A$ as $A$ moves in $\gamma'$. 
	Now take general points $A_1,A_2,A_3\in\gamma'$ and call $V_{i}$ the cone over $\sigma_{A_i}$ with vertex at $A_i$. Note that $V_3$ cannot be in the pencil spanned by $V_1$ and $V_2$, since $A_1$, $A_2$, $A_3$ are general and the pencil of quadrics contains at most four singular elements. So $V_3$ cannot contain the complete intersection $V_1\cap V_2$.
	
	If the intersection of $V_1$ and $V_2$ is irreducible, then we get a contradiction because then $V_1\cap V_2 \cap V_3$ contains at most 8 points.
	
	Now assume that the intersection of $V_1$ and $V_2$ is reducible. Since $Z$ is in LGP, it is easy to see that this can happen only if $V_1\cap V_2$ consists of a line and an irreducible twisted cubic. If $V_3$ does not contain the twisted cubic, then it cuts out at most six points of the twisted cubic by B{\'e}zout's Theorem. But the line contains at most two points of $Z$, so
	we are done. If $V_3$ contains the twisted cubic, then $V_1, V_2, V_3$ generate the ideal of the twisted cubic. This is a contradiction since $P_1, \ldots, P_7$ are not on a twisted cubic.

	So there are at least four points in $Z$ which are not projected to $\sigma_A$. These points are in linear general position and in particular they are not contained in a plane, so their projection from $A$ does not lie on a line, and hence the projection of the whole set $Z$ does not lie on a cubic. This contradicts the assumption that $Z$ is a $(3,b)$-geproci set.
	\medskip
	
   \textbf{Case 2. $Z$ is contained in a twisted cubic $\Gamma$.}
A set of $3b$ points on a twisted cubic $\Gamma$ certainly belongs to the linear series $|\mathcal{O}_\Gamma(b)|$,
which maps $\Gamma$ to a rational normal curve $\Delta \subset \PP^{3b}$. Projecting generically $\Gamma$ to a plane we get a nodal cubic $\Gamma'$, whose arithmetic genus is of course $1$.
The set $Z$ projects to a set $Z'$ which is a complete intersection, so $Z$ belongs to the pull-back $V$ of the series cut on $\Gamma'$ by forms of degree $b$ in the plane.
It is immediate to compute by the Riemann-Roch Theorem that $V$ has codimension $1$ in $|\mathcal{O}_\Gamma(b)|$. Then $V$ maps $\Gamma'$ to some projection $\Delta'$ of $\Delta$ to $\PP^{3b-1}$ (the projection is nongeneral, since $\Delta'$ is nodal). Both $\Delta$ and $\Delta'$ have degree $3b$. We claim  that $\Delta'$ corresponds to a general projection of $\Delta$ from a point on its (3-dimensional) secant variety.

Now, consider that $\Delta'$ is the image of $\Gamma'$ in the $b$-Veronese embedding of $\PP^2$. Since $Z$ is a divisor in $|\mathcal{O}_\Gamma(b)|$ its image $Z_0$ spans a hyperplane in $\PP^{3b}$, in particular $Z_0$ is degenerate.

By the same token, the set $Z'$ is the complete intersection of $\Gamma'$ with a curve of degree $b$ in $\PP^2$ so its image $Z'_0$ in $\Delta'$ is degenerate. A projection of the set $Z_0$ spanning a hyperplane of $\PP^{3b}$ cannot be degenerate, unless the projection comes from a point of the hyperplane containing $Z_0$. Since the secant variety $S$ of $\Delta$ is nondegenerate, a general projection of $Z_0$ from a point of $S$ is nondegenerate. This contradiction shows that $Z$ cannot be a $(3,b)$-geproci set.   
\end{proof}
It follows immediately from Theorem \ref{thm:LGP} that $Z$ has either at least $4$ coplanar or at least $3$ collinear points. We will show now that the first property implies the other. 
\begin{proposition}[Coplanar points imply collinear points]\label{3alin}
   Let $Z$ be a nondegenerate $(3,b)$-geproci set with $b\geq 4$.
   If the maximal number $q$ of coplanar points in $Z$ satisfies $q>3$, then $Z$ contains $3$ or more collinear points.
\end{proposition}
\begin{proof}
Let $\Pi$ be a plane containing $q$ points from $Z$ and let $Z_0=Z\cap \Pi$. The proof is quite long and we split it into some steps.

\textbf{Step 1. A lower bound on the number of points in $Z\setminus Z_0$.}
To begin with we assume that $q=3b-1$. Take a general projection of $Z$ to $\Pi$. The point $P_0=Z\setminus Z_0$ is then projected to the unique point $P$
of $\Pi$ such that  $Z_0\cup \{P\}$ is a complete intersection of type $(3,b)$. This is  possible only if $P_0$ already lies in $\Pi$ by the Cayley-Bacharach property, which contradicts the fact that
$Z$ is nondegenerate. 

Next we want to exclude the case $q=3b-2$. Assume that $Q$ is one of the points in $Z$ not contained in $\Pi$. Notice that since the $h$-vector of $Z_0$ is of the form $(1,2,3,3,\ldots)$, there is a unique cubic curve $\Gamma$ containing $Z_0$.
Hence the locus of points $P$ in $\PP^3$ such that projection from $P$ of $Q$ is contained in $\Gamma$ is a surface (more precisely the cone with vertex at $Q$ over $\Gamma$). 
 Thus a general point in $\PP^3$ determines a projection such that the image of $Q$ in this projection is not on $\Gamma$, a contradiction to the assumption that $Z$ is a $(3,b)$-geproci.

Hence from now on we may and do assume $\mathbf{q\leq 3b-3}$. 
\medskip

\textbf{Step 2. Assuming $q\geq 5$.}
If there is a conic $\sigma$ in $\Pi$ containing $7$ or more points from $Z_0$, then we are done. Indeed, this conic is a component of any cubic containing a projection of $Z$ to $\Pi$. Under a projection from a general point in $\PP^3$ all points in $Z_0\setminus\sigma$ and all points from $Z\setminus Z_0$ must project to a line in $\Pi$. Hence in particular all points in $Z\setminus Z_0$ must be aligned.

From now on we assume that the intersection of any conic in $\Pi$ with $Z_0$ consists of at most $6$ points. Moreover, we may assume that any conic through at least $5$ points of $Z_0$ is irreducible, as otherwise there would be $3$ (or more) collinear points in $Z_0$. In our argument we play the number of points in $Z_0$ against the number of points in $Z\setminus\Pi$.

Assume that there is a conic $\sigma\subset \Pi$ containing $6$ points from $Z_0$. Let $P_1\in Z\setminus Z_0$ and let $\pi$ be a projection from a point $P\in\PP^3$ such that $\pi(P_1)\in\sigma$. Thus $P$ is a point on the cone $\calc$ with vertex at $P_1$ over $\sigma$. Taking $P$ general on this cone, we may assume that line $P_1\pi(P_1)$ does not contain any other point from $Z$ except for $P_1$. Since $\sigma$ contains images of $7$ points from $Z$ projected to $\Pi$ (the $6$ points from $Z_0$ already in $\sigma$ and $\pi(P_1)$) it is a component of any cubic containing $\pi(Z)$. Thus the images under $\pi$ of the points from $Z\setminus(Z_0\cup P_1)$ and $Z_0\setminus\sigma$ must all be contained in a line in $\Pi$. If there are at least $3$ points in $Z_0\setminus\sigma$, then we are done since these points must have been already collinear before the projection. Otherwise, there are at most $8$ points from $Z$ in $\Pi$, so that there are at least $4$ points from $Z$ off $\Pi$, say $P_1$ and $P_2,P_3,P_4$. The points $P_2,P_3,P_4$ being projected from $P$ to a line in $\Pi$ must be contained in some plane $H$ through $P$. Projecting from another point $P'$ on the cone $\calc$ over the smooth conic $\sigma$ we conclude in the same way that $P_2,P_3,P_4$ are contained in another plane $H'$. Hence they are all contained in the line $H\cap H'$.

In the remaining case, let $\sigma\subset\Pi$ be a conic through exactly $5$ points from $Z_0$. 

Call $P_1,P_2,P_3$ three points of $Z$ which lie outside $\Pi$. If the three points are collinear, we are done, so assume they are not.
The plane $\langle P_1,P_2,P_3\rangle$ meets $\sigma$ in at most  two points. If both lines $P_1P_2$ and $P_1P_3$ meet $\sigma$, then the line $P_2P_3$ cannot meet $\sigma$. Thus, without loss of generality, we may assume that the line $L$ spanned by $P_1,P_2$ does not meet $\sigma$. Then, for a general choice of a plane $\Pi'\supset L$ the intersection $\Pi'\cap \sigma$ consists of two points $Q_1,Q_2$, and neither $Q_1$ nor $Q_2$ can belong to $L$. The two lines $Q_1P_1$ and $Q_2P_2$ meet in a point $A$ which depends on the choice of $\Pi'\supset L$. Similarly, the two lines $Q_1P_2$ and $Q_2P_1$ meet in a point $A'$ which depends on $\Pi'$. Since $Q_1,Q_2\notin L$, then also $A,A'\notin L$. Call $C_i$ the cone over $\sigma$ with vertex $P_i$. Then, by construction, $A,A'$ belong to the intersection of the cones $C_1,C_2$.


By Lemma \ref{p. 2cones} the intersection of cones $C_1\cap C_2$ is a union $\sigma\cup\sigma'$, where $\sigma'$ is an irreducible conic which avoids $L$. Thus $\sigma'$ does not lie in $\Pi'$. Hence $\sigma'\cap\Pi'$ consists of  the points $A,A'$. Consider the projection $\pi_A$ from $A$ to $\Pi$. It sends a point $P\notin\Pi$ to $\sigma$ only if $A$ belongs to the cone over $\sigma$ with vertex $P$.  Now recall that there are at least $5$ points of $Z$, besides $P_1,P_2$, which do not lie in $\sigma$.
Call these points $P_3,\dots,P_7$. If $P_i\in\Pi$, then $\pi_A$ cannot send $P_i$ to $\sigma$. Assume $P_i\notin \Pi$. Then for a general choice of the plane $\Pi'\supset L$, i.e., for a general choice of $A\in\sigma'$, the projection $ \pi_A$ maps $P_i$ to $\sigma$ only if the cone $C_i$ over $\sigma$ with vertex $P_i$ contains $\sigma'$, in which case  $C_i$ contains the intersection $C_1\cap C_2$, hence $C_i$ lies in the pencil spanned by $C_1,C_2$. As the pencil contains at most two singular quadrics other than $C_1,C_2$ by Remark \ref{r. 4singcones}, then there are at most two points of $Z\setminus \Pi$, which map to $\sigma$ in $\pi_A$, when $A$ is general in $\sigma'$. (Indeed, by Remark \ref{r. 3singcones} the quadrics are $3$, namely $C_1,C_2$, and the union of the planes spanned by $\sigma,\sigma'$, and note that the vertices of the union of the two planes lies in $\Pi$.) Thus, for a general $A\in\sigma'$ the projection $\pi_A$ maps at least 7 points of $Z$ (including $P_1,P_2$) to $\sigma$, and at least three points of $Z$ (call them $P_5,P_6,P_7$) outside $\sigma$. By semicontinuity, the image of $Z$ in $\pi_A$ lies in a cubic of $\Pi$, which contains $\sigma$ by B{\'e}zout's Theorem. Thus, the three points $\pi_A(P_5),\pi_A(P_6),\pi_A(P_7)$ are collinear, i.e., $A,P_5,P_6,P_7$ are coplanar. Since this holds for $A\in\sigma'$ general, then $P_5,P_6,P_7$ must be collinear.

\medskip

\textbf{Step 3. Assuming at most $4$ points in $Z$ are coplanar.} In this step our argument is combinatorial.

We denote by $[3b]^t$ the collection of all subsets of  $\{1,\ldots, 3b\}$ containing $t$ distinct elements. We assume now by way of contradiction that no $3$ points in $Z$ are collinear and 
consider the set 
	\[
	\mathcal{S}=\{\{i,j,k,\ell\}\subseteq [3b]^4\ :\ P_i,P_j,P_k,P_{\ell}\  \text{span a plane} \}.
	\]
	We \textbf{claim} first that $b=4$ and second that for each set $T\subseteq [12]^3$ there is a unique element $T' \in \mathcal S$ such that $T\subseteq T'$. Taking this for granted for a while, we come to a contradiction since the set $\mathcal S$, with these properties, would be a Steiner system\index{Steiner system} of type $(3,4,12)$, usually written $\mathcal S= S(3,4,12)$. But the parameters $(3,4,12)$ are not compatible with the existence of a Steiner system. In particular in \cite[Theorem 3]{hanani} it is shown that a Steiner system $S(3,4,v)$ exists if and only if either $v\equiv 2 \mod 6$ or $v\equiv 4 \mod 6.$  
	
Turning to the proof of the claim let $\alpha$ be a plane containing exactly four points of $Z$; it exists by hypothesis. 
Let $P_1,P_2,P_3$ be points of $Z$ not in the plane $\alpha$. Consider a general point $P$ in the line $M=\langle P_1,P_2,P_3\rangle\cap \alpha$ and let $\pi$ be the projection from this point to a plane $\beta$. The points in $\alpha$ map to $4$ points on the line $N=\alpha\cap\beta$. By semicontinuity, the set $\pi(Z)$ is contained in a cubic, hence $N$ is a component of this cubic. Let $C$ be the conic residual to $N$ in the cubic.
The points $\pi(P_1),\pi(P_2),\pi(P_3)$ are projected to some line $L$ which is then a component of $C$. Then the cubic is the union of $3$ lines: $N$, $L$ and some line $K$, each of them containing images of at most $4$ points in $Z$ (since there are at most $4$ points from $Z$ on any plane). As there are at least $12$ points in $Z$, there must be exactly images of $4$ of them on each of the lines. This implies that $|Z|=12$ so that $b=4$. We observe for later use that the argument provides that
\begin{itemize}
    \item[($*$)] \textit{there is a partition of $Z$ into $3$ sets of $4$ coplanar points}. 
\end{itemize}

In order to conclude the proof we need to show that for any three distinct points in $Z$, call them $P_1,P_2,P_3$, the plane $\langle P_1,P_2,P_3\rangle$ contains a fourth point in $Z$ (any $3$ points in $Z$ span a plane since we assumed that no $3$ are collinear).
	
The following cases may occur:
	\begin{itemize}
		\item[(1)]  $P_1,P_2,P_3\in \alpha$ and we are done. 
		\item[(2)] None of the points $P_1,P_2,P_3$ is in $\alpha$. Now the argument used in order to prove $b=4$ implies, in particular, that there is a fourth point from $Z$ in the plane $\langle P_1,P_2,P_3\rangle$.  
		\item[(3)] Some of the points $P_1,P_2,P_3$ are in $\alpha$, but not all three of them. Suppose that $P_1$ and $P_2$ lie on $\alpha$ and $P_3$ does not. (The case when only one of the three points lies on $\alpha$ is similar.) For convenience let us denote by $Q_1,\dots,Q_9$ the remaining points in $Z$. We want to show that one of these nine points lies on the plane spanned by $P_1, P_2, P_3$. 
	
	By ($*$) we obtain a partition of $Z$ into three sets of four coplanar points (in many ways).
	
	Since $\alpha$ contains four points of $Z$, say $Q_1$ and $Q_2$ are the remaining points of $Z$ on $\alpha$. So $Q_3, \dots, Q_9, P_3$ all lie off of the plane $\alpha$. 
	Using case (2), the plane $\beta$ spanned by $Q_3, Q_4, Q_5$ contains a fourth of these nine points. If it is $P_3$ we make a different choice of the three points. By the partition property ($*$), and renumbering the points if necessary, without loss of generality we may assume that $Q_6$ is the fourth point on $\beta$. Now we replace $\alpha$ by $\beta$, noting that $P_1, P_2, P_3$ all lie away from $\beta$ and we apply case (2) to the points $P_1, P_2, P_3$.
	\end{itemize}
This ends the proof.
\end{proof}
In the sequel we will use arguments involving projections from special points. To this end the following lemma turns out to be useful.
\begin{lemma}\label{series} 
In $\PP^2$ consider a set of $b$ points $W=\{P_1,\dots,P_b\}$ and a line $L$ avoiding these points, i.e., $W\cap L=\emptyset$.
Then the minimal linear subseries of $|\mO_L(b)|$ which contains all general projections of $W$ to $L$ is the complete series $|\mO_L(b)|$ itself.
\end{lemma}
\begin{proof} One can represent $|\mO_L(b)|$ as the projective space $\PP^b$ of forms of degree $b$ on $L$. In this representation, linear subseries correspond to linear subspaces. The statement thus translates to say that the set $T$ of all polynomials of degree $b$ in $\PP^2$
vanishing at $W$ and having a $b$-tuple point (cones), restricted to $L$, determines in $\PP^b$ a set of points not contained in a proper linear subspace. 

Notice that $T$ contains in the closure cones over $W$ with vertices in general points of~$L$. This amounts to saying  that the family of divisors that we consider contains in its closure divisors of type $bQ$ for $Q$ general in $L$. Thus $T$ determines in $\PP^b$ a set whose closure contains the image of the $b$-th Veronese embedding of $L$, which is nondegenerate. The claim follows.
\end{proof}
In the next proposition we turn to $(3,b)$-geproci sets contained in the union of $3$ lines. 
\begin{proposition}[Geproci subsets of $3$ lines]\label{prop:in_3_lines}
   Let $Z$ be a nondegenerate $(3,b)$-geproci set  contained in $3$ lines. Then $Z$ is a $(3,b)$-grid.
\end{proposition}
\begin{proof}
We assume $b\geq 4$ as the case $b=3$ is already discussed in Proposition \ref{SmallProp}.

Let $L_1,L_2,L_3$ be the lines containing $Z$.
Notice to begin with that each line must contain exactly $b$ points of $Z$, since otherwise there exists a line
containing more that $b$ points, which contradicts that a general projection of $Z$
lies in a cubic and in a curve of degree $b$ with no common components.

\medskip

\textbf{Case 1.} We
assume first that some two lines, say $L_1$ and $L_2$ intersect. It is irrelevant if $L_3$ intersects either of these lines or not. Note that it cannot intersect both in distinct points, as then they would be all contained in a plane, hence the set $Z$ would be degenerate. 

Let $\Pi$ be the plane that $L_1$ and $L_2$ span. 
The set $Z$ cannot contain the intersection point $L_1\cap L_2$, as otherwise $L_3$ would contain more than $b$ points of $Z$. The intersection $Z'$ of $Z$ with $L_1\cup L_2$ is a complete intersection of type $(2,b)$, hence it is not separated in degree $b-1$, i.e., $H^1(\Pi, \mathcal I_{Z'}(b-1))\neq 0$,  but it is separated in degree $b$.

Let $\Pi'$ be a general plane through $L_3$, which meets $\Pi$ in a line $L$, and consider the projection of $Z$ from a general point $P$ in $\Pi'$ to $\Pi$. The generality of $L$ implies that $L$ does not meet $Z'$. We have the following exact sequence
\begin{align*}
   0 &\to H^0(\Pi, \mI_{Z'}(b-1)) \to H^0(\Pi, \mI_{Z'}(b)) \to H^0(L, \mO_L(b))\\    
   &\to H^1(\Pi, \mI_{Z'}(b-1)) \to H^1(\Pi, \mI_{Z'}(b)) =0 
\end{align*}  
which implies that the linear series cut on $L$ by curves of degree $b$ through $Z'$ is not complete. By Lemma \ref{series}, for $P$ general in $\Pi'$, a curve of degree $b$ in $\Pi$
through $Z'$ cannot contain the projection of $Z\cap \Pi'=Z\cap L_3$, unless it contains $L$. Thus all plane curves of degree $b$ that contain the projection
of $Z$ contain $L$ and hence also $L_1$ and $L_2$, so they contain the unique cubic that contains the projection of $Z$. Hence we get a contradiction by the geproci assumption. 
\medskip

\textbf{Case 2.} In the second step we assume that the lines $L_1, L_2$ and $L_3$ are pairwise disjoint.

Our strategy in this case is to find a way to distinguish a grid from a nongrid in geometric terms, and then conclude that the geproci property forces $Z$ to be a grid.

Let $C = L_1 \cup L_2 \cup L_3$. Note that $C$ lies on a unique smooth quadric surface $\mathcal{Q}$. Write $Z = Z_1 \cup Z_2 \cup Z_3$, where $Z_i$ is the set of $b$ points on $L_i$. Say $Z_i = \{ P_{i,1}, P_{i,2}, \dots, P_{i,b} \}$ for $i = 1,2,3$.  Consider the plane $H$ spanned by (say) $L_3$ and $P_{2,1}$. $H$ certainly meets $L_1$ in a unique point, say $A$. Note that the point $P_{2,1}$ determines a line, say $M$,  in the opposite ruling on $\mathcal{Q}$ from $L_2$, and this line certainly meets $L_3$ and $L_1$. The line $M$ lies on $H$ (it contains two points from $H$, namely $P_{2,1}$ and one point $B=M\cap L_3$), so $M$ contains the point $A$ and it also contains the point $B$. And the issue is whether $A$ and $B$ are points of $Z$. 
By varying the choices of lines and points, we make the following conclusion:

\begin{quotation}
{\it $Z$ is a grid\index{grid} if and only if for any choice of one of the lines, say $L_3$, and for any choice of a point from a second line, say $P_{2,1}$, the line and the point span a plane $H$ that intersects the remaining line in a point of $Z$.
}
\end{quotation}

We also get that  $H$ contains $b+2$ points of $Z$ if and only if the unique trisecant of $C$ that contains $P_{2,1}$ also contains a point of $Z_1$, say $P_{1,1}$, if and only if $Z$ is a grid.

Now assume for a contradiction that $Z$ is not a grid. Choose a line and a point as above (say $L_3$ and $P_{2,1}$) so that the plane $H$ that they span contains only $b+1$ points of $Z$ (and not $b+2$). Let $Z'$ be the remaining $2b-1$ points of $Z$, so $Z'$ consists of the $b$ points on 
$L_1$ and the $b-1$ points on $L_2$ other than $P_{2,1}$. 
Let $\pi$ be the projection from a general point of $H$ to a general plane $\Pi$. Notice that $\pi(Z)$ consists of $b+1$ points on $\pi(L_3)$ and
$b$ each on $\pi(L_1)$ and $\pi(L_2)$
(with $\pi(P_{2,1})$ on both $\pi(L_2)$ and $\pi(L_3)$). Using B{\'e}zout's Theorem successively on these sets of collinear points, we see that any curve of degree $b$ containing $\pi(Z)$ must contain these three lines as components. 
But this is impossible:  by semicontinuity and the geproci property, there must be at least one curve of degree $b$ which does not contain the curve defined by the cubic generator of the ideal. Therefore we have shown that  $Z$ must be a grid.
\end{proof}
Now we are ready for another reduction step.
\begin{theorem}\label{thm:contains_3_3} 
Let $Z$ be a nondegenerate $(3,b)$-geproci set, with $b\geq 3$. Then $Z$  contains a $(3,3)$-grid. Moreover, if $Z$ has 4 points on a line then $Z$ is a $(3,b)$-grid.
\end{theorem}
\begin{proof} Let $L$ be a line containing the maximal number of points in $Z$. Let $Z'=L\cap Z$. By Theorem \ref{thm:LGP} and Proposition \ref{3alin} there are at least $3$ points in $Z'$.

If there are $4$ or more points in $Z'$, then in a general projection
$\pi$, the cubic containing $\pi(Z)$ splits in a conic and the line $\pi(L)$. Since
$\pi(Z)$ is complete intersection of type $(3,b)$ the line $\pi(L)$ contains $b$ points of 
$\pi(Z)$. Since $\pi$ is general, it must be that already $L$ contains $b$ points of $Z$. The set $Z'$
is certainly a $(1,b)$ geproci and, by construction, $\pi(Z')$ shares with $\pi(Z)$ a generator of degree $b$. By Lemma \ref{lem subset}, the residual set $Z''=Z\setminus Z'$ is a $(2,b)$-geproci, hence it is a grid, by Proposition \ref{SmallProp}. Then $Z$ lies in three lines with
$b$ points on each line, thus it is a $(3,b)$-grid by Proposition \ref{prop:in_3_lines} and we are done in this case.
\medskip

In the remaining case assume that $L$ meets $Z$ in exactly $3$ points. Applying Lemma \ref{lem subset} as above, we get that the residue $Z''=Z\setminus Z'$ is a $(3,b-1)$-geproci. If $b=4$ then $Z''$ is a $(3,3)$-grid by \cite[Theorem 5.12]{CM} and we are done. Then we conclude by induction on $b$.
\end{proof}
Now we are in a position to prove Theorem \ref{thm:classification_of_3xb}, i.e., the classification of nondegenerate $(3,b)$-geproci sets $Z$,  one of the main results of the chapter.
\begin{proof}[Proof of Theorem \ref{thm:classification_of_3xb}.]
By Theorem \ref{thm:contains_3_3} we may assume that $Z$ contains a $(3,3)$-grid $G$. As discussed in Section \ref{sec:Segre_Embeddings}, there is only one $(3,3)$-grid. So we can assume that the points in $G$ have specific coordinates, which allows for a computational proof.
Specifically, we assume that $G$ consists of the following points:
$$1000,
0100,
0010,
0001,
1111,
1100,
0011,
1010,
0101,$$
shown in Figure \ref{GridOnTetrahedron}.
This grid lies on a unique quadric $\mathcal Q$ given by the equation $yz-xw=0$.
The three grid lines in one ruling are $y=w=0$, $x-y=z-w=0$ and $x=z=0$.
The three grid lines in the other ruling are $x=y=0$, $x-z=y-w=0$ and $z=w=0$.

Given a point $Q\in\PP^3$ not in $G$,  there is a unique cubic cone $C_{P}(Q)$ with vertex at a general point $P\in\PP^3$ containing $G$ and $Q$. Call 
$X_Q=\bigcap_P C_P(Q)$ the intersection of all such cones, where $P$ ranges in a suitable nonempty open subset of $\PP^3$. 

We are interested in $X_Q$ because if there is a $(3,b)$-geproci set $Z$ containing $Q$ and $G$,
then 
$$\{Q\}\cup G\subseteq Z\subseteq X_Q.$$
We will compute $X_Q$ as a function of $Q$. 
Indeed, consider $$\Gamma=\left\{(Q,Q') \ :\ Q' \in  X_Q \right\}\subset \PP^3\times \PP^3,$$
so $X_Q$ is the fiber over $Q$ in the first projection of $\Gamma.$

We can write the equation of $C_P(Q)$ as a homogeneous polynomial $F$ of degree $3$ in $x,y,z,w$
with coefficients depending polynomially on the coordinates of the points $Q=[A:B:C:D]$ and $P=[a:b:c:d]$. Regarding $F$ as a polynomial in the variables $a,b,c,d$ and with coefficients in ${\field}[A,B,C,D][x,y,z,w]$, 
 the set $\Gamma$ corresponds to the ideal $I\subset {\field}[A,B,C,D][x,y,z,w]$ 
generated by these coefficients. 

We are interested in when $|X_Q|>10$; if we colon out the ideals of $Q$ and of $G$ from 
$I\subset {\field}[A,B,C,D][x,y,z,w]$
to get an ideal $J$, those $Q$ with $|X_Q|>10$ are in the complement of the locus
where $J$ has empty zero locus. Thus all such $Q$ are contained in the zero locus
of the ideal $T=J\cap {\field}[A,B,C,D]$. To find the zero locus of $T$,
we might as well replace $T$ by its radical and then compute its primary decomposition.

The primary components of $T$ are  prime  and define
the 6 grid lines, and the following 26 points (where $*$ represents $-1$):
$$1001,
*001,
101*,
1021,
1011,
*011,
110{*},
1201,
1101,
*101,
1110,
2110,
0110,
1{*}10,
11{*}0,
0{*}10,$$$$
01{*}1,
1121,
2111,
0111,
1211,
0{*}11,
1112,
0112,
[1:1-t:t:1],
[t-1:t:t:1],$$ 
where $t^2-t+1=0$.
Note that only the last two points are on $\mathcal Q$, but they are not on a grid line.
Thus, if $Z$ is $(3,b)$-geproci, $Q$ must be one of the above points or on a grid line.
If $Q$ is on a grid line, then we have 4 points in a line and we are done by Theorem \ref{thm:contains_3_3}.

If $Q$ is one of the $26$ listed points, we can compute $X_Q$ explicitly. We find that then the number of points in $X_Q$
is either $11$ or $12$ for each of them (or $10$ in the case of the last two points). When it is less than $12$, clearly $Z$ cannot be $(3,b)$-geproci.
The points for which it is $12$ come in two sets of three collinear points, namely
$$0111, 100{*} \mbox{ and } 1110,\;\mbox{ and }\;
1011, 1101 \mbox{ and } 0{*}10.$$
These are the $6$ Brianchon points\index{Theorem! Brianchon}; each set of 3 together with the $(3,3)$-grid gives a $(3,4)$-geproci set
projectively equivalent to the $D_4$ configuration.
\end{proof}
This result completes the numerical classification of pairs $(a,b)$ for which there exists a nontrivial $(a,b)$-geproci set.

\chapter{Nonstandard geproci sets of points in projective space}\label{chap.extending}
Chapter \ref{chap.Geography} addresses the question of the existence of nontrivial $(a,b)$-geproci sets in $\PP^3$ for all the possible values of $a$ and $b$.
We classified the $(a,b)$-geproci sets with $3=a\le b$ (see Theorem \ref{thm:classification_of_3xb}),  and we constructed a nontrivial $(a,b)$-geproci set for any $4\le a\le b$ (see Theorem \ref{t. (a,b)-geproci}).


In this chapter we turn our attention to a different construction of $(a,b)$-geproci sets, for all $b\ge a\ge 4$, which are {combinatorially} different from those coming from the standard construction.


In the last part of the chapter we will relate our constructions to three isolated examples of geproci sets, which are different from those obtained above. One is a configuration of points due to Klein, first shown to be geproci in \cite{PSS}; another due to Penrose, which we will prove to be a geproci set;
and the configuration of points arising from the $H_4$ root system, which is proved to be geproci in \cite{FZ}.

{ The set $H_4$ has the interesting property of being a geproci set which is not a half grid. We will show that the Penrose is also not a half grid. The Penrose example is also remarkable for being the only nondegenerate arithmetically Gorenstein geproci set known up to now.}

\section{Beyond the standard construction}\label{s.extending geproci}
In this section we construct an infinite family of new geproci sets in $\PP^3$ by extending those in the standard construction.
Namely, Proposition \ref{p.add points to geproci} and Proposition \ref{p.add points to geproci2} describe how to get an $(a,a)$-geproci set by adding points to a set as in the hypothesis of Theorem \ref{t. geproci infinite class}.
We encode such a hypothesis in the following notation.

\begin{definition}[Standard construction hypothesis]\label{n.standard}
An $(a,a+1)$-geproci set $Z$, $a\ge 3$ (see Figure \ref{Fig: ExtStConst n=3, i}),  satisfies the {\it standard construction hypothesis 1} (SCH1) 
\index{SCH}\index{standard construction! hypothesis}
if there exists
a decomposition $Z=X\cup Y_1\subseteq \PP^3$ such that  $X$ is an $(a,a)$-grid on a quadric $\calq$ and $Y_1$ is a set of $a$ collinear points on a line $\ell_1 \not\subseteq \calq$, where $\ell_1$ does not
meet any of the grid lines of $X$.
(This implies that $X$ and $Y_1$ are disjoint.)

An $(a,a+2)$-geproci set $Z$ (see Figure \ref{Fig: ExtStConst n=3, ii}) satisfies the {\it standard construction hypothesis 2} (SCH2) if there exists a decomposition $Z=X\cup Y_1\cup Y_2$ such that both $X\cup Y_1$ and $X\cup Y_2$ satisfy SCH1 and the lines $\ell_1$, $\ell_2$, spanned by $Y_1$ and $Y_2$ resp., meet the quadric $\calq$ in four points which are a $(2,2)$-grid on~$\calq$.

We will say that a set $Z$ satisfies the {\it standard construction hypothesis} (SCH) if it satisfies either SCH1 or SCH2.  
\end{definition}

\begin{example}
The sets $Z=X\cup Y_i$, $i=1,2$ of 
Theorem \ref{t. geproci infinite class}(a)
satisfy SCH1, while the set 
$Z=X\cup Y_1\cup Y_2$ of 
Theorem \ref{t. geproci infinite class}(b)
satisfies SCH2, with the points of $\calq\cap \ell_1$ being $Q=[1:0:0:0]$ and $Q'=[0:0:0:1]$
and the points $\calq\cap \ell_2$ being
$P=[0:1:0:0]$ and $P'=[0:0:1:0]$. (Note that the line $s_1$ through $Q$ and $P$ (as shown in Figure \ref{Fig: ExtStConst n=3, ii}) is a line in the same ruling of $\calq$ as are the lines
$M_i$ of section \ref{sec:standard construction}. The point $Q_{1i}$ in Proposition \ref{p.add points to geproci2} 
of intersection of $L_i$ 
of section \ref{sec:standard construction} with $s_1$
is $[1:u^i:0:0]$, while the point $Q_{2i}$ of intersection of $L_i$ with the line $s_2$
through $P'$ and $Q'$ is $[0:0:1:u^i]$.)
\end{example}

\begin{remark}
   Note that the decomposition in Definition \ref{n.standard} may not be unique. For example $D_4$ satisfies SCH1 but decomposes in a $(3,3)$-grid and $3$ extra collinear points in 16 different ways, see Proposition \ref{D4factsProp}.
\end{remark}

\begin{remark}Below we show how to add points to an SCH set $Z$ to obtain a larger geproci set. This is useful since there are many SCH sets $Z$. Indeed, the sets of points constructed in Theorem \ref{t. geproci infinite class} satisfy SCH. In particular, the lines containing $Y_1$ and $Y_2$ intersect the quadric defined by $xw-yz=0$ in the coordinate points, which determine a $(2,2)$-grid on this quadric.\\
We currently don't know if there exist sets, up to projective equivalence, satisfying SCH other than those coming from Theorem \ref{t. geproci infinite class}.
\end{remark}

\begin{lemma}\label{l. no linear components}
Let $Z=X\cup Y_1\subseteq \PP^3$ be an $(a,a+1)$-geproci set satisfying the SCH. Then for a general point $P$, the unique cone of degree $a$ containing $Z$ with vertex at $P,$ has no linear components. 
\end{lemma}
\begin{proof}Let $\pi$ be the projection from $P$ to a general plane.
Let $\gamma_a$ be the unique plane curve of degree $a$ containing $\pi(Z)$. 
Assume that $\gamma_a$ contains a line $\ell$.

By liaison, see  \cite[Corollary 5.2.19]{MiglioreBook}, $Z\setminus \ell$ is contained in a curve of degree $a-1$ if and only if $\ell$ contains $a+1$ points of $Z$. This is impossible since, by the SCH, $Z$ has no collinear subsets of $a+1$ points.
\end{proof}

\begin{proposition}\label{p.add points to geproci}\index{geproci! extended construction}
	Let $Z=X\cup Y_1\subseteq \PP^3$ be an $(a,a+1)$-geproci set satisfying the SCH. Take  points $Q, Q_1, \ldots, Q_a$ such that
\begin{enumerate}
\item[(i)]  $Q$ is a point of intersection of $\ell_1$ and $\calq$;
\item[(ii)]  $Q, Q_1,\ldots, Q_a$ are on a line of the quadric $\calq$;
\item[(iii)] $X\cup \{ Q_1,\ldots, Q_a \}$ is an $(a,a+1)$-grid.
\end{enumerate}
	
	Then, the set 
	 $\widetilde Z=Z\cup \{Q, Q_1, \ldots, Q_a\}$ is an $(a+1,a+1)$-geproci set. 
\end{proposition}
\begin{proof}
Let $P$ be a general point.  Since $Z$ is an $(a,a+1)$-geproci set there is a unique cone $F_a$ of degree $a$ containing $Z$, with vertex at $P$.  

Let $r_1,\ldots, r_a$ be the grid lines in $\calq$ whose union contains the grid $X\cup \{Q_1,\ldots, Q_a \}$.     
Let $G_{a+1}$ be the cone of degree $a+1$ given by the union of the planes defined by $P$ and the lines  $\ell_1, r_1,\ldots, r_a$.
The cone $G_{a+1}$ also contains $Z$.

Let $s_Q$ be the line on the quadric containing the points $Q, Q_1,\ldots, Q_a$ and  let $H$ be the plane through $P$ and the line $s_Q$. Then the set $\widetilde Z$  is contained in the pencil of cones of degree $a+1$ with vertex at $P$  defined by $G_{a+1}$ and $F_a\cup H$. Since $F_a$ has no linear components by Lemma \ref{l. no linear components}, then one sees that such a pencil has no fixed components. Thus $\widetilde Z$ is an $(a+1,a+1)$-geproci.
\end{proof}

\begin{figure}
\definecolor{ffffff}{rgb}{1.,1.,1.}
\definecolor{cqcqcq}{rgb}{0.7529411764705882,0.7529411764705882,0.7529411764705882}
\begin{tikzpicture}[line cap=round,line join=round,>=triangle 45,x=1.25cm,y=1.25cm]
\clip(-1.4919773485738415,0) rectangle (7.765906454712707,5);
\draw [line width=1.pt] (1.,0.5)-- (1.,4.5);
\draw [line width=1.pt] (2.,0.5)-- (2.,4.5);
\draw [line width=1.pt] (3.,4.5)-- (3.,0.5);
\draw [line width=1.pt] (3.5,1.)-- (0.5,1.);
\draw [line width=1.pt] (0.5,2.)-- (3.5,2.);
\draw [line width=1.pt] (3.5,3.)-- (0.5,3.);
\draw [line width=1.pt] (0.5,4.)-- (4.5,4.);
\draw [line width=1.pt,dashed] (3.6,4.4)-- (6.4,1.6);
\draw (0.8384554708741516,0.5) node[anchor=north west] {$r_1$};
\draw (1.844053194334587,0.5) node[anchor=north west] {$r_2$};
\draw (2.8496509177950227,0.5) node[anchor=north west] {$r_3$};
\draw (-.1,1.3) node[anchor=north west] {$M_0$};
\draw (-.1,2.3) node[anchor=north west] {$M_1$};
\draw (-.1,3.3) node[anchor=north west] {$M_2$};
\draw (0,4.2) node[anchor=north west] {$s_Q$};
\draw (1,4) node[anchor=north west] {$Q_1$};
\draw (2,4) node[anchor=north west] {$Q_2$};
\draw (3,4) node[anchor=north west] {$Q_3$};
\draw (3.7,4) node[anchor=north west] {$Q$};
\draw (6.297414541087944,1.6) node[anchor=north west] {$\ell_1$};
\begin{scriptsize}
\draw [fill=cqcqcq] (1.,3.) circle (3pt);
\draw [fill=cqcqcq] (1.,2.) circle (3pt);
\draw [fill=cqcqcq] (1.,1.) circle (3pt);
\draw [fill=cqcqcq] (2.,1.) circle (3pt);
\draw [fill=cqcqcq] (2.,2.) circle (3pt);
\draw [fill=cqcqcq] (2.,3.) circle (3pt);
\draw [fill=cqcqcq] (3.,3.) circle (3pt);
\draw [fill=cqcqcq] (3.,2.) circle (3pt);
\draw [fill=cqcqcq] (3.,1.) circle (3pt);
\draw [fill=ffffff] (1.,4.) circle (3pt);
\draw [fill=ffffff] (2.,4.) circle (3pt);
\draw [fill=ffffff] (3.,4.) circle (3pt);
\draw [fill=white] (4.,4.) circle (3pt);
\draw [fill=black] (4.,4.) circle (1pt);
\draw [fill=black] (4.5693700325662485,3.430629967433751) circle (3pt);
\draw [fill=black] (5.255730383499562,2.7442696165004383) circle (3pt);
\draw [fill=black] (5.862281391301094,2.1377186086989055) circle (3pt);
\end{scriptsize}
\end{tikzpicture}
\caption[The extended standard construction from Proposition \ref{p.add points to geproci}.]{The extended standard construction from Proposition \ref{p.add points to geproci} for $n=3$ with the $(n,n+1)$-grid (the gray and white dots) and
the $(n,n+1)$-geproci set $Z$ from Theorem \ref{t. geproci infinite class}(i) (the gray and black dots).
The solid lines are lines on the quadric containing the grid.}\label{Fig: ExtStConst n=3, i}
\end{figure}

For a geproci set constructed using Proposition \ref{p.add points to geproci}, see Remark \ref{R. Prop 5.4 applied to D4}
(see also Example \ref{ex.(4,4)-geproci not real}).

Analogously, we have the following result.
\begin{proposition}\label{p.add points to geproci2}\index{geproci! extended construction}
	Let $Z=X\cup Y_1\cup Y_2\subseteq \PP^3$ be an $(a,a+2)$-geproci set satisfying the SCH.
Take the points $P,Q, Q_{11}, \ldots, Q_{1a}$ and $P',Q', Q_{21}, \ldots, Q_{2a}$ such that
\begin{enumerate}
\item[(i)] $P,Q,P',Q'$  are the points of intersection of $\ell_1\cup \ell_2$ and $\calq$;
\item[(ii)] $P,Q, Q_{11}, \ldots, Q_{1a}$ are on a line $s_1$ and $P',Q', Q_{21}, \ldots, Q_{2a}$ are on a line $s_2$, moreover $s_1$ and $s_2$ are grid lines on the quadric $\calq$ and they belong to the same ruling;
\item[(iii)] $X\cup \{  Q_{11}, \ldots, Q_{1a}\}\cup \{  Q_{21}, \ldots, Q_{2a}\} $ is an $(a,a+2)$-grid.
\end{enumerate}
Then the set $\widetilde Z=Z\cup \{Q,P, Q_{11}, \ldots, Q_{1a}\}\cup \{ Q',P', Q_{21}, \ldots, Q_{2a}\}$ is an $(a+2,a+2)$-geproci set.  
\end{proposition}
\begin{proof}    The proof of this result follows similarly as in Proposition \ref{p.add points to geproci}. 

Let $P$ be a general point.  Since $Z$ is an $(a,a+2)$-geproci set there is a unique cone $F_a$ of degree $a$ containing $Z$, with vertex at $P$. Note that $X\cup Y_1$ is an $(a,a+1)$-geproci set and $F_a$ is also the unique cone of degree $a$ containing $X\cup Y_1$, with vertex at $P$. Hence $F_a$ has no linear components by Lemma \ref{l. no linear components}.

Let $r_1,\ldots, r_a$ be the grid lines in $\calq$ whose union contains the grid $X\cup \{  Q_{11}, \ldots, Q_{1a}\}\cup \{  Q_{21}, \ldots, Q_{2a}\}$.  Let $G_{a+2}$ be the cone of degree $a+2$ given by the union of the planes defined by $P$ and the lines  $\ell_1,\ell_2, r_1,\ldots, r_a$. The cone $G_{a+2}$ contains $Z$.

Let $H_1,H_2$ be the planes spanned by $P$ and the lines $s_1,s_2$. Then, the set $\widetilde Z$  is contained in the pencil of cones of degree $a+2$ with vertex at  $P$  defined by $G_{a+2}$ and $F_a\cup H_1\cup H_2$. Since such a pencil has no fixed components, thus $\widetilde Z$ is an $(a+2,a+2)$-geproci set. 
\end{proof}

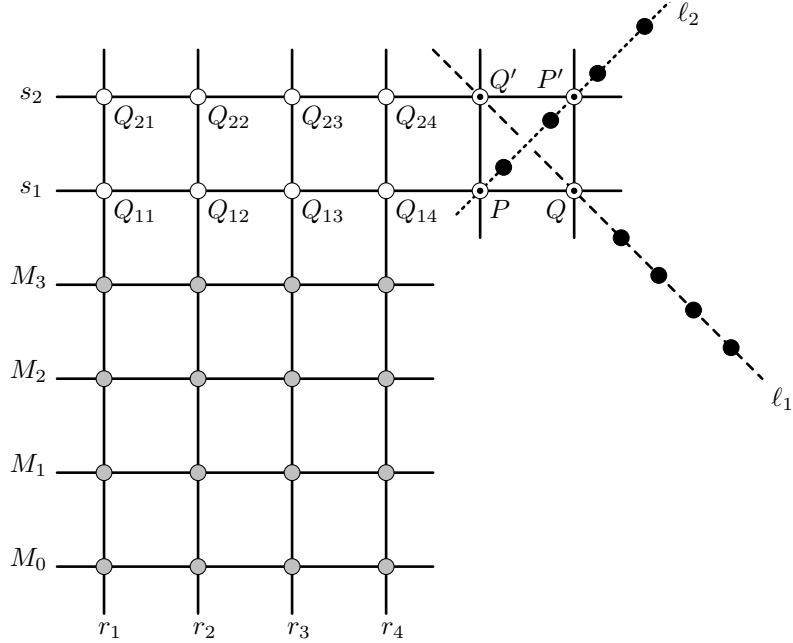
\begin{figure}
\definecolor{ffffff}{rgb}{1.,1.,1.}
\definecolor{cqcqcq}{rgb}{0.7529411764705882,0.7529411764705882,0.7529411764705882}
\begin{tikzpicture}[line cap=round,line join=round,>=triangle 45,x=1.25cm,y=1.25cm]
\clip(-1.4919773485738415,0) rectangle (10,7);
\draw [line width=1.pt] (1.,0.5)-- (1.,6.5);
\draw [line width=1.pt] (2.,0.5)-- (2.,6.5);
\draw [line width=1.pt] (3.,6.5)-- (3.,0.5);
\draw [line width=1.pt] (4.,6.5)-- (4.,0.5);
\draw [line width=1.pt] (4.5,1.)-- (0.5,1.);
\draw [line width=1.pt] (0.5,2.)-- (4.5,2.);
\draw [line width=1.pt] (4.5,3.)-- (0.5,3.);
\draw [line width=1.pt] (0.5,4.)-- (4.5,4.);
\draw [line width=1.pt] (0.5,5.)-- (6.5,5.);
\draw [line width=1.pt] (0.5,6.)-- (6.5,6.);
\draw [line width=1.pt, dashed] (4.5,6.5)-- (8,3);
\draw [line width=3.pt, color=white] (5.45,5.55)-- (5.55,5.45);
\draw [line width=1.pt,dotted] (4.75,4.75)-- (9.5,9.5);
\draw (0.8384554708741516,0.5) node[anchor=north west] {$r_1$};
\draw (1.844053194334587,0.5) node[anchor=north west] {$r_2$};
\draw (2.8496509177950227,0.5) node[anchor=north west] {$r_3$};
\draw (3.8496509177950227,0.5) node[anchor=north west] {$r_4$};
\draw (-.1,1.3) node[anchor=north west] {$M_0$};
\draw (-.1,2.3) node[anchor=north west] {$M_1$};
\draw (-.1,3.3) node[anchor=north west] {$M_2$};
\draw (-0.1,4.3) node[anchor=north west] {$M_3$};
\draw (0,5.2) node[anchor=north west] {$s_1$};
\draw (0,6.2) node[anchor=north west] {$s_2$};
\draw (1,5) node[anchor=north west] {$Q_{11}$};
\draw (2,5) node[anchor=north west] {$Q_{12}$};
\draw (3,5) node[anchor=north west] {$Q_{13}$};
\draw (4,5) node[anchor=north west] {$Q_{14}$};
\draw (1,6) node[anchor=north west] {$Q_{21}$};
\draw (2,6) node[anchor=north west] {$Q_{22}$};
\draw (3,6) node[anchor=north west] {$Q_{23}$};
\draw (4,6) node[anchor=north west] {$Q_{24}$};
\draw (5.6,5) node[anchor=north west] {$Q$};
\draw (5.5,6.4) node[anchor=north west] {$P'$};
\draw (5,5) node[anchor=north west] {$P$};
\draw (5,6.4) node[anchor=north west] {$Q'$};
\draw (8,3) node[anchor=north west] {$\ell_1$};
\draw (7,7.1) node[anchor=north west] {$\ell_2$};
\begin{scriptsize}
\draw [fill=cqcqcq] (1.,3.) circle (3pt);
\draw [fill=cqcqcq] (1.,2.) circle (3pt);
\draw [fill=cqcqcq] (1.,1.) circle (3pt);
\draw [fill=cqcqcq] (2.,1.) circle (3pt);
\draw [fill=cqcqcq] (2.,2.) circle (3pt);
\draw [fill=cqcqcq] (2.,3.) circle (3pt);
\draw [fill=cqcqcq] (3.,3.) circle (3pt);
\draw [fill=cqcqcq] (4.,3.) circle (3pt);
\draw [fill=cqcqcq] (3.,2.) circle (3pt);
\draw [fill=cqcqcq] (4.,2.) circle (3pt);
\draw [fill=cqcqcq] (3.,1.) circle (3pt);
\draw [fill=cqcqcq] (4.,1.) circle (3pt);
\draw [fill=cqcqcq] (1.,4.) circle (3pt);
\draw [fill=cqcqcq] (2.,4.) circle (3pt);
\draw [fill=cqcqcq] (3.,4.) circle (3pt);
\draw [fill=cqcqcq] (4.,4.) circle (3pt);
\draw [fill=ffffff] (1.,5.) circle (3pt);
\draw [fill=ffffff] (2.,5.) circle (3pt);
\draw [fill=ffffff] (3.,5.) circle (3pt);
\draw [fill=ffffff] (4.,5.) circle (3pt);
\draw [fill=ffffff] (1.,6.) circle (3pt);
\draw [fill=ffffff] (2.,6.) circle (3pt);
\draw [fill=ffffff] (3.,6.) circle (3pt);
\draw [fill=ffffff] (4.,6.) circle (3pt);
\draw [line width=1.pt] (6,4.5)-- (6,6.5);
\draw [line width=1.pt] (5,4.5)-- (5,6.5);
\draw [fill=white] (6.,5.) circle (3pt);
\draw [fill=black] (6.,5.) circle (1pt);
\draw [fill=white] (5.,6.) circle (3pt);
\draw [fill=black] (5.,6.) circle (1pt);
\draw [fill=white] (5.,5.) circle (3pt);
\draw [fill=black] (5.,5.) circle (1pt);
\draw [fill=white] (6.,6.) circle (3pt);
\draw [fill=black] (6.,6.) circle (1pt);
\draw [fill=black] (6.5,4.5) circle (3pt);
\draw [fill=black] (6.9,4.1) circle (3pt);
\draw [fill=black] (7.2693700325662485,3.730629967433751) circle (3pt);
\draw [fill=black] (7.6693700325662485,3.330629967433751) circle (3pt);
\draw [fill=black] (5.25,5.25) circle (3pt);
\draw [fill=black] (5.75,5.75) circle (3pt);
\draw [fill=black] (6.25,6.25) circle (3pt);
\draw [fill=black] (6.75,6.75) circle (3pt);
\end{scriptsize}
\end{tikzpicture}
\caption[The extended standard construction from Proposition \ref{p.add points to geproci2}.]{The extended standard construction from Proposition \ref{p.add points to geproci2} for $n=4$ with the $(n,n+2)$-grid (the gray and white dots) and
the $(n,n+2)$-geproci set $Z$ from Theorem \ref{t. geproci infinite class}(ii) (the gray and black dots).
The solid lines are lines on the quadric containing the grid.}\label{Fig: ExtStConst n=3, ii}
\end{figure}

For a geproci set constructed using Proposition \ref{p.add points to geproci2}, see Example \ref{ex.from F4 to Klein}.

\begin{remark}\label{r. subsets of extended geproci} Note that the procedure to construct new geproci sets from those in the standard construction, described in the proof of Theorem \ref{t. (a,b)-geproci}, also works for the sets~$\widetilde Z$. 

Indeed, fix integers $a,b$, where $4 \leq a \leq b$, we can construct an $(a,b)$-geproci set that is nontrivial by removing maximal sets of collinear points from a suitable extended geproci set~$\widetilde Z$. 

Precisely, starting from a $(b,b)$-geproci set $\widetilde Z$ as constructed in Proposition \ref{p.add points to geproci} and, using the notation as in its proof, we remove from $\widetilde Z$ the set $Y$ of $b(b-a)$ points in a selection (say $r_1, \ldots, r_{b-a}$) of $b-a$ of the grid lines $r_1, \ldots, r_{b-1}$. So, the remaining $a-1$ grid lines, in this case $r_{b-a+1},\ldots, r_{b-1}$, still determine the same quadric $\calq$.  By Lemma \ref{lem subset}, the set $\widetilde{Z}\setminus Y$ is an $(a,b)$-geproci set.  

The same considerations apply for the construction given in Proposition \ref{p.add points to geproci2}.
\end{remark}

\begin{remark}\label{r.substantially different}
The sets constructed in Remark \ref{r. subsets of extended geproci} are substantially different from the sets in the standard construction and their subsets built in Theorem \ref{t. (a,b)-geproci}.

Indeed, if $\widetilde{Z}$ is a $(b,b)$-geproci set, as constructed in Proposition \ref{p.add points to geproci}, then $\widetilde{Z}$ has $(b+1)$ lines containing exactly $b$ points.
Precisely, these lines are $r_1,\ldots, r_{b-1}$,  $s_Q$ and $\ell_1$.
On the contrary, a $(b,b)$-geproci set from the standard construction has only $b$ lines containing exactly $b$ points.

Moreover, for $a<b$, and for an $(a,b)$-geproci subset of an extended $(b,b)$-geproci set $\widetilde{Z}$ the lines  $s_Q$ and $\ell_1$ contain exactly $a$ points and exactly $b$ points respectively.
The $(a,b)$-geproci subsets of those in the standard construction have either no lines containing exactly $a$ points or no lines containing exactly $b$ points.
\end{remark}

In Chapter \ref{chap. non iso and real} we will discuss in more detail the notion of {\it substantially different} $(a,b)$-geproci sets, giving precise definitions.

\section{The Klein configuration}
\label{sec: The Klein configuration}
We will see in this section that it is possible to extend the $(6,6)$-geproci set $\widetilde{Z}$ in order to get a $(6,10)$-geproci set  called the {\it Klein configuration} (see \cite[Theorem 5.5]{PSS}).\index{configuration! Klein}

Theorem \ref{t. geproci  infinite class} shows that it is possible to extend a very special class of $(n,n)$-grids to a new class of geproci sets.
In particular, the proof of this theorem makes use of the pencil of cones of degree $n$ containing an $(n,n)$-grid, with vertex at a general point $P$, to construct an element of the pencil which contains extra points not depending on $P$.

From this point of view, one could try to repeat the same procedure by replacing the $(n,n)$-grid with an $(n,n)$-geproci set, for instance one of those constructed in the previous section.
In the next example,  we use this procedure to show that the Klein configuration is indeed the union of $\widetilde{Z}$, where $Z$ is the $(4,6)$-geproci set coming from Theorem \ref{t. geproci infinite class}, and a $(4,6)$-grid. 
\begin{example}\label{ex.from F4 to Klein}
For $n=4$, the $(4,6)$-geproci set coming from the standard construction, see  Theorem~\ref{t. geproci infinite class}, is as follows 
\[Z= \ \
\begin{array}{ccccccccccc}
    &      L_1    &    L_2     &    L_3        &     L_4        \\
M_1:& \ [1:1,1,1] & [1:i:1:i] & [1: -1: 1: -1]  & [1: -i: 1: -i] \\ 
M_2:& \  [1: 1: i:i] &  [1:i: i: -1] & [1:-1:i:-i]  & [1:-i:i: 1] \\
M_3:&  \ [1:1:-1:-1] & [1:i:-1:-i] & [1:-1:-1:1] &  [1:-i: -1: i] \\
M_4:& \  [1: 1: -i: -i] & [1: i: -i: 1]& [1: -1: -i: i]& [1: -i: -i: -1]\\
\\
\ell_1:& \ [1:0:0:1]& [1:0:0:i]&[1:0:0:-1]&[1:0:0:-i]\\
\ell_2:& \ [0:1:1:0]& [0:1:i:0]&[0:1:-1:0]&[0:1:-i:0].\\
\end{array}
\]
The lines $\ell_1$ and $\ell_2$ meet the quadric $\calq$, defined by the form $xw-yz=0$, at the coordinate points. These are a $(2,2)$-grid on $\calq$. Thus, we apply Proposition \ref{p.add points to geproci2} to extend the set $Z$
to a $(6,6)$-geproci set $\widetilde{Z}.$ 
{\footnotesize	\[
\begin{array}{ccccccccccc}
    &      L_1    &    L_2     &    L_3        &     L_4    &    L_5       &     L_6        \\
M_1:&  [1:1:1:1] & [1:i:1:i] & [1: -1: 1:-1]  & [1: -i: 1: -i] & [1:0:1:0] & [0:1:0:1] \\ 
M_2:&  [1: 1: i:i] &  [1:i: i: -1] & [1:-1:i:-i]  & [1:-i: i: 1]  & [1:0:i:0] & [0:1:0:i] \\
M_3:&  [1:1:-1:-1] & [1:i:-1:-i] & [1:-1:-1:1] &  [1:-i:-1: i]  & [1:0:-1:0] & [0:1:0:-1] \\
M_4:&  [1: 1: -i: -i] & [1: i: -i: 1]& [1: -1: -i: i]& [1: -i: -i: -1] & [1:0:-i:0] & [0:1:0:-i] \\
\\
\ell_1:& \ [1:0:0:1]& [1:0,0,i]&[1:0:0:-1]&[1:0:0:-i] &  [1:0:0:0] & [0:0:0:1]\\
\ell_2:& \ [0:1:1:0]& [0:1:i:0]&[0:1:-1:0]&[0:1:-i:0]  & [0:0:1:0] &  [0:1:0:0] \\
\end{array}
\]}

Using the notation in \cite{PSS},  the Klein configuration is the union of $\widetilde{Z}$ and  the following $(4,6)$-grid on the quadric defined by $xw+yz=0$:
\[
\begin{array}{llll}
\ [1:  i: -1:  i],& [1: - i: -1: - i],&  [1: 1: -1: 1],&  [1: -1: -1: -1],\\
\ [1:  i: 1: - i],& [1: - i: 1:  i],&    [1: 1: 1: -1],&  [1: -1, 1: 1],  \\
\ [1:  i:  i:  1],&  [1: - i:  i: -1],&  [1: 1:  i: - i],&  [1: -1:  i:  i],\\ 
\ [1:  i:- i: -1],& [1: - i: - i: 1],&  [1: 1: - i:  i],& [1: -1: - i: - i],\\
\ [1: i: 0: 0 ], & [1: -i: 0: 0 ], &  [1: 1: 0: 0],& [1: -1: 0: 0],  \\
\ [0: 0: 1: -i ],& [0: 0: 1: i ],& [0: 0: 1: -1],& [0: 0: 1: 1].  \\
\end{array}
\]
Note that only the eight points in the last two rows are also on $\calq$, precisely they are on the two grid lines determined by the $(2,2)$-grid, different from $L_5$ and $L_6$. Thus, all the points on $\calq$ form a $(6,6)$-grid. 

Hence, this construction is similar to the one described in Theorem \ref{t. geproci infinite class}. In the pencil of cones of degree 6 with vertex at a general point $P$ containing $\widetilde{Z}$, there is one special element containing all the remaining 24 points.  
\end{example}

\begin{question}\label{q. Klein-like}
Are there other examples like the Klein configuration? That means, are there standard geproci sets $Z$, different from $F_4$, for which the set $\widetilde{Z}$ is an $(a,a)$-geproci which can be further extended to a larger $(a,b)$-geproci set? I.e., is there a special element in the pencil of cones of degree $a$ through $\widetilde{Z}$ with vertex at a general point $P$ that contains a larger configuration of points that is geproci? 
\end{question}

The question has a positive answer in the case $a=4$, i.e., $Z=D_4$,  as we will show in Section \ref{s.Penrose}.

\begin{remark}
Arguing as in the proof of Theorem \ref{t. geproci infinite class} we give a new proof that the Klein configuration is a $(6,10)$-geproci. 

Let $W$ be the set of $60$ points in the Klein configuration. Let $G$  be  the $(6,6)$-grid on the quadric $xw-yz=0$, see Example \ref{ex.from F4 to Klein}.  Similarly as   \cite[Remark 6.2]{PSS} points out, the set $W\setminus G$ is
isomorphic to the $F_4$ root system. 
Let $P=[a:b:c:d]$ be a general point. Then two cones of degree 6 containing $G$ and having vertex at $P$ are
\[
\begin{array}{lcl}
C_1([x:y:z:w]) & = & \displaystyle \prod_{i=0}^{3} \left [ (u^ic-d)(u^i x-y) + (b-u^ia)(u^iz-w) \right ](bx-ay)(dz-cw), \\[2pt]
C_2([x:y:z:w]) & = & \displaystyle \prod_{i=0}^{3} \left [ (u^ib-d)(u^ix-z) + (c-u^ia)(u^iy-w) \right ](cx-az)(dy-bw).
\end{array}
\]
As in Theorem \ref{t. geproci infinite class}, these cones are simply the union of the planes spanned by $P$ and each of the horizontal and vertical lines of the grid $G$, respectively.

An easy computation shows that the element of the pencil vanishing, for instance, at $[1:0:0:-1]$ is
$C_1-C_2$. This form also vanishes in all the other points of $W\setminus G$, as one can check directly.
\end{remark}

The grid  $G$ in the previous example has a nice property that we explore in the next section.

\subsection{Biharmonic sets of points}
The Klein configuration contains a $(6,6)$-grid of points on $\calq$. It is given by the Segre product of the set $\{[1:1],[1:i],[1:-1], [1:-i],[1:0],[0:1]\}$ by itself. 

\begin{definition} We say that an ordered set $(P_1,P_2,P_3,P_4,P_5,P_6)$ of points in $\PP^1$  is \emph{biharmonic} if the following sets are harmonic, see Section \ref{harmonic points},
$$ (P_1,P_2,P_3,P_4), \quad (P_1,P_2,P_5,P_6), \quad (P_5,P_6,P_3,P_4).$$\index{points!harmonic}\index{points!biharmonic}
\end{definition}

\begin{example}\label{harm} A simple biharmonic set is given by the  the points
\[
([1:0],[0:1],[1:1],[1:-1],[1:i],[1:-i]).
\]
\end{example}

\begin{proposition}\label{1biharm} For any harmonic set $(P_1,P_2,P_3,P_4)$ there are only two points $P_5,P_6$ such that $(P_1,P_2,P_3,P_4,P_5,P_6)$ is biharmonic.
\end{proposition}
\begin{proof} Since harmonic sets can be interchanged by an automorphism, we can reduce to the case
$(P_1,P_2,P_3,P_4)=([1:0],[0:1],[1:1],[1:-1])$. Since $P_5$, $P_6$ cannot coincide with 
$P_1$, we may assume $P_5=[1:a]$, $P_6=[1:b]$. An easy computation now proves that $\{a,b\}=\{i,-i\}$.
\end{proof}

\begin{remark}\label{biperm} If $(P_1,P_2,P_3,P_4, P_5,P_6)$ is biharmonic, then every automorphism of $\PP^1$ that produces a permutation on
$(P_1,P_2,P_3,P_4)$ also produces a permutation on $(P_5,P_6)$.\\
This can be checked easily on examples.
\end{remark}

\section{The Penrose configuration}\label{s.Penrose}
In this section we introduce two new examples. 

First, we give a $(5,8)$-geproci set, $Z_P$, which we will call the {\it Penrose configuration}\index{configuration! Penrose}. It is arithmetically Gorenstein\index{arithmetically! Gorenstein}. It arose in \cite{ZP} due to connections with quantum mechanics.

Second, we describe a subset of the Penrose configuration giving a  $(4,5)$-geproci set, which we will call the {\it half Penrose configuration}. \index{configuration! half Penrose}

\begin{proposition} 
Fix an integer $d\ge 3$. The ideal 
\[I_d=(xyzw, w(x^d-y^d+z^d), z(x^d+y^d+w^d), y(-x^d+z^d+w^d), x(y^d+z^d-w^d))\]
in the coordinate ring $\field[x,y,z,w]=\field[\PP^3]$,
defines a subset $Z_d$ of $4d^2+4$ points of $\PP^3$.
Moreover, $Z_d$ is arithmetically Gorenstein.
\end{proposition}
\begin{proof}
We note that $I_d$  is generated by the Pfaffians of the following skew symmetric matrix
\[
\left( \begin{array}{ccccc} 
  0 &  - zw &  - yw & yz & x^{d-1} \\
  zw & 0 &  - xw &  - xz & y^{d-1} \\
  yw & xw & 0 & xy & z^{d-1} \\
   - yz & xz &  - xy & 0 & w^{d-1} \\
   - x^{d-1} &  - y^{d-1} &  - z^{d-1} &  - w^{d-1} & 0 \end{array}\right). \]
Let $Z_d$ be the scheme defined by $I_d$.  
The scheme $Z_d$ lies on the coordinate planes. 
We first note that the only points of $Z_d$ in the coordinate lines are the four coordinate points.

Then we claim that $Z_d$ contains $d^2$ points in each of the coordinate planes which lie in no coordinate lines. Indeed, for instance the points of $Z_d$ in the plane $x=0$ are defined by 
\[(w(-y^d+z^d), z(y^d+w^d), y(z^d+w^d))\]
for $y,z,w$ different from zero. Then, setting $w=1$, we get that $y$ and $z$ can be any of the $d$-th roots of $-1$. 

From the computation above, we get that $Z_d$ is reduced and zero-dimensional, and it consists of $4d^2+4$ points.  
Thus $I_d$ defines a  zero-dimensional and arithmetically Gorenstein scheme  $Z_d$, by  \cite[Theorem~2.1]{BE}.    
\end{proof}

Experimentally,  $Z_d$ is not a geproci set for $d>3$.
The set $Z_3$ is the set $Z_P$ mentioned above.
The next result shows that it is a geproci set.

We can explicitly list the 40 points of $Z_P$ in terms of $t$, where $t^2-t+1=0$:
\begin{equation}\label{points Penrose}
\renewcommand{\arraystretch}{1.2}
\begin{array}{llll}
P_1=[1:t:t^2:0],& P_2=[1:0:0:0],& P_3=[0:1:0:0], & P_4=[0:0:1:0],\\ 
P_5=[-1:0:t^2:1],& P_6=[0:-1:t:1],&
P_7=[t^2:1:0:1],& P_8=[t:0:1:1],\\
P_9=[0:t^2:1:-1],& P_{10}=[1:t^{-2}:0:1],& P_{11}=[0:t:-1:1],& P_{12}=[t^{-1}:0:1:1],\\
P_{13}=[1:0:t:-1], & P_{14}=[1:t^2:0:1],& P_{15}=[t^{-2}:1:0:1], & P_{16}=[0:1:t^2:-1],\\
P_{17}=[1:1:0:1],& P_{18}=[0:1:1:-1], & P_{19}=[-1:0:1:1],&P_{20}=[t^2:t:1:0],\\
P_{21}=[0:0:0:1],& P_{22}=[0:t^2:1:t],& P_{23}=[t:0:1:t^2],&P_{24}=[t^{-2}:t^2:0:1],\\
P_{25}=[0:1:1:t], & P_{26}=[1:0:-1:t^{-1}], & P_{27}=[1:0:-1:t],& P_{28}=[1:1:0:t^2],\\ P_{29}=[1:1:0:t^{-2}],& P_{30}=[0:1:1:t^{-1}],& P_{31}=[-1:1:t^{-1}:0],& P_{32}=[-1:1:t:0],\\ P_{33}=[t^2:-1:1:0],& P_{34}=[t:1:-1:0],&
P_{35}=[1:t^{-1}:1:0],&P_{36}=[1:t:1:0],\\
P_{37}=[1:t^2:0:t^{-2}],& P_{38}=[0:1:t^2:t],&P_{39}=[t:0:t^2:1],&
P_{40}=[1:-1:1:0].
\end{array}
\end{equation}

\begin{theorem}\label{thm: Penrose is geproci}
The set $Z_P$ is a nontrivial $(5,8)$-geproci set. Moreover, it is not a half grid.   
\end{theorem}
\begin{proof}
The proof is mainly computational. 
In the ring $T=\field[x,y,z,w][a,b,c,d]$ we computed the intersection $J$ of the extension $I_3\cdot T$ with the ideal $(a,b,c,d)^5$. 

It turned out that $J$ has a unique generator of bidegree $(5,5)$, namely\\ 
\begin{align*}
F_5&=cd\cdot(2b^{3}-c^{3}+d^{3})\cdot x^{4}y-6ab^{2}cd\cdot
x^{3}y^{2}+6a^{2}bcd\cdot x^{2}y^{3}-cd\cdot(2a^{3}+c^{3}
+d^{3})\cdot xy^{4}\\
&-bd\cdot(b^{3}-2c^{3}-d^{3})\cdot x^{4}z
+2ad\cdot(b^{3}+c^{3}-d^{3})\cdot x^{3}yz
-2bd\cdot (a^{3}-c^{3}-d^{3})\cdot xy^{3}z\\
&+ad\cdot (a^{3}+2c^{3}-d^{3})\cdot y^{4}z-6abc^{2}d\cdot
x^{3}z^{2}-6abc^{2}d\cdot y^{3}z^{2}+6a^{2}bcd\cdot
x^{2}z^{3}\\
&-2cd\cdot(a^{3}+b^{3}+d^{3})\cdot xyz^{3}+6ab^{2}cd\cdot
y^{2}z^{3}-bd(2a^{3}-b^{3}d-d^{3})\cdot
xz^{4}\\
&+ad\cdot (a^{3}-2b^{3}+d^{3})\cdot yz^{4}-bc\cdot
(b^{3}+c^{3}+2d^{3})\cdot x^{4}w+2ac\cdot (b^{3}+c^{3}-d^{3})\cdot
x^{3}yw\\
&-2bc\cdot (a^{3}-c^{3}-d^{3})\cdot
xy^{3}w+ac(a^{3}-c^{3}+2d^{3})\cdot y^{4}w+2ab\cdot
(b^{3}+c^{3}-d^{3})\cdot x^{3}zw\\
&-6a^{2}\cdot (b^{3}+c^{3}-d^{3})\cdot
x^{2}yzw+6b^{2}\cdot(a^{3}-c^{3}-d^{3})\cdot xy^{2}zw-2ab\cdot
(a^{3}-c^{3}-d^{3})\cdot
y^{3}zw\\
&+6c^{2}\cdot (a^{3}+b^{3}+d^{3})\cdot
xyz^{2}w-2bc(a^{3}+b^{3}+d^{3})\cdot xz^{3}w-2ac\cdot
(a^{3}+b^{3}+d^{3})\cdot yz^{3}w\\
&+ab\cdot (a^{3}+b^{3}-2d^{3})\cdot z^{4}w+6abcd^{2}\cdot
x^{3}w^{2}-6abcd^{2}\cdot y^{3}w^{2}-6d^{2}\cdot
(a^{3}-b^{3}+c^{3})\cdot xyzw^{2}\\
&+6abcd^{2}\cdot z^{3}w^{2}-6a^{2}bcd\cdot x^{2}w^{3}+2cd\cdot
(a^{3}-b^{3}+c^{3})\cdot xyw^{3}+6ab^{2}cd\cdot
y^{2}w^{3}\\
&+2bd\cdot (a^{3}-b^{3}+c^{3})\cdot xzw^{3}+2ad\cdot
(a^{3}-b^{3}+c^{3})\cdot yzw^{3}-6abc^{2}d\cdot
z^{2}w^{3}\\
&+bc\cdot (2a^{3}+b^{3}-c^{3})\cdot xw^{4}-ac\cdot
(a^{3}+2b^{3}+c^{3})\cdot yw^{4}-ab\cdot (a^{3}-b^{3}-2c^{3})\cdot
zw^{4}.
\end{align*}

By specializing $F_5$ to $[a:b:c:d]=[1:2:3:4]$ and intersecting with $w=0,$ we  obtain a smooth plane curve. 
This implies that, for a general point $P$ in $\PP^3$, the unique cone of degree 5 with vertex at $P$ containing $Z_P$ is irreducible.
  
A similar computation shows that there is an irreducible cone of degree $4$, call it $G_1$,  containing the set~$Z_{20}$ of 20 points in $Z$, with index in the set $$\begin{array}{c}
\{1,  2, 32, 35,\ 
29,  7, 37,  3,\ 
30, 15,  5, 36,\ 
11, 14, 39, 40,\ 
8, 17, 33, 38\}
\end{array}$$ with vertex at a general point $P$. The same is true for the set $Z_{20}'= Z\setminus Z_{20}$, call $G_2$ the cone of degree 4 containing $Z_{20}'$, with vertex at a general point $P$.

Then  $G_1\cup G_2$ is a cone of degree 8, with vertex at a general point $P$, containing $Z$. Therefore, $Z$ is a $(5,8)$-geproci set. 

Moreover, $Z_P$ is not contained in 8 line since it has $\binom{40}{2}=780$ pairs of points and, by brute force check, one can find that these pairs define only 330 lines, 90 of which contain 4 of the points,
and 240 of which contain exactly 2 points. In particular no line contains 5 points of $Z_P$.
\end{proof}

The following result follows from the proof of Theorem \ref{thm: Penrose is geproci}.
\begin{corollary}\label{c. half Penrose geproci}
The sets $Z_{20}\subseteq Z$ whose points have index in the set  $$\begin{array}{c}
\{1,  2, 32, 35\}\cup 
\{29,  7, 37,  3\}\cup 
\{30, 15,  5, 36\}\cup
\{11, 14, 39, 40\}\cup 
\{8, 17, 33, 38\}
\end{array}$$ and $Z_{20}'=Z\setminus Z_{20}$ are half grid $(4,5)$-geproci sets.
\end{corollary}
\begin{proof}
The existence of an irreducible cone of degree 4 containing $Z_{20}$ with vertex at a general point $P$ is guaranteed by the computation in the proof of Theorem \ref{thm: Penrose is geproci}. Moreover, $Z_{20}$ is contained in 5 lines. Indeed we checked that the points corresponding to the five sets listed above are collinear. 
A similar argument holds for the set $Z_{20}'$. 
\end{proof}
We will show in Remark \ref{r. different (4,5)-geproci} that $Z_{20}$ is not one of the $(4,5)$-geproci sets constructed either in Section \ref{sec:standard construction} or in Section \ref{s.extending geproci}.

\begin{remark}
The set $Z_{20}$ provides a positive answer to Question \ref{q. Klein-like} in the case $a=4.$ Precisely, in the pencil of quartic cones, with vertex at a general point, containing the $(4,4)$-geproci set extension of $D_4$,\index{configuration! extended $D_4$} there is a special element which contains a larger geproci set.
\end{remark}

\begin{remark}\label{R. Prop 5.4 applied to D4} Here we focus on $Z_{20}$, however a similar description can be done for $Z_{20}'$.
We note that~$Z_{20}$ contains the $D_4$ configuration. With respect to list in \eqref{points Penrose} the points with indices  $$\begin{array}{ccc}
1,&  2,& 32,\\
29,&  7,& 37,\\
30,& 15,&  5
\end{array}$$ are a $(3,3)$-grid $X$ which is completed to the $D_4$ by the set of collinear points $Y_1$ whose indices are $\{11, 14, 39\}$.
Set $Z=X\cup Y_1,$ then $Z_{20}$ also contains $\widetilde Z$ as defined in Proposition \ref{p.add points to geproci}.  The points of $\widetilde Z$ correspond to the set
\[\begin{array}{cccc}
1,&  2,& 32,& 35,\\
29,&  7,& 37,&  3,\\
30,& 15,&  5,& 36,\\[2pt]

11,& 14,& 39,& 40\\
\end{array}
\]
where $P_{40}$ is a point in the intersection  of the line $\ell_1$ containing $Y_1$ and the quadric $\calq$ defined by $X$, and the points $P_{35}, P_3,P_{36},P_{40}$ are collinear. Picture \eqref{fig:tildeZ.jpg} summarize the situation.
\begin{figure}[ht]
\definecolor{rvwvcq}{rgb}{0.08235294117647059,0.396078431372549,0.7529411764705882}
\definecolor{wrwrwr}{rgb}{0.3803921568627451,0.3803921568627451,0.3803921568627451}
\begin{tikzpicture}[line cap=round,line join=round,>=triangle 45,x=1cm,y=1cm,scale=0.6]
\clip(-8.890373787243591,-5.231848296267393) rectangle (9.051500055622716,3.7344235513636636);
\draw [line width=1.0pt,loosely dotted, color=wrwrwr,domain=-8.890373787243591:9.051500055622716] plot(\x,{(--33.6686--10.88*\x)/6.86});
\draw [line width=1.0pt,loosely dotted, color=wrwrwr,domain=-8.890373787243591:9.051500055622716] plot(\x,{(--7.0162--10.68*\x)/7.34});
\draw [line width=1.0pt,loosely dotted, color=wrwrwr,domain=-8.890373787243591:9.051500055622716] plot(\x,{(-19.3236--10.56*\x)/7.08});
\draw [line width=1.0pt,loosely dotted, color=wrwrwr,domain=-8.890373787243591:9.051500055622716] plot(\x,{(--24.415--4.46*\x)/-17.02});
\draw [line width=1.0pt,loosely dotted, color=wrwrwr,domain=-8.890373787243591:9.051500055622716] plot(\x,{(-4.444--3.84*\x)/-14.96});
\draw [line width=1.0pt,loosely dotted, color=wrwrwr,domain=-8.890373787243591:9.051500055622716] plot(\x,{(-28.0896--3.44*\x)/-15.52});
\draw [line width=1.0pt,loosely dotted, color=wrwrwr,domain=-8.890373787243591:9.051500055622716] plot(\x,{(-38.6376--8.48*\x)/6});
\draw [line width=1.0pt,loosely dotted, color=wrwrwr,domain=-8.890373787243591:9.051500055622716] plot(\x,{(-52.8224-4.5*\x)/16.12});
\draw [line width=1.0pt,dash pattern=on 5pt off 4pt,color=wrwrwr] (1.8687122596788223,-3.7984866729872637)-- (4.302,7.068);

\begin{scriptsize}
\draw [fill=white] (-1.7138567568691023,2.1897723739452135) circle (5pt);
\draw [fill=white] (-2.5022654663598214,0.9393515635576012) circle (5pt);
\draw [fill=white] (-3.431966861500872,-0.535160270135494) circle (5pt);
\draw [fill=white] (-1.3921114353643187,-1.0696934781595262) circle (5pt);
\draw [fill=white] (0.7383983190413993,-1.6279821682094382) circle (5pt);
\draw [fill=black] (2.987451908239893,-2.2173346363542845) circle (5pt);
\draw [fill=black] (4.033884382204764,-0.7383767398172656) circle (5pt);
\draw [fill=black] (5.04561688175207,0.6915385262095929) circle (5pt);
\draw [fill=white] (2.6495945739201394,1.222615635677495) circle (5pt);
\draw [fill=white] (1.7311312733305683,-0.14729572791372863) circle (5pt);
\draw [fill=white] (-0.3848904438428929,0.39585423157464644) circle (5pt);
\draw [fill=white] (0.5093435083810115,1.6970011811320442) circle (5pt);
\draw [fill=black] (1.8687122596788223,-3.7984866729872637) circle (5pt);
\draw [fill=white] (2.432410349281785,-1.2811445285475535) circle (5pt);
\draw [fill=white] (2.7719880060664907,0.235328850555649) circle (5pt);
\draw [fill=white] (2.1065340575845406,-2.736430856626124) circle (5pt);
\end{scriptsize}
\end{tikzpicture}
\caption[An extension of the $D_4$ configuration.]{The picture shown is for illustration purposes only. The 16 points represent $\widetilde Z$, the white dots $D_4$.}
	\label{fig:tildeZ.jpg}
\end{figure}
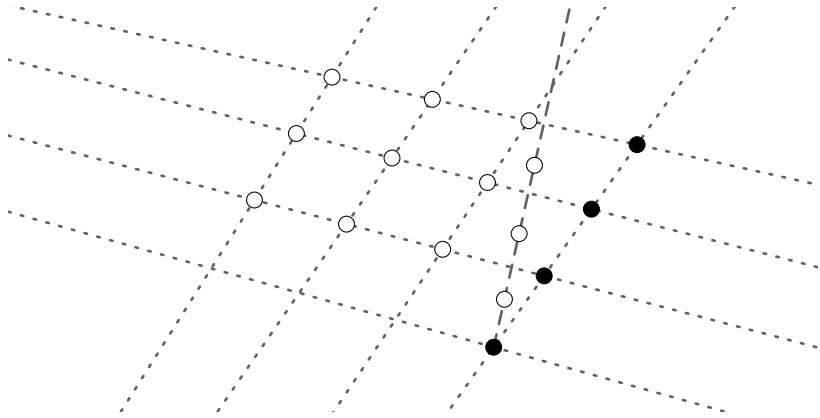

The remaining 4 points $P_8,P_{17},  P_{33}, P_{38}$ are collinear. Furthermore, the point $P_{17}$ is in the grid line spanned by $P_2, P_7, P_{15}$. 
We also note that the points indexed by $\{1,29,30\}\cup\{32,37,5\}\cup\{35,3,36\} \cup\{8,33,38\}$ are also a $D_4$ configuration.
The set $Z_{20}$ is represented in Figure \ref{fig:tildeZplus.jpg}
 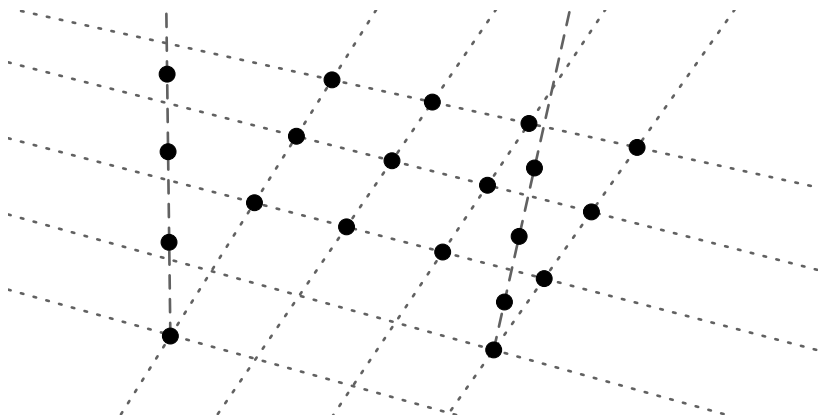
\begin{figure}[ht]
\definecolor{rvwvcq}{rgb}{0.08235294117647059,0.396078431372549,0.7529411764705882}
\definecolor{wrwrwr}{rgb}{0.3803921568627451,0.3803921568627451,0.3803921568627451}
\begin{tikzpicture}[line cap=round,line join=round,>=triangle 45,x=1cm,y=1cm,scale=0.6]
\clip(-8.890373787243591,-5.231848296267393) rectangle (9.051500055622716,3.7344235513636636);
\draw [line width=1.0pt,loosely dotted, color=wrwrwr,domain=-8.890373787243591:9.051500055622716] plot(\x,{(--33.6686--10.88*\x)/6.86});
\draw [line width=1.0pt,loosely dotted, color=wrwrwr,domain=-8.890373787243591:9.051500055622716] plot(\x,{(--7.0162--10.68*\x)/7.34});
\draw [line width=1.0pt,loosely dotted, color=wrwrwr,domain=-8.890373787243591:9.051500055622716] plot(\x,{(-19.3236--10.56*\x)/7.08});
\draw [line width=1.0pt,loosely dotted, color=wrwrwr,domain=-8.890373787243591:9.051500055622716] plot(\x,{(--24.415--4.46*\x)/-17.02});
\draw [line width=1.0pt,loosely dotted, color=wrwrwr,domain=-8.890373787243591:9.051500055622716] plot(\x,{(-4.444--3.84*\x)/-14.96});
\draw [line width=1.0pt,loosely dotted, color=wrwrwr,domain=-8.890373787243591:9.051500055622716] plot(\x,{(-28.0896--3.44*\x)/-15.52});
\draw [line width=1.0pt,loosely dotted, color=wrwrwr,domain=-8.890373787243591:9.051500055622716] plot(\x,{(-38.6376--8.48*\x)/6});
\draw [line width=1.0pt,loosely dotted, color=wrwrwr,domain=-8.890373787243591:9.051500055622716] plot(\x,{(-52.8224-4.5*\x)/16.12});
\draw [line width=1.0pt,dash pattern=on 5pt off 4pt,color=wrwrwr] (1.8687122596788223,-3.7984866729872637)-- (4.302,7.068);

\begin{scriptsize}
\draw [fill=black] (-1.7138567568691023,2.1897723739452135) circle (5pt);
\draw [fill=black] (-2.5022654663598214,0.9393515635576012) circle (5pt);
\draw [fill=black] (-3.431966861500872,-0.535160270135494) circle (5pt);
\draw [fill=black] (-1.3921114353643187,-1.0696934781595262) circle (5pt);
\draw [fill=black] (0.7383983190413993,-1.6279821682094382) circle (5pt);
\draw [fill=black] (2.987451908239893,-2.2173346363542845) circle (5pt);
\draw [fill=black] (4.033884382204764,-0.7383767398172656) circle (5pt);
\draw [fill=black] (5.04561688175207,0.6915385262095929) circle (5pt);
\draw [fill=black] (2.6495945739201394,1.222615635677495) circle (5pt);
\draw [fill=black] (1.7311312733305683,-0.14729572791372863) circle (5pt);
\draw [fill=black] (-0.3848904438428929,0.39585423157464644) circle (5pt);
\draw [fill=black] (0.5093435083810115,1.6970011811320442) circle (5pt);
\draw [fill=black] (1.8687122596788223,-3.7984866729872637) circle (5pt);
\draw [fill=black] (2.432410349281785,-1.2811445285475535) circle (5pt);
\draw [fill=black] (2.7719880060664907,0.235328850555649) circle (5pt);
\draw [fill=black] (2.1065340575845406,-2.736430856626124) circle (5pt);
\draw [line width=1.0pt,loosely dotted, color=wrwrwr,domain=-8.890373787243591:9.051500055622716] plot(\x,{(-78.199176-4.422*\x)/15.796});
\draw [line width=1.0pt,dash pattern=on 5pt off 4pt,color=wrwrwr] (-5.295582789511225,-3.490865998525092)-- (-5.422,6.738);
\draw [fill=black] (-5.295582789511225,-3.490865998525092) circle (5pt);
\draw [fill=rvwvcq] (-5.422,6.738) circle (2.5pt);
\draw [fill=black] (-5.324981338263386,-1.412124887450931) circle (5pt);
\draw [fill=black] (-5.346095053358437,0.5962608742922009) circle (5pt);
\draw [fill=black] (-5.367350709776518,2.316131613178228) circle (5pt);
\end{scriptsize}
\end{tikzpicture}
 	\caption[A half Penrose configuration.]{A representation of the half Penrose configuration.}
 	\label{fig:tildeZplus.jpg}
\end{figure}
\end{remark}

\begin{remark}
The following table lists the 90 sets, mentioned in the proof of Theorem \ref{thm: Penrose is geproci} of indices corresponding to subsets of 4 collinear points. 

{\footnotesize \begin{equation}\label{table Penrose}
\centering
\renewcommand{\arraystretch}{1.2}
\begin{array}{ccccccc}
\{1, 2, 32, 35\}&
\{1, 3, 31, 34\}&
\{1, 4, 33, 36\}&
\{1, 11, 13, 15\}&
\{1, 12, 14, 16\}&
\{1, 17, 25, 26\}\\
\{1, 18, 27, 28\}&
\{1, 19, 29, 30\}&
\{1, 37, 38, 39\}&
\{2, 5, 27, 39\}&
\{2, 7, 15, 17\}&
\{2, 8, 12, 19\}\\
\{2, 10, 28, 37\}&
\{2, 13, 23, 26\}&
\{2, 14, 24, 29\}&
\{2, 20, 31, 36\}&
\{2, 33, 34, 40\}&
\{3, 6, 30, 38\}\\
\{3, 7, 29, 37\}&
\{3, 9, 11, 18\}&
\{3, 10, 14, 17\}&
\{3, 15, 24, 28\}&
\{3, 16, 22, 25\}&
\{3, 20, 32, 33\}\\
\{3, 35, 36, 40\}&
\{4, 5, 13, 19\}&
\{4, 6, 16, 18\}&
\{4, 8, 26, 39\}&
\{4, 9, 25, 38\}&
\{4, 11, 22, 30\}\\
\{4, 12, 23, 27\}&
\{4, 20, 34, 35\}&
\{4, 31, 32, 40\}&
\{5, 6, 24, 40\}&
\{5, 7, 9, 20\}&
\{5, 8, 21, 23\}\\
\{5, 11, 29, 33\}&
\{5, 15, 30, 36\}&
\{5, 18, 32, 37\}&
\{5, 28, 35, 38\}&
\{6, 8, 10, 20\}&
\{6, 9, 21, 22\}\\
\{6, 12, 28, 36\}&
\{6, 14, 27, 33\}&
\{6, 19, 31, 37\}&
\{6, 29, 34, 39\}&
\{7, 8, 22, 40\}&
\{7, 10, 21, 24\}\\
\{7, 11, 26, 32\}&
\{7, 13, 25, 35\}&
\{7, 19, 34, 38\}&
\{7, 30, 31, 39\}&
\{8, 14, 30, 32\}&
\{8, 16, 29, 35\}\\
\{8, 17, 33, 38\}&
\{8, 25, 36, 37\}&
\{9, 10, 23, 40\}&
\{9, 13, 28, 34\}&
\{9, 15, 27, 31\}&
\{9, 17, 36, 39\}\\
\{9, 26, 33, 37\}&
\{10, 12, 25, 31\}&
\{10, 16, 26, 34\}&
\{10, 18, 35, 39\}&
\{10, 27, 32, 38\}&
\{11, 14, 39, 40\}\\
\{11, 16, 21, 38\}&
\{11, 17, 23, 35\}&
\{11, 24, 27, 34\}&
\{12, 13, 21, 39\}&
\{12, 15, 38, 40\}&
\{12, 17, 22, 34\}\\
\{12, 24, 30, 35\}&
\{13, 16, 37, 40\}&
\{13, 18, 24, 31\}&
\{13, 22, 29, 36\}&
\{14, 15, 21, 37\}&
\{14, 18, 23, 36\}\\
\{14, 22, 26, 31\}&
\{15, 19, 22, 33\}&
\{15, 23, 25, 32\}&
\{16, 19, 24, 32\}&
\{16, 23, 28, 33\}&
\{17, 20, 27, 30\}\\
\{17, 21, 28, 29\}&
\{18, 20, 26, 29\}&
\{18, 21, 25, 30\}&
\{19, 20, 25, 28\}&
\{19, 21, 26, 27\}&
\{20, 22, 23, 24\}
\end{array}
\end{equation}}



\end{remark}

\begin{remark}
A brute force search of subsets of 4 skew 4-point lines
shows that $Z_P$ contains no $(4,4)$-grids.
(It is enough to check sets of 4 skew 4-point lines
for which one of the lines contains point 0, since Penrose says the
group of projective symmetries of $Z_P$ is transitive.)

Then, we also see that there can be no $(4,6)$-geproci subset of $Z_P$ projectively equivalent to the $F_4$ configuration
(since the latter contains a $(4,4)$-grid; see Example \ref{e:F4exOf4.2b}) .
\end{remark}

\begin{remark}\label{r.Penrose not real} We can compute the cross ratio of the four collinear points in the $(3,4)$-grid mentioned in Remark \ref{R. Prop 5.4 applied to D4}, by using the coordinates of $D_4$ as in the standard construction and its extension as introduced in Proposition \ref{p.add points to geproci}.
This gives the points $$\{[1:1], [1:\varepsilon], [1:\varepsilon^2], [1:0]\}\subseteq \PP^1.$$ 
These points are equianharmonic,\index{points!equianharmonic}, i.e., their cross ratios\index{cross ratio} assume only one of the following two values (depending on the order of the points) $$\{-\varepsilon, -1/\varepsilon \},$$ where $-\varepsilon$ is a third root of $-1$. 

Thus, since the cross-ratio is preserved under projective transformations, see Remark \ref{r. cr and autom}, as mentioned also in \cite{ZP}, the set  $Z_P$ cannot be realized over the real space. Of course, this is also true for  $\widetilde{Z}$ and the half Penrose configuration since both contain the same  $(3,4)$-grid.
\end{remark}


\section{The \texorpdfstring{$H_4$}{H4} configuration}\index{configuration! $H_4$}

	As shown in Theorem \ref{t. geproci infinite class} in the $n$ odd cases, both sets $X\cup Y_1$ and $X\cup Y_2$ are geproci sets, but $X\cup Y_1\cup Y_2$ is not.  In particular, the set  $X\cup Y_1$ is contained in a unique cone of degree $n$ and vertex at a general point $P$. Then adding the plane spanned by $P$ and $Y_2$ we obtain a cone of degree $n+1$ with vertex at $P$ containing $X\cup Y_1\cup Y_2$. Swapping the roles of $Y_1$ and $Y_2$ we obtain a second cone 
	of degree $n+1$ with vertex at $P$ containing $X\cup Y_1\cup Y_2$ independent from the previous one.
	
	This implies that, for a general point $P$, $$\dim[I(X\cup Y_1\cup Y_2)\cap I(P)^{n+1} ]_{n+1}\ge2.$$ Hence, there is a pencil of cones (with no fixed components) through $X\cup Y_1\cup Y_2$ of degree $n+1$ having vertex at $P$. So, it is natural to ask if  there is a special cone in such a pencil which  contains a larger configuration of geproci points. This actually happens in (at least) two cases:
	\begin{itemize}
		\item[$(n=3)$] There is a cone of degree 4 through the 15 points in  $X\cup Y_1\cup Y_2$ with vertex at a general point $P$  and also containing an extra $(3,3)$-grid $X'$ not depending on $P$.  Both $X'\cup Y_1$ and  $X'\cup Y_2$ are equivalent to $D_4$. The set of $X\cup X'\cup Y_1\cup Y_2$ consists of 24 points and it is a $(4,6)$-geproci projectively equivalent to the $F_4$ configuration.  In Section \ref{Sec: D4 in F4} we give more details about the combinatorics behind the fact that $F_4$ is union of two $D_4.$ 
		\item[$(n=5)$]  One can check by a computer calculation, that there is a cone of degree 6 through the 35 points in $X\cup Y_1\cup Y_2$ with vertex at a general point $P$ and also containing an extra (5,5)-grid $X'$ not depending on $P$. To construct the new $(5,5)$-grid, set $q=\dfrac{1}{u}+u-1$ and consider the Segre embedding\index{Segre embedding} $s_q:\PP^1\times \PP^1 \to \PP^3$ defined by
		\[s_q([a:b],[c:d])= [ac:qad:qbc:bd]. \]
		Set $X'$ to be the grid $\{[1:1],\ldots, [1:u^4]\} \otimes_{s_q}  \{[1:1],\ldots, [1:u^4]\}$. (Or just consider the automorphism on $\PP^3$, $[x:y:z:w]\mapsto[x:qy:qz:w]$; this fixes $Y_1$ and $Y_2$). The 60 points determine  a $(6,10)$-geproci set $Z$, union of two $(5,6)$-geproci sets.  One can check that $Z$ is projectively equivalent to the $H_4$ configuration studied in \cite{HMNT} and \cite{FZ}.
	\end{itemize}
Surprisingly, we note that, in both of the above cases, the element of the pencil of cones of degree $n$ with vertex at a general point $P=[a:b:c:d]$ and containing the larger configuration is given by $qA+B$ where $A=(F-G)\cdot(d x-a w)$ and $B=(F+G)\cdot(cy-bz)$ and $F$ and $G$ are defined in the proof of Theorem \ref{t. geproci infinite class}. In particular, when  $n=3$ we get  $q=-2$.

However, the computation described in the case $n=5$, using again $q=\dfrac{1}{u}+u-1$, does not produce geproci sets in the next cases. We checked that for $n=7,9,11$ we get 
$2n^2+2n$ points whose general projection is contained in a unique curve of degree $2n-4.$

\begin{question}
Is there a construction similar to the one given in the $n=5$ case which produces geproci sets for some $n\ge 7$? 
\end{question}

\chapter{Equivalences of geproci sets}\label{chap. non iso and real}

The standard construction in Theorem \ref{t. geproci infinite class} uses complex coordinates to construct an infinite class of $(a,b)$-geproci sets that are not grids. We know, from \cite{CM} and \cite{HMNT}, that ${D_4}$ and ${F_4}$ are geproci sets  realizable  with real (in fact even integer) coordinates. 
In this chapter, the question is discussed whether the geproci sets introduced in Chapter~\ref{chap.Geography} and Chapter~\ref{chap.extending} can be realized in the real space.  We show that all the geproci sets of the standard construction are realizable in $\PP_{\mathbb R}^3$, i.e., there is an automorphism which maps the points $X\cup Y_1 \cup Y_2$ of Theorem~\ref{t. geproci infinite class} into points with real coordinates, see Proposition \ref{p. real-izable}. Furthermore, in Proposition \ref{p. extension is not real}, we prove that the extensions of the geproci sets introduced  in Proposition \ref{p.add points to geproci} and Proposition~\ref{p.add points to geproci2} cannot be realized over the real numbers. As a consequence we give a new proof of the nonrealizability in $\PP_{\mathbb R}^3$ of the Klein configuration. A fact first established in \cite{PSS}. 

The realizability in the real projective space  gives an easy criterion to distinguish two nonprojectively equivalent $(a,b)$-geproci. However, as we saw for the $(3,b)$-geproci case in Theorem \ref{thm:classification_of_3xb}, 
  it is not easy to say how many nonprojectively equivalent $(a,b)$-geproci sets exist for given $a,b$. 
We will discuss this problem in Section \ref{sec.how many} in connection with the notion of combinatorial equivalence.\index{equivalent! combinatorially}

\section{Realizability over the real numbers}\index{geproci! realizability}
The motivating question of this section is the following.
\begin{question}\label{q. are real?}
    Can the geproci sets constructed in Chapter \ref{chap.Geography} and Chapter \ref{chap.extending} be realized in $\PP_{\mathbb R}^3$?
\end{question}
In order to show that all the sets of points $X\cup Y_1\cup Y_2$ in Theorem \ref{t. geproci infinite class} can be realized using real coordinates we need the following lemma.

\begin{lemma}\label{l. cr u} Let $u$ be a primitive $n$-th root of unity, $n\ge 4$.
The points 
$$([1:u^t],\; [1:u^{t+1}],\; [1:u^{t+2}],\; [1:u^{t+3}])$$ 
have, for $t=0,\ldots, n-4$, the cross ratio\index{cross ratio} equal to 	 
$$ 1+ \dfrac{1}{2r+1}$$
where $r=\mathfrak R (u)$ is the real part of $u$.
\end{lemma}
\begin{proof}
	Computing the cross ratio of the points $([1:u^t], [1:u^{t+1}], [1:u^{t+2}], [1:u^{t+3}])$ as in equation \eqref{eq. cross ratio}, we get
	\begin{equation}\label{eq.cr u}
	\dfrac{(u^{t+2}-u^t)(u^{t+3}-u^{t+1})}{(u^{t+3}-u^t)(u^{t+2}-u^{t+1})}=
	\dfrac{(u^{2}-1)(u^{2}-1)}{(u^{3}-1)(u-1)}=
	\dfrac{(u+1)^2}{u^{2}+u+1}=1+ \dfrac{u}{u^2+u+1}.
\end{equation}
Since $|u|=1$ all its powers are contained in the (real) unit circle on the complex line $\field$. It follows from \cite[Theorem 17.2]{Richter-Gebert2011} that the cross ratio of the four points is real.  

Turning to the exact value, recall that $u$ can be written as $u=\cos\dfrac{k\pi}{n}+i \sin\dfrac{k\pi}{n} $ for some integer $k,$ and set 
$$r=\mathfrak{R}(u)=\cos\dfrac{k\pi}{n}\ \text{and}\ s=\mathfrak{I}(u)=\sin\dfrac{k\pi}{n},$$ 
so $r^2+s^2=1$. 
The inverse of the last term in the identity \eqref{eq.cr u} is equal to 	$$\begin{array}{rcl}
	      \dfrac{u^2+u+1}{u}&=&\mathfrak{R}\left(\dfrac{u^2+u+1}{u}\right)=\mathfrak{R}\left(\dfrac{(r^2-s^2+2rsi+r+si+1)(r-si)}{r^2+s^2}\right)  \\
	     &=&r^3-rs^2+r^2+r+2rs^2+s^2= r(r^2+s^2)+r +r^2+s^2\\
	     &=&2r+1. 
	\end{array}$$%

Thus, we are done.	
\end{proof}

\begin{proposition}\label{p. real-izable} Using the notation as in Theorem \ref{t. geproci infinite class},  
	all the sets $X\cup Y_1\cup Y_2$ are realizable in $\PP^3_{\mathbb R}$.
\end{proposition}
 \begin{proof}
 As in Theorem \ref{t. geproci infinite class}, let $u$ be a primitive $n$-th root of unity, for $n \ge 3.$ 
 We claim that the grid $X$ is projectively equivalent to a grid of real points. This will be enough to prove the statement. Indeed, each point in $Y_1\cup Y_2$ is an intersection of planes through three noncollinear points of $X$.  (For instance, the point $[1:0:0:-u^t]$ is the intersection of the planes defined by the linear forms $y-z$, $y-uz$ and $u^tx+w$. Each of these planes contains (at least) three noncollinear points of $X$.)
Thus, after proving the claim, the automorphism that maps $X$ into a grid of real points, will send the points of $Y_1$ and $Y_2$ into real points. 
 
The case $n=3$ holds since any two nondegenerate $(3,3)$-grids are projectively equivalent, see Lemma \ref{3by3GridUniqueLem}. So, assume $n\ge 4$.

In order to construct a grid of real points which is projectively equivalent to $X$, using the Segre embedding\index{Segre embedding}, it will be enough to show that $\{[1:1], \ldots, [1:u^{n-1}]\}$ is projectively equivalent to a set of $n$ real points $\{P_0, \ldots, P_{n-1} \}$.
 	
To do this, we set $P_0=[0:1], P_1=[1:0], P_2=[1:1]$ and we iteratively construct the points $P_3, \ldots, P_{n-1}$  by imposing that the cross ratio of $P_i,P_{i+1},P_{i+2},P_{i+3}$ is equal to $1+ \dfrac{1}{2\mathfrak R (u)+1}$ as in Lemma~\ref{l. cr u}. 
Since this number is real, by Remark \ref{r. cr and autom} we are done.
  \end{proof}

\begin{example}
In the case $n=5$ we can take  
$u= \cos\dfrac{2\pi}{5}+i\sin\dfrac{2\pi}{5}$  
as a primitive fifth root of unity. Hence, by Lemma \ref{l. cr u} the cross ratio\index{cross ratio} of the points $([1:1], [1:u], [1:u^{2}], [1:u^{3}])$ is  the golden ratio\index{golden ratio}  $\varphi=\dfrac{1+\sqrt{5}}{2}.$ So, the proof of Proposition \ref{p. real-izable}, in the case of the $(5,6)$-geproci set constructed in Theorem \ref{t. geproci infinite class}, will give the following five real points:   $$P_0=[0:1],\  P_1=[1:0],\ P_2=[1:1],\ P_3=[1:\varphi],\  P_4=[1:1+\varphi].$$

In the case $n=6$ the real part of a primitive sixth root of unity is the rational number $1/2$. Thus, the standard $(6,8)$-geproci set can be written, up to a projective transformation, using integer coordinates. For instance,  using the procedure in the proof of Proposition \ref{p. real-izable}, we get
$$\begin{array}{c}
    \left\{ \{[0:1],[1:0],[1:1],[2:3],[1:2],[1:3]\}\otimes \{[0:1],[1:0],[1:1],[2:3],[1:2],[1:3]\}\right\} 
   \   \cup \\[2pt] 
 \{     [1:  3:  0:  3],  [1:  0:  3:  3], [0:  1:  -1: 0], [1:  2:  1:  3],[1:  3:  3:  6], [2:  3:  3:  6]\} \ \cup \\[2pt]
\{  [1:  1:  1:  0], [-1: 0: 0: 3 ],  [0:  1:  1:  3], [1:  1:  2:  3],[1:  2:  2:  3],[2:  3:  3:  3]\}.  \\
\end{array}$$
The Segre product $\otimes$ is defined in Section \ref{sec:Segre_Embeddings}.

\end{example}

 In the next proposition we show that the extensions of a geproci set  as in Chapter  \ref{chap.extending} cannot be realized over the real numbers. As a consequence we derive the promised nonrealizability in $\PP_{\mathbb R}^3$ of the Klein configuration. 

\begin{proposition}\label{p. extension is not real}
	Let $Z=X\cup Y_1$ be an $(n,n+1)$-geproci set as in Theorem \ref{t. geproci infinite class}. Then, using the notation as in Proposition \ref{p.add points to geproci}, the $(n+1,n+1)$-geproci set $\widetilde Z$ is not realizable in~$\PP_{\mathbb R}^3$.
\end{proposition}
\begin{proof}
The line $\ell_1$ containing $Y_1$ has equation $y=z=0$, thus it meets the quadric $\calq\colon xw-yz=0$ in the points $[1:0:0:0]$	and $[0:0:0:1]$.
From Proposition \ref{p.add points to geproci}, the set $\widetilde Z$ contains an $(n,n+1)$-grid given by the Segre product of the two sets of points  $\{[1:1],\ldots,[1:u^{n-1}],[1:0]\}$ and  $\{[1:1],\ldots,[1:u^{n-1}]\},$ 
where $u$ is a primitive $n$-th root of unity.
Then computing the cross ratio of $([1:0],[1:1],[1:u],[1:u^2])$ as in equation \eqref{eq. cross ratio}, we get
\[
\dfrac{(u^2-1)u}{-(u+1)}=-(u+1).
\]
This is not a real number, so from Remark \ref{r. cr and autom} we are done. 
\end{proof}

\begin{remark}
Given a $(n,n+2)$-geproci set $Z=X\cup Y_1\cup Y_2$ as in Theorem \ref{t. geproci infinite class},
it is an immediate consequence of Proposition \ref{p. extension is not real} that also the $(n+2,n+2)$-geproci set $\widetilde Z$ (as in Proposition \ref{p.add points to geproci2}) is not realizable in $\PP_{\mathbb R}^3$.
\end{remark}

\begin{corollary}
	The Klein configuration\index{configuration! Klein} is not realizable in $\PP_{\mathbb R}^3$.
\end{corollary}
\begin{proof}
	It follows from Example \ref{ex.from F4 to Klein} that the Klein configuration contains the $(6,6)$-geproci set $\widetilde Z$ which is not realizable in $\PP_{\mathbb R}^3$  by Proposition  \ref{p. extension is not real}.
\end{proof}


\begin{remark}
An alternative approach to prove the sets $\widetilde Z$ in Proposition \ref{p. extension is not real} are  not projectively equivalent to sets of real points follows from the Sylvester-Gallai Theorem. 
	
Let $Z=X\cup Y_1$ be an $(n,n+1)$-geproci set as in Theorem \ref{t. geproci infinite class}. Then, using the notation as in Proposition~\ref{p.add points to geproci}, consider the $(n+1,n+1)$-geproci set $\widetilde Z$.

Let  $L$ and $M$ be the grid lines of $\calq \colon xw-yz=0$ through the point $P=[1:0:0:0]\in \widetilde Z$.

Let $\pi$ be the projection from $P$ into the plane defined by the equation $x=0$. 

Then, $\pi$ maps the lines $L,M$ and $\ell_1:y=z=0$  into three points
$$A=\pi(L),\ B=\pi(M),\ C=\pi(\ell_1).$$

 Call $\lambda_1,\ldots, \lambda_n$ and $\mu_1,\ldots, \mu_n$ the projection of the lines whose union contains the grid $X$. Note that both $\{\lambda_1,\ldots, \lambda_n\}$ and $\{\mu_1,\ldots, \mu_n\}$ are sets of concurrent lines. Indeed, they concur in $A$ and $B$, respectively.

Let $u$ be an $n$-th root of unity and let $n_0\in \{0,\ldots, n-1\}$. Note that the set of points 
$$\{[1:u^r:u^s:u^{r+s}]\in X\ | \ r-s\equiv n_0 \mod n  \}$$ consists of $n$ points on the plane $N_{n_0}=y-u^{r-s}z=0$ which also contains the line $\ell_1$ and then $Y_1$ and $C$. So $\pi(N_{n_0})$ is a line, call it $\nu_{n_0}$. 
Then the lines $\nu_0, \ldots, \nu_{n-1}$ are concurrent in~$C$.

The configuration of the $3n$ lines, $\{\lambda_i\}\cup \{\mu_j\}\cup \{\nu_k\}$, is  known as a Fermat configuration.\index{configuration! Fermat} It has the property that any point where two of the lines intersect is on a third line.
Thus the Fermat configuration cannot be realized over the real plane by the dual version of the Sylvester-Gallai Theorem,\index{Theorem! Sylvester-Gallai} see Theorem~\ref{thm:SG dual}. 
\end{remark}

We illustrate the alternative approach with an example. 
\begin{example}\label{ex.(4,4)-geproci not real}\index{configuration! extended $D_4$} We consider the  configuration $Z=D_4$ in the coordinates of the standard construction. 
So, letting $\varepsilon$ be a primitive third root of unity, we have $Z=X\cup Y_1$ where
\[
X= \begin{array}{ccc}
\	[1:1:1:1], & [1:\varepsilon:1:\varepsilon], &[1:\varepsilon^2:1:\varepsilon^2],\\
\	[1:1:\varepsilon:\varepsilon], & [1:\varepsilon:\varepsilon:\varepsilon^2], &[1:\varepsilon^2:\varepsilon:1],\\
\	[1:1:\varepsilon^2:\varepsilon^2], & [1:\varepsilon:\varepsilon^2:1], &[1:\varepsilon^2:\varepsilon^2:\varepsilon].\\
\end{array}
\]  
\[Y_1=[1:0:0:-1], \ [1:0:0:-\varepsilon], \ [1:0:0:-\varepsilon^2].
\]

From Proposition \ref{p.add points to geproci}, the set of points
\[
\widetilde Z= \begin{array}{cccc}
	\	[1:1:1:1], & [1:\varepsilon:1:\varepsilon], & [1:\varepsilon^2:1:\varepsilon^2], &[1:0:1:0],\\
	\	[1:1:\varepsilon:\varepsilon], & [1:\varepsilon:\varepsilon:\varepsilon^2], &[1:\varepsilon^2:\varepsilon:1],&[1:0:\varepsilon:0],\\
	\	[1:1:\varepsilon^2:\varepsilon^2], & [1:\varepsilon:\varepsilon^2:1], &[1:\varepsilon^2:\varepsilon^2:\varepsilon] &[1:0:\varepsilon^2:0],\\
\	[1:0:0:-1], & [1:0:0:-\varepsilon],& [1:0:0:-\varepsilon^2]& [1:0:0:0].
\end{array}
\]
is a $(4,4)$-geproci containing $Z$.

Let  $L$ and $M$ be the rulings of $\calq\colon xw-yz=0$ through the point $P=[1:0:0:0]\in \widetilde Z$. 
Let $\pi$ be the projection from $P$ into the plane defined by the equation $x=0$.  Then, $\pi$ maps the lines $L,M$ and $\ell_1$  into three points
$$A=\pi(L)=[0:1:0],\ B=\pi(M)=[1:0:0],\ C=\pi(\ell_1)=[0:0:1].$$ Moreover, 
\[
\pi(X)= \begin{array}{ccc}
	\	[1:1:1], & [1:\varepsilon:1], &[1:\varepsilon^2:1],\\
	\	[1:1:\varepsilon], & [1:\varepsilon:\varepsilon], &[1:\varepsilon^2:\varepsilon],\\
	\	[1:1:\varepsilon^2], & [1:\varepsilon:\varepsilon^2], &[1:\varepsilon^2:\varepsilon^2].\\
\end{array}
\]  
The nine lines represented in Figure \ref{fig:Hesse-dual} show the $n=3$ Fermat configuration also known as the  dual of the Hesse configuration\index{configuration! Hesse}, see \cite[Proof of Theorem 1]{triple}.  
	
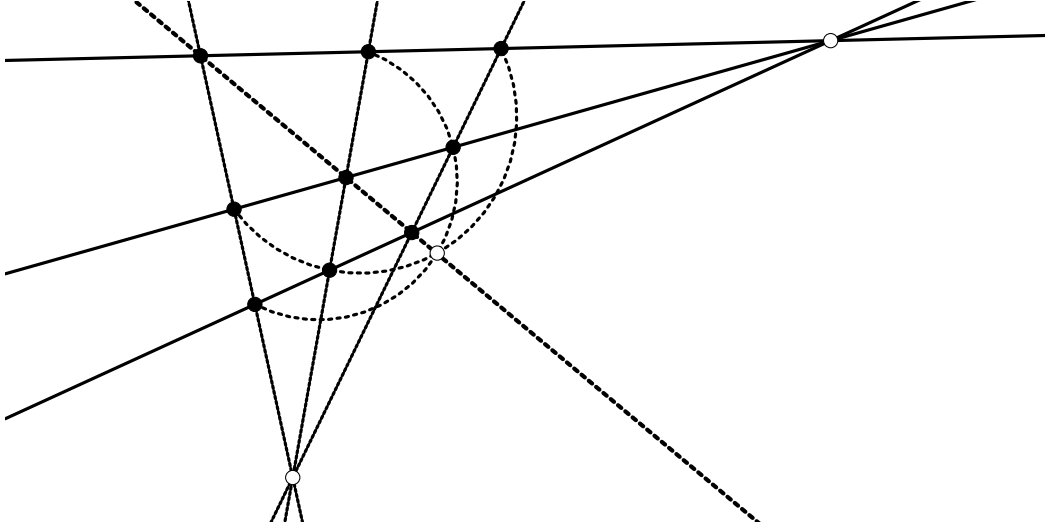
\begin{figure}[!ht]
\definecolor{wrwrwr}{rgb}{0,0,0}
\begin{tikzpicture}[line cap=round,line join=round,>=triangle 45,x=1cm,y=1cm,scale=0.6]
\clip(-8.408065145359737,-7.477438217719601) rectangle (14.62783274151729,4.034521152685364);
\draw [line width=1.2pt,color=wrwrwr,domain=-8.408065145359737:14.62783274151729] plot(\x,{(--19.5672--0.16*\x)/6.66});
\draw [line width=1.2pt,color=wrwrwr,domain=-8.408065145359737:14.62783274151729] plot(\x,{(--12.578429259392365-4.255436928095541*\x)/-9.28256213197689});
\draw [line width=1.2pt,dash pattern=on 1pt off 1pt,color=wrwrwr,domain=-8.408065145359737:14.62783274151729] plot(\x,{(--19.071812686451594--5.513128560332182*\x)/-1.2048421970787703});
\draw [line width=1.2pt,dash pattern=on 1pt off 1pt,color=wrwrwr,domain=-8.408065145359737:14.62783274151729] plot(\x,{(--4.5864-4.08*\x)/-1.98});
\draw [line width=1.2pt,dash pattern=on 1pt off 1pt,color=wrwrwr,domain=-8.408065145359737:14.62783274151729] plot(\x,{(--4.048719626193839--4.59762587432516*\x)/0.8158431568040476});
\draw [line width=1.6pt,dotted,color=wrwrwr,domain=-8.408065145359737:14.62783274151729] plot(\x,{(-2.7024-3.92*\x)/4.68});
\draw [line width=1.2pt,color=wrwrwr,domain=-8.408065145359737:14.62783274151729] plot(\x,{(-4.095814596407697-3.0359929640492695*\x)/-10.738428905379072});
\draw [shift={(-1.464439428321043,0.059202957337566406)},line width=1.2pt,dotted,color=wrwrwr]  plot[domain=-2.047403252222831:1.2037028126211772,variable=\t]({1*3.0750222527464834*cos(\t r)+0*3.0750222527464834*sin(\t r)},{0*3.0750222527464834*cos(\t r)+1*3.0750222527464834*sin(\t r)});
\draw [shift={(-0.5381339711908228,1.4845888902746371)},line width=1.2pt,dotted,color=wrwrwr]  plot[domain=-2.5101652530703195:0.4523844930921552,variable=\t]({1*3.466876157836231*cos(\t r)+0*3.466876157836231*sin(\t r)},{0*3.466876157836231*cos(\t r)+1*3.466876157836231*sin(\t r)});
\begin{scriptsize}
\draw [fill=black] (-4.08,2.84) circle (4.5pt);
\draw [fill=black] (2.58,3) circle (4.5pt);
\draw [color=black,fill=white] (9.88256213197689,3.1754369280955412) circle (4.5pt);
\draw [fill=black] (0.6,-1.08) circle (4.5pt);
\draw [fill=black] (-2.8751578029212297,-2.6731285603321817) circle (4.5pt);
\draw [color=black,fill=white] (-2.036180423111641,-6.5121293567148974) circle (4.5pt);
\draw [fill=black] (-1.2203372663075935,-1.9145034823897377) circle (4.5pt);
\draw [fill=black] (-0.360801296815344,2.9293501189954267) circle (4.5pt);
\draw [fill=black] (-0.8558667734021816,0.13944396404627166) circle (4.5pt);
\draw [fill=black] (1.5174108704498912,0.8104223997149271) circle (4.5pt);
\draw [fill=black] (-3.3365469677869215,-0.561899564189583) circle (4.5pt);
\draw [color=black,fill=white] (1.1649956682013782,-1.535109724686316) circle (4.5pt);
\end{scriptsize}
\end{tikzpicture}
	\caption[The dual Hesse configuration.]{The lines for the $n=3$ Fermat configuration from Example \ref{ex.(4,4)-geproci not real}. The open dots represent the points $A,B,C$.}
	\label{fig:Hesse-dual}
\end{figure}
\end{example}

\section{Projective and combinatorial equivalence of geproci sets}\label{sec.how many}
It is interesting to ask Terao type questions.
Recall that Terao's Conjecture\index{Terao's Conjecture} concerns whether certain properties of hyperplane arrangements
are determined combinatorially. Likewise, we can ask what properties of geproci sets 
are determined combinatorially.

Two subsets of $\PP^n$ are projectively equivalent\index{equivalent! projectively} if they are the same up to change of coordinates
(i.e., there is a linear automorphism of $\PP^n$ taking one set to the other).
We consider also a weaker notion of equivalence between projective sets.
We say two finite sets $Z_1$ and $Z_2$ of $\PP^n$ 
are combinatorially equivalent (or share the same
combinatorics) if they have the same linear dependencies; i.e., the 
same collinearities, coplanarities, etc. 
This can be stated formally as follows,
in terms of the linear span $\langle S\rangle$ of a subset $S\subseteq \PP^n$.

\begin{definition}\label{DefCombEq}
We say two finite subsets $Z_1$ and $Z_2$ of $\PP^n$ are {\it combinatorially equivalent}\index{equivalent! combinatorially} (or {\it share the same
combinatorics})
if there is a bijection $f\colon Z_1\to Z_2$ such that, whenever $Z\subseteq Z_1$, then
$f$ induces a bijection from $\langle Z\rangle\cap Z_1$ to $\langle f(Z)\rangle\cap Z_2$.
\end{definition}

We now show this is equivalent to having a bijection 
$f\colon Z_1\to Z_2$ such that $\dim \langle Z\rangle = \dim \langle f(Z)\rangle$
for every nonempty subset $Z\subseteq Z_1$.

\begin{proposition}\label{p. comb equiv}
Consider two finite subsets $Z_1$ and $Z_2$ of $\PP^n$.
Let $f\colon Z_1\to Z_2$ be a bijection.
Then
$f$ induces a bijection $\langle Z\rangle\cap Z_1$ to $\langle f(Z)\rangle\cap Z_2$
for every nonempty subset $Z\subseteq Z_1$ if and only if
$\dim \langle Z\rangle = \dim \langle f(Z)\rangle$ holds
for every nonempty subset $Z\subseteq Z_1$.
\end{proposition}

\begin{proof}
Assume $f\colon Z_1\to Z_2$ is a bijection inducing
a bijection from $\langle Z\rangle\cap Z_1$ to $\langle f(Z)\rangle\cap Z_2$
for each nonempty subset $Z\subseteq Z_1$.
Thus $p\in \langle Z\rangle\cap Z_1$ if and only if $f(p)\in \langle f(Z)\rangle\cap Z_2$.

Clearly, for every nonempty subset $Z\subseteq Z_1$ with $|Z|<3$ we have $\dim \langle Z\rangle = \dim \langle f(Z)\rangle$.
So, if $\dim \langle Z\rangle = \dim \langle f(Z)\rangle$ does not always hold,
there is a subset $Z\subseteq Z_1$ of minimum cardinality such that 
$\dim \langle Z\rangle \neq \dim \langle f(Z)\rangle$, and such a $Z$ has at least 3 points.
Let $p\in Z$ and let $Y=Z\setminus\{p\}$.
By minimality, we have $\dim \langle Y\rangle = \dim \langle f(Y)\rangle$, 
so either $\dim \langle Z\rangle < \dim \langle f(Z)\rangle$ and hence 
$p\in \langle Y\rangle$ but $f(p)\not\in \langle f(Y)\rangle$,
or $\dim \langle Z\rangle > \dim \langle f(Z)\rangle$ and hence 
$p\not\in \langle Y\rangle$ but $f(p)\in \langle f(Y)\rangle$.
Either way we have a contradiction.

Conversely, assume for every nonempty subset $Z\subseteq Z_1$ that we have 
$\dim \langle Z\rangle = \dim \langle f(Z)\rangle$.
Assume for some $Z$ that $f$ does not give a bijection from
$\langle Z\rangle\cap Z_1$ to $\langle f(Z)\rangle\cap Z_2$;
we may assume $Z$ has minimal cardinality.
Clearly, for every point $p\in Z_1$ we have $\langle p\rangle\cap Z_1=\{p\}$ and
$\langle f(p)\rangle\cap Z_2=\{f(p)\}$, so $|Z|\geq 2$.
Let $p\in Z$ and let $Y=Z\setminus\{p\}$, so $Y$ is nonempty and by minimality
$f$ gives a bijection from $\langle Y\rangle\cap Z_1$ to $\langle f(Y)\rangle\cap Z_2$.
Since $f$ does not give a bijection from $\langle Z\rangle\cap Z_1$ to $\langle f(Z)\rangle\cap Z_2$,
we must have $p\not\in\langle Y\rangle$ and $f(p)\not\in \langle f(Y)\rangle\cap Z_2$.
Now, either $|\langle Z\rangle\cap Z_1|>|\langle f(Z)\rangle\cap Z_2|$
or 
$|\langle Z\rangle\cap Z_1|<|\langle f(Z)\rangle\cap Z_2|$.

Say $|\langle Z\rangle\cap Z_1|>|\langle f(Z)\rangle\cap Z_2|$.
Then there is a point $q\in (\langle Z\rangle\cap Z_1)\setminus Z$ such that $f(q)\not\in \langle f(Z)\rangle\cap Z_2$.
But this would mean that $\dim \langle (Z\cup\{q\})\rangle\cap Z_1=\dim \langle Z\rangle\cap Z_1$
while $\dim \langle (f(Z)\cup\{f(q)\})\rangle\cap Z_2=1+\dim \langle f(Z)\rangle\cap Z_2
=1+\dim \langle Z\rangle\cap Z_1=1+\dim \langle (Z\cup\{q\})\rangle\cap Z_1$,
and this is a contradiction.

The case that $|\langle Z\rangle\cap Z_1|<|\langle f(Z)\rangle\cap Z_2|$ is the same;
just use the inverse bijection.
\end{proof}
The following is an immediate consequence of Proposition \ref{p. comb equiv}.
\begin{corollary}\label{c. comb equiv cor}
Let $Z_1$, $Z_2$ be combinatorially equivalent finite sets. Then for each $d$ and $r$ the number of maximal subsets of $Z_1$ of cardinality $r$ having span of dimension $d$ is the same as for $Z_2$.
\end{corollary}

\begin{remark}\label{ProjEqVsCombEq}\

\begin{enumerate}
\item[(a)] Two finite subsets of $\PP^1$ have the same combinatorics if and only if they have the same cardinality.
Thus sharing the same combinatorics is very different from being projectively equivalent, since
subsets of 4 or more points of $\PP^1$, up to projective equivalence, have positive dimensional moduli.
(For sets of 4 points in $\PP^1$, projective equivalence is determined by the cross ratio.\index{cross ratio})

\item[(b)] Things are more subtle in higher dimensions. Any set $Z_1$ of
6 points on an irreducible plane conic has the same combinatorics as a set $Z_2$ of 
6 general points of $\PP^2$, but $Z_1$ and $Z_2$ are not projectively equivalent.
In fact, although two sets of 6 points on an irreducible conic
always have the same combinatorics, they
need not be projectively equivalent.
For example, 6 general points of an irreducible conic are not projectively equivalent to six
points obtained from a transverse intersection
of the conic with three concurrent lines.

\item[(c)] Moreover, although any six points on an irreducible conic 
give a complete intersection of type $(2,3)$,
not all complete intersections of type $(2,3)$ have the same combinatorics,
and thus two degenerate $(2,3)$-geproci sets need not share the same combinatorics. 
For example, a degenerate $(2,3)$-geproci set is the transverse intersection of a plane conic 
with a plane cubic, but the conic could either be irreducible or not.

\item[(d)] In contrast, for $2\leq a\leq 3$, two $(a,3)$-grids are always projectively equivalent
(see Lemma \ref{3by3GridUniqueLem} for the case of $a=3$).
However, two $(a,b)$-grids are not always projectively equivalent when $b>3$
(even if they share the same combinatorics, an easy case of which is given by
$(2,b)$-grids),
since two sets of $b>3$ points on a line need not be projectively equivalent.
 Two $(a,b)$-grids do not even always have the same combinatorics.
For example, consider a $(4,4)$-grid. It is easy to construct such a grid for which there are no planes that contain
exactly four points of the grid but no grid lines; it is also easy to construct $(4,4)$-grids where there are such planes.

\item[(e)] Two sets of the same cardinality in LGP are clearly combinatorially equivalent. But certainly they need not be projectively equivalent.
They can be distinguished also for other properties. For example, 6 points in $\PP^2$ are in LGP if and only if the 15 pairs of points
define distinct lines, regardless of whether any three of those lines are concurrent.

\end{enumerate}
\end{remark}
In Chapter \ref{chap.Geography} we saw that any two nontrivial $(3,4)$-geproci sets
are projectively equivalent.  In the next remark, we see that 
two nontrivial $(4,4)$-geproci sets need not share the same combinatorics.

\begin{remark}\label{r.two different (4,4)-geproci}

We use the details from Remark \ref{r.substantially different} to show now that there are at least two nontrivial $(4,4)$-geproci 
sets which do not share the same combinatorics. So, they are also not projectively equivalent. 
Indeed, one of them has exactly four lines containing four points. The other has five such lines.
Thus they do not have the same combinatorics by Corollary \ref{c. comb equiv cor}.
\end{remark}

We do not know if being a nondegenerate 
geproci set is a combinatorial property.
\begin{question}\label{quest. ce geproci}
If $Z$ and $Z'$ are combinatorially equivalent and $Z$ is a nondegenerate geproci set, must $Z'$ also be a nondegenerate geproci set?  
\end{question}
 The answer is yes if $Z$ is a grid, since, as the next result shows, being a grid is a combinatorial property.

\begin{proposition}\label{CombEqToGridIsGrid}
Let $Z$ and $Z'$ share the same combinatorics. If $Z'$ is an $(a,b)$-grid, then $Z$ is an $(a,b)$-grid.
\end{proposition}

\begin{proof}
Let $H'_1,\ldots,H'_a$ and $V'_1,\ldots,V'_b$ be the grid lines of $Z'$.
Let $z'_{ij}=H'_1\cap V'_j$ be the points of the grid.
Let $f\colon Z'\to Z$ be a bijection with respect to which  
$Z$ and $Z'$ are combinatorially equivalent, and let $z_{ij}=f(z'_{ij})$.
Then $H_i=\{z_{i1},\ldots,z_{ib}\}$ are collinear sets of points,
and likewise for $V_j=\{z_{1j},\ldots,z_{aj}\}$.
Moreover the $H_i$ are pairwise disjoint 
since $\dim \langle H_i\cup H_j\rangle = \dim \langle H'_i\cup H'_j\rangle=3$
whenever $i\neq j$ (and likewise for the $V_j$).
Thus $Z$ is a grid with grid lines $H_i, V_j$ and grid points $z_{ij}$.
\end{proof}


We saw in Remark \ref{r.two different (4,4)-geproci} that nontrivial $(a,b)$-geproci sets
need not share the same combinatorics, but, in contrast with grids,
we know of no examples of nontrivial geproci sets which share the same combinatorics
but are not projectively equivalent.

Furthermore, whereas $(a,b)$-grids have positive dimensional moduli, i.e., there are continuous families of nonprojectively equivalent grids,  we know no examples of  pairs $(a,b)$ for which
there is an infinite number of nonprojectively equivalent nontrivial  $(a,b)$-geproci sets.
(By Remark \ref{r.two different (4,4)-geproci}, we do know that nontrivial  $(a,b)$-geproci sets
do not all share the same combinatorics, but the examples of this we know have only finitely many for a given $a$ and~$b$.)

 
The following result follows immediately by Remark \ref{r.substantially different}  and Corollary \ref{c. comb equiv cor}.
 \begin{corollary}\label{c. not combinatorially geproci}
For any $4\le a\le b$  there exist (at least) two  nontrivial $(a,b)$-geproci sets which are not combinatorially equivalent (hence they are not projectively equivalent).
\end{corollary}

\begin{proposition}\label{p. two non-isomorphic}
If $6 \leq a \leq b$, with $b$ even,  then there exist two nontrivial $(a,b)$-geproci sets constructed in Theorem \ref{t. (a,b)-geproci}
which are not projectively equivalent.
\end{proposition} 
\begin{proof} 
Using Theorem \ref{t. geproci infinite class}, since $b$ is even we know that we can construct a $(b+2,b)$-geproci set (not a grid) consisting of a $(b,b)$-grid and two additional sets of $b$ collinear points. Following the approach of Theorem \ref{t. (a,b)-geproci}, we can take away $b+2-a$ subsets of $b$ collinear points, in one of the two following ways, to produce an $(a,b)$-geproci set. If all of the removed sets come from the grid, the result will consist of at least three grid rows, lying on the unique quadric $\calq$ of the grid, and two rows not on $\calq$. On the other hand, if all but one of the removed sets come from the grid and the last one does not, then the result will consist of all but one row on $\calq$ and the last row not on $\calq$. These two outcomes are not projectively equivalent because in the second case we have a set of $a-1$ grid lines contained in a quadric, which does not exist in the first case.
\end{proof}

In the next proposition we consider the particular case where $b$ is odd and $b=a+1$. 

\begin{proposition}\label{p. two non-isomorphic2}
If $4 \leq a $, with $a$ even,  then there exist two nontrivial $(a,a+1)$-geproci sets constructed in Theorem \ref{t. (a,b)-geproci}
which are not combinatorially equivalent (and hence they are not projectively equivalent).
\end{proposition}
\begin{proof}
The first $(a,a+1)$-geproci set is $X\cup Y_1$ in the standard construction for $a$ even. 
In order to construct the other $(a,a+1)$-geproci set, we start from the $(a+1,a+2)$-geproci set in the standard construction of Theorem \ref{t. geproci infinite class} and remove two sets of $a+1$ collinear points as in the proof of Theorem \ref{t. (a,b)-geproci}. 
The second set has subsets of $a+1$ collinear points which are not present in the first set. Then the two constructed sets are not combinatorially equivalent by Corollary \ref{c. comb equiv cor}.
\end{proof}

\begin{remark}
Collecting the results in this section we obtain
\begin{itemize}
    \item[(i)] at least three nonprojectively equivalent nontrivial geproci sets in the cases covered by Proposition \ref{p. two non-isomorphic};
    \item[(ii)] at least three noncombinatorially equivalent nontrivial geproci sets in the cases covered by Proposition \ref{p. two non-isomorphic2};
     \item[(iii)]  at least five nonprojectively equivalent nontrivial $(6,10)$-geproci sets: two from the standard construction, one from the extension of the standard construction, the Klein configuration and the $H_4$ configuration.
\end{itemize}
\end{remark}

We can consider a {\it weaker} version of Definition \ref{DefCombEq}  by looking at collinear points only.
\begin{definition}\label{d.weakly combinatorially equivalent}
We say two finite subsets $Z_1$ and $Z_2$ of $\PP^n$ are {\it weakly combinatorially equivalent}\index{equivalent! weakly combinatorially} 
if there is a bijection $f\colon Z_1\to Z_2$ such that, whenever $Z\subseteq Z_1$ and $\dim\langle Z\rangle=1$,  then
$f$ induces a bijection from $\langle Z\rangle\cap Z_1$ to $\langle f(Z)\rangle\cap Z_2$.
\end{definition}
As in Remark \ref{r.two different (4,4)-geproci}, two $(a,b)$-grids, where $b\ge 4$, do not necessarily have the same combinatorics, but they are weakly combinatorially equivalent.
So, the weaker definition leads to the following stronger questions.
\begin{question}\label{q. weak comb equiv} Let $Z$ be a nontrivial $(a,b)$-geproci set and let $Z'$ be weakly combinatorially equivalent to $Z$. 
Is $Z'$ an $(a,b)$-geproci set?
\end{question}
\begin{question}\label{q.ce->pe}
If two nontrivial geproci sets are weakly combinatorially equivalent, must they be projectively equivalent?
\end{question}

One way to address Question \ref{q.ce->pe} could be to start from a geproci set in the standard construction and, in order to get not projectively equivalent $(a,b)$-geproci sets, to apply Theorem \ref{t. (a,b)-geproci} by removing different sets of collinear points.
We illustrate this idea in the following remark, in the case of $(4,5)$-geproci sets, which however does not give a negative answer to Question \ref{q.ce->pe} since the constructed sets are not even weakly combinatorially equivalent, see Remark \ref{r. three (5,6)-geproci not w c eq}.
\begin{remark}\label{r. three (5,6)-geproci}
Let $Z$ be the $(6,8)$-geproci set in the standard construction, see Theorem \ref{t. geproci infinite class}, i.e., $Z$  is the set of points whose coordinates are
$\{[1:1],\ldots,[1:u^5]\} \otimes \{[1:1],\ldots,[1:u^5]\} \cup Y_1 \cup Y_2$, where $u$ is a primitive $6$-th root of  unity. 

Consider the following three $(5,6)$-geproci sets contained in $Z$ \[\begin{array}{rl}
    Z_1:& \{[1:1], [1:u\;], [1:u^2], [1:u^3]\} \otimes \{[1,1],\ldots,[1,u^5]\} \cup Y_1;\\
    Z_2:&\{[1:1], [1:u\;], [1:u^3], [1:u^4]\}\otimes
    \{[1,1],\ldots,[1,u^5]\} \cup Y_1;\\
    Z_3:& \{[1:1], [1:u^2],[1:u^3],[1:u^4]\}\otimes
    \{[1:1],\ldots,[1:u^5]\} \cup Y_1.\\
\end{array}\]
They are not projectively equivalent since the three sets of 4 points in the above first copy of $\mathbb P^1$ have different cross ratios.
Now we describe what we found experimentally.

The following table summarizes the number of lines containing exactly $n$ points of $Z_i$.
\[\begin{array}{c|ccccc}
        n & 2 & 3 & 4 & 5 & 6 \\
         \hline
        Z_1 & 216& 36& 6& 0& 5 \\ 
        Z_2 & 216& 36& 6& 0& 5 \\ 
        Z_3 & 216& 36& 6& 0& 5 \\ 
\end{array}
\]
and precisely for each point $P$ in $Z_i$ we get the following table counting the number of lines through $P$ and containing exactly $n$ points of $Z_i$.

\begin{equation}\label{table (5,6)-geproci}
\begin{array}{c|ccccc}
        n & 2 & 3 & 4 & 5 & 6 \\
         \hline
        P\in (4,6)\text{-grid}  & 15 & 3& 1& 0& 1  \\
        P\in Y_1  & 12 & 6& 0& 0& 1  
\end{array}
\end{equation}

\noindent (Given $P$ in the $(4,6)$-grid,  there are $15$ lines containing $P$ and exactly one other point, $3$ lines containing $P$ and exactly other $2$ points, there is $1$ line containing $P$ and exactly 3 points, etc.). 
These tables do not allow us to distinguish $Z_1,$ $Z_2,$ and $Z_3$ from a weakly combinatorial point of view. We will do that in Remark \ref{r. three (5,6)-geproci not w c eq}.

However, the following table gives the number of planes containing exactly $n$ points of $Z_i$.
\[\begin{array}{c|cccccccc}
        n  & 3 & 4 & 5 & 6& 7 & 8 & 9 & 10 \\
         \hline
        Z_1 & 366&  168&  30&  0&  0&  0&  0&  30 \\ 
        Z_2 & 408&  192&  18&  0&  0&  0&  0&  30 \\ 
        Z_3 & 324&  144&  42&  0&  0&  0&  0&  30 \\ 
\end{array}
\]
Thus $Z_1,Z_2$ and $Z_3$ are not combinatorially equivalent.
\end{remark}

\begin{remark}\label{r. three (5,6)-geproci not w c eq} In this remark we show that the three $(5,6)$-geproci sets of Remark  \ref{r. three (5,6)-geproci} are also not weakly combinatorially equivalent.

Let us denote the points of $Z_i$, $i=1,2,3,$ by $1,\ldots,30$, where the following are the five sets of six collinear points and, in particular, the last one is $Y_1$  
 \[
\begin{array}{ccccc}
\{1,  2,  3,  4,  5,  6\}, &
\{7,  8,  9,  10,  11,  12\},& 
\{13,  14, 15,  16,  17,  18\}, &
\{19,  20,  21,  22,  23,  24\},&
\{ 25,  26,  27, 28,  29,  30\}.
\end{array}
\]

Let $i\neq h$. We note that any bijection $f\colon Z_i\to Z_h$,  as in Definition \ref{d.weakly combinatorially equivalent}, has the following properties:
\begin{itemize}
\item[(i)] it maps each of the five sets above into another, i.e., $$f(\{1+6r, 2+6r, \ldots, 6+6r \})=(\{1+6s, 2+6s, \ldots, 6+6s \})$$ for $r,s=0,\ldots, 4$ not  necessarily distinct; and
\item[(ii)]  it preserves the set $Y_1$ (this is clear from table \eqref{table (5,6)-geproci}), i.e., $$f(\{ 25,  26,  27, 28,  29,  30\})=\{ 25,  26,  27, 28,  29,  30\}.$$
\end{itemize}
Thus each 2-point line of $Z_i$ containing a point in $Y_1$ will be mapped by $f$ into a 2-point line of $Z_h$ containing a point in $Y_1$.

 As table \eqref{table (5,6)-geproci} shows, each point in $Y_1$  is contained in twelve 2-point lines, but in order to prove that $Z_1$, $Z_2$ and $Z_3$ are not weakly combinatorially equivalent we need to know more. The following tables have on the left column the points in $Y_1$ and on the right the twelve points which are on a 2-point line with them (these are divided in subsets of three points, each contained in a 6-point line).

\[
Z_1:\begin{array}{lcccc}
25,&\{1,  2,  3\}& \{ 8,  9,  11\}& \{   14,  15,  16\}& \{   20,  21,  24\}\\
26,&\{2,  4,  6\}& \{  9,  10,  12\}& \{  13,  16,  18\}& \{  22,  23,  24\}\\
27,&\{2,  3,  6\}& \{  7,  8,  12\}& \{  14,  17,  18\}& \{   20, 22,  24\}\\
28,&\{1,  4,  5\}& \{  7,  11,  12\}& \{  13,  14,  17\}& \{ 19,  21,  23\}\\
29,&\{1,  3,  5\}& \{  7,  9,  10\}& \{  13,  15,  18\}& \{  19,  20,  21\}\\
30,&\{4,  5,  6\}& \{  8,  10,  11\}& \{  15,  16,  17\}& \{ 19,  22,  23\}
\end{array}
\]

\[
Z_2:\begin{array}{lcccc}
25,&\{1,  3,  4\}& \{  9,  10,  11\}& \{  15,  16,  18\}& \{  20,  21,  22\}\\
26,&\{1,  2,  6\}& \{ 7,  9,  12\}& \{  13,  17,  18\}& \{  19,  22,  24\}\\
27,&\{2,  3,  5\}& \{  7,  8,  11\}& \{  14,  16,  17\}& \{  20,  23,  24\}\\
28,&\{2,  4,  5\}& \{  8,  11,  12\}& \{  14,  15,  17\}& \{  19,  20,  23\}\\
29,&\{1,  5,  6\}& \{  7,  10,  12\}& \{  13,  14,  18\}& \{  19,  21,  24\}\\
30,&\{3,  4,  6\}& \{  8,  9,  10\}& \{  13,  15,  16\}& \{  21,  22,  23\}
\end{array}
\]

\[
Z_3:\begin{array}{lcccc}
25,&\{1,  3,  5\}& \{  9,  10,  11\}& \{  15,  17,  18\}& \{  20,  21,  23\}\\
26,&\{2,  3,  6\}& \{  7,  9,  12\}& \{ 15,  17,  18\}& \{  21,  22,  24\}\\
27,&\{1,  2,  3\}& \{  7,  8,  11\}& \{  13,  14,  16\}& \{  19,  20,  24\}\\
28,&\{4,  5,  6\}& \{ 8,  11,  12\}& \{  15,  17,  18\}& \{  19,  23,  24\}\\
29,&\{1,  4,  5\}& \{ 7,  10,  12\}& \{  13,  14,  16\}& \{ 19,  21,  22\}\\
30,&\{2,  4,  6\}& \{ 8,  9,  10\}& \{  13,  14,  16\}& \{ 20,  22,  23\}

\end{array}
\]

If $Z_i$ and $Z_h$, $i\neq h$,  are weakly combinatorially equivalent then there is a permutation of $1,\ldots, 30$ with the properties in items (i) and (ii) above, which applied to the table of $Z_i$ produces the table of $Z_h$.  

To show that this is not the case, we associate to each table a bipartite graph \index{graph! bipartite} $G_i=(V_i,E_i)$, $i=1,2,3$,  such that the vertex set is $$V_i=\{1, \ldots, 30\}=\{1, \ldots, 24\}\cup \{25, \ldots, 30\}$$
and the edge set is determined by the above 2-point lines, i.e., $$E_i=\{\{c,d\}\ |\ c \in \{25, \ldots, 30\}\ \text{and}\ d\ \text{appears in the row of}\ c\ \text{in the table corresponding to}\ Z_i\}.$$

Thus, if $Z_i$ and  $Z_h$ are weakly combinatorially equivalent then the graphs $G_i$ and $G_h$ are isomorphic.

We first show that $G_3$ is not isomorphic to $G_1$ or $G_2$. Let $H=\{13,  14,  16\} \cup \{27, 29, 30\}$ and consider the subgraph  $G_3|_H$ of the restriction of $G_3$ to $H.$ We note that  $G_3|_H$ is a complete bipartite graph, usually denoted by $K_{3,3}$. An easy check shows that in the tables of $Z_1$ and $Z_2$ none of the 3-sets is repeated, so this implies that both $G_1$ and $G_2$ do not contain $K_{3,3}$ as subgraph. 

Now we show that $G_1$ is not isomorphic to $G_2$. Let now $H=\{1,2,3,4,5,6\}\cup\{ 25,30\}$ and note that the subgraph $G_1|_H$ is the disjoint union of two $K_{1,3}$ graphs. There are not any disjoint $K_{1,3}$ graphs in $G_2$. Indeed, one can check that fixing any two points in $Y_1$, the 3-sets in the same columns always have a nonempty intersection. 
\end{remark}

\begin{question}\label{quest. (5,a)-geproci} Let $a>5$ and consider a standard $(a,a+1)$-geproci set $Z$. Construct $(5,a)$-geproci subsets of $Z$ which are not projectively equivalent as in Remark \ref{r. three (5,6)-geproci}.  Are they also not weakly combinatorially equivalent?  
\end{question}

\subsection{The combinatorics determines the \texorpdfstring{$F_4$}{F4} configuration}\index{configuration! $F_4$}\label{sec.geometryF4} 
The answer to Question \ref{q. weak comb equiv} is positive for $Z={F_4}$. Indeed, as Lemma \ref{l. Justyna's lemma} shows, any set weakly combinatorially equivalent to ${F_4}$ is projectively equivalent to ${F_4}$ and then it is $(4,6)$-geproci.

With the notation in \cite{CM} and \cite{HMNT}, recall that the set ${F_4}$ consists of the points

$${F_4}=\begin{array}{lllll}
P_{1j}:\ &\ [1:0:0:0], & [0:1:0:0], & [1:1:0:0],& [1:-1:0:0],  \\
P_{2j}:\ &\ [0:0:1:0], & [0:0:0:1], & [0:0:1:1],& [0:0:1:-1],  \\
P_{3j}:\ &\ [1:0:1:0], & [0:1:0:1], & [1:1:1:1],& [1:-1:1:-1],  \\
P_{4j}:\ &\ [1:0:-1:0], & [0:1:0:-1], & [1:1:-1:-1],& [1:-1:-1:1], \\
\\
A_{j}:\ &\ [1:0:0:1],&[0:1:-1:0], &[1:1:-1:1], &[1:-1:1:1],\\
B_{j}:\ &\ [1:0:0:-1], &[0:1:1:0],& [1:1:1:-1],& [1:-1:-1:-1].\end{array}$$

In the table above the points are arranged in a $(4,4)$-grid on the quadric $xw-yz=0$ and on two lines, which immediately gives ten subsets of 4 collinear points. However, it was observed in \cite{CM} that ${F_4}$ has 18 sets of 4 collinear points, and each point of $F_4$ lies on three 4-point lines 
(this makes the set a $(24_3, 18_4)$ point-line configuration in $\PP^3$), see the table below. 
{\small \begin{equation}\label{eq.18 lines F4}
\begin{array}{cccccccccccccccccccc}
	\ell_1 & \ell_2 & \ell_3 & \ell_4 & \ell_5 & \ell_6 & \ell_7 & \ell_8 & \ell_9 & \ell_{10} & \ell_{11} & \ell_{12} & \ell_{13} & \ell_{14} & \ell_{15} & \ell_{16} & \ell_{17} & \ell_{18} \\
	P_{11}& P_{21}& P_{31}& P_{41}& P_{11}& P_{12}& P_{13}& P_{14}& A_1 & B_1 & P_{11}& P_{33}&P_{34}&P_{12}&P_{13}& P_{32}&P_{31}&P_{23}\\   
	P_{12}& P_{22}& P_{32}& P_{42}& P_{21}& P_{22}& P_{23}& P_{24}& A_2 & B_2 & P_{22}&
	P_{44}&P_{43}&P_{21}&P_{24}&P_{41}&P_{42}&P_{14}\\
	P_{13}& P_{23}& P_{33}& P_{43}& P_{31}& P_{32}& P_{33}& P_{34}& A_3 & B_3 &    A_1&
	A_1&     A_2& A_2& A_3&A_3&A_4&A_4\\
	P_{14}& P_{24}& P_{34}& P_{44}& P_{41}& P_{42}& P_{43}& P_{44}& A_4 & B_4 &    B_1&
	B_2&     B_1& B_2& B_3&B_4&B_3&B_4\\
\end{array}
\end{equation}}
Moreover, one can also check that there are  32 lines each containing 3 points of ${F_4}$.  
These collinearities are collected in the following table
\begin{equation}\label{eq.24 3-lines F4}
\begin{array}{ccccccccccccccccc}
A_1   &A_1&  A_1&  A_1&  A_2&  A_2& A_2&A_2&  A_3&  A_3&  A_3&  A_3&A_4&A_4&A_4&A_4 \\
P_{24}& P_{13}&P_{14}& P_{23}& P_{13}&  P_{14}& P_{23}& P_{24}&    P_{21}&P_{22}& P_{12}&P_{11}&P_{12}&P_{22}&P_{21}&P_{11}&\\
P_{31}& P_{42}& P_{32}& P_{42}& P_{31}&    P_{41}& P_{32}& P_{42}&  P_{33}&P_{43}&P_{44}&P_{34}&P_{33}&P_{34}&P_{44}&P_{43}&\\
\\ 
B_1 &B_1 &B_1 &B_1 &B_2 &B_2 &B_2 &B_2 &B_3 &B_3 &B_3 &B_3 & B_4&B_4&B_4&B_4&\\
P_{23}&P_{13}&P_{24}&P_{14}&P_{14}&P_{24}&P_{13}&P_{23}&P_{22}&P_{21}& P_{12}&  P_{11}&P_{21}& P_{12}&  P_{22}& P_{11}& \\
P_{31}&P_{32}&P_{41}&P_{42}&P_{31}&P_{32}&P_{41}&P_{42}& P_{33}&P_{43}&P_{34}&P_{44}& P_{34}&P_{43}& P_{44}& P_{33}& \\
\end{array}
\end{equation}

%
%

In the next lemma we show that ${F_4}$ is determined by its 4-point lines.

\begin{lemma}\label{l. Justyna's lemma}
	Let $Z$ be a set of 24 points in $\PP^3$ such that the collinearities in table \eqref{eq.18 lines F4} hold. Assume $Z$ is nondegenerate. 
Then $Z$ is projectively equivalent to ${F_4}.$
\end{lemma}
\begin{proof}
Since $Z$ is nondegenerate, from table \eqref{eq.18 lines F4} it follows that the nine points $P_{ij}$, $1\le i,j\le 3$, cannot be contained in a plane. 

So, without loss of generality, by Lemma \ref{3by3GridUniqueLem}, we can assume 
\[
\begin{array}{cccccc}
          & \ell_5  &    \ell_6 & \ell_7\\
P_{1j}:\ &\ [1:0:0:0], & [0:1:0:0], & [1:1:0:0] & \ell_1\\
P_{2j}:\ &\ [0:0:1:0], & [0:0:0:1], & [0:0:1:1]& \ell_2 \\
P_{3j}:\ &\ [1:0:1:0], & [0:1:0:1], & [1:1:1:1] & \ell_3 \\
\end{array}
\]
We now compute the coordinates of the remaining points by using the table of collinearities.
We start by computing the point $B_4$ as the intersection of the lines $\ell_{16}$ and $\ell_{18}$.
Note that the points $P_{14}$ and $P_{41}$,  being on the line $\ell_1$ and $\ell_5$,  have coordinates $P_{14}=[1:\alpha:0:0]$ and $P_{41}=[1:0:\beta:0]$, respectively, for some $\alpha,\beta\neq 0$. 
Then, depending on the parameters $\alpha,\beta$, we get the following equations   
 $$\ell_{16}: y-w=\beta x-z=0 \ \text{and}\ \ell_{18}:z-w=\alpha x-y=0.$$ So, must be
 \[
 \det \left( 
 \begin{array}{cccc}
 0 & 1 & 0 & -1\\
 \beta & 0 & -1 & 0\\
 0 & 0 & 1 & -1\\
 \alpha & -1 & 0 & 0\\
 \end{array}
  \right)=\alpha-\beta=0.
 \] 
 Hence we get $B_4=[1:\alpha:\alpha:\alpha]$ and, up to the parameter $\alpha$,  $\ell_4: \alpha x-z=y-\alpha w=0$ and $\ell_8: \alpha x-y=z-\alpha w=0$. Therefore,  $P_{24}=[0:0:1:\alpha]$ , $P_{34}=[1:\alpha:1:\alpha]$ , $P_{42}=[0:1:0:\alpha]$, $P_{43}=[1:1:\alpha:\alpha]$, and $P_{44}= [1: \alpha: \alpha: 1].$ 
  Thus
$$B_3=[1:1:1:\alpha] \ \text{and}\ B_2=[0:1:1:0]$$  the former being intersection of $\ell_{15}: x-y=\alpha z-w=0$ and $\ell_{17}: x-z=\alpha y-w=0$ and the latter being intersection of $\ell_{14}: x=w=0$ and $\ell_{12}: x-w=y-z=0$.

The points $B_2,B_3,B_4$ are collinear, hence we get $\alpha=-1$. From now on, it is a matter of computation to  determine the coordinates of all the other points.
\end{proof}


\begin{remark}\label{rem. F4 equivalent to SC}
The above lemma  can be used to check that ${F_4}$ is projectively equivalent to the $(4,6)$-geproci set in the standard construction, see Theorem \ref{t. geproci infinite class} and Example \ref{e:F4exOf4.2b}. 
\end{remark}

Using the results in this chapter we now show that the half Penrose configuration does not have the same combinatorics as the $(4,5)$-geproci sets  constructed either in Section \ref{sec:standard construction} or in Section \ref{s.extending geproci}. 
\begin{remark}\label{r. different (4,5)-geproci}
The following are the $(4,5)$-geproci sets we know from Chapter \ref{chap.Geography} and Chapter \ref{chap.extending}.
\begin{itemize}
    \item[(a)] The set obtained by removing 4 points in one line from the $(4,6)$-geproci set as in the standard construction. We see from table \eqref{eq.18 lines F4} that there are nine lines containing 4 points of the configuration. No line contains 5 points of the configuration.
    \item[(b)] The set obtained by removing 10 points in two lines from the $(5,6)$-geproci set as in the standard construction. In this case there are 4 lines containing 5 points of the configuration. No line contains only 4 points of the configuration.
    \item[(c)] The set obtained by removing 5 points in one line from the $(5,5)$-geproci set from the extension of the $(4,5)$-geproci set  standard construction. In this case there are 4 lines containing 5 points of the configuration. One line contains only 4 points of the configuration.
    \item[(d)] The half Penrose configuration,\index{configuration! half Penrose} see Corollary \ref{c. half Penrose geproci}. In this case no line contains 5 points. We see from table \eqref{table Penrose} that there are ten lines containing 4 points of the configuration. Indeed,  the twenty points of the configuration are indexed by $$\{1,  2, 3, 5, 7, 8,
11, 14, 15, 17,
29,
30, 32, 33, 35, 36, 37, 38, 39,
40\}$$ 
and the ten sets of five collinear points are indexed by
  \[
\begin{array}{ccccccc}
\{1, 2, 32, 35\}&
\{1, 37, 38, 39\}&
\{2, 7, 15, 17\}&
\{3, 7, 29, 37\}&
\{3, 35, 36, 40\}\\
\{5, 11, 29, 33\}&
\{5, 15, 30, 36\}&
\{8, 14, 30, 32\}&
\{8, 17, 33, 38\}&
\{11, 14, 39, 40\}.\\
\end{array}
\]
\end{itemize}
Thus the four $(4,5)$-geproci sets listed above do not have the same combinatorics by Corollary \ref{c. comb equiv cor}.  
\end{remark}

\chapter{Unexpected cones} \label{ch: unexp hypersurf}\index{unexpected!cone}

The notion of unexpected hypersurfaces started with curves in the plane. The authors of \cite{DIV} observe the existence of a quartic curve passing through specific $9$-point configuration $B_3$\index{configuration! $B_3$} in $\PP^2$ and a general triple point which is not expected, i.e., that a general triple point does not impose independent conditions on the ideal of $B_3$ in degree four. The theory of unexpected curves is developed in \cite{CHMN}; and \cite{FGST} proves that the configuration $B_3$ is the unique configuration of  points in $\PP^2$ which admits this kind of unexpected quartic. The authors of \cite{HMNT} expand the definition of an unexpected curve to an unexpected hypersurface of any dimension.

The results of the paper \cite{CM} suggest there is a strong connection between geproci sets and unexpected cones. In this chapter we explore these connections and we explore additional occurrences of unexpectedness.  

\section{Geproci sets and unexpected cones}\label{sec:geproci and unexpected}
Let $Z\subset\PP^3$ be a finite set of points.
Note that for any point $P$, $\dim [R/I(tP)]_t=\binom{t+2}{3}$.
Following \cite{CM}, we will say that $Z$ satisfies $C(t)$ 
if for a general point $P$ we have
$$\dim [I(Z+tP)]_t >\max\left(0, \dim [I(Z)]_t-\binom{t+2}{3}\right).$$
In terms of \cite{HMNT,HMT}, the set $Z$ satisfies $C(t)$ if it admits unexpected cones of degree $t$.

Let $Z$ be a nondegenerate $(a,b)$-geproci set. 
The only case we know for which
$Z$ does not satisfy $C(a)$ is when  $a=b=2$. 
And the only cases we know for which
$Z$ does not satisfy $C(b)$ are when  $a=2\leq b$. 
In this section we give results which show that $Z$ satisfies $C(a)$ and $C(b)$
in a great range of cases.

By \cite[Theorem 3.5]{CM},
if $Z$ is a grid with $a=2<b$, then $Z$ satisfies $C(a)$,
while if $Z$ is a grid with $2<a\leq b$, then $Z$ satisfies $C(b)$.
Based on computer calculations, \cite[Remark 3.6]{CM} suggests for $(2,b)$-grids $Z$ with 
$b\geq2$, that $Z$ does not satisfy $C(b)$. We prove this below.

\begin{remark}\label{r. D4 C3 C4}
Based on our classification of $(3,b)$-geproci sets in Section \ref{sec:small_a_b}, and the fact that $D_4$ satisfies $C(3)$ and $C(4)$ by \cite{HMNT}, we know
a nondegenerate $(3,b)$-geproci set
satisfies both $C(3)$ and $C(b)$.
\end{remark}

We now pose some questions whose answers we do not know. 

\begin{question}\label{unexpQuest1}
If $Z$ is $(a,b)$-geproci with $3<a\leq b$ and nondegenerate,
must $Z$ satisfy $C(a)$ and $C(b)$?
\end{question}

Our results apply mostly (but not entirely) to nontrivial $(a,b)$-geproci half grids $Z$ with $a\leq b$
contained in $a$ skew lines. This raises questions as to how many nondegenerate geproci $Z$
do not fall into this category.

\begin{question}\label{unexpQuest2}
For each $a> 3$, are there infinitely many $b$ such that there is
a nontrivial $(a,b)$-geproci set $Z$ not contained in $a$ skew lines? 
(The sets constructed in Theorem \ref{t. (a,b)-geproci}
are half grids, contained in $a$ skew lines.)
If the answer is negative, can one find a bound $B$ on $b$ (possibly depending on $a$)
such that a nontrivial $(a,b)$-geproci set $Z$ not contained in $a$ skew lines
must have $b\leq B$? 
\end{question}
When $a=3$ we know there is a bound since $b$ must be $4$.

We know examples with $b-a$ as large as 4, and $b$ as large as 10. For example,
for ${D_4}$ we have $a=3, b=4$; 
for ${F_4}$ we have $a=4, b=6$; and 
for ${H_4}$ we have $a=6, b=10$, and none of these are contained in $a$ skew lines.
When $Z_K$ is the 60 point Klein configuration \index{configuration! Klein}\cite{PSS} we also have $a=6$ and $b=10$, and when
$Z_{P}$ is the 40 points given by the Penrose configuration \index{configuration! Penrose} construction, we have
$a=5$ and $b=8$; neither of these is contained in $a$ skew lines. Note that
${H_4}$ and $Z_{P}$ are not half grids; the other three are half grids, since they are
contained in $b$ skew lines.


The following lemma will be useful.

\begin{lemma}\label{projectLem}
Let $Z\subset \PP^3$ be an $(a,b)$-geproci set and let $\overline{Z}$ be its projection from a
general point $P$ to a plane $H$.
Then for any $t>0$, we have $\dim [I(Z+tP)]_t =\dim [J(\overline{Z})]_t$, where $J(\overline{Z})$ is the ideal of $\overline{Z}$ in the coordinate ring of $H$.
In particular, $\dim [I(Z+aP)]_a =1$ (when $a<b$) and $\dim [I(Z+bP)]_b =\binom{b-a+2}{2}+1$,
so there is a unique 2 dimensional cone of degree $a$ with vertex $P$ containing $Z$ (when $a<b$), and
there is a unique 2 dimensional cone of degree $b$ with vertex $P$ containing $Z$
modulo the cone of degree $a$.
\end{lemma}

\begin{proof}
After a change of coordinates, we 
can assume $H$ is a coordinate hyperplane and $P$ is the coordinate vertex opposite $H$.
We may, for example, assume the variables are $x,y,z,w$ with $H:\;w=0$ and $P=[0:0:0:1]$.
The cones of degree $t$ through $Z$ with vertex $P$ are then the forms of degree $t$ in $x,y,z$
vanishing on $\overline{Z}$. Thus $[I(Z+tP)]_t=[J(\overline{Z})]_t$,
hence $I(Z+tP)$ is generated by a form $F(x,y,z)$ of degree 
$a$ and a form $G(x,y,z)$ of degree $b$ which has no irreducible factors in common with $F$.
But $\dim [J(\overline{Z})]_t$ is 0 for $t<a$, $\binom{t-a+2}{2}$ for $a\leq t < b$ and $1+\binom{t-a+2}{2}$
for $t=b$.
\end{proof}

We now consider $(a,b)$-geproci sets $Z$, with $a\geq3$, which impose independent conditions in
degree either $a$ or $b$. 

\begin{proposition}\label{NondegProp}
Let $3\leq a \leq b$. Let $Z\subset \PP^3$ be $(a,b)$-geproci.
\begin{enumerate}
\item[(a)] If $a<b$ and $Z$ imposes independent conditions on forms of degree $a$,
then $Z$ satisfies $C(a)$ and $Z$ is nondegenerate.
\item[(b)] If $Z$ imposes independent conditions on forms of degree $b$,
then $Z$ satisfies $C(b)$ and $Z$ is nondegenerate.
\end{enumerate}
\end{proposition}

\begin{proof}
If $Z$ satisfies $C(t)$ for some degree $t$, then $Z$ must be nondegenerate
(since cones over sets of reduced points in a plane are not unexpected; see \cite[Lemma 2.12]{CM}).

Let $P$ be a general point and assume $Z$ imposes independent conditions in degree $t$.
Then $Z$ satisfies $C(t)$ if $\dim [I(Z+tP)]_t$
is positive and bigger than
$$\dim [I(Z)]_t-\binom{t+2}{3}=\binom{t+3}{3}-ab-\binom{t+2}{3}=\frac{t^2+3t+2-2ab}{2}.$$
Of course, $\dim [I(Z+tP)]_t>0$ for all $t\geq a$ since $Z$ is $(a,b)$-geproci.

(a) If $t=a$ we have $\dim [I(Z)]_a-\binom{a+2}{3}=\frac{a^2+3a+2-2ab}{2}=-\frac{2a(b-a)+(a-2)(a-1)-4}{2}
<-\frac{2a-4}{2}\leq 0$ 
for $3\leq a<b$. Thus $Z$ satisfies $C(a)$.

(b) Now say $t=b$. Then 
$$\dim [I(Z)]_b-\binom{b+2}{3}=\frac{1}{2}(b^2+3b+2)-ab=\frac{1}{2}(b+2)(b+1)-ab.$$
By Lemma \ref{projectLem},
$$\dim [I(Z+bP)]_b=\binom{b-a+2}{2}+1=\frac{1}{2}(b+2)(b+1)+\frac{1}{2}(a-2)(a-1)-ab.$$ 
For $a\geq3$, we thus have $\dim [I(Z+bP)]_b>\dim [I(Z)]_b-\binom{b+2}{3}$, so 
$Z$ satisfies $C(b)$.
\end{proof}

\begin{remark}\label{unexpRem}
The ${H_4}$ configuration and the Klein configuration $Z_K$ \cite{PSS} 
impose independent conditions in degrees $t\geq 6$, hence by 
Proposition \ref{NondegProp} satisfy $C(6)$ and $C(10)$.
However, ${F_4}$ imposes independent conditions only in degrees $t\geq 5$, hence by 
Proposition \ref{NondegProp} it satisfies $C(6)$. If it imposed
independent conditions in degree 4 we would have $\dim [I(Z)]_4=11$,
but in fact we have $\dim [I(Z)]_4=12$. Since a point of multiplicity 4 
imposes $\binom{4+3-1}{3}=20$ conditions and since 
$\dim [I(Z+4P)]_4=1>\dim [I(Z)]_4-20$, it
nonetheless satisfies $C(4)$.
Likewise, the Penrose configuration $Z_{P}$ imposes independent
conditions only for $t\geq 7$, hence it satisfies $C(8)$ 
by Proposition \ref{NondegProp}. If it imposed
independent conditions in degree 5 we would have $\dim [I(Z)]_5=16$,
but in fact we have $\dim [I(Z)]_5=20$. Since a point of multiplicity 5 
imposes $\binom{5+3-1}{3}=35$ conditions and since therefore 
$\dim [I(Z+5P)]_5=1>\dim [I(Z)]_5-35$, it satisfies $C(5)$.
\end{remark}

To obtain further applications of Proposition \ref{NondegProp}, the following lemma will be useful.

\begin{lemma}\label{GVT-Lemma}
Let $Z'$ be an $(a,b)$-grid with $2\leq a\leq b$. 
Then $Z'$ imposes independent conditions in degrees $t\geq b-1$.
\end{lemma}

\begin{proof}
The idea is that for any point $P_0\in Z'$ there are planes $H_1,\ldots,H_{b-1}$ whose union $\cup H_i$ 
contains $Z'\setminus\{P_0\}$ but does not contain $P_0$, hence the points 
of $Z'$ impose independent conditions on forms of degree $b-1$.

First assume that $a\geq3$, so the grid lines of $Z'$ come from the rulings on a unique quadric $\calq$. 
Then add $b-a$ grid lines so $Z'$ is an $(a,b)$-subgrid of a $(b,b)$-grid $Z''\subset \calq$.
It's enough to show $Z''$ imposes independent conditions on forms of degree $b-1$.
Starting with any point $P_0\in Z''$, we can choose a sequence of points $P_0, P_1,\ldots,P_{b-1}\in Z''$ such that 
$P_{i+1}$ is not on the two grid lines through $P_j$ for any $0\leq j<i+1$.
Let $H_i$ be the plane spanned by the grid lines through $P_i$.
It is easy to check that $\cup_{i=1}^{b-1} H_i$ now has the required containment properties.

If $a=2$, there are two disjoint grid lines with $b$ points each; 
call them $L_1$ and $L_2$ (say $L_1$ contains $P_0$, and say
$P_1,\ldots,P_{b-1}$ are the grid points on $L_1$ other than $P_0$).
In this case $H_1$ is the plane spanned by $L_2$ and $P_1$
and $H_i$ for $1<i\leq b-1$, is any plane containing $P_i$ but not $P_0$.
\end{proof}

The next lemma recovers the result
of \cite[Theorem 3.5]{CM} that an $(a,b)$-grid with $3\leq a \leq b$ satisfies $C(a)$ and $C(b)$.

\begin{lemma}\label{HMNTandCMLemma}
Let $Z$ be a nondegenerate $(a,b)$-geproci set contained in $a$ skew lines
with $b$ points on each line (so $Z$ is a grid or a half grid). 
\begin{enumerate}
\item[(i)] If $a<b$, then $Z$ satisfies $C(a)$. 
\item[(ii)] If $Z$ is an $(a,b)$-grid with $3\leq a\leq b$, then $Z$ satisfies $C(b)$.
\end{enumerate}
\end{lemma}

\begin{proof}
(i) This follows from \cite[Corollary 2.5]{HMNT}, as in the first paragraph of the proof
of \cite[Theorem 3.5]{CM}.

(ii) This follows from Proposition \ref{NondegProp} and Lemma \ref{GVT-Lemma}.
\end{proof}

The next two results show situations in which geproci sets satisfy $C(a)$ or $C(b)$.

\begin{proposition}\label{unexpPropForCertainHalfgrids1}
Let $Z$ be $(a,b)$-geproci where $3\leq a<b$, with $a$ points on each of $b$ skew lines $L_1,\ldots, L_b$.
\begin{enumerate}
\item [(i)] Assume that $Z'=Z\cap (L_1\cup\cdots\cup L_{a})$ is an $(a,a)$-grid. 
Then $Z$ imposes independent conditions in degree $b-1$ and satisfies $C(a)$ and $C(b)$.
\item [(ii)] Assume that $Z'=Z\cap (L_1\cup\cdots\cup L_{a-1})$ is an $(a-1,a)$-grid. 
Then $Z$ imposes independent conditions in degree $b$ and satisfies $C(b)$ and, if $a\neq4,5$, also $C(a)$.
\end{enumerate}
\end{proposition}

\begin{proof} 
(i) For $C(a)$, we must check for a general point $P$ that $\dim [I(Z+aP)]_a > \dim [I(Z)]_a-\binom{a+2}{3}$.
But $\dim [I(Z+aP)]_a\geq 1$, while $\dim [I(Z)]_a-\binom{a+2}{3}\leq \dim [I(Z')]_a-\binom{a+2}{3}=
\binom{a+3}{3}-a^2-\binom{a+2}{3}=\frac{1}{2}(3a-a^2+2)$ holds by Lemma \ref{GVT-Lemma}.
Thus $C(a)$ holds for $a>3$, while for $a=3$, then $Z$ is either a grid or $D_4$, so again $C(a)$ holds.

For $C(b)$, it is enough by Proposition \ref{NondegProp} to show that
$Z$ imposes independent conditions in degree $b-1$. We proceed by induction. 
Let $Z_t$ consist of the points of $Z$ in $Z' \cup L_{a+1} \cup \dots \cup L_{a+t}$, for $1 \leq t < b-a$ and we will understand
$Z_0 = Z'$ and $Z_{b-a} = Z$. Let $H=H_{t+1}$ be a general plane containing $L_{a+t+1}$. Let $\overline Z_{t+1} = Z \cap H$ (a set of $a$ collinear points in $H$). 

We know by Lemma \ref{GVT-Lemma} that $Z_0 = Z'$ imposes independent conditions on forms of degree $a-1$, so $h^1(\mathcal I_{Z_0}(a-1)) = 0$. Since $\overline Z_{t+1}$ is a set of $a$ collinear points no matter what $t$ is, we also have $h^1(\mathcal I_{\overline Z_{t+1},H }(s)) = 0$ for all $s \geq a-1$. For $0 \leq t < b-a$ consider the short exact sequence of sheaves
of ideals
\[
0 \rightarrow \mathcal I_{Z_t}(a-1+t) \stackrel{\times H}{\longrightarrow} \mathcal I_{Z_{t+1}}(a+t) \rightarrow \mathcal I_{\overline Z_{t+1},H}(a+t) \rightarrow 0.
\]
It then follows by induction that $h^1(\mathcal I_Z(b-1)) = 0$, i.e. $Z$ imposes independent conditions in degree $b-1$.

(ii) For $C(a)$ the argument is the same as before except now
$\dim [I(Z')]_a-\binom{a+2}{3}=
\binom{a+3}{3}-a(a-1)-\binom{a+2}{3}=\frac{1}{2}(5a-a^2+2)$, so $C(a)$ holds for $a>5$, and, as before, for $a=3$.

For $C(b)$ the argument runs as before, except now
$Z_t$ consists of the points of $Z$ in $Z' \cup L_{a} \cup \dots \cup L_{a+t}$, for $0 \leq t < b-a$.
Thus the induction runs one extra step longer, so we end up in degree $b$ rather than in degree $b-1$.
\end{proof}

\begin{proposition}\label{unexpPropForCertainHalfgrids2}
Let $Z$ be $(a,b)$-geproci where $a\leq b$, with $b$ points on each of $a$ skew lines $L_1,\ldots, L_a$.
Assume that $Z'=Z\cap (L_1\cup\cdots\cup L_{a-1})$ is an $(a-1,b)$-grid. 
Then $Z$ satisfies~$C(b)$.
\end{proposition}

\begin{proof} 
We mimic the proof of Proposition \ref{unexpPropForCertainHalfgrids1}.
It is enough by Proposition \ref{NondegProp} to show that
$Z$ imposes independent conditions in degree $b$.

Let $H$ be a general plane containing $L_a$.
Consider the short exact sequence of sheaves
of ideals
\[
0 \rightarrow \mathcal I_{Z'}(b-1) \stackrel{\times H}{\longrightarrow} \mathcal I_{Z}(b) \rightarrow \mathcal I_{\overline Z\cap L_a,H}(b) \rightarrow 0.
\]

As in the proof of Proposition \ref{unexpPropForCertainHalfgrids1} we have
$h^1(\mathcal I_{\overline Z\cap L_a,H}(b))=0$
(since $a$ collinear points in a plane impose independent conditions in degrees $t\geq a-1$).
We also have $h^1(\mathcal I_{Z'}(b-1))=0$ by 
Lemma \ref{GVT-Lemma}.
The long exact sequence of cohomology of the short exact sequence
now gives $h^1(\mathcal I_{Z}(b)=0$,
hence $Z$ imposes independent conditions in degree $b$.
\end{proof}

\begin{example}
The examples of $(a,b)$-geproci sets $Z$ coming from the standard construction, 
as given in Theorem \ref{t. geproci infinite class}(i) and applied in Theorem \ref{t. (a,b)-geproci}, 
all satisfy both $C(a)$ and $C(b)$. Recall that the construction starts with
an $(n,n+1)$-geproci set $Z$ with $n\geq3$, where $Z$ consists of a set of $n$ points $Y_1$
on a line $\ell_1$, together with an $(n,n)$-grid $X$
whose grid lines in one ruling are $L_0,\ldots,L_{n-1}$ and whose grid lines in the other ruling are
$M_0,\ldots,M_{n-1}$ (see Figure \ref{Fig: StConst n=3, i}
for the case $n=3$).
With respect to a permutation $j_k$, the sets $Z_i$ of the form $Z_i=Z\cap(\ell_1\cup L_{j_0}\cup\cdots\cup L_{j_i})$ for $2\leq i\leq n-1$
(we could equally well use $Z_i=Z\cap(\ell_1\cup M_{j_0}\cup\cdots\cup M_{j_i})$ for $2\leq i\leq n-1$)
are obtained by removing grid lines of points, one at a time, from $Z=Z_{n-1}$.
The set $Z_i$ is, as in the proof of Theorem \ref{t. (a,b)-geproci}, $(a,b)$-geproci, with $a=i+2$ and $b=n$ when $i<n-1$,
and with $a=n, b=n+1$ when $i=n-1$. 

By Proposition \ref{unexpPropForCertainHalfgrids1}(i), $Z_{n-1}$ imposes independent conditions on forms of degree $n$
and $C(a)$ and $C(b)$ both hold for $Z_{n-1}$.
For $2\leq i<n-2$, $Z_i$ satisfies $C(a)$ by Lemma \ref{HMNTandCMLemma}(i), and $Z_i$
satisfies $C(b)$ for $2\leq i<n-1$ by Proposition \ref{NondegProp}, 
since $Z_i\subset Z_{n-1}$ for $2\leq i<n-1$, so we also have that $Z_i$ imposes independent conditions
on forms of degree $n$ since $Z_{n-1}$ does. (Note that when $i=n-2$ we have $b=a$, so 
$C(a)$ and $C(b)$ coincide.)
\end{example}

\begin{example}
Similarly, examples of $(a,b)$-geproci sets $Z$ can be obtained from the standard construction, 
as given in Theorem \ref{t. geproci infinite class}(ii), by removing grid lines of points
(see Remark \ref{r. 4.2(b) geprocis}).
Recall that the construction starts with
an $(n,n+2)$-geproci set $Z$ with $n\geq4$ even, where $Z$ consists of a set of $n$ points $Y_1$
on a line $\ell_1$ and a set $Y_2$ of $n$ points on a line $\ell_2$, together with an $(n,n)$-grid $X$
whose grid lines in one ruling are $L_0,\ldots,L_{n-1}$ and whose grid lines in the other ruling are
$M_0,\ldots,M_{n-1}$ (see Figure \ref{Fig: StConst n=3, ii}
for the case $n=4$).
With respect to a permutation $j_k$, the sets $Z_i$ of the form $Z_i=Z\cap(\ell_1\cup \ell_2\cup L_{j_0}\cup\cdots\cup L_{j_i})$ for $2\leq i\leq n-1$
(we could equally well use $Z_i=Z\cap(\ell_1\cup\ell_2\cup M_{j_0}\cup\cdots\cup M_{j_i})$ for $2\leq i\leq n-1$)
are obtained by removing grid lines of points, one at a time, from $Z=Z_{n-1}$.
The set $Z_i$ is, as in the proof of Theorem \ref{t. (a,b)-geproci}, $(a,b)$-geproci, with $a=i+3$ and $b=n$ when $i<n-2$
and with $a=n, b=i+3$ when $i=n-2$ or $i=n-1$. 

Thus $Z_{n-1}$ has $a=n$ and $b=n+2$ and, by Proposition \ref{unexpPropForCertainHalfgrids1}(i), 
imposes independent conditions on forms of degree $n+1$ and $C(a)$ and $C(b)$ both hold.
Also, $Z_{n-2}$ has $a=n$ and $b=n+1$, and since $Z_{n-2}\subset Z_{n-1}$, we know
$Z_{n-2}$ imposes independent conditions in degree $n+1$, hence
$Z_{n-2}$ has $C(b)$ by Proposition \ref{NondegProp}, and it has
$C(a)$ by Proposition \ref{unexpPropForCertainHalfgrids1}(ii) when $n>4$ (recall $n$ is even).
We do not know if $Z_{n-3}$ satisfies $C(a)$ or $C(b)$; here $b=a=n$.
And for $2\leq i<n-3$, where $b=n$ and $a=i+3$, $Z_i$ satisfies $C(a)$ by Lemma \ref{HMNTandCMLemma}(i)
but we do not know if it satisfies $C(b)$.
\end{example}

\begin{example}
Examples of $(a,b)$-geproci sets $Z$ can be obtained also from the extended standard construction, 
as given in Propositions \ref{p.add points to geproci} and \ref{p.add points to geproci2}
and discussed in Remark \ref{r. subsets of extended geproci}.

For Proposition \ref{p.add points to geproci} we start
with an $(n,n+1)$-geproci set $Z=X\cup Y_1$ coming from Theorem \ref{t. geproci infinite class}(i).
So $n\geq 3$ and (as in Figure \ref{Fig: ExtStConst n=3, i})
$X$ is an $(n,n)$-grid with grid lines $r_1,\ldots,r_n$ and $M_0,\ldots,M_{n-1}$,
$Y_0$ is a set of $n$ points on a line $\ell_1$, and $X$ is contained in 
a larger grid $X'$ with grid lines $r_1,\ldots,r_n,s_Q$ and $M_0,\ldots,M_{n-1}$.
The point $Q$ is the intersection of $s_Q$ with $\ell_1$. The result
is the $(n+1,n+1)$-geproci set $\widetilde{Z}=X'\cup Y_1\cup\{Q\}$.
We can now remove points of $\widetilde{Z}$ on grid lines
to obain other geproci sets.
Specifically, given a permutation $j_k$, let $Z_i=Z\cap(\ell_1\cup r_{j_1}\cup\cdots\cup r_{j_i})$
for $3\leq i\leq n$. Then $Z_i$ is $(a,b)$-geproci
with $a=i+1, b=n+1$.
Note that $Z_i$ satisfies $C(b)$ by Proposition \ref{unexpPropForCertainHalfgrids2}
(and hence also $C(a)$ when $i=n$),
and it satisfies $C(a)$ by Lemma \ref{HMNTandCMLemma}(i) when $i<n$.

For Proposition \ref{p.add points to geproci2} we start
with an $(n,n+2)$-geproci set $Z=X\cup Y_1\cup Y_2$ coming from Theorem \ref{t. geproci infinite class}(ii).
So $n\geq 3$ and (as in Figure \ref{Fig: ExtStConst n=3, ii})
$X$ is an $(n,n)$-grid with grid lines $r_1,\ldots,r_n$ and $M_0,\ldots,M_{n-1}$,
$Y_i$ is a set of $n$ points on a line $\ell_i$ ($i=1,2$), and $X$ is contained in 
a larger grid $X'$ with grid lines $r_1,\ldots,r_n,s_1,s_2$ and $M_0,\ldots,M_{n-1}$.
The point $Q$ (resp. $P$) is the intersection of $s_1$ with $\ell_1$ (resp. $\ell_2$),
and the point $Q'$ (resp. $P'$) is the intersection of $s_2$ with $\ell_1$ (resp. $\ell_2$).
Moreover, the lines $\overline{PQ'}$ and $\overline{QP'}$ are lines
on the unique quadric containing $X$. The result
is the $(n+2,n+2)$-geproci set $\widetilde{Z}=X'\cup Y_1\cup Y_2\cup\{P.P',Q,Q'\}$.
We can now remove points of $\widetilde{Z}$ on grid lines
to obtain other geproci sets.
Specifically, given a permutation $j_k$, let $Z_i=Z\cap(\ell_1\cup\ell_2\cup r_{j_1}\cup\cdots\cup r_{j_i})$
for $3\leq i\leq n$. Then $Z_i$ is $(a,b)$-geproci
with $a=i+2, b=n+2$ and
satisfies $C(a)$ by Lemma \ref{HMNTandCMLemma}(i) for $3\leq i<n$.
We do not know which $Z_i$ satisfy $C(b)$, but
note for the $(6,6)$-geproci set $Z$ given as an instance of
Proposition \ref{p.add points to geproci2} in
Example \ref{ex.from F4 to Klein}, that $C(6)$ holds
(on account of $Z$ imposing independent conditions on forms of degree 5 or more, which we confirmed by direct computation).
\end{example}

We close this section with a criterion for an $(a,b)$-geproci set $Z$ to satisfy $C(a)$ or $C(b)$.
Note that if $\overline{Z}$ is its projection  from a
general point $P$ to a plane,  then $\overline{Z}$ is a complete intersection of forms of degrees 1, $a$ and $b$, so its ideal in $\PP^3$ satisfies  $\dim [I(\overline{Z})]_a=\binom{a+2}{3}+1$  
and $\dim [I(\overline{Z})]_b=\binom{b+2}{3}+\binom{b-a+2}{2}+1$.

\begin{proposition}\label{UnexpectedCriterion}
Let $Z$ be $(a,b)$-geproci.
\begin{enumerate}
\item[(a)] Assume $a<b$. Then $Z$ satisfies $C(a)$ if and only if $\dim [I(Z)]_a<\binom{a+2}{3}+1 = \dim [I(\overline{Z})]_a$.
\item[(b)] Assume $a\leq b$. Then $Z$ satisfies $C(b)$ if and only if $\dim [I(Z)]_b<\binom{b+2}{3}+\binom{b-a+2}{2}+1 = \dim [I(\overline{Z})]_b$.
\end{enumerate}
\end{proposition}

\begin{proof} 
(a) By Lemma \ref{projectLem}, $\dim [I(Z+aP)]_a=1$, and by definition, $Z$ satisfies $C(a)$ if and only if
$\dim [I(Z+aP)]_a >\max(0,\dim[I(Z)]_a-\binom{a+2}{3})$. Thus $Z$ satisfies $C(a)$ if and only if
$\dim [I(Z)]_a < \binom{a+2}{3}+1$.

(b) By Lemma \ref{projectLem}, $\dim [I(Z+bP)]_b=\binom{b-a+2}{2}+1$.
But by definition, $Z$ satisfies $C(b)$ if and only if
$\dim [I(Z+bP)]_b >\max(0,\dim[I(Z)]_b-\binom{b+2}{3})$. Thus $Z$ satisfies $C(b)$ if and only if
$\binom{b+2}{3}+\binom{b-a+2}{2}+1 > \dim[I(Z)]_b$.
\end{proof}

The fact that a $(2,b)$-grid $Z$ does not satisfy $C(b)$
was observed in \cite{CM}, based on computer calculations, but not proved.
Here is a proof.

\begin{corollary}\label{2bGeprociExpected}
Let $Z$ be a $(2,b)$-grid, $b\geq 2$.
Then $Z$ does not satisfy $C(b)$.
\end{corollary}

\begin{proof} 
Applying Lemma \ref{GVT-Lemma}, we get $\dim [I(Z)]_b=\binom{b+3}{3}-2b
=\binom{b+2}{3}+\binom{b}{2}+1$, so the result follows by
Proposition \ref{UnexpectedCriterion}(b).
\end{proof}

\begin{remark} \label{cone-summary}
Let $Z$ be $(a,b)$-geproci. The criterion given by the above proposition can be restated as saying that
$Z$ satisfies $C(a)$ (resp. $C(b)$) if and only if $\dim [I(Z)]_t < \dim [I(\overline{Z})]_t$
for $t=a$ (resp. $t=b$), where $\overline{Z}$ is the projection of $Z$ from a
general point $P$ to a plane $H$. After a change of coordinates we may assume $P=[0:0:0:1]$ and $H:w=0$.
Then, for scalars $s$, we get a family of subschemes $Z_s$, where the points of $Z_s$
are the points $p_s=[c_0:c_1:c_2:sc_3]$ where $p=[c_0:c_1:c_2:c_3]$ is a point of $Z$.
(Note that $Z_1=Z$, $Z_0=\overline{Z}$ and for $s\neq0$, $Z_s$ is projectively equivalent to $Z$.)
By semicontinuity, we therefore have $\dim [I(Z)]_t\leq \dim [I(\overline{Z})]_t$.
Thus the criterion is that we have unexpectedness for $t=a$ or $t=b$ if the
known inequality is strict in that degree. 
It is not clear to what extent strictness depends on geprociness, or geprociness depends on strictness. Thus it might be informative to try to classify degrees $t$ and sets $Z$, geproci or not, for which 
$\dim [I(Z)]_t = \dim [I(\overline{Z})]_t$.
\end{remark}

\begin{question}\label{quest. geprociness and strictness}
Which $t$ and finite sets $Z$ in $\PP^3$ satisfy $\dim [I(Z)]_t = \dim [I(\overline{Z})]_t$?
\end{question}

\section{Unexpected cubic cones for \texorpdfstring{$B_{n+1}$}{B n+1} configurations of points}\index{unexpectedness! $B_{n+1}$}
A number of root systems give geproci sets in $\PP^3$ and also give unexpected cones. 
Although $B_{n+1}$ configurations of points do not seem to give geproci sets, computer calculations suggested that they have unexpected cones, see \cite{HMNT}.    
Here we will give a proof of the unexpectedness. 

Let ${B_{n+1}}\subset\PP^n$\index{configuration! ${B_{n+1}}$} be the  $(n+1)^2$ points whose homogeneous coordinates
are either $\pm1$ or 0, with at most two nonzero coordinates, see \cite{HMNT}.

Note that every coordinate line contains exactly 4 points of ${B_{n+1}}$.
Thus any cubic vanishing on ${B_{n+1}}$ must vanish on every coordinate line.
I.e., $[I({B_{n+1}})]_3=[I(S_1)]_3$, where $S_1$ is the union of the $\binom{n+1}{2}$ coordinate lines.
Hence $S_1$ is the 1-skeleton of the coordinate simplex of $\PP^n$. (Although $S_1$ depends on $n$,
when we write $S_1$ it will always be clear from context which $\PP^n$ it is the simplicial 1-skeleton of.)

So for any point $P\not\in S_1$, we have $[I({B_{n+1}}+3P)]_3=[I(S_1+3P)]_3$.
However, when $P$ is a general point, there is a linear change of coordinates on $\PP^n$ where the image of 
$S_1$ is $S_1$ and which takes $P$ to the point $[1:\cdots:1]$. So to compute
the dimension of $[I({B_{n+1}}+3P)]_3=[I(S_1+3P)]_3$ we can assume 
$P=[1:\cdots:1]$.

By the following theorem, $S_1$ has unexpected cubic cones for $n\geq5$, hence so does ${B_{n+1}}$.
In fact, for $n\geq 5$, there are exactly $n+1$ more cones than are expected.

\begin{theorem}\label{BnUnexpThm}
Let $P\in\PP^n$ be a general point and let $x_0,\ldots,x_n$ be the variables on $\PP^n$.
For $n\geq2$, using the notation above, we have
$$\dim [I({B_{n+1}})]_3=\dim [I(S_1)]_3= \binom{n+1}{3}.$$
For $n\geq 4$ we have
$$\dim [I({B_{n+1}}+3P)]_3=\dim [I(S_1+3P)]_3=\binom{n+1}{3}-\binom{n+2}{2}+n+1;$$
this is positive for $n>4$ so for $n\geq 5$ the cubic cones with vertex $P$ and containing ${B_{n+1}}$ (resp. $S_1$)
are unexpected. Moreover, $[I({B_{n+1}}+3P)]_3=[I(S_1+3P)]_3$ is spanned by elements
of the form $(x_{i_1}-x_{i_2})(x_{j_1}-x_{j_2})(x_{k_1}-x_{k_2})$,
where $i_1,i_2,j_1,j_2,k_1,k_2$ are distinct integers between 0 and $n$.
\end{theorem}

The proof will follow from some lemmas.

\begin{lemma}\label{Lem1BnUnexp}
    Let $n\geq 2$ and let $I\subseteq k[x_0,\ldots,x_n]$. Then $I(S_1)$
    is a monomial ideal, generated by all monomials $x_ix_jx_k$
    with $0\leq i<j<k\leq n$. Hence 
$$[I({B_{n+1}})]_3=[I(S_1)]_3= \binom{n+1}{3}=\binom{n+3}{3}-(n+1)^2.$$
Thus the points of ${B_{n+1}}$ impose independent conditions on cubics.
\end{lemma}

\begin{proof}
It is well-known that $I(S_1)$ is monomial with the given generators (see, 
for example, \cite[Proposition 2.9]{GHM}
or more generally \cite[Proposition 2.3(4)]{GHMN})
and there are $\binom{n+1}{3}$ such degree 3 monomials. 
Now $\binom{n+1}{3}=\binom{n+3}{3}-(n+1)^2$ is a calculation, 
which shows that ${B_{n+1}}$ imposes independent conditions on cubics.
\end{proof}

\begin{lemma}\label{GFLBlemma}
    Let $n\geq 5$, let $I\subseteq k[x_0,\ldots,x_n]$ be the ideal generated by the 
    forms $(x_a-x_b)(x_c-x_d)(x_e-x_f)$ for any choice of 6 different indexes $a,b,c,d,e,f$,
    and let $P=[1:\cdots:1]$.
    Then $I\subset I(S_1+3P)$ and
    $$\dim [I(S_1+3P)]_3\geq \dim [I]_3\ge \binom{n+1}{3}-\binom{n+1}{2}=\binom{n+1}{3}-\binom{n+2}{2}+n+1.$$
\end{lemma}

\begin{proof}
Any point $p\in S_1$ has at most two nonzero coordinates, hence one of the factors of
$(x_a-x_b)(x_c-x_d)(x_e-x_f)$ must vanish at $p$, and 
each factor vanishes at $P$. Thus $I\subseteq I(S_1+3P)$.

    Now consider the lexicographic monomial order with respect to $x_0>x_1>\cdots > x_n$.
    Since $\dim [I]_3=\dim [In(I)]_3$, it is enough to look at the leading terms of the forms in $I$.
Note that
    \begin{itemize}
        \item $x_ax_bx_c\in In(I)$  whenever there are  $a',b',c'$, all different from $a,b,c,$  such that $x_a>x_{a'}$, $x_b>x_{b'}$ and $x_c>x_{c'}$. 
        Indeed, $x_ax_bx_c$ is the leading term of  $(x_a-x_{a'})(x_b-x_{b'})(x_c-x_{c'})$.
    \end{itemize}
This fact keeps out of the count the following monomials 
\begin{itemize}
    \item[$(i)$] those containing $x_n$;
    \item[$(ii)$] those not containing $x_n$, but containing both $x_{n-2}$ and $x_{n-1}$;
    \item[$(iii)$] $x_{n-4}x_{n-3}x_{n-2}$ and $x_{n-4}x_{n-3}x_{n-1}$.
\end{itemize}
    Therefore we get
    \[
    \dim[I]_3\ge \binom{n+1}{3}- \stackrel{(i)}{\binom{n}{2}}-\stackrel{(ii)}{(n-2)} -\stackrel{(iii)}{2} =\binom{n+1}{3}-\binom{n}{2}-n.
    \] 
  But $\binom{n+1}{3}-\binom{n+1}{2}=\binom{n+1}{3}-\binom{n+2}{2}+n+1$
  is equivalent to $\binom{n+2}{2}-\binom{n+1}{2}=n+1$, which is clear.
    \end{proof}

\begin{lemma}\label{UBlemma}
      Let $n\geq 4$ and let $P=[1:\cdots:1]$. Then
    $$\dim [I(S_1+3P)]_3\leq \binom{n+1}{3}-\binom{n+1}{2}=\binom{n+1}{3}-\binom{n+2}{2}+n+1.$$
\end{lemma}

\begin{proof} 
The case $n=4$ is a calculation showing that 
$\dim [I(S_1+3P)]_3=0=\binom{n+1}{3}-\binom{n+1}{2}=\binom{n+1}{3}-\binom{n+2}{2}+n+1$.
So now assume $n>4$.

The idea of the proof is to look at two exact sequences.
The first is given by restriction to the plane $H$ defined by $x_0-x_1=0$:
$$0\to [I(S_1'+2P)]_2\to[I(S_1+3P)]_3\to [\overline{I(S_1\cap H+3P)}]_3,$$
where the overline indicates working in the ring $\CC[H]=\CC[\PP^{n-1}]$.
The first map is given by multiplication by $x_0-x_1$ and
$S_1'$ is the union of the coordinate lines not contained in $H$.

Note that a coordinate line consists of all points
having all but two coordinates equal to 0, and the
two coordinates (say $x_i$ and $x_j$ with $i<j$) which are not obliged to be zero can take
arbitrary values, as long as both are not zero. Thus the coordinate lines in $H$
always have $x_0=x_1=0$, so there are $\binom{n-1}{2}$ coordinate lines contained
in $H$, corresponding to the possible choices of $i$ and $j$ selected from $\{2,\ldots,n\}$.
In particular, there are $\binom{n+1}{2}-\binom{n-1}{2}$ coordinate lines in $S_1'$.
Since $H$ contains every coordinate vertex except $p_0=[1:0:\ldots:0]$ or $p_1=[0:1:0:\ldots:0]$,
each line in $S_1'$ contains either $p_0$ or $p_1$, and no such coordinate line is in $H$,
and so every such coordinate line is in $S_1'$. Thus there are $2(n-1)+1$ lines in $S_1'$,
where $2(n-1)$ is the number 
of coordinate lines from $p_0$ or $p_1$ to the vertices of the $\PP^{n-2}$ defined by $x_1=x_2=0$,
and the plus 1 is for the coordinate line defined by $p_0$ and $p_1$. And indeed, one can easily check that
$\binom{n+1}{2}-\binom{n-1}{2}=2n-1$.

We also note that $S_1\cap H$ is the union of the $\binom{n-1}{2}$
coordinate lines in $H$, together with the point $q=[1:1:0:\ldots:0]$.

The second exact sequence is given by restriction to the plane $H'$ defined by $x_2-x_3=0$:
$$0\to[I(S_1''+P)]_1\to[I(S_1'+2P)]_2\to [\overline{\overline{I(S_1'\cap H'+2P)}}]_2,$$
where the double overline indicates working in the ring $\CC[H']=\CC[\PP^{n-1}]$,
and $S_1''$ is the union of the coordinate lines not contained in $H\cup H'$.
Since $H$ and $H'$ each contain $\binom{n-1}{2}$ coordinate lines
and $H\cap H'$ contains $\binom{n-3}{2}$ coordinate lines, we see
$H\cup H'$ contains $2\binom{n-1}{2}-\binom{n-3}{2}$, hence
$S_1''$ contains the remaining $\binom{n+1}{2}-2\binom{n-1}{2}+\binom{n-3}{2}=4$ coordinate lines.
If we denote by $p_i$ the coordinate vertex with $x_i\neq0$, then
these 4 coordinate lines are just the lines defined by $p_i$ and $p_j$ where $i$ is either 0 or 1 and
$j$ is either 2 or 3.

Also note that $S_1'\cap H'$ consists of the $\binom{n-1}{2}-\binom{n-3}{2}=2n-5$ coordinate lines
in $H'$ which are not also in $H$. As a consistency check, we note that
there are $2n-1$ coordinate lines not in $H$ and 4 not in $H\cup H'$, so again we see there are $2n-1-4$
coordinate lines which are in $H'$ but not in $H$.  
Recall that the coordinate lines not in $H$ are 
the line $L_{01}$ defined by $p_0$ and $p_1$, and the lines
$L_{ij}$ defined by $p_i$ and $p_j$ where $i$ is either
0 or 1, and $p_j$ is a coordinate vertex in $H$. 
But $p_0,p_1\in H'$, so the coordinate lines in $H'$ but not in $H$
are the lines defined by $p_0$ and $p_1$, and by $p_i$ and $p_j$ where $i$ is either
0 or 1, and $p_j$ is a coordinate vertex in $H\cap H'$.
Since there are $n-3$ coordinate vertices in $H\cap H'$, again we see there are
$2(n-3)+1=2n-5$ coordinate lines in $S_1'\cap H'$.

To compute an upper bound on $\dim [\overline{\overline{I(S_1'\cap H'+2P)}}]_2$,
note that a quadric $Q\subset H'$ containing $S_1'\cap H'$ contains $L_{01}, L_{0i}, L_{1i}$, for
each vertex $p_i\in H\cap H'$. Thus $Q$ contains the plane spanned by
$p_0,p_1,p_i$. Thus we have 
$$\dim [\overline{\overline{I(S_1'\cap H'+2P)}}]_2=\dim [\overline{\overline{I(T+2P)}}]_2$$
where $T$ is the join of the $n-3$ coordinate vertices in $H\cap H'$
with $L_{01}$. Thus the elements of $[\overline{\overline{I(S_1'\cap H'+2P)}}]_2$
are the cones (in $H'$) with vertex $P$ over the cones with vertex $p_0$ over 
the cones with vertex $p_1$ over quadrics in the span $W$ of the vertices
in $H\cap H'$ where the quadrics contain these vertices.
(Thus $W$ is the $\PP^{n-4}$ defined by $x_0=\cdots=x_3=0$.)
But the dimension of the quadrics in $\PP^{n-4}$ containing the
coordinate vertices of $\PP^{n-4}$ is $\binom{n-3}{2}$.
Thus $\dim [\overline{\overline{I(S_1'\cap H'+2P)}}]_2=\binom{n-3}{2}$.

We also see that the linear span of $S_1''+P$ is the span of the points $p_0,\ldots,p_3$ and $P$,
which has dimension 5, hence imposes 5 conditions on linear forms, so
$\dim [I(S_1''+P)]_1 = n-4$.

Thus from the second exact sequence above we have
$$\dim [I(S_1'+2P)]_2\leq \binom{n-3}{2}+n-4.$$

Now consider the cubic cones (with vertex $P$) in $H$ vanishing on $C$, where $C$
is the union of the $\binom{n-1}{2}$ coordinate lines in $S_1$.
Since $C$ consists of the coordinate lines in a $\PP^{n-2}\subset H$,
the cubic cones in $H$ vanishing on $C$ with vertex $P$ are bijective 
with the cubics in $\PP^{n-2}$ containing $C$.
So we see that $\dim [\overline{I(C+3P)}]_3=\binom{n-1}{3}$.
Note that $S_1\cap H=C\cup \{q\}$, but $q$ is not a base point of
$[\overline{I(C+3P)}]_3$. Indeed, after a change of coordinates
fixing $H\cap H'$, we may assume $P=[1:0:\cdots:0]$
and that $q=[1:0:-1:\cdots:-1]$. Now $x_2x_3x_4$ vanishes on
$C$ with order 3 at $P$ but does not vanish at $q$.
Thus $[\overline{I(S_1\cap H+3P)}]_3= [\overline{I(C+\{q\}+3P)}]_3$ and
$$\dim [\overline{I(S_1\cap H+3P)}]_3= \dim [\overline{I(C+\{q\}+3P)}]_3=\binom{n-1}{3}-1$$
so we have from the first exact sequence that
$$\dim [\overline{I(S_1+3P)}]_3\leq \dim [I(S_1'+2P)]_2+\dim [\overline{I(S_1\cap H+3P)}]_3=\binom{n-3}{2}+n-4+\binom{n-1}{3}-1.$$
But $\binom{n-3}{2}+n-4+\binom{n-1}{3}-1=\binom{n+1}{3}-\binom{n+1}{2}$.
\end{proof}

\begin{proof}[Proof of Theorem \ref{BnUnexpThm}]
The first statement of Theorem \ref{BnUnexpThm} is done in Lemma \ref{Lem1BnUnexp}.
The second statement follows from Lemmas \ref{GFLBlemma} and \ref{UBlemma}, which also shows that
$[I]_3=[I(S_1+3P]_3$, for the ideal $I$ of Lemma \ref{GFLBlemma} and hence gives
the last statement of the theorem. 
\end{proof}

\begin{remark}
More generally, let $Z_n$ be the simplicial 1-skeleton in $\PP^n$.
Computer experiments suggest
that all cones on $Z_n$, $n\geq 5$, of degree $m>3$ with general vertex point $P$ come 
from the cubic cones, in the following sense:
$$[I(Z_n+mP)]_m = I[(Z_n+3P)]_3\cdot I[(m-3)P]_{m-3}$$ 
(there is no unexpected degree 3 cone for $n < 5$).
Moreover, the experiments suggest that 
$$\dim [I(Z_n+mP)]_m = n+1 + \dim [I(Z_n)]_m -\binom{m+n-1}{n}$$
for $m \geq 5$, which means $Z_n$ would have
unexpected cones for all $m\geq 3$ and $n\geq 5$.
\end{remark}

We will study such unexpectedness in the next section. Similar behavior occurs for simplicial skeleta of codimension 2, as we 
see in Section \ref{subs: cod 2 sim skeleta}.

\section{Unexpected cones for simplicial skeleta}\index{unexpectedness! simplicial skeleta}
\index{skeleton}


\begin{lemma} \label{XIsACM}
Let $A$ be the set of coordinate points in $\PP^n$, $|A| = n+1$. Let $Q = [1,\dots,1]$. Let $B = A \cup \{ Q \}$. Let $X$ be the 1-skeleton of $B$. 
Then the 
Hilbert polynomial of $X$ is 
\[
\binom{n+2}{2} t - \binom{n+2}{2}.
\]
\end{lemma}

\begin{proof}

We prove the proposition via a series of remarks.

\begin{itemize}
\item $X$ is a union of $\binom{n+2}{2}$ lines; note that the sum of the entries of the $h$-vector is $\binom{n+2}{2}$.
\item The 1-skeleton, $D_1$, of $A$ is a star configuration (see \cite{GHM}) coming from the $n+1$ coordinate hyperplanes. Hence $D_1$ is ACM with $h$-vector $(1, n-1, \binom{n}{2})$. Note that $D_1$ is a union of $\binom{n+1}{2}$ lines.

\item Using $Q$ as the vertex, the cone $D_2$ over the points of $A$ is an ACM curve of degree $n+1$ with $h$-vector $(1,n-1,1)$. In fact, $D_2$ is arithmetically Gorenstein: the projection from $Q$ of the $n+1$ points of $A$ is arithmetically Gorenstein in $\PP^{n-1}$, and $D_2$ is a cone over these points.

\item Note that $X = D_1 \cup D_2$. We will use this decomposition to show that $X$ is ACM with the claimed $h$-vector.

\item From the $h$-vectors we get that
\[
\left. 
\begin{array}{l}
I_{D_1} \hbox{ is generated by $\binom{n+1}{3}$ cubics.} \\ 
I_{D_2} \hbox{ is generated by $\binom{n}{2}-1$ quadrics.} \\ 
\end{array}
\right \} (*)
\]

\item $D_1$ and $D_2$ have no common components, so $I_{D_1} + I_{D_2}$ is the (for now possibly unsaturated) ideal of a zero-dimensional scheme, $Z$. We want to show that this ideal is saturated.

\item $Z$ is supported on the $n+1$ points of $A$.

\item By symmetry, $Z$ has the same multiplicity at each of the $n+1$ points of $A$. Thus $\deg Z$ is one of $n+1, 2(n+1), 3(n+1), \dots$ depending on the multiplicity at each point. We want to determine that multiplicity.

\item Let $P \in A$. $P$ is a vertex of $D_1$. The components of $D_1$ passing through $P$ form a cone, $D_P$, with vertex at $P$, over the remaining $n$ points of $A$. In fact, we can view these components as the cone with vertex $P$ over the coordinate points of $\PP^{n-1}$ (spanned by the remaining $n$ points of $A$). So without loss of generality we can write $P = [1,0,\dots,0]$ and $I_{D_P}$ is generated by the degree 2 squarefree monomials in $x_1,\dots,x_n$.

\item $D_2$ has exactly one component passing through $P$, call it $\lambda$. We can write $I_\lambda = \langle x_1 - x_2, x_1 - x_3, \dots, x_1 - x_n \rangle$.

\item Putting this together, $I_\lambda + I_{D_P}$ is the saturated ideal of a 0-dimensional scheme of degree 2 supported at $P$. 

\item Thus $\deg Z = 2(n+1)$. 

\item \underline{Claim}: $[I_{D_2}]_2 = [I_Z]_2$.

We know $\subseteq$, so we just have to prove $\supseteq$. 

If the assertion is not true there is a hypersurface, $F$, of degree 2 that contains $Z$ but does not contain $D_2$. By symmetry it does not contain any component of $D_2$. Thus $Z \in | \mathcal O_{D_2}(2)|$. In particular, $F$ vanishes to multiplicity 2 at each point of $A$, so it is in $I_A^{(2)}$. But any such polynomial has to contain all of $D_1$ by B{\'e}zout's Theorem, and $I(D_1)$ starts in degree 3. So we have the claim.

\item Thus 
\[
\dim [I_Z]_2 = \binom{n}{2} - 1.
\]
Therefore the $h$-vector of $Z$ begins
\[
\left (1, n, \binom{n+1}{2} - \binom{n}{2}+1, \dots \right ) = (1, n, n+1, \dots).
\]
Since the degree of $Z$ is $2n+2$, the $h$-vector is precisely $(1, n, n+1)$. 

\item We know that

\begin{itemize}

\item  $[I_Z]_2 = [I_{D_2}]_2$, 

\item  $D_2$ is ACM of degree $n+1$ 

\item The Hilbert function of $D_2$ has first difference
\[
\Delta h_{D_2} = (1, n, n+1, n+1, \dots)
\]

\item $i_{D_2}$ is generated by quadrics.

\end{itemize}

\noindent It follows that $I_Z$ must have exactly $n+1$ minimal generators of degree 3  in addition to the $\binom{n}{2}-1$ generators of degree 2.

%

%
%

\item 
Remembering that $D_1$ and $D_2$ are both ACM, we sheafify and take cohomology:
\[
\begin{array}{cccccccccccccccccc}
0 & \rightarrow & I_X & \rightarrow & I_{D_1} \oplus I_{D_2} & \longrightarrow & \hspace{.2in} I_Z & \rightarrow & \bigoplus_{t \in \mathbb Z} H^1 (\mathcal I_X (t)) & \rightarrow & 0 \\
&&&& \hfill \searrow && \nearrow \hfill & \\
&&&&& I_{D_1} + I_{D_2} \\
&&&& \hfill \nearrow && \searrow \hfill & \\
&&&& 0 && \hspace{.2in} 0
\end{array}
\]

\item Since $I_Z$ is the saturation of $I_{D_1} + I_{D_2}$, the quotient 
\[
\frac{I_Z}{I_{D_1} + I_{D_2}}
\]
has finite length, so is zero in large enough degree.

\item Hence, for any  $t \gg 0$, we have
\[
- \dim [I_X]_t + \dim [I_{D_1}]_t + \dim [I_{D_2}]_t - \dim [I_Z]_t = 0
\]
so 
\[
h_{R/I_X} (t) = h_{R/I_{D_1}}(t) + h_{R/I_{D_2}}(t) - h_{R/I_Z}(t).
\]
From the $h$-vectors a quick computation gives that the Hilbert polynomials of $D_1, D_2$ and $Z$ are

\begin{tabular}{lll}
$\displaystyle P_{D_1}(t) = \binom{n+2}{2} + (t-2) \binom{n+1}{2}$; \\
$\displaystyle P_{D_2}(t) = t(n+1)$; \\
$P_Z(t) = 2n+2.$
\end{tabular}
\end{itemize}
The computation of the 
Hilbert polynomial is  straightforward.
This concludes the proof of the lemma.
\end{proof}

For the next result we recall from \cite{FGHM} the definition of $AV$-sequences.\index{$AV$-sequence}

\begin{definition} 
Let $X \subset \PP^n$ be a closed subscheme.  Let $P \in \PP^n$ be a general point. The {\it actual dimension} and the {\it virtual dimension} are given  respectively by

\begin{enumerate}

\item $\hbox{adim} (X,t,m) = \hbox{adim} (I_X,t,m) = \dim [I_X \cap I_P^m]_t$.\index{actual dimension}\index{dimension! actual}

\item $ \hbox{vdim} (X,t,m) = \hbox{vdim} (I_X,t,m) = \dim [I_X ]_t - \binom{m+n-1}{n}$.\index{virtual dimension}\index{dimension! virtual}

\item Fixing a nonnegative integer $j$, we define the sequence $AV_{X,j}$ as follows. 
\[
AV_{X,j}(m) = \hbox{adim}(X,m+j,m) - \hbox{vdim} (X,m+j,m), \ \ \ m \geq 1.
\]
In particular, $AV_{X,0}(m)$ measures the ``number" of  unexpected cones of degree $m$.
\end{enumerate}
\end{definition}

\begin{theorem} \label{AV of skeleton}
Let $A$ be the set of coordinate points of $\PP^n$, $|A| = n+1$. Let $C$ be the 1-skeleton of $A$. Then the AV-sequence for $C$ measuring unexpected cones is
\[
AV_{C,0} (m) = n+1
\]
for all 
$m \gg 0$.  In particular, for $m \gg 0$ the dimension of the linear system of cones of degree $m$ with general vertex $P$  containing $C$ is at least $n+1$ more than one expects.
\end{theorem}

\begin{proof}
Let $P$ be a general point in $\PP^n$ and let $X$ be the projection of $C$ to $\PP^{n-1}$. After a change of variables, we can view $X$ as the curve in Lemma \ref{XIsACM} (but remembering that now $X$ is in $\PP^{n-1}$, not $\PP^n$, so we have to adjust the statement). In particular, $X$ lies in $\PP^{n-1}$ with Hilbert polynomial 

\[
P_X(t) = \binom{n+1}{2} (t-1). 
\]
Since $X$ is the projection of $C$ from $P$,  we take $t = m \gg 0$ and this gives the actual dimension for the cones over $C$ of degree $m$ with vertex $P$ in $\PP^n$ (remembering that  we are applying Lemma \ref{XIsACM} to the projection, $X$, of $C$, so now $X$ is in $\PP^{n-1}$): 
\[
\hbox{adim} (C, m, m) = \binom{m+n-1}{n-1} - \binom{n+1}{2} (m-1) \ \ \ \hbox{ for } m \gg 0.
\]

For the virtual dimension we have
\[
\hbox{vdim} (C,m,m) = \dim[I_C]_m - \binom{m-1+n}{n}.
\]
Since the $h$-vector of $C$ is $(1, n-1, \binom{n}{2})$, we get the Hilbert function
\[
\begin{array}{rcl}
\Delta h_C(t) & = & (1,n,\binom{n+1}{2}, \binom{n+1}{2}, \dots ) \\
h_C (t) & = & (1, n+1, \binom{n+2}{2}, \binom{n+2}{2} + \binom{n+1}{2}, \binom{n+2}{2} + 2 \binom{n+1}{2}, \dots \\
& = & \binom{n+2}{2} + (t-2) \binom{n+1}{2} \ \ \ \hbox{ for } t \geq 2.
\end{array}
\]
Hence for  $m \gg 0$
we get
\[
\hbox{vdim} (C,m,m) = \binom{m+n}{n} - \left [ \binom{n+2}{2} + (m-2)\binom{n+1}{2} \right ] - \binom{m-1+n}{n}.
\]
A calculation then gives
\[
AV_{C,0}(m) = n+1
\]
as desired.
\end{proof}

We believe that Theorem \ref{AV of skeleton} is actually true for all $m \geq 3$, and this rests on proving in Proposition \ref{XIsACM} that $X$ is ACM. Nevertheless, we can still say something about when $m$ is large enough in Theorem \ref{AV of skeleton}.

By Theorem \ref{BnUnexpThm}, for $n \geq 5$ we have $\hbox{adim} (C,m,m) > 0$ for $m\ge3$ since $C$ has unexpected cubic cones. But we can say more. By \cite[Theorem 3.4]{FGHM}, a leftward shift of the $AV$-sequence\index{$AV$-sequence} by 1 is an $O$-sequence, i.e., it is the Hilbert function of some graded algebra. That is,
\[
AV_{C,j} (d+1) = h_{R/J}(d) 
\]
for $d \geq 0$, for any $j \geq 0$ and for a suitable ideal $J$ that depends only on $C$ and $j$. We know by Theorem \ref{AV of skeleton} that for $d$ sufficiently large, this takes the value $n+1$. We also know that $h_{R/J}(t) \geq t+1$ for all $t \leq n$ since otherwise Macaulay's bound prevents us from reaching the value $n+1$. Combining, we get 
\[
AV_{C,0}(m) = h_{R/J}(m-1) \geq m
\]
for all $0 \leq m \leq n+1$, $AV_{C,0}(m) \geq n+1$ for $m \geq n+1$, and $AV_{C,0}(m) = n+1$ for $m \gg 0$. Thus we can write

\begin{corollary}
Let $C$ be the 1-skeleton consisting of the coordinate lines of $\PP^n$, for $n \geq 5$. Then $C$ has unexpected cones of all degrees $m \geq 3$. Moreover, the dimension of the linear system of cones of degree $m$ with general vertex $P$ containing $C$ exceeds the expected dimension by: at least $m$ for all $3 \leq m \leq n+1$; at least $n+1$ for all $m \ge n+1$; and exactly $n+1$ for $m\gg 0$.
\end{corollary}


\subsection{Unexpectedness for codimension 2 simplicial skeleta}\label{subs: cod 2 sim skeleta}\index{unexpectedness! simplicial skeleta}\index{skeleton}
The previous section shows unexpectedness for simplicial skeleta is of interest.
Here we consider skeleta of codimension 2 in $\PP^n$, hence $n\geq 2$.

So let $Z$ be the set of codimension 2 faces of the coordinate simplex in $\PP^n$, $n\geq 2$;
i.e., $Z$ is the codimension 2 simplicial skeleton for $\PP^n$.
There is a bijection between the $$r=\binom{n+1}{n-1}=\binom{n+1}{2}$$
subsets of $n-1$ of the $n+1$ coordinate vertices and the components of $Z$,
since each such subset minimally spans a component of $Z$. 
Thus the number of components of $Z$ is $r$, so we can enumerate them as $Z_1,\ldots, Z_r$.

\begin{lemma}\label{cod2simplexLem1}
Let $Z$ be codimension 2 simplicial skeleton for $\PP^n$, $n\geq 2$.
Then $\dim [I(Z)]_m = 0$ for $m<n$, and 
$\dim [I(Z)]_m = \binom{m-1}{n} +(n+1)\binom{m-1}{n-1}$ for $m\geq n$.
\end{lemma}

\begin{proof}
Since $I(Z)$ is generated by all square-free monomials of degree $n$ \cite{GHM},
we see that $\dim [I(Z)]_m = 0$ for $m<n$, and $[I(Z)]_m$, for $m\geq n$, is spanned by all
degree $m$ monomials in which at least $n$ of the variables appear.
There are $\binom{m-1}{n}$ monomials of degree $m-(n+1)$; multiplying these by
$x_0\cdots x_n$ gives the set of all monomials of degree $m$ in which every variable appears.
Similarly, there are $\binom{m-1}{n-1}$ monomials of degree $m-n$ in the variables $x_1,\ldots,x_n$; multiplying these by
$x_1\cdots x_n$ gives the set of all monomials of degree $m$ in which every variable appears except $x_0$.
Doing this also for $x_1, x_2, ..., x_n$, we get the total count $(n+1)\binom{m-1}{n-1}$ 
of all monomials of degree $m$ in which some variable is missing.
Hence $\dim [I(Z)]_m = \binom{m-1}{n} +(n+1)\binom{m-1}{n-1}$ for $m\geq n$.
\end{proof}

\begin{lemma}\label{cod2simplexLem2}
Let $Z$ be the codimension 2 simplicial skeleton for $\PP^n$, $n\geq 2$.
Let $P\in\PP^2$ be a general point.
Then $\dim [I(Z+mP)]_m = 0$ for $m<r$, and 
$\dim [I(Z+mP)]_m = \binom{m-r+n-1}{n-1}$ for $m\geq r$.
\end{lemma}

\begin{proof}
Let $F\in [I(Z+mP)]_m$. Then $F$ vanishes on $Z_i$ and $P$,
but $F$ defines a cone with vertex $P$, so $F$ vanishes on 
the hyperplane spanned by $P$ and each $Z_i$.
Thus $F=CH_1\cdots H_r$, where $C$ is a form of degree $m-r$
vanishing to order $m-r$ at $P$, and every such form $F$
is in $[I(Z+mP)]_m$.
Thus $\dim [I(Z+mP)]_m=0$ for $m<r$, and 
$[I(Z+mP)]_m=H_1\cdots H_r [I((m-r)P)]_{m-r}$ for $m\geq r$.
In particular, 
$\dim [I(Z+mP)]_m = \dim [I((m-r)P)]_{m-r}=\binom{m-r+n-1}{n-1}$ for $m\geq r$.
\end{proof}

\begin{lemma}\label{cod2simplexLem3}
Let $Z$ be the codimension 2 simplicial skeleton for $\PP^n$, $n\geq 2$.
Then $Z$ has unexpected cones
of degree $m$ if and only if both $m\geq r$ and
$$\binom{m-r+n-1}{n-1} > \binom{m-1}{n} +(n+1)\binom{m-1}{n-1}-\binom{m+n-1}{n}.$$
\end{lemma}

\begin{proof}
The condition $m\geq r$ implies $\dim [I(Z+mP)]_m>0$ by Lemma \ref{cod2simplexLem2}, 
and, using Lemmas \ref{cod2simplexLem1} and \ref{cod2simplexLem2}, 
the displayed inequality then implies by definition that $Z$ has unexpected cones of degree $m$.

Conversely, if $Z$ has unexpected cones of degree $m$, then $\dim [I(Z+mP)]_m>0$
hence $m\geq r$ by Lemma \ref{cod2simplexLem2}, and, using  Lemmas \ref{cod2simplexLem1} and \ref{cod2simplexLem2},
by definition the displayed inequality must hold.
\end{proof}

\begin{remark}\label{cod2UnexpSkelRem}
If we write 
$$f(m,n)=\binom{m-r+n-1}{n-1} - \binom{m-1}{n} - (n+1)\binom{m-1}{n-1}+\binom{m+n-1}{n},$$
then for any specific $n$ it is not hard to check whether $f(m,n)>0$ for $m\geq r$.
For example, $f(m,2)=0$, so $Z$ has no unexpected cones for $n=2$.

But $f(m,3)=7$, so $Z$ has unexpected cones for all $m\geq r=6$ when $n=3$.
And $f(m,4)=25m-80$, so $Z$ has unexpected cones for all $m\geq r=10$ when $n=4$.
Similarly, $f(m,5)=(65/2)m^2-(625/2)m+996$ and this is certainly positive for $m\geq 10$, 
so $Z$ has unexpected cones for all $m\geq r=15$ when $n=5$.
\end{remark}

As a starting point for proving Theorem \ref{cod2simplexThm} by induction, we have the following proposition.

\begin{proposition}\label{cod2simplexProp}
Let $Z$ be the codimension 2 simplicial skeleton for $\PP^n$, $n\geq 3$.
Then $Z$ has an unexpected cone of degree $r=\binom{n+1}{2}.$
\end{proposition}

\begin{proof}
Since $\dim [I(Z+rP)]_r = \binom{n-1}{n-1}=1$ by Lemma \ref{cod2simplexLem2},
there is a unique cone of degree $r$.
By Lemma \ref{cod2simplexLem3}, it will be unexpected if $f(r,n)>0$.
Since
$$f(r,m)=1-\Big(\binom{r-1}{n} + (n+1)\binom{r-1}{n-1}-\binom{r+n-1}{n}\Big),$$
we must show 
$$\binom{r-1}{n} + (n+1)\binom{r-1}{n-1}-\binom{r+n-1}{n}\leq 0.$$
Factoring out common factors and simplifying slightly, this can be rewritten as
 $$\frac{1}{n!}\frac{\Big(\frac{n^2+n-2}{2}\Big)!}{\Big(\frac{n^2-n-2}{2}\Big)!}\Big(3+\frac{4}{n-1}- 
\frac{(n^2+3n-2)(n^2+3n-4)\cdots(n^2+n)}{(n^2+n-2)(n^2+n-4)\cdots(n^2-n)}\Big)\leq0.$$
Thus it suffices to show 
$$(*)=\frac{(n^2+3n-2)(n^2+3n-4)\cdots(n^2+n)}{(n^2+n-2)(n^2+n-4)\cdots(n^2-n)}$$
is bigger than $3+\frac{4}{n-1}$.

Note that $(*)$
is a product of $n$ fractions; they are
$$\frac{n^2+3n-2i}{n^2+n-2i}$$
for $1\leq i\leq n$ and for each we have
$$\frac{n+3}{n+1}<\frac{n^2+3n-2i}{n^2+n-2i},$$
hence
$$\Big(\frac{n+3}{n+1}\Big)^n<\frac{(n^2+3n-2)(n^2+3n-4)\cdots(n^2+n)}{(n^2+n-2)(n^2+n-4)\cdots(n^2-n)}.$$
Since $\Big(\frac{n+3}{n+1}\Big)^n$ is an increasing function of $n$ while $3+\frac{4}{n-1}$ is decreasing,
and for $n=5$ we have $\Big(\frac{n+3}{n+1}\Big)^n>3+\frac{4}{n-1}$, we see that
$(*)$, which is bigger than $\Big(\frac{n+3}{n+1}\Big)^n$, is also bigger than $3+\frac{4}{n-1}$ for all $n\geq5$.
For $n=3,4$ we check directly that $(*)$ is bigger than $3+\frac{4}{n-1}$.
(For $n=2$ we have $3+\frac{4}{n-1}=7$ but $(*)$ is 6.)
\end{proof}

It will be useful to have a recursion formula for $f(m,n)$.

\begin{lemma}\label{RecFormLem}
For $m>r$ and $n>2$ we have
$$f(m,n)=f(m-1,n)+f(m-1,n-1)+$$
$$\binom{m-\binom{n+1}{2}+n-2}{n-2}-\binom{m-\binom{n}{2}+n-3}{n-2}
-\binom{m-2}{n-2}+\binom{m+n-3}{n-2}.$$
\end{lemma}

\begin{proof}
Consider $f(m-1,n)+f(m-1,n-1)= $
$$\binom{m-\binom{n+1}{2}+n-2}{n-1} - \binom{m-2}{n} - (n+1)\binom{m-2}{n-1}+\binom{m+n-2}{n}+$$
$$\binom{m-\binom{n}{2}+n-3}{n-2} - \binom{m-2}{n-1} - (n)\binom{m-2}{n-2}+\binom{m+n-3}{n-1}.$$
By reordering the terms  and  using the well known fact $\binom{a}{b}=\binom{a-1}{b}+\binom{a-1}{b-1}$, 
this is equal to
$$
\left[\binom{m-\binom{n+1}{2}+n-2}{n-1}+\binom{m-\binom{n}{2}+n-3}{n-2}\right]
-\left[ \binom{m-2}{n}+ \binom{m-2}{n-1}\right] $$
$$-\left[ (n+1)\binom{m-2}{n-1}+ (n)\binom{m-2}{n-2}\right] +\left[\binom{m+n-2}{n}
  +\binom{m+n-3}{n-1}\right]
$$

$$
=\left[\binom{m-\binom{n+1}{2}+n-2}{n-1}+\binom{m-\binom{n}{2}+n-3}{n-2}\right]
-\left[ \binom{m-1}{n}\right] $$
$$-\left[(n+1)\binom{m-1}{n-1}-\binom{m-2}{n-2}\right] +\left[\binom{m+n-2}{n}
+\binom{m+n-3}{n-1}\right].
$$

Now we have $f(m,n)-f(m-1,n)-f(m-1,n-1)=$

$$
\binom{m-\binom{n+1}{2}+n-1}{n-1}-\binom{m-\binom{n+1}{2}+n-2}{n-1}-\binom{m-\binom{n}{2}+n-3}{n-2}$$
$$ - \binom{m-1}{n}+\binom{m-1}{n}$$
$$- (n+1)\binom{m-1}{n-1}+(n+1)\binom{m-1}{n-1}-\binom{m-2}{n-2}$$
$$+\binom{m+n-1}{n}-\binom{m+n-2}{n}
-\binom{m+n-3}{n-1}.$$
Again using the property $\binom{a}{b}-\binom{a-1}{b}=\binom{a-1}{b-1}$, this is equal to
$$
\binom{m-\binom{n+1}{2}+n-2}{n-2}-\binom{m-\binom{n}{2}+n-3}{n-2}$$
$$+0$$
$$-\binom{m-2}{n-2}$$
$$+\binom{m+n-3}{n-2}.$$
\end{proof}

\begin{lemma}\label{TBD-Lem}
For $m\geq r$ and $n> 3$ we have
$f(m,n)-f(m-1,n)-f(m-1,n-1)>0$.
\end{lemma}

\begin{proof}
By Lemma \ref{RecFormLem}, $f(m,n)-f(m-1,n)-f(m-1,n-1)$ equals
$$\binom{m-\binom{n+1}{2}+n-2}{n-2}-\binom{m-\binom{n}{2}+n-3}{n-2}
-\binom{m-2}{n-2}+\binom{m+n-3}{n-2}=$$
$$\binom{m-\binom{n}{2}-2}{n-2}-\binom{m-\binom{n}{2}+n-3}{n-2}
+\binom{m+n-3}{n-2}-\binom{m-2}{n-2}=$$
$$\Big[\binom{m+n-3}{n-2}-\binom{m-2}{n-2}\Big]
-\Big[\binom{m-\binom{n}{2}+n-3}{n-2}-\binom{m-\binom{n}{2}-2}{n-2}\Big]=$$
$$\sum_{i=n-3}^{m+n-4}\binom{i}{n-3}-\sum_{i=n-3}^{m-3}\binom{i}{n-3}-\Big[\sum_{i=n-3}^{m-\binom{n}{2}+n-4}\binom{i}{n-3}-\sum_{i=n-3}^{m-\binom{n}{2}-3}\binom{i}{n-3}\Big]=$$
$$\sum_{i=m-2}^{m+n-4}\binom{i}{n-3}-\sum_{i=m-\binom{n}{2}-2}^{m-\binom{n}{2}+n-4}\binom{i}{n-3}
=\sum_{i=m-\binom{n}{2}-2}^{m-\binom{n}{2}+n-4}\Big(\binom{i+\binom{n}{2}}{n-3}-\binom{i}{n-3}\Big),$$
where the terms of the last expression are positive since $\binom{a}{b}$ is an increasing function of $a$ 
for $a\geq b$ for fixed $b$ with $b>0$, and
where the third to last equality comes from 
$\binom{a}{b}=\sum\limits_{i=b-1}^{a-1}\binom{i}{b-1}$.
\end{proof}

\begin{theorem}\label{cod2simplexThm}
Let $Z$ be the codimension 2 simplicial skeleton for $\PP^n$. Then
$Z$ has unexpected cones of degree $m$ if and only if both $m\geq r=\binom{n+1}{2}$ and $n\geq 3$.
\end{theorem}

\begin{proof}
Codimension 2 implies $n\geq 2$, but by Remark \ref{cod2UnexpSkelRem}, 
$Z$ has no unexpected cones of any degree $m$ when $n=2$.
So assume $n\geq 3$.
By Lemma \ref{cod2simplexLem3}, if $Z$ has unexpected cones of degree $m$, then $m\geq r$, 
and conversely, if $m\geq r$, then $Z$ has unexpected cones of degree $m$ if
$f(m,n)>0$. We now verify $f(m,n)>0$ by induction.

By Proposition \ref{cod2simplexProp} and Lemma \ref{cod2simplexLem3}
we know that $f(m,n)>0$ when $n\geq3$ and $m=r$, and by Remark \ref{cod2UnexpSkelRem}
we know $f(m,3)>0$ when $n=3$ and $m\geq r$.

We will use a double induction.
We know that $f(i,3)>0$ for all $i\geq r$.
Suppose, by induction, that we know $f(i,j)>0$ for all $3\leq j<n$ and $i\geq r$.
Our first induction will show that then $f(i,j)>0$ for all $i\geq r$ with $j=n$. 
Our second induction, that
$f(i,j)>0$ for all $i\geq r$ and all $j\geq 3$, now follows.

For our first induction, we thus assume $f(i,j)>0$ for all $3\leq j<n$ and $i\geq r$.
Our base case is that $f(r,n)>0$. Assume by induction that $f(i,n)>0$ for all $r\leq i<m$. 
We want to show that $f(m,n)>0$. It will then follow that
$f(i,j)>0$ for all $3\leq j\leq n$ and $i\geq r$.

Since $m>r$ and $n>3$ we can apply Lemma \ref{TBD-Lem}:
$$f(m,n)>f(m-1,n)+f(m-1,n-1).$$
But by assumption we  have $f(m-1,n)>0$ and $f(m-1,n-1)>0$.
\end{proof}


Turning to skeleta of higher dimension and codimension we include the following amusing observation.
\begin{remark}
   Let $k=\lfloor\frac{N}{2}\rfloor$, $\ell=\lceil\frac{N}{2}\rceil$ and consider in $\PP^N$ a skeleton\index{skeleton} $S_{k-1,N+2}$ of $(k-1)$ dimensional projective subspaces spanned by $N+2$ general points in $\PP^N$.
   We assume these points to be the coordinate points $P_0,\ldots,P_N$ and the point $P=[1:\ldots:1]$. Let $Q=[a_0:\ldots:a_N]\in\PP^N$ be a general point.

   For all $N\geq 2$ there exists a hypersurface $T$ of degree $k+1$ vanishing along the skeleton $S_{k-1,N+2}$ and vanishing to order $\ell$ at $Q$. 

   One can write down an explicit equation of $T$:
   $$T=\sum_{\sigma\in \Sigma_{N+1}} \sgn(\sigma)\cdot
   x_{\sigma(0)}\cdots x_{\sigma(k)}\cdot
   a_{\sigma(1)}a_{\sigma(2)}^2\cdots a_{\sigma(k)}^k
   a_{\sigma(k+1)}a_{\sigma(k+2)}^2\cdots a_{\sigma(N)}^{N-k},$$
   where as usual $\Sigma_n$ is the group of permutations of $n$ elements.

   For $N\geq 4$ we believe that this hypersurface is  unique and unexpected.

   For example, in $\PP^2$, we have $k=1$ and $S_{0,4}$ is the union of $4$ general points. Certainly there is a unique conic passing through these points and a general point $Q$.
   
   For $\PP^3$, we have again $k=1$ and $S_{0,5}$ is the union of $5$ general points. There is a unique quadric cone with vertex at a general point $Q$ vanishing in the given $5$ points. 
\end{remark}

\chapter{Future work}  \label{ch:FW}

\section{Quantum mechanics\index{quantum mechanics} and geproci sets}
This section contains some experimental evidence of a possible interconnection between the so called KS-sets and the unexpected hypersurfaces. In particular, some KS-sets are geproci sets.

\subsection{Geometrical configurations in Quantum Mechanics}\label{QuantMech}
Consider $\field^{n+1}$ as the Hilbert space equipped with the standard scalar product,  \[\cdot \colon \field^{n+1}\times \field^{n+1} \to \field\  \text{ defined by}\ \
\ [p_0:\ldots: p_n]\cdot [q_0:\ldots: q_n]= \sum_{j=0}^np_jq_j.\]

Let $X$ be a finite subset of $\field^{n+1}$ (that we see as points in $\PP^n$),  a function $f\colon X\to \{0,1\}$ is called a {\it truth assignment on $X$}\index{truth assignment}. A set  $X$ is called a {\it KS-set}\index{KS-set} if there is not a truth assignment $f$ on $X$ such that
	\begin{itemize}
\item $f(P)f(Q)=0$ for any $P,Q$ orthogonal vectors; 
\item  $\sum_{j=0}^{n} f(P_j)=1$, for each $P_0,\ldots,P_n\in X$ giving an orthogonal basis of $\field^{n+1}$.
\end{itemize}

It takes the name from Kochen and Specker which used in 1967, in the quantum mechanics context, the existence of a KS-set in $\mathbb R^{3}$ to prove the so called {\it quantum contextuality}, see \cite{budroni2021quantum} for a survey on the problem. This result takes the name of Kochen-Specker Theorem\index{Theorem!Kochen-Specker} (also known as Bell-Kochen-Specker or BKS Theorem\index{Theorem!Bell-Kochen-Specker}) and each KS-set can be used to prove such a  result.

It is clear that KS-sets do not exist in $\field^2$ and it is known from\cite{cabello1996} that they exists in $\field^{n+1}$ for any $n>1$.

In $\PP^{3}$ several KS-sets are surprisingly related to the geproci property. In the next example we give a partial list of the known KS-sets in $\field^{4}$ (some subsets of the mentioned configurations are also KS-sets).
\begin{example}
	The following are KS-sets
\begin{itemize}
	\item[(a)] $F_4$, due to Peres (1991), see \cite{peres1991}, it is a (4,6)-geproci set;

	\item[(b)] $H_4$, that arises as the 60 rays of the 600-cell, due to Aravind and Lee-Elkin (1998), see \cite{aravind1998}, it is a (6,10)-geproci set;

	\item[(c)]  the Penrose configuration, due to Zimba and Penrose (1993),  see \cite{ZP}, it is a $(5,8)$-geproci set;

	\item[(d)] the 300 rays of the 120-cell, due to Aravind and Lee-Elkin (1998), see \cite{aravind1998}; 
	To construct the points of such a configuration we use the procedure described in \cite{waegell2014}. 

	Set $\varphi=\dfrac{1+\sqrt{5}}{2}$, i.e., the golden ratio\index{golden ratio} that is a root of $t^2-t-1=0$, and set  
	\[\small U=\dfrac{1}{2}\left(\begin{array}{cccc}
		1&1&1&-1\\
		1&1&-1&1\\
		1&-1&1&1\\
		1&-1&-1&-1
	\end{array}\right),\ V=\dfrac{1}{2}\left(\begin{array}{cccc}
	\varphi&0&-1&1/\varphi\\
	0&\varphi&-1/\varphi&-1\\
	1&1/\varphi&\varphi&0\\
	-1/\varphi&1&0&\varphi
\end{array}\right),\ W=\dfrac{1}{2}\left(\begin{array}{cccc}
1/\varphi&-\varphi&0&1\\
\varphi&1/\varphi&1&0\\
0&-1&1/\varphi&-\varphi\\
-1&0&\varphi&1/\varphi
\end{array}\right).\]
	These orthonormal matrices  have finite order, indeed $U^3$, $V^5$, $W^5$ are the identity matrix $id_4$.
	The columns of the matrices $W^nV^mU^l$, for $l=0,1,2$ and $m,n=0,\ldots,4$, give all the 300 elements.   
	This configuration contains several copies of $H_4$.  According to a computer calculation, it is not geproci. However, it admits several unexpected cones. Indeed, we get
	\[
	\begin{array}{rrrrr}
		m=& 22&23&24&25\\
		\adim(X,m,m)=&2&6&28&52\\ 
		\vdim(X,m,m)=&-24&0&25&51
	\end{array}
	\]
	It would be interesting to know if it can be extended to a larger set which is geproci.
\end{itemize}    
\end{example}

\begin{remark}\label{rem:120nonhalfgrid}
The method of generating sets of points by taking the columns of suitable matrices does not produce a geproci set in the case of the 300 rays of the 120-cell; however it works in some other cases.

Indeed, we note that: 
\begin{itemize}
    \item[(i)] The columns of $id_4$, $U$, $U^2$ give exactly the 12 points of the $D_4$ configuration.\index{configuration! $D_4$}
    \item[(ii)] Set \[T=\dfrac{1}{\sqrt{2}}\left(
    \begin{array}{cccc}
    1& -1& 0& 0\\
  -1& -1& 0& 0\\
  0& 0& -1& -1\\
  0&0& -1& 1
    \end{array}\right)\]
    Then $T^2=id_4$ and the columns of $T^iU^j$, for $i=0,1$ and $j=0,1,2$, give exactly the 24 points of the $F_4$ configuration.\index{configuration! $F_4$}
\item[(iii)]     The columns of $V^kU^j$, for $i=0,1$ and $k=0,1,2,3,4$, give exactly the 60 points of the $H_4$ configuration.\index{configuration! $H_4$}
\item[(iv)]     The columns of $V^kT^iU^j$, for $i=0,1$, $j=0,1,2$ and $k=0,1,2,3,4$, give 120 points of a new configuration that is a $(10,12)$-geproci set and not an half grid. (Indeed, experimentally, the point $[1:0:0:0]$ is collinear with sets of at most five other points of the configuration.) This set consists of 10 disjoint copies of $D_4$, or 5 disjoint copies of  $F_4$, or 2 disjoint copies of  $H_4$.   
\end{itemize} 

\end{remark}

\begin{remark}
	The examples  $F_4$ and $H_4$ come from root systems. The relation between  KS-sets and root systems is explored in \cite{KS-rootsystem} where, in particular, the author shows that also $E_7$ and $E_8$ are KS-sets but not $E_6.$ In analogy with the study of the unexpected hypersurfaces, as it was computed in \cite[Section 3.5]{HMNT}, we have that $E_7\subseteq \PP^6$ and $E_8\subseteq \PP^7$ configurations both admit unexpected cones; precisely
\[
\begin{array}{rcccl}
	\adim(E_8, 4,4)=&99&>&90&=\vdim (E_8, 4,4);\\
	\adim(E_8, 5,5)=&343&>&342&=\vdim (E_8, 5,5);\\
	\adim(E_7, 4,4)=&64&>&63&=\vdim (E_7, 4,4).
\end{array}
\]
The nonzero AV-sequences\index{$AV$-sequence}, see Definition 2.8 in \cite{FGHM}, of $E_6$ are $AV_{E_6,0}=(1,6,1)$ and $AV_{E_6,1}=(1)$. Then, from \cite[Proposition 3.8]{FGHM}, the degrees checked in \cite[Section 3.5]{HMNT} are enough to exclude of any unexpected hypersurfaces for $E_6$. 
\end{remark}

The orthogonality relations among the elements of a KS-set can be visualized by associating to it a simple graph.

Given a set $X$ of vectors in $\field^{n+1}$ (that we see as points in $\PP^n$) the graph $ G_X^{\perp}=(X,E)$ whose vertices are identified with the elements in $X$ and  $\{v,w\}\in E$ if and only if $v\perp w$, where $v,w\in X$, is called the {\it orthogonality graph of $X$}\index{graph! orthogonality}.

In the next example we explore some of the known KS-sets in $\PP^2$ and we check if some relation with unexpected curves occur.
\begin{example}
		Case $n=2$   
	\begin{itemize}
		\item In \cite{BKS13points} the authors consider a set of 13 vectors, 	\[X = \{100, 010, 001, 011, 01{*}, 101, 10{*}, 110, 1{*}0, {*}11, 1{*}1, 11{*}, 111\}.\] 
		The symbol ${*}$ denotes $-1$, colons and brackets are omitted. It was proved in \cite[Theorem 5]{fractionalchromatic-BKS} that this is the smallest KS-set in $\field^3.$
	
	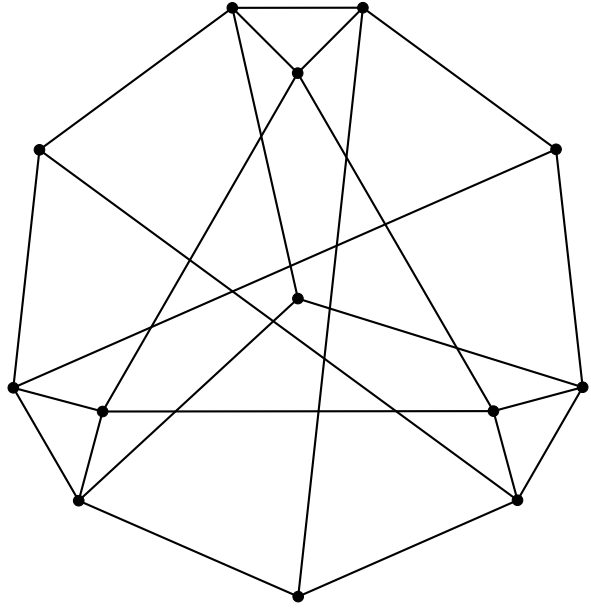
\begin{figure}
    \centering
\begin{tikzpicture}[line cap=round,line join=round,>=triangle 45,x=1cm,y=1cm,scale=0.8]
\draw [line width=0.8pt] (0.9877095533530323,4.079321342571693)-- (-1.1817637369486937,4.077148983049752);
\draw [line width=0.8pt] (-4.821382642897628,-2.2414590533069547)-- (-3.734764679214612,-4.119191855779091);
\draw [line width=0.8pt] (3.557119849801273,-4.111890270246643)-- (4.639975176419981,-2.231985108252567);
\draw [line width=0.8pt] (-0.09218441324777386,-0.7580093269939678)-- (-1.1817637369486937,4.077148983049752);
\draw [line width=0.8pt] (-0.09218441324777386,-0.7580093269939678)-- (-3.734764679214612,-4.119191855779091);
\draw [line width=0.8pt] (-0.09218441324777386,-0.7580093269939678)-- (4.639975176419981,-2.231985108252567);
\draw [line width=0.8pt] (-4.3870353923997385,1.715893268642893)-- (3.557119849801273,-4.111890270246643);
\draw [line width=0.8pt] (4.197703570637992,1.724489428601672)-- (-4.821382642897628,-2.2414590533069547);
\draw [line width=0.8pt] (-0.08722141798157573,-5.7144106782264705)-- (0.9877095533530323,4.079321342571693);
\draw [line width=0.8pt] (-1.1817637369486937,4.077148983049752)-- (-4.3870353923997385,1.715893268642893);
\draw [line width=0.8pt] (-4.3870353923997385,1.715893268642893)-- (-4.821382642897628,-2.2414590533069547);
\draw [line width=0.8pt] (-3.734764679214612,-4.119191855779091)-- (-0.08722141798157573,-5.7144106782264705);
\draw [line width=0.8pt] (-0.08722141798157573,-5.7144106782264705)-- (3.557119849801273,-4.111890270246643);
\draw [line width=0.8pt] (4.639975176419981,-2.231985108252567)-- (4.197703570637992,1.724489428601672);
\draw [line width=0.8pt] (4.197703570637992,1.724489428601672)-- (0.9877095533530323,4.079321342571693);
\draw [line width=0.8pt] (-0.09593909130284786,2.991680202702893)-- (-1.1817637369486937,4.077148983049752);
\draw [line width=0.8pt] (-0.09593909130284786,2.991680202702893)-- (0.9877095533530323,4.079321342571693);
\draw [line width=0.8pt] (-0.09593909130284786,2.991680202702893)-- (-3.3376334632422404,-2.6361057384211204);
\draw [line width=0.8pt] (-3.3376334632422404,-2.6361057384211204)-- (3.157019314801767,-2.629602445263676);
\draw [line width=0.8pt] (3.157019314801767,-2.629602445263676)-- (-0.09593909130284786,2.991680202702893);
\draw [line width=0.8pt] (3.157019314801767,-2.629602445263676)-- (4.639975176419981,-2.231985108252567);
\draw [line width=0.8pt] (3.157019314801767,-2.629602445263676)-- (3.557119849801273,-4.111890270246643);
\draw [line width=0.8pt] (-3.3376334632422404,-2.6361057384211204)-- (-3.734764679214612,-4.119191855779091);
\draw [line width=0.8pt] (-3.3376334632422404,-2.6361057384211204)-- (-4.821382642897628,-2.2414590533069547);
\begin{scriptsize}
\draw [fill=black] (-0.09218441324777386,-0.7580093269939678) circle (2.5pt);
\draw [fill=black] (0.9877095533530323,4.079321342571693) circle (2.5pt);
\draw [fill=black] (3.557119849801273,-4.111890270246643) circle (2.5pt);
\draw [fill=black] (-4.821382642897628,-2.2414590533069547) circle (2.5pt);
\draw [fill=black] (-1.1817637369486937,4.077148983049752) circle (2.5pt);
\draw [fill=black] (-3.734764679214612,-4.119191855779091) circle (2.5pt);
\draw [fill=black] (4.639975176419981,-2.231985108252567) circle (2.5pt);
\draw [fill=black] (-4.3870353923997385,1.715893268642893) circle (2.5pt);
\draw [fill=black] (4.197703570637992,1.724489428601672) circle (2.5pt);
\draw [fill=black] (-0.08722141798157573,-5.7144106782264705) circle (2.5pt);
\draw [fill=black] (-0.09593909130284786,2.991680202702893) circle (2.5pt);
\draw [fill=black] (-3.3376334632422404,-2.6361057384211204) circle (2.5pt);
\draw [fill=black] (3.157019314801767,-2.629602445263676) circle (2.5pt);
\end{scriptsize}
\end{tikzpicture}
	\caption[The orthogonality graph for a set of 13 vectors.]{The orthogonality graph of $X$.}
    \label{fig:orth-graph}
\end{figure}
		The set of 13 vectors $X$, seen as a set of 13 points in $\PP^2$, admits an unexpected curve of degree 6 and multiplicity 5 (See also \cite[Example 3.3]{FGHM} for more details)
		
		\[\qquad \qquad \qquad \quad  h_V=(1, 2, 3, 4, 2, 1)\]
		\[\begin{array}{rrrrrrrrrrrrrrrrrrrrr}
			d=                & 5&6&7 \\
			\vdim(X, d, d-1)= &  -2& 0 &2\\
			\adim(X, d, d-1)= &  0& 1& 2
		\end{array}
		\]
		Moreover, there are 12 pairs of orthogonal vectors not in a basis. Thus, such configuration can be extended by completing all these pairs to a basis. With the addition of the points \[1{*}2,\ {*}{*}2,\  {*}12\ ,112,\
		\]
		and all their permutations, it is possible to construct a larger KS-set consisting of 25 elements which again admits an unexpected curve of degree 12 with respect a general point of multiplicity 11.

		\item The authors in \cite{BKS21points} consider the following set of 21 vectors of $\field^3$, where $q$ denotes a primitive third root of unity.
		\[
		X= \begin{array}{ccccccc}
			(0,1,-1), & (0,1,-q),& (0,1,-q^2),& 
			(-1,0,1),& (-q,0,1),& (-q^2,0,1),\\
			(1,-1,0),& (1,-q,0),& (1,-q^2,0),& 
			(1,0,0),& (0,1,0),& (0,0,1),\\
			(1,1,1),& (1,q,q^2),& (1,q^2,q),& 
			(1,q^2,q^2),& (q^2,1,q^2),& (q^2,q^2,1),\\
			(1,q,q),& (q,1,q),& (q,q,1). 
		\end{array}
		\]
		Computing with CoCoA  the actual and the virtual dimension of $I(X)\subseteq \field[\PP^2]$ we see an unexpected curve of degree $8$
		\[h_V=\ (1, 2, 3, 4, 5, 3, 2, 1)\]
		\[\begin{array}{rrrrrrrrrrrrrrrrrrrrr}
			d=                 &7&8&9 &10&11 &12 &13    \\
			\vdim(X, d, d-1)= &  -6& -4& -2& 0& 2&4&6\\
			\adim(X, d, d-1)= &  0& 1& 2& 3& 4&5 &6\\ 
		\end{array}
		\]
	\item The following are the 33 rays of Peres, see \cite{peres1991}, where $v=\sqrt{2}$.	
		\[
		\begin{array}{l}
	100,\ 010,\ 001,\
	110,\ 101,\ 011,\ 1{*}0,\ 10{*},\ 01{*},\\
	01v,\ 10v,\ 1v0,\ 0{*}v,\ {*}0v,\ {*}v0,\ 
\	0v1,\ v01,\ v10,\ 0v{*},\ v0{*},\ v{*}0,\\
	11v,\ {*}{*}v,\ 1{*}v,\ 
	{*}1v,\ 
1v1,\ 1v{*},\ {*}v{*},\ {*}v1,\
v11,\ v1{*},\ v{*}1,\ v{*}{*}. 		
		\end{array}
	\]
	
	However, the computer calculation shows that the generic initial ideal with respect the lex ordering of these 33 points in $\PP^2$ is a lex segment ideal. Thus, by \cite[Proposition 3.9]{FGHM} no unexpected curves of any kind are admitted. 
	It can be extended to a set of 57 points completing all the pairs of orthogonal vectors to a basis. However, even the larger configuration admits no expected curves since its lex-gin is a lex segment ideal.
	\end{itemize}
\end{example}

%

A different approach to look at KS-sets is in terms of fractional chromatic numbers.

\subsection{The fractional chromatic number of a graph} 

A $k$-coloring of a graph\index{graph! $k$-coloring} $G=(V,E)$ is a function $f:V\to \{1,\ldots, k\}$ such that for each $\{u,v\}\in E$ we have $f(u)\neq f(v)$. A graph G is $k$-colorable\index{graph! $k$-colorable} if there exists a $k$-coloring of $G.$ The chromatic number\index{graph! chromatic number}, $\chi(G)$, of a graph $G$ is the smallest $k$ such that $G$ is $k$-colorable. A graph G is $k$-chromatic\index{graph! $k$-chromatic} if $\chi(G)=k$.
A graph $G$ is {\it critical $k$-chromatic}\index{graph! critical $k$-chromatic} if $G$ is $k$-chromatic and every proper subgraph of $G$ is $(k-1)$-colorable. For instance, the complete graph $K_m$ is critical $m$-chromatic.

We say that $W \subseteq V$ is an independent set of $G=(V,E)$ if $e \not\subseteq W$ whenever $e \in E.$
An independent set of $G$ is  maximal if it is not contained in a larger independent set.

Maximal independent sets have an algebraic interpretation in terms of associated primes of the face ideal of $G$. Recall that, given a graph $G=(V,E)$ the {\it face ideal of $G$}\index{graph! face ideal}, denoted by  $I(G)\subseteq K[V]=K[x_v\ | \ v\in V],$ is the ideal generated by  $x_vx_w$ for $\{v,w\}\in E.$ Then, see for instance \cite[Lemma 2.6]{wc-squarefree}, if $I(G)=\mathfrak p_1\cap \ldots\cap \mathfrak p_r$ is the  minimal primary decomposition of $I(G)$ then the sets $W_i = {v\in V\ |\  x_v\notin\mathfrak p_i}$, where $i=1,\ldots,r$, are maximal independent sets of~$G.$

Let $\mathcal I=\{W_1,\ldots, W_r\}$ be the set of maximal independent sets of a graph $G=(V,E)$, for a vertex $v\in V$ we denote by $\mathcal I_v=\{W\in \mathcal I \ |\ v\in W \}$,  then the fractional chromatic number\index{graph! fractional chromatic number} of $G$ is
\[
\chi_f(G)=\min\left\{\sum_{i=1}^r y_i\ |\  y_1,\ldots,y_r\ge 0\ \text{and}\  \sum_{i\colon W_i \in \mathcal I_v} y_i\ge 1 \text{ for all}\ v\in V  \right\}.
\]

It is known that $\chi(G)\ge \chi_f(G)$.

A consequence of the more general \cite[Theorem 4.6]{wc-squarefree} 
is that the fractional chromatic number of $G$ can be written in terms of the Waldschmidt constant of the face ideal of $G$,
precisely 
$$\chi_f(G)=\dfrac{\widehat{\alpha}(I(G))}{\widehat{\alpha}(I(G))-1}.$$

The chromatic number of the orthogonality graph strictly characterizes the KS-sets. Indeed \cite[Theorem~2]{fractionalchromatic-BKS} shows that
\begin{center}\it
	A set of vectors $X\subseteq \field^{n+1}$ is a KS-set 
	if and only if
	$\chi_f(G_X^{\perp})>n+1.$
\end{center}

There is also an interesting reformulation of the Kochen-Specker theorem in terms of spectral presheaf, see for instance \cite{IB98,AB11}.

\begin{question}
What is exactly the connection between the KS-sets and the geproci property?
\end{question}
\begin{question}
Are the sets in Theorem \ref{t. geproci infinite class}, up to a change of coordinate system, KS-sets?  
\end{question}

\section{A possible geproci pairing property}
\subsection{Coupling two copies of \texorpdfstring{$D_4$}{$D_4$} to form \texorpdfstring{$F_4$}{$F_4$}}\label{Sec: D4 in F4}
We continue the discussion of the geometry of lines determined by points in the $D_4$ configuration\index{configuration! $D_4$} initiated in \ref{GridColinearitiesProp}. We keep working with coordinates of points in $D_4$ as in Remark \ref{rem:onD4}, where $*$ denotes $-1$:
$$1100,\;
{*}100,\;
1010,\;
{*}010,\;
1001,\;
{*}001,\;
0110,\;
0{*}10,\;
0101,\;
0{*}01,\;
0011,\;
00{*}1.
$$
It is convenient to explicitly list points lying on the 18 $2$-point lines found in Proposition \ref{GridColinearitiesProp} and the 16 $3$-point lines found in Proposition \ref{D4factsProp}. 

Here are the 16 3-point lines (given by listing the 3 points on each line).
There are $4$ lines in each coordinate plane as follows:
\medskip

\begin{minipage}{.25\textwidth}
\baselineskip=18pt
\begin{center}
    $w=0$

1100, 0110, $*$010;

1100, 1010, 0$*$10;

0110, 1010, $*$100;

$*$100, $*$010, 0$*$10;
\end{center}
\end{minipage}%
\begin{minipage}{.25\textwidth}%
\baselineskip=18pt
\begin{center}
    $z=0$

1100, 0101, $*$001;

1100, 1001, 0$*$01;

0101, 1001, $*$100;

$*$100, $*$001, 0$*$01;
\end{center}
\end{minipage}%
\begin{minipage}{.25\textwidth}%
\baselineskip=18pt
\begin{center}
    $y=0$

1010, 0011, $*$001;

1010, 1001, 00$*$1;

0011, 1001, $*$010;

$*$010, $*$001, 00$*$1;
\end{center}
\end{minipage}%
\begin{minipage}{.25\textwidth}%
\baselineskip=18pt
\begin{center}
    $x=0$

0110, 0011, 0$*$01;

0110, 0101, 00$*$1;

0011, 0101, 0$*$10;

0$*$10, 0$*$01, 00$*$1.
\end{center}
\end{minipage}%
\medskip

Now we list the 18 2-point lines. There are 6 for each of the 3 pairs of disjoint coordinate lines. We explicitly exhibit only those for the disjoint pair
$x=y=0$ and $z=w=0$. The points 0011, 00$*$1, 1100 and $*$100 are on one or the other of
these two lines, and each of the 6 pairs of points taken from this set of 4 points
defines a 2-point line. There are 6 more 2-point lines coming in the same way from
the disjoint pair of lines $x=z=0$ and $y=w=0$ and the remaining 6 come 
from the disjoint pair $x=w=0$ and $y=z=0$.

There are 12 points of concurrency for the 18 2-point lines, where each of 
the 12 points is on exactly 3 of the 18 2-point lines. 
We begin by exhibiting 12 such points and the 3 2-point lines for each point.
(We also recall that these 12 points are dual to the 6-point planes of ${D_4}$, see Lemma \ref{CoplanarLem}.)
\medskip\medskip

\noindent
\begin{minipage}{.5\textwidth}%
\baselineskip=18pt
\begin{center}
1000: 1100, $*$100; 1010, $*$010; 1001, $*$001.

0100: 1100, $*$100; 0101, 0$*$01; 0110, 0$*$10.

0010: 1010, $*$010; 0110, 0$*$10; 0011, 00$*$1.

0001: 1001, $*$001; 0101, 0$*$01; 0011, 00$*$1.

1111: 1100, 0011; 1010, 0101; 0011, 1100.

$*$111: $*$100, 0011; $*$010, 0101; $*$001, 0110. 
\end{center}
\end{minipage}%
\begin{minipage}{.5\textwidth}%
\baselineskip=18pt
\begin{center}
1$*$11: 1$*$00, 0011; 0$*$10, 1001; 0$*$01, 1010. 
\\

11$*$1: 0$*$10, 1001; $*$010, 0101; 00$*$1, 1100. 

111$*$: 0$*$01, 1010; 00$*$1, 1100; $*$001, 0110. 

$**$11: 1100, 0011; $*$010, 0$*$01; $*$001, 0$*$10.

$*$1$*$1: 1010, 0101; $*$100, 00$*$1; $*$001, 0$*$10.

$*$11$*$: 1001, 0110; $*$100, 00$*$1; $*$010, 0$*$01.
\end{center}
\end{minipage}%

\begin{remark}[$F_4$ configuration as a union of two $D_4$]\label{def:F4}\index{configuration! $F_4$} We note that the $12$ points of $D_4$ together with these $12$ points of concurrency give the $F_4$ configuration as shown in Section \ref{sec.geometryF4}.
\end{remark}
Since $F_4$\index{configuration! $F_4$} contains a $D_4$, and the latter is unique up to a projective change of coordinates, and $F_4$ extends $D_4$ in a geometrically determined way, we see again that $F_4$ is unique up to projective change of coordinates, see Lemma \ref{l. Justyna's lemma}. 
\begin{lemma}
   The $12$ points
of concurrency, as listed above, are projectively equivalent to~${D_4}$.    
\end{lemma}
\begin{proof}
   We conclude simply by providing explicitly a projective transformation transforming the points of $D_4$ to the set of concurrency points:
$$
M=\begin{pmatrix}
1 & 0 &  1 &  0\\ 
1 & 0 & -1 &  0\\ 
0 & 1 &  0 &  1\\ 
0 & 1 &  0 & -1\\ 
\end{pmatrix}.
$$
   The rest is a small computation.
\end{proof}
Since the 12 points of concurrency comprise a ${D_4}$ (which we will denote 
by $D_4'$), we know that $D_4'$ has
16 3-point lines. There is a canonical correspondence between
the 16 3-point lines of ${D_4}$ and those of $D_4'$.
Removing the three points of a 3-point line $L$ from ${D_4}$ gives
a $(3,3)$-grid. This grid has two sets of 3 collinear points 
(the Brianchon points, see Remark \ref{BrianchonRem})\index{points!Brianchon}. One set consists of the three removed points;
the other set consists of the 3 points on a 3-point line $L'$ of $D_4'$.
Under the canonical correspondence, $L$ corresponds to $L'$.
All of the 3-point lines of ${D_4}$ are contained in coordinate planes
and each coordinate plane contains 4 3-point lines, none of which  
go through a coordinate vertex; thus
none of the 3-point lines of ${D_4}$ contains a point of $D_4'$,
and likewise none of the 3-point lines of $D_4'$ contains a point of ${D_4}$.
Thus ${F_4}$ has 32 3-point lines (16 from ${D_4}$ and 16 from $D_4'$).

Here is an alternative description of the correspondence between the 
3-point lines of ${D_4}$ and those of $D_4'$.
Given a 3-point line $L$ of $D_4$, say the line through
$*$100, $*$010 and 0$*$10, the corresponding 3-point line of $D_4'$ is defined by
the planes dual to the 3 points of $L$, namely $-x+y$, $-x+z$ and $-y+z$;
this is a 3-point line of $D_4'$ whose 3 points are
1111, 111$*$ and~0001.

There is an even simpler correspondence between the 18 2-point lines of
${D_4}$ and the 18 2-point lines of $D_4'$: each 2-point line of ${D_4}$ is also a 2-point
line of $D_4'$ and hence is a 4-point line for ${F_4}$. 

We summarize the discussion with the following proposition.
\begin{proposition}[Collinearities in $F_4$]\label{prop:coll_in_F4}
   There are $60$ lines passing through exactly $2$ points in the $F_4$ configuration, $32$ through exactly $3$ points and $18$  through exactly $4$ points.
\end{proposition}
\begin{proof}
It amounts to a combinatorial count. Since ${F_4}$ has 24 points, it has $\binom{24}{2}=276$ pairs of points.
The 36 3-point lines of ${F_4}$ account for $3\cdot 36=108$ of these pairs
and the 18 4-point lines account for another $6\cdot 18=108$ pairs. The remaining $60$ possibilities of pairing points in $F_4$ account for the $60$ 2-point lines.
\end{proof}

We now determine the group of projective automorphisms of ${F_4}$. 
\begin{proposition}\label{F4GroupProp}
Let $H=\Aut(F_4)$ be the group of linear transformations of $\PP^3$ mapping a fixed $F_4$ to itself. And let $G=\Aut(D_4)$ be the group of projective transformation of a $D_4$ contained in this $F_4$.
Then $H\cong G\rtimes {\mathbb Z}_2$. 
\end{proposition}
\begin{proof}
The group $G$ of automorphism of ${D_4}$ (see Proposition \ref{GridGroupProp})
preserves its 6-point planes and thus the points of $D_4'$ and thus $G$ is a
subgroup of $H$. The fact that $M\in H$ shows that $G$ is a proper subgroup.
On the other hand, the 3-point lines of ${D_4}$ are disjoint from the
3-point lines of $D_4'$. Thus the incidence graph for the 3-point lines of ${F_4}$
has two connected components; one component comes from the 16 3-point lines of
${D_4}$, the other component comes from the 16 3-point lines of
$D_4'$. Thus there is a surjective homomorphism $H\to {\mathbb Z}_2$ with kernel
$G$, i.e., $|H|=2|G|$ and $H\cong G\rtimes {\mathbb Z}_2$.
\end{proof}
Note that ${D_4}$ and $D_4'$ are not the only subsets projectively equivalent to ${D_4}$ contained in ${F_4}$. Swapping the 3 points on a 3-point line of ${D_4}$ for the 3 points 
on the dual 3-point line in $D_4'$ gives a subset of ${F_4}$ of 12 points 
projectively equivalent to ${D_4}$. This gives 16 such subsets in addition to ${D_4}$
itself. Doing likewise with $D_4'$ gives 17 more for at least 34 altogether.


\medskip

We now consider another, more classical, way to view ${D_4}$ and ${F_4}$.
The 8 vertices of a cube have $\binom{8}{2}=28$ pairs, which define 28 lines.
These 28 lines meet in 24 points; these are the points of ${F_4}$.
There are 12 points where exactly two of the lines meet; these are the points
of ${D_4}$. There are 12 points where 4 or more of these lines are concurrent;
these are the remaining points of ${F_4}$, and so 
give $D_4'$, and hence is projectively
equivalent to ${D_4}$. To be more specific, the points 
of $D_4'$ are the vertices of the cube, the center of the cube, and the three points at infinity where parallel edges meet.
The points of ${D_4}$ are the centers of the faces of the cube
(i.e., where diagonals on the faces intersect), and the 6 points at infinity where these diagonals meet the plane at infinity
(i.e., the six points where parallel face diagonals meet).

Note that 6 of the points of ${D_4}$ (the 6 contained in real affine
3-space, where the last coordinate is 1) form the vertices of a regular
octahedron. The other 6 points are where the lines defined by the edges of
the octahedron intersect the plane at infinity. It is easy to pick 3 edges of the octahedron which define skew lines; call the cubic curve they form $C$. The edges opposite these three
define three skew lines forming another cubic curve, $D$.
Together, $C$ and $D$ give the six grid lines of a $(3,3)$-grid.
Three of the 6 points of ${D_4}$ at infinity are on $C\cap D$; the other three
of these 6 points are collinear and are off $C\cup D$. The line $L$ through
these 3 points, together with either $C$ or $D$, contains ${D_4}$,
which shows that there is a quartic cone, whose vertex is a general point,
which contains ${D_4}$. Moreover, a linear combination of the forms defining
the cones over $C$ and $D$, with a general point as vertex,
can be made which vanishes at a 10th point of ${D_4}$. This defines
a cubic cone with general vertex that in fact contains all of ${D_4}$.
Not every $(3,3)$-grid contained in ${D_4}$ has its grid lines coming from edges
of the octahedron. The complement in ${D_4}$ of any set $S$ of three collinear points 
in ${D_4}$ is a $(3,3)$-grid whose grid lines are disjoint from the line
containing $S$. Thus there are 16 $(3,3)$-grids contained in ${D_4}$; these 16 are
the only ones.

There are 16 lines containing exactly 3 points of $D_4'$.
These are given by the edges of the cube and the 4 main diagonals of the cube.
There are 18 lines containing exactly four points each of ${F_4}$ (two points each from ${D_4}$ and two from $D_4'$). These 18 lines
consist of the 12 diagonals of the faces of the cube mentioned above, and the 
6 coordinate lines of $\PP^3$. There are 9 $(4,4)$-grids in ${F_4}$.
The 8 points of ${F_4}$ complementary to a given $(4,4)$-grid are contained in
two skew lines, each containing 4 of the points, and neither of these two lines
is in the complement of any other $(4,4)$-grid contained in ${F_4}$.
Thus the 18 lines come in 9 pairs, and the set of 8 points on each pair 
is complementary to one of the 9 $(4,4)$-grids. Moreover, the 8 grid lines consist
of two sets of 4 skew lines, and each set of four consists of two of the 
aforementioned 9 pairs. There are 6 ways to choose 3 of these 9 pairs
to get a set of 6 skew lines. The cone (with a general point as vertex) 
over such a set of 6 skew lines gives a sextic cone containing ${F_4}$.
Each pair of the 9 pairs occurs in two sets of skew lines
(since the complement in ${F_4}$ of the 8 points in the pair of skew lines
is a $(4,4)$-grid, and thus contains two sets of 4 skew lines,
both sets disjoint from the lines in the pair).

In terms of the 18 lines through pairs of vertices of the cube
coming as 9 pairs of two skew lines,
the two lines in each pair are as follows. Each diagonal on a 
face of the cube is paired with the diagonal on the opposite face
skew to it. This gives 6 of the pairs. The remaining 6 lines form
a coordinate tetrahedron; each coordinate line of this tetrahedron
is paired with the one coordinate line skew to it (i.e., each
coordinate line is paired with the coordinate line
opposite to it). This means there is no projective transformation
which takes each line to its partner, because it would have to
transpose the partners for all 6 coordinate lines, and no such
transformation exists. 

The incidence matrix of the 9 pairs and the 6 sets of three pairs forming each
set of 6 skew 4-point lines is the same as that for 
a $(3,3)$-grid, namely:

\renewcommand{\arraystretch}{1.2}
\begin{table}[ht]
\centering
\begin{tabular}{|c|c|c|c|c|c|c|c|c|}
\hline
1&1&1&0&0&0&0&0&0\\
\hline
0&0&0&1&1&1&0&0&0\\
\hline
0&0&0&0&0&0&1&1&1\\
\hline
1&0&0&1&0&0&1&0&0\\
\hline
0&1&0&0&1&0&0&1&0\\
\hline
0&0&1&0&0&1&0&0&1\\
\hline
\end{tabular}
\end{table}

In analogy with what happened with ${D_4}$, the grid lines of a $(4,4)$-grid
in ${F_4}$ give two curves, $C$ and $D$, but now they are quartics, each consisting of 4 skew lines. The two cones, one over $C$ and 
one over $D$, each with a common general vertex,
define a pencil of cones and some member of the pencil contains ${F_4}$.
Thus ${F_4}$ is contained in a quartic cone with a general vertex.

If the 4 points on any of the 18 4-point lines of ${F_4}$ is removed
from ${F_4}$, we get a half grid $(4,5)$-geproci set.
If the eight points on the 2 skew lines forming one of the 9 pairs
is deleted from ${F_4}$ we get a $(4,4)$-grid. If we delete from ${F_4}$
eight points on two skew lines but the two lines do not form one of the 9 pairs,
then we get a half grid $(4,4)$-geproci set.

\subsection{Other examples of the coupling property}

We note that ${D_4}$ is a union of two disjoint $(2,3)$-grids.
The quartic cone with respect to which ${D_4}$ is $(3,4)$-geproci
can be taken to be the union of the cones over the two 3-point lines of
the two $(2,3)$-grids. 

Similarly, ${F_4}$ is a union of 2 disjoint subsets, $Z$ and $Z'$, both
projectively equivalent to ${D_4}$.
The sextic with respect to which ${F_4}$ is $(4,6)$-geproci
can be taken to be the union of the cones over the two cubics coming from
$Z$ and $Z'$. And the quartic for $Z$ and $Z'$ can be taken to be
the unique quartic which comes from ${F_4}$.

In the same way, the Penrose configuration $Z_{P}$ is the union of two $(4,5)$-geproci sets $Z_{20}$ and $Z'_{20}$.
The quintic cone for $Z_P$ gives us a quintic cone which works for
both $Z_{20}$ and $Z'_{20}$, and the quartics which work for
$Z_{20}$ and $Z'_{20}$ together give us an octic cone which works for~$Z_P$.

A similar situation holds for ${H_4}$ \cite{FZ} and for the 60 point Klein configuration of \cite{PSS}.

\begin{question} \label{half} How often does an $(a,b)$-geproci set $Z$ have an 
$(a,b)$-geproci partner $Z'$ such that $Z\cup Z'$ is $(b,2a)$-geproci or $(2a,b)$-geproci?
\end{question}

\begin{remark}
We can get many examples of the behavior suggested in Question \ref{half} for free using the ideas behind Theorem \ref{t. (a,b)-geproci}. Specifically, if $b$ is even and $4 \leq 2a \leq b+2$ then one can start with a $(b,b)$-grid (as in Theorem \ref{t. geproci infinite class}), add two nongrid rows (since $b$ is even), and then remove grid rows until there are $2a$ remaining rows. Then our two ``half" subsets will consist of $a-1$ rows from the grid and one nongrid row.
\end{remark}

\begin{question}More generally, is a geproci set always a union of two disjoint geproci subsets?
\end{question}

\section{The Cayley-Bacharach property and the geproci property}

A set  of $N$ points in $Z\subseteq \PP^n$ has the Cayley-Bacharach Property\index{Cayley-Bacharach Property} (CBP) if all its subsets of  $N-1$ points have the same Hilbert function.  See   for instance the references \cite{EGH, GKR}. 
The CBP is equivalent to having
$\dim[I(Z)]_d=\dim[I({Z\setminus \{P\}})]_d\neq 0$  for $d$ equal to the regularity of $Z$, that is the regularity $R/I(Z)$, and for each $P\in Z$, 
any hypersurface of degree $d$ containing  $Z\setminus\{P\}$ also contains $Z$.


In our context, it is natural to look at the CBP in terms of general projections.

\begin{definition}\label{d.geproCB}
We say that $Z$ has the Cayley-Bacharach Property with respect to a general projection, so $Z$ is geproCB\index{geproCB} for short, if a general projection $Z'$ of $Z$ has the CBP.   
\end{definition}
 This is equivalent to say that if $P$ is any point in $Z$, any cone of degree $d$ with general vertex containing $Z\setminus \{P\}$ also contains $P$. 
Thus we have the following observation.
\begin{proposition}\label{p.geproci c-CBP}
Let $Z$ be an $(a,b)$-geproci in $\PP^3$. Then $Z$ is a geproCB set.
\end{proposition}
\begin{proof}
The projection of $Z$ from a general point $P$ gives a set $\pi(Z)$ which is a complete intersection of $ab$ points in $\PP^2$. 
From \cite[Theorem CB4]{EGH}, the general curve of degree $a+b-3= \reg(\pi(Z))$ containing all but one point of $\pi(Z)$ also contains all the points in $\pi(Z)$. 
Therefore, the cone of degree $a+b-3$ and vertex at $P$ containing all but one point of $Z$ also contains all the points in $Z$.
\end{proof}

\begin{question}
	Let $Z$ be a set of reduced points in $\PP^n$ with the CBP. Is $Z$ a geproCB set? That is, does the general projection of $Z$ into  $\PP^{n-1}$ also have the CBP? 
\end{question}
The converse of the above question is false, i.e., the geproCB property does not imply the CBP. 

\begin{example}Let $Z$ be a set of 11 points in $\PP^3$; where ten of them are generically chosen on a quadric and the eleventh point is a general point of $\PP^3$. Then, the last point has a separator of degree 2 that all the other points don't have. Thus, the $h$-vector of the ten points on the quadric is different from the $h$-vector of any other subset of ten points in $Z$, so $Z$ does not have the CBP. 

On the other hand, the general projection of the 11 points gives a set of general points in $\PP^2$, whose $h$-vector is $(1,2,3,4,1)$, and there is no subset of ten points on a cubic. Hence, the removal of any point has $h$-vector $(1,2,3,4)$, so the projection has the CBP and thus $Z$ is geproCB.
\end{example}

Experimentally, all examples of geproci sets we checked have the CBP, so one can also ask the following question.

\begin{question}
Does the geproci property imply the CBP? 
\end{question}

\begin{question}
Can one characterize the geproCB sets of points in $\PP^{n}$?
\end{question}

It is shown in \cite{DGO1985} that sets of points with the CBP and a symmetric $h$-vector are precisely the 
arithmetically Gorenstein (AG) sets of points. So, one can also investigate in this direction.
Let $Z\subseteq\PP^n$ be a nondegenerate set of reduced points. Then the general projection of $Z$ is AG if and only if $Z$ is geproCB set and the general projection of $Z$ has symmetric $h$-vector.
We call a such set $Z$ a {\it geproAG}. 

\begin{remark} At least one geproAG set of points exists in any dimension. 
Indeed, the general projection of $n+1$ general points in $\PP^n$ gives an arithmetically Gorenstein set in $\PP^{n-1}.$
This extends the fact that the set of four general points in $\PP^3$ is geproci.
\end{remark}

Then, the following questions arise.
\begin{question}
Can one characterize the sets of points in $\PP^{n}$ such that the general projection to $\PP^{n-1}$ has a symmetric $h$-vector?
\end{question}

\begin{question}
	Is there an infinite class of geproAG sets of points in $\PP^n$, $n\ge 4$?
\end{question}

Given a set  $Z$ which is   $(a,b)$-geproci, and given $Q\in Z$, the geproCB property implies that  the intersection of the general cones of degree $a+b-3$ through $Z\setminus \{Q\}$ is $Z$. 
We now generalize this property.
\begin{definition}
Let $Z\subset \PP^3$ be a
finite set of points and $W\subseteq Z$. Let $P\in \PP^3$ be a general point and let
$m\geq 2$ be an integer. Denote by  
$C_{P,m}(W)$ the general cone with vertex at $P$ and containing $W$, and by
   $$W_m=\bigcap_{P\in U} C_{P,m}(W),$$
where the intersection is taken over all $P\in U\subset \PP^3$
Zariski open.
We say that the set $W$ {\it remembers} the set
   $Z$ at cones of degree $m$ if $Z\subseteq W_m.$ When the equality holds, i.e., $Z= W_m$, we say that the set $W$ {\it precisely remembers} the set
   $Z$ at cones of degree $m$.
\end{definition}
Of course,  in the above definition $W_m$ is the intersection
over a finite set $\{P_1, \ldots, P_s\}$ of general points, for $s$
sufficiently large.

\begin{remark}
Note that any set $Z$ precisely remembers itself at cones of sufficiently large degree, i.e., $Z=Z_m$, for $m\gg 0$. Indeed, if $m\ge |Z|$ then the intersection of completely reduced cones containing $Z$ with vertex at general points gives $Z$. Therefore, the same is true for general cones.
\end{remark}

If $m$ is not sufficiently large, then  the containment $Z\subseteq Z_m$ could be strict.
\begin{example}
   Let $Z=X\cup Y_1$ where $X$ and $Y_1$ are the $(4,4)$-grid and the four aligned points as in the standard construction for $n=4$, see Theorem \ref{t. geproci infinite class}.
   
   A computer calculation shows that  $Z\subsetneq\bigcap_{P\in U}  C_{P,4}(Z)=F_4=X\cup Y_1 \cup Y_2$. 
   So, the set $Z$ precisely remembers $F_4$ at cones of degree 4, but it does not precisely remember itself. 
 \end{example}

The following proposition can be applied to all the sets constructed in Theorem \ref{t. geproci infinite class} and to the Klein configuration.
\begin{proposition}\label{p.memory geproci}
Let $Z$ be an $(a,b)$-geproci set, $a<b$, such that there is a quadric $\mathcal Q$ which intersects $Z$ in an $(a,a)$-grid. 
 Then, there is a set $W\subseteq Z$ of $\binom{a+1}{2}+a$ points which remembers $Z$ at cones of degree $a$.
\end{proposition}
\begin{proof}
Note that if the general projection of $W$ has $h$-vector  \[
(1,2,\ldots, a,a)
 \]
then we are done.
Indeed, there is only one cone,  $C_{P,a}(W)$, of degree $a$  with vertex at the general point $P$ containing $W$. Since the same is true for $Z$, then $C_{P,a}(W)$ contains all the points in $Z$. 
 This implies $Z\subseteq W_a$.

So, it is enough to construct a set $W$ whose projection has the required $h$-vector. Let $P_{ij}$ be the points on the $(a,a)$-grid, where $1\le i,j\le a$.
 Set $W'=\{P_{ij}\ |\ i+j\le a+2 \}$.
 Let $\pi_P\colon \PP^3\to \PP^2$ be the projection from a general point $P$.
 Then the $h$-vector of $\pi_P(W')$ is the following
 \[
(1,2,\ldots, a,a-1).
 \]
 So, in particular $\pi_P(W')$ is contained in the base locus of the pencil of two planar curves of degree $a$, with no fixed components (by Bertini's Theorem the general element in the pencil is irreducible).
 
 Then any curve of degree $a$ containing $\pi_P(W')$ will contain all the points of intersection of the two curves of degree $a$ (these are the projections of all the points in the grid).
 Hence, any cone of degree $a$ containing $W'$ contains the $a^2$ points of the grid.
 
 Now, let $Q$ be a point of $Z$ not in the grid and set $W=W'\cup \{Q\}$. Since, the projection is general, the point $\pi_P(Q)$ imposes one condition to the above pencil and this forces the $h$-vector of $\pi_P(W)$ to be $(1,2,\ldots, a,a)$, as required.\end{proof}

\begin{example}
      Let $Z$ be the set of $60$ points of the Klein configuration. 
   Then a computer calculation shows that the set $W$ of $\binom{7}{2}+6=27$ points as constructed in Proposition \ref{p.memory geproci} precisely remembers $Z$ at cones of degree $6$. This manifests in a certain sense a {\it very long memory} of $Z$ since we can reconstruct the whole set starting from less than half of its points.  

  We also note that, experimentally, if we impose the additional condition that the equation of $C_{P,6}(W)$ may not contain any monomial of the form $x_i^6$,
   then there are sets consisting of 26 points which remember $Z$. For instance, using the enumeration as in \cite{PSS}, such a set is  $\{P_{35},\ldots,P_{60}\}$.
\end{example}

\begin{question}
Let $Z$ be an $(a,b)$-geproci set of points.
What is the least number $c$ such that every subset of $c$ points of $Z$ remembers the whole configuration at cones of degree $a$?
\end{question}


\section{Graded Betti numbers of general projections}
\label{Betti}

All the possible graded Betti numbers (see section \ref{basic facts}) of complete intersections in $\PP^2$ can be obtained by general projections of grids, see \cite{CM}. The next result shows that all the graded Betti numbers of sets of points in $\PP^2$ can be obtained by projecting points in a quadric surface.

\begin{proposition} 
	Let $\beta_{X}$ be the Betti table of a set of reduced points $X\subseteq \PP^2$. Then there exists $Z\subseteq \PP^3$, a set of reduced points contained in a smooth quadric surface, such that 
	$\beta_{\pi(Z)}=\beta_{X},$ where $\pi\colon \PP^3\to \PP^2$ is a general projection. 
\end{proposition}
\begin{proof} 
Let $J\subseteq S=\field[x,y]$ be an artinian monomial ideal such that $\beta_{S/J}=\beta_X$; for the existence of such an ideal $J$ see for instance Chapter 3 in  \cite{eisenbud2005geometry}.

Let $T_1=x^{a_1}y^{b_1},\ldots, T_s=x^{a_s}y^{b_s}$ be the minimal generators of $J$, where
$a_1>\cdots > a_s=0$ and $0=b_1<\cdots < b_s$. 

Then, a set of points in $\PP^2$ with Betti table $\beta_{X}$ can be constructed by {\it lifting} the monomial ideal $J$, see  \cite[Proposition~2.6]{MN2000lifting}.
This procedure consists in considering two sets of linear forms in $\field[\PP^2]$, say $\mathcal L=\{\ell_1,\ldots, \ell_{a_1}\}$ and $\mathcal M=\{m_1,\ldots, m_{b_s}\}$ such that no three elements in $\mathcal L\cup \mathcal M$ define concurrent lines.
Then, construct the ideal  $\widetilde J$ generated by $\{  \ell_1\cdots \ell_{a_i}\cdot m_1\cdots m_{b_i} \ |\  i=1,\ldots,s \}$.
In fact, the ideal $\widetilde J$ defines a set of reduced points $V$ in $\PP^2$ and  $\beta_{S/J}=\beta_V.$  
The set $V$ is contained in the complete intersection defined by $(\ell_1\cdots\ell_{a_1}, m_1\cdots m_{b_s})$. Hence, we can consider the set of indexes $(i,j)$ such that the point intersection of the lines defined by $\ell_i$ and $m_j$ is in $V.$ Call such set~$U$.

Take any $(a_1,b_s)$-grid in an irreducible quadric of $\PP^3$, let $Z$ be the subset of the grid whose points are indexed by the elements in $U$, let $P$ be a general point in $\PP^3$ and let $\mathcal L=\{\ell_1,\ldots, \ell_{a_1}\}$ and $\mathcal M=\{m_1,\ldots, m_{b_s}\}$ be the projection from $P$ of the horizontal and vertical rulings of the grid. 
From what was noticed above, the projection from $P$ of $Z$ gives a set of reduced points in $\PP^2$, $\pi_P(Z)$, with Betti table equal to $\beta_X$.
\end{proof}

\begin{question}
	Which are the Betti tables of sets of reduced points in $\PP^n$ that can be obtained by a general projection of a set of reduced points in $\PP^{n+1}$, $n>2$? 
\end{question}
As a consequence of Theorem \ref{no 222} we will get that not all the Betti tables of points in $\mathbb P^3$ can be obtained by a general projection of a set of reduced points in $\PP^{4}$. 

\begin{question}
	Which are the Betti tables of sets of reduced points in $\PP^2$ that can be obtained by a general projection of a set of reduced points in $\PP^{3}$ not lying on a quadric surface? 
\end{question}
As shown in Chapter \ref{chap.Geography} there are no nontrivial $(3,b)$-geproci sets for $b> 4$ in $\PP^3$. Thus the Betti tables of the complete intersections of type $(3,b)$ cannot be obtained by a general projection of a set of reduced points in $\PP^{3}$ not lying on a quadric surface.

\section{Analyzing other properties \texorpdfstring{$\mathcal P$}{P}}

The scheme that we used to define geproci sets can be proposed in more generality for many other properties of finite sets in projective spaces.
\smallskip

Namely, fix a property $\mathcal P$ of finite sets of cardinality $d$ in $\PP^2$. Then it makes sense to ask:

\smallskip

\noindent\emph{which (nondegenerate) finite sets of cardinality $d$ in $\PP^3$ have a general plane projection for which property $\mathcal P$ holds?}

\smallskip

When $d=ab$ and $\mathcal P=$ being a complete intersection of type $(a,b)$, then we get back the notion of $(a,b)$-geproci, which is the main topic of our investigation. But several other properties can replace complete intersection, yielding the notion of \emph{gepro$\mathcal P$ sets}.

Of course the question is particularly meaningful when general sets of cardinality $d$ in $\PP^2$ do not share property $\mathcal P$.

Let us give some examples.
\begin{enumerate}
    \item Assume  $d\geq \binom{s+2}s$ and let $\mathcal P$ be the property of lying in some plane curve of degree $d$. The corresponding notion of gepro$\mathcal P$ sets is closely related to the notion of sets satisfying the unexpected cone property $C(d)$.
    \item If $\mathcal P$ is the property of having some special free resolution, then the corresponding gepro$\mathcal P$ sets are sets whose general projection has some special Betti numbers. Compare with the discussion in Section
    \ref{Betti}.
    \item If $\mathcal P$ is the property of having an ideal with socle degree $s$ which enjoys the Cayley-Bacharach property, then the corresponding gepro$\mathcal P$ sets are closely related with geproCB sets discussed above in Section
    \ref{Betti}. 
\end{enumerate}

Many other properties can be analyzed under this  point of view.
\smallskip

Furthermore, if a set $Z\subset \PP^3$ is not a gepro$\mathcal P$ set, then it makes sense to ask about the locus of points $P$ such that the projection of $Z$ from $P$ enjoys property $\mathcal P$. We obtain an analogue for the Weddle loci of finite sets.
\smallskip

All these generalizations provide a mine of future directions of investigations on the structure of finite sets in projective spaces.

\section{Geproci sets in higher dimension}

One can certainly extend the definition of geproci set of points to higher dimensional projective space.

We should observe, however, that apart from the obvious examples of degenerate geproci sets, we have no instances of finite sets in $\PP^n$, $n>3$, with the geproci property.
Furthermore, in the next subsection we show that 8 points in $\PP^4$ are never a nondegenerate geproci set.

\subsection{Results on geproci finite sets in $\PP^4$}

 Not even grids can be  easily generalized to higher dimensions: e.g., we have no examples of nondegenerate sets of $abc$ points in $\PP^4$ that correspond to the locus where three surfaces, of degrees $a,b,c$, meet.

In the first nontrivial numerical case, we have indeed a nonexistence result.

\begin{theorem} \label{no 222}
There are no $(2,2,2)$-geproci sets in $\PP^4$.
\end{theorem}
\begin{proof}
Assume by contradiction that there exists a set $Z$ of $8$ points in $\PP^4$ which is a $(2,2,2)$-geproci set.
\vskip\baselineskip

{\bf Case 1}. Assume that $Z$ is in LGP. Then the $h$-vector of $Z$ is $(1, 4, 3)$ by \cite{maroscia}. In particular $Z$ is separated by quadrics.
Call $\phi$ the rational map induced by the system of quadrics through $Z$, which maps $\PP^4$ to $\PP^6$. If the image of $\phi$ has dimension $4$, we get a contradiction. Indeed, hyperplanes in $\PP^6$ correspond to quadrics in $\PP^4$ containing $Z$, and quadric cones in $\PP^4$ with a general vertex $P$ correspond to hyperplanes in $\PP^6$ containing the tangent space to $\phi(\PP^4)$  at $\phi(P)$. 
Since the tangent space at a general point $\phi(P)$ of the image is $4$-dimensional, this means that we can have only a pencil of quadrics through $Z$ which are singular (cones) at $P$, contradicting the geprociness.

Thus, for $P$ general, the system of quadrics through $Z\cup\{P\}$ has a base curve $C_P$. Pick two general points $P_1,P_2$ of $C_P$ and consider the set $Z'=Z\cup\{P,P_1,P_2\}$. $Z'$ is a set of $11$ points, with a subset $Z\cup\{P\}$ of $9$ points in LGP, and $Z'$ imposes only $9$ conditions to the quadrics of $\PP^4$. By Lemma 5.4 of \cite{ACV}, the set $Z'$ lies in a rational normal curve. This is impossible because $Z$ lies in at most one  rational normal curve and $P$ is general.
\vskip\baselineskip

\textbf{Case 2.} We assume now that $Z$ is not in LGP.

{\bf Subcase 2.1}. If $Z$ contains three aligned points, then certainly its general projection cannot be generated by quadrics.

{\bf Subcase 2.2}.
If $Z$ contains a subset $W$ of at least five points in a plane, then the same holds for a general projection of $Z$. Since $W$ is in LGP then it is not contained in a pencil of conics, which excludes that a general projection of $Z$ is a complete intersection of quadrics.  

If $Z$ meets a plane $\Pi$ in a set $W$ of exactly $4$ points in LGP, then there is only a pencil of conics in $\Pi$ passing through $W$. Thus if $Z$ is a geproci set, then the general projection of $W'=Z-W$ is also contained in a plane, which means that $W'$ itself sits in a plane $\Pi'$.
In the  projection $\pi$ from a general $P$, $\pi(\Pi)$ and $\pi(\Pi')$ meet in a line $L$, image of the (unique) plane $\Pi_P$ through $P$ which meets both $\Pi,\Pi'$ in lines $\ell,\ell'$ resp. The union $\pi(W)\cup \pi(W')$ is a complete intersection only if the pencil of conics of $\pi(\Pi)$ through $\pi(W)$ and the pencil of conics of $\pi(\Pi')$ through $\pi(W')$ induce on $L$ the same linear series $g^1_2$. Going back to $\PP^4$, the point $P$ induces an isomorphism between the lines $\ell,\ell'$. (The isomorphism takes a point $Q\in \ell$ to the point of $\ell'$ collinear with $Q$ and $P$.)
If we move $P$ in $\Pi_P$, the isomorphism changes. Thus for a general choice of $P'\in\Pi_P$ it is impossible that the pencils of conics through $W,W'$ resp. induce on $\ell,\ell'$ resp. two series $g^1_2$ which are identified by the isomorphism induced by $P'$. This implies that a general projection of $Z$ is not a geproci set.

{\bf Subcase 2.3}. It remains to exclude the case in which $Z$ contains a set of $5$ or more points in LGP in a hyperplane $H$ of $\PP^4$. 

Assume $H\cap Z=W$ is a set of $7$ points. Since a general projection of $Z$ to $H$ is a complete intersection of type $(2,2,2)$, by the CBP the image of the last point of $Z$ is fixed, which is a contradiction.


Assume that $H\cap Z=W$ contains exactly six points. Call $W'=Z-W$. The general projection from a point in $H$ of six points in LGP in $\PP^3$ is not contained in a conic (as otherwise it would be a $(2,3)$-geproci and by Proposition \ref{SmallProp} we would be back in Subcase 2.1).
Then, the quadrics containing the projection (from a general point of $H$) of $Z$,  must contain the plane spanned by the projection of $W$, and also the line containing the projection of $W'$. This is not possible by semicontinuity.

Assume that $H\cap Z=W$ contains exactly five points. Call $W'=Z-W$. 
If we project $Z$ from a general point of $H$, then $H$ is contracted to a plane $\Pi$ and there is at most a unique conic of $\Pi$ through the projection of $W$. Since the projection of $Z$, by semicontinuity, sits in three independent quadrics, this means that there is a pencil of planes passing through the projection of $W'$. 

Thus the projection of $W'$ is aligned. Since this is true for a general choice of $P$ in $H$, we get that $W'$ itself is aligned. This is excluded by Case 2.1 if $W'$ has three points.
\end{proof}

Consequently, we can exclude the existence in $\PP^4$ of the following analogue of grids. 

\begin{corollary}
There are no nondegenerate sets $Z$ of $abc$ points in $\PP^4$ that are the intersection of surfaces of degrees $a,b,c$ splitting in unions of planes.
\end{corollary}
\begin{proof}
The set $Z$ would contain a $(2,2,2)$-geproci set, whose existence is excluded by Theorem \ref{no 222}. 
\end{proof}

\begin{remark} \label{proj8fromP4}
In view of the results in Chapter \ref{Chap.Weddle} about $d$-Weddle loci, an interesting follow-up to Theorem \ref{no 222} is to ask what is the locus of points in $\PP^4$ from which the projection of a general set of 8 points is a complete intersection of quadrics in $\PP^3$. 

So let $Z$ be a general set of 8 points in $\PP^4$. 
Using the computer with random points we found that the 2-Weddle locus has the expected codimension 3. More precisely, there is a curve from the points of which $Z$ projects to a set of points lying on a net of quadrics in $\PP^3$.

Applying the methods used in Chapter \ref{Chap.Weddle}, especially the exact sequence \eqref{OrigMacDualSeqPrime}  and Lemma \ref{J1 lemma}, we see that the 2-Weddle locus is a curve of degree 35. However, it is clear that the 28 lines joining two points of $Z$ are all components of this locus. What remains is a curve of degree 7. 
We have confirmed by computer calculation that this curve is smooth of genus 3 and is arithmetically Cohen-Macaulay, lying on a surface of degree 3. But we do not see a geometric reason for this fact, or see how to visualize this curve from a geometric point of view.
\end{remark}


\subsection{Geproci curves in $\PP^4$}\index{geproci! curve}

In view of the fact that we know only degenerate examples of finite geproci sets in $\PP^4$, one can extend the definition to sets of positive dimension.

\begin{definition}
 A reduced variety $X\subset\PP^n$ is a {\it geproci} variety\index{geproci! variety} if a general projection of $X$ to $\PP^{n-1}$ is a complete intersection. 
\end{definition}

Of course, any variety of codimension 2 is a geproci variety, so the problem of constructing geproci varieties becomes interesting in codimension  $\geq 3$.

We show that there are at least two constructions that provide many examples of geproci curves in $\PP^4$.
\medskip

Let $Z$ be a nondegenerate geproci set of points in  a hyperplane $H_1=\PP^3$ of $\PP^4$. 
If $Q$ is a general point in $\PP^3$, it is clear that the cone $C_0$ over $Z$ with vertex $Q$ is a complete intersection curve in $\PP^3$. 
Now let $Q'$ be any point in $\PP^4$, not lying in the $\PP^3$ containing $Z$, and let $C$ be the cone over $Z$ with vertex $Q'$. Note that $C$ is nondegenerate in $\PP^4$. 
Let $P$ be a general point of $\PP^4$. The projection of $C$ from $P$ to $H_1$ is a cone over $Z$, thus it is a complete intersection. 

We claim that even the projection of $C$ from $P$ to a general hyperplane $H_2$ is a complete intersection. Indeed, the projection from $P$ gives an isomorphism of $H_1$ and $H_2$ so this is immediate. Thus we have shown:

\begin{proposition} \label{p.cones over geproci} 
Any nondegenerate cone in $\PP^4$ over a geproci set of points in $\PP^3$ is a geproci curve.
\end{proposition}

However, we can do more. As before, let $Z$ be a geproci set of points in $H_1 = \PP^3$;  we will view $H_1$ as a linear space of dimension 3 in $\PP^n$ ($n \geq 4$). Now let $\Lambda$ be a linear space of dimension $n-4$ disjoint from $H_1$. The cone $V$ over $Z$, with vertex $\Lambda$, is a union of $(n-3)$-dimensional linear spaces in $\PP^n$. 

We claim that the general projection of $V$ to $\PP^{n-1}$ is again a complete intersection. Indeed, note that the projection, $\Lambda'$, of $\Lambda$ to $\PP^{n-1}$ again has dimension $n-4$, so it has codimension 3 in $\PP^{n-1}$. Now a projection $W$ of $V$ from a general point $P$ is the union in $\PP^{n-1}$ of linear spaces of dimension $n-3$ which form a cone with vertex $\Lambda'$. $W$ has codimension 2, so the intersection of $W$ with $H_1$ is a union of lines that form a one-dimensional cone over $Z$. As already noted, such a cone is a complete intersection, so $W$ is also a complete intersection since after a finite number of successive hyperplane sections we obtain a complete intersection curve. So we have shown:

\begin{proposition}
Any nondegenerate cone in $\PP^n$ over a geproci set of points in a linear subspace  $H=\PP^3$, whose vertex is a linear space of codimension 4 in $\PP^n$ that is disjoint from $H$, is a geproci variety of dimension $n-3$.
\end{proposition}

\begin{remark}
If $C$ is a nondegenerate geproci curve in $\PP^4$, then $C$ cannot be smooth. Indeed, if $C$ were smooth then a general projection of $C$ to $\PP^3$ is a smooth nonlinearly normal curve, while a complete intersection curve is ACM, hence linearly normal.

Indeed, one can say more on the singularities of geproci curves $C$ in $\PP^4$. Namely, even if $C$ is singular, the linear series which embeds $C$ in $\PP^3$ must be complete (since the image is a complete intersection), but it cannot be such if the projection is an isomorphism. Thus, the general projection of $C$ to $\PP^3$ cannot be an isomorphism. On the other hand, it is well known that a general projection of a curve $C\in \PP^4$ to a hyperplane is isomorphic, outside the singular locus. It follows immediately that the projection is not isomorphic at some singular point $P\in C$. This is possible only if the \emph{embedding dimension} of $C$ at $P$, i.e. the  dimension of the span of the tangent cone at $C$, is $4$.

We get then that curves with ordinary double or triple points cannot be geproci curves in $\PP^4$.

Compare with Proposition \ref{p.cones over geproci}: when $C$ is a nondegenerate cone, then the embedding dimension of $C$ at the vertex is $4$.
\end{remark}

\begin{proposition}\label{r. geproci hyp}
Let $C$ be an $(a,b)$-geproci curve in $\PP^4$. Then a general hyperplane section of $C$ is an $(a,b)$-geproci set in $\PP^3$.
\end{proposition}
\begin{proof}
Let $C$ be an $(a,b)$-geproci curve in $\PP^4$,  let $H$ be a general hyperplane in $\PP^4$ and consider the general hyperplane section $C \cap H$ of $C$. Let $H'$ be the target hyperplane for our projections, and notice that $H \cap H'$ is a plane in $H'$ and also a plane in $H$. Let $P$ be a general point of $H$. Since $H$ was itself general with respect to $C$, we can view $P$ as being a general point of $\PP^4$ with respect to $C$. 

Ignoring $H$ for a moment, the projection $\pi(C)$ is a complete intersection in $H'$. Since $P \in H$, the projection from $P$ sends any point of $H$ to the plane $H \cap H'$, so in particular it sends the points $C \cap H$ into this plane. Thus $\pi(C \cap H)$ is a hyperplane section  of $\pi(C)$ in $H'$. Since the latter curve is a complete intersection in $H'$, also $\pi(C \cap H)$ is a complete intersection in $H \cap H'$. But now looking inside $H = \PP^3$, we have that the projection  from the general point $P$ of the finite set $C \cap H$ to the plane $H \cap H'$ is   a complete intersection, hence $C \cap H$ is geproci in $H$.
\end{proof}

As a consequence of the previous result we get the following.
\begin{proposition}
There are no irreducible nondegenerate $(3,b)$-geproci curves in $\PP^4$, for all $b\geq 3$.
\end{proposition}
\begin{proof}
We argue by contradiction. Let $C$ be a nondegenerate irreducible $(3,b)$-geproci curve in $\PP^4$.  A general hyperplane section $Z$ of $C$ is  a $(3,b)$-geproci set, by Proposition \ref{r. geproci hyp}. Since $Z$ is nondegenerate, then by Theorem \ref{thm:classification_of_3xb}, $Z$ is either a grid or a $D_4$ configuration. Then $Z$ contains three collinear points. This contradicts the well known fact that a general hyperplane section of an irreducible curve is a set of points in uniform position, see \cite{montreal}.
\end{proof}

\begin{question}
Are there nondegenerate $(a,b)$-geproci curves in $\PP^4$ which are not cones?
\end{question}

\section{Open problems}


Let $\calP$ be a property possesed by (some) finite point sets $\overline{Z}\subset\PP^{n-1}$.
Recall we say a finite set $Z\subset\PP^n$ is gepro-$\calP$ if $\overline Z\subset H\cong \PP^{n-1}$
has property $\calP$
(where $\overline Z$ is the image of $Z$ under projection $\PP^n\dashrightarrow H$ 
from a general point $P$ to a hyperplane $H\subset\PP^n$).

\begin{prob*}
For what properties $P$ are there interesting gepro-$\calP$ sets?
In such cases, classify the gepro-$\calP$ sets.
\end{prob*}

For example, say $\calP$ means ``$\overline Z$ is arithmetically Gorenstein." Then a set $Z$ of $n+1$  points in $\PP^n$ in linearly general position is gepro-$\calP$
since the image $\overline Z$ is a set of $n+1$  points in $H$ in linearly general position, which is arithmetically Gorenstein.

\begin{prob*}
What other gepro-Gorenstein sets are there?
\end{prob*}

The image $\overline Z$ of a geproci set $Z\subset\PP^3$ under a general projection, being a complete intersection in $\PP^2$,
has the Cayley-Bacharach property. 

\begin{prob*}
Study gepro-$\calP$ where $\calP$ is the Cayley-Bacharach property. 
\end{prob*}

Suppose $P$ is the property that $\overline Z$ is a complete intersection (ci).
Then gepro-$\calP$ is geproci.


\begin{prob*}
Classify nondegenerate geproci sets in $\PP^3$.
\end{prob*}

\begin{prob*}
Let $X$ be a set of points in a smooth quadric. 
If $X$ is geproci, must $X$ be a grid? (See Question \ref{q. quadric->grid}.)
\end{prob*}


\begin{prob*}
Are there other examples like the Klein configuration and the half Penrose configuration? That means, are there standard 
geproci sets $Z$, different from $F_4$, for which the set $\widetilde{Z}$ is an 
$(a,a)$-geproci set which can be further extended to a larger $(a,b)$-geproci set? 
I.e., is there a special element in the pencil of cones of degree $a$ through 
$\widetilde{Z}$ with vertex at a general point $P$ that contains a larger 
configuration of points that is geproci?  (See Question \ref{q. Klein-like} and Section \ref{s.Penrose}.)
\end{prob*}

We say a finite set $Z\subset\PP^n$ satisfies $C(d)$ if the dimension of the space of
cones of degree $d$ containing $Z$ with vertex a general point is unexpected.
Unexpectedness is closely related to the Weak Lefschetz Property (WLP).
Every $(a,b)$-grid with $3\leq a\leq b$ satisfies both $C(a)$ and $C(b)$.

\begin{prob*}
Does every nontrivial $(a,b)$-geproci $Z\subset\PP^3$ satisfy both $C(a)$ and $C(b)$? (See Section \ref{sec:geproci and unexpected}.)
\end{prob*}

\begin{prob*}
Which $t$ and finite sets $Z$ in $\PP^3$ satisfy $\dim [I(Z)]_t = \dim [I(\overline{Z})]_t$, where $\overline{Z}$ is the projection of $Z$ from a
general point to a plane but viewed in $\mathbb P^3$ rather than $\mathbb P^2$?  (See Question \ref{quest. geprociness and strictness}.)
\end{prob*}

\begin{prob*}
Can one formulate a property $\calP$ based on unexpectedness (and hence related to the WLP)
such that there are interesting gepro-$\calP$ point sets in $\PP^n$?
\end{prob*}

We know only a few examples of nontrivial geproci non-half grids:
\begin{enumerate}
\item[$\bullet$] A 60 point set coming from the $H_4$ root system.
\item[$\bullet$] A 40 point set originally constructed by Penrose, who applied it to quantum mechanics.
\item[$\bullet$] A 120 point set.
\end{enumerate}

\begin{prob*}
Are there other nontrivial geproci non-half grids?
Are there only finitely many nontrivial geproci non-half grids?
(See Section \ref{seq: Q and R}.)
\end{prob*}

Every nontrivial geproci set in $\PP^3$ that we know of has multiple subsets
of at least 3 collinear points.
The only nondegenerate linearly general geproci set we know of is the set of four general points in $\PP^3$,
but it is a $(2,2)$-grid.

\begin{prob*}
Do all nontrivial geproci sets contain a subset of 3 collinear points?
If the answer is no, can we classify those which do not?
In fact, does every nontrivial geproci 
set contain a $(3,3)$-grid? 

In the other direction, can a nontrivial geproci set be linearly general? (See Section \ref{seq: Q and R} and Question~\ref{q. LGP->4}.)
\end{prob*}

\begin{prob*}
The 40 point Penrose set is Gorenstein. 
Are there other finite Gorenstein geproci sets?
\end{prob*}

The following questions relate to projective and combinatorial equivalence. 
\begin{prob*}\label{q. are unique?}
If two nontrivial geproci sets are combinatorially equivalent, must they be
projectively equivalent? Geproci sets seem to have additional collinearities compared to grids
which might make the answer be yes.
\end{prob*}

\begin{prob*}
Are there pairs $(a,b)$ for which
there are infinitely many  nonprojectively equivalent nontrivial $(a,b)$-geproci sets?
\end{prob*}


\begin{prob*}
Does Proposition \ref{CombEqToGridIsGrid} extend to nontrivial geproci sets? If a set $Z$ has the same combinatorics
as a nontrivial $(a,b)$-geproci set, is $Z$ also geproci (or even an $(a,b)$-geproci set)?
\end{prob*}

There are, up to projective equivalence,  uncountably many grids.

\begin{prob*}
Up to projective equivalence, is there any $(a,b)$ with infinitely many nontrivial $(a,b)$-geproci sets?
\end{prob*}

Terao's Conjecture concerns whether a certain property of hyperplane arrangements is
a combinatorial property. A geproci set also has combinatorics (e.g., its collinear subsets).

\begin{prob*}
If a set has the same combinatorics as a geproci set, is it geproci?
I.e., if $Z$ and $Z'$ are combinatorially equivalent and $Z$ is a 
nondegenerate geproci set, must $Z'$ also be a nondegenerate geproci set? (See 
Problem \ref{q. are unique?} and Question \ref{quest. ce geproci}.)
\end{prob*}

\begin{prob*}
If two geproci sets have the same combinatorics, are they projectively equivalent?
I.e., say $Z$ is a nontrivial $(a,b)$-geproci set and $Z'$ is weakly combinatorially equivalent to $Z$. 
Is $Z'$ an $(a,b)$-geproci set? (See Question \ref{q. weak comb equiv}.)
\end{prob*}

\begin{prob*}
If two nontrivial geproci sets are weakly combinatorially equivalent, 
must they be projectively equivalent? (See Question \ref{q.ce->pe}.)
\end{prob*}

\begin{prob*}
Let $a>5$ and consider $(5,a)$-geproci subsets of a standard $(a,a+1)$-geproci set $Z$ 
which are not projectively equivalent.  Are they also not weakly combinatorially equivalent? 
(See Remark \ref{r. three (5,6)-geproci} and Question \ref{quest. (5,a)-geproci}.)
\end{prob*}

\begin{prob*}
For each $a> 3$, are there infinitely many $b$ such that there is
a nontrivial $(a,b)$-geproci set $Z$ not contained in $a$ skew lines? 
(The sets constructed in Theorem \ref{t. (a,b)-geproci}
are half grids, contained in $a$ skew lines.)
If the answer is negative, can one find a bound $B$ on $b$ (possibly depending on $a$)
such that a nontrivial $(a,b)$-geproci set $Z$ not contained in $a$ skew lines
must have $b\leq B$?  (See Question \ref{unexpQuest2}.)
\end{prob*}

Classification of geproci sets is far from being understood. We wonder whether the following problem has a negative answer.
\begin{prob*}\label{quest positive moduli}
   Do there exist nontrivial examples of geproci sets with positive dimensional moduli?
\end{prob*}

Another challenging question along these lines is the following.

\begin{prob*}
Can a nongrid geproci set degenerate to a grid through geproci sets? 
\end{prob*}

Moving beyond the setting of points in $\PP^3$, we first ask:

\begin{prob*} 
Are there any geproci sets of points in $\PP^n$ for $n>3$? 
\end{prob*}

We can define a geproci variety as any variety (of any codimension $\geq 3$) whose general projection is a complete intersection.
A cone with a general vertex over a finite geproci set is a geproci curve, the cone over that is a geproci surface, etc.
These geproci varieties all have codimension 3.

\begin{prob*}
Are there other kinds of geproci varieties? Are there any with codimension greater than 3?
(See Section \ref{seq: Q and R}.)
\end{prob*}

\backmatter

\newpage

\printindex

\end{document}